\documentclass[11pt,a4paper,oneside]{amsart}

\usepackage[utf8]{inputenc}
\usepackage[T1]{fontenc}
\usepackage{lmodern}
\usepackage[protrusion=false]{microtype}

\usepackage{amsmath,amssymb,amsxtra,mathtools} 
\usepackage[lite]{amsrefs}
\usepackage{graphicx}
\usepackage{mathabx}
\usepackage{tensor}
\usepackage{a4wide}

\usepackage{booktabs} 
\usepackage{multirow} 

\usepackage{tikz,tikz-cd}

\usepackage{bbm}
\usepackage[all]{xy}

\renewcommand{\PrintDOI}[1]{\href{http://dx.doi.org/\detokenize{#1}}{doi: \detokenize{#1}}}

\BibSpec{book}{%
	+{}  {\PrintPrimary}                {transition}
	+{,} { \textit}                     {title}
	+{.} { }                            {part}
	+{:} { \textit}                     {subtitle}
	+{,} { \PrintEdition}               {edition}
	+{}  { \PrintEditorsB}              {editor}
	+{,} { \PrintTranslatorsC}          {translator}
	+{,} { \PrintContributions}         {contribution}
	+{,} { }                            {series}
	+{,} { \voltext}                    {volume}
	+{,} { }                            {publisher}
	+{,} { }                            {organization}
	+{,} { }                            {address}
	+{,} { \PrintDateB}                 {date}
	+{,} { }                            {status}
	+{}  { \parenthesize}               {language}
	+{}  { \PrintTranslation}           {translation}
	+{;} { \PrintReprint}               {reprint}
	+{.} { }                            {note}
	+{.} {}                             {transition}
	+{} { \PrintDOI}                   {doi}
	+{} { available at \url}            {eprint}
	+{}  {\SentenceSpace \PrintReviews} {review}
}

\usepackage{enumitem} 
\setlist[enumerate,1]{label=\textup{(\arabic*)}}

\numberwithin{equation}{section}
\theoremstyle{plain}
\newtheorem{theorem}[equation]{Theorem}
\newtheorem{lemma}[equation]{Lemma}
\newtheorem{proposition}[equation]{Proposition}
\newtheorem{corollary}[equation]{Corollary}

\newtheorem{thmx}{Theorem}

\newtheorem{corx}[thmx]{Corollary}

\theoremstyle{definition}
\newtheorem{definition}[equation]{Definition}
\theoremstyle{remark}
\newtheorem{remark}[equation]{Remark}
\newtheorem{convention}[equation]{Convention}
\newtheorem{example}[equation]{Example}
\newtheorem{notation}[equation]{Notation}
\newtheorem{question}[equation]{Problem}

\DeclareMathOperator{\Aut}{Aut}
\DeclareMathOperator{\Bis}{Bis}
\DeclareMathOperator{\Iso}{Iso}

\DeclareMathOperator{\supp}{supp}%
\DeclareMathOperator{\clsp}{\overline{span}}
\DeclareMathOperator{\linspan}{span} 
\DeclareMathOperator{\sing}{sing} 
\DeclareMathOperator{\reg}{reg} 
\DeclareMathOperator{\src}{src} 
\DeclareMathOperator{\Hau}{H}
\DeclareMathOperator{\PHomeo}{PHomeo}
\DeclareMathOperator{\MPIso}{MPIso}

\DeclareMathOperator{\PIso}{PIso}

\DeclareMathOperator{\PAut}{PAut}
\DeclareMathOperator{\tight}{tight}

\newcommand{\OO}{\mathcal O}
\newcommand{\TT}{\mathcal T}
\newcommand{\LL}{\mathcal L}
\newcommand{\EE}{\mathcal{E}}
\newcommand{\RR}{\mathcal{R}}

\newcommand{\Gg}{\mathcal G}

\newcommand{\tv}{\widetilde{V}}
\newcommand{\vv}{V}
\newcommand{\tw}{\widetilde{W}}
\newcommand{\ww}{W}
\newcommand{\ttt}{\widetilde{T}}
\renewcommand{\tt}{T}

\newcommand{\rg}{\textbf{\textup{r}}}
\newcommand{\sr}{\textbf{\textup{s}}}

\newcommand{\tc}{\mathrm{t}}

\newcommand{\CC}{\mathcal C}
\newcommand{\DD}{\mathcal D}

\newcommand{\bone}{\mathbbm{1}}

\newcommand*{\Z}{\mathbb Z}
\newcommand*{\R}{\mathbb R}
\newcommand*{\N}{\mathbb N}
\newcommand*{\NN}{\mathcal N}

\newcommand*{\T}{\mathbb T}

\newcommand*{\Gr}{G}

\newcommand{\la}{\rhd}
\newcommand{\ra}{\lhd}

\newcommand*{\Bound}{\mathbb B}
\newcommand*{\Comp}{\mathbb K}
\newcommand*{\red}{\mathrm{red}}
\newcommand*{\ess}{\mathrm{ess}}
\newcommand*{\rd}{\mathrm{r}}
\newcommand*{\es}{\mathrm{e}}
\newcommand*{\M}{\mathfrak M}

\newcommand*{\Cst}{\textup C^*}

\newcommand*{\Cont}{\textup C}
\newcommand*{\Contb}{\textup C_\textup b}
\newcommand*{\Contu}{\textup C_\textup u}
\newcommand*{\Contc}{\Cont_\textup c} 

\renewcommand*{\tight}{\mathrm{tight}}

\newcommand*{\B}{\mathcal B}
\newcommand*{\defeq}{ \coloneqq }

\renewcommand*{\emptyset}{\varnothing}
\newcommand*{\ox}{\otimes}
\newcommand*{\fibre}[2]{\tensor*[^{}_{#1}]{\times}{^{}_{#2}}}
\newcommand*{\ol}[1]{\overline{#1}}

\DeclarePairedDelimiterX{\braket}[2]{\langle}{\rangle}{#1\,\delimsize\vert\,\mathopen{}#2}
\DeclarePairedDelimiterX{\BRAKET}[2]{\langle}{\rangle}{\!\delimsize\langle#1\,\delimsize\vert\,\mathopen{}#2\delimsize\rangle\!}
\DeclarePairedDelimiterX{\setgiven}[2]{\{}{\}}{#1\,{:}\,\mathopen{}#2}

\newcommand*{\into}{\rightarrowtail}
\newcommand*{\onto}{\twoheadrightarrow}
\newcommand*{\Min}{\textup{(Min)}}
\newcommand*{\Fin}{\textup{(Fin)}}
\newcommand*{\Evr}{\textup{(Evr)}}
\newcommand*{\Rec}{\textup{(Rec)}}

\newcommand*{\Cyc}{\textup{(Cyc)}}
\newcommand*{\Con}{\textup{(Con)}}
\newcommand*{\Sla}{\textup{(Slack)}}
\newcommand*{\Hum}{\textup{(Hum)}}

\newcommand\Vspace[1][.4em]{\textup{\mbox{\vrule height.4ex}%
		\vbox{\hrule width#1}%
		\hbox{\vrule height.4ex}}}
\newcommand{\Space}{\Vspace[0.5em]}

\usepackage[unicode=true,
colorlinks=true,
linkcolor=blue!75!black,
citecolor=blue!75!black,
urlcolor=purple,
breaklinks=true]{hyperref}

\begin{document}
	\thispagestyle{empty}
	\title[Twisted self-similar groupoid actions]{
		Twisted operator algebras of 	self-similar
		\\
	 groupoid actions on arbitrary graphs}

\author[Kwa\'sniewski]{Bartosz K. Kwa\'sniewski}  
\email{b.kwasniewski@uwb.edu.pl}

\author[Mundey]{Alexander Mundey} 
\email{alex.mundey@adelaide.edu.au}




\thanks{We thank Benjamin Steinberg for telling us the trick that produces closed but non-Hausdorff inverse semigroups, and to Josiah Aakre for drawing our attention to \cite{Exel-Pardo:Self-similar}*{Theorem 18.8(2)(b)} where a special case of condition $\Evr$ appears. Bartosz Kwa\'sniewski was supported by the
National Science Centre, Poland, through the WEAVE-UNISONO grant 
no.~2023/05/Y/ST1/00046. 
Alexander Mundey was supported by University of Wollongong AEGiS CONNECT grant 141765.
} 

\begin{abstract} 
We study self-similar groupoid actions on arbitrary directed graphs together with $\T$-valued twists that exhaust the second cohomology group of the associated Zappa-Sz\'ep product category. We define and analyse the associated universal, reduced, and essential $C^*$-algebras, along with their Toeplitz versions and core subalgebras. In fact, we develop our theory in the more general setting of $L^P$-operator algebras, where $P\subseteq [1,\infty]$ is any non-empty set of parameters. This includes $C^*$-algebras, $L^p$-operator algebras  and symmetrised $L^{p,*}$-operator algebras for  $p\in [1,\infty]$, as special cases. 

We use three complementary approaches: twisted inverse semigroups, twisted ample groupoids, and $C^*$-correspondences.
We provide, in terms of the self-similar action, general characterisations of topological freeness, minimality, Hausdorffness, finite non-Hausdorffness, effectiveness, and local contractiveness for the associated ample groupoids. We generalise the classical Cuntz--Krieger Uniqueness and Coburn--Toeplitz Uniqueness Theorems for graph $C^*$-algebras to twisted $L^P$-operator algebras of self-similar groupoid actions.
We characterise when the natural inclusions are Cartan, 
give checkable criteria for simplicity and pure infiniteness of the essential algebras, and discuss when the universal and reduced algebras coincide. 

We also provide conditions that ensure the singular ideals vanish. Using the groupoid model we show that for any $P\subseteq [1,\infty]$, the  $L^P$-operator algebra of a  contracting self-similar action is simple if and only if the corresponding  Steinberg algebra is simple. Using the Toeplitz-Pimsner model, we prove that for the universal groupoid of any self-similar groupoid action on a row-finite graph, the singular $C^*$-algebraic ideal always vanishes.
\end{abstract}

\maketitle

\setcounter{tocdepth}{1}
\tableofcontents
\section*{Introduction}

Groups acting self-similarly on sets, also known as self-similar groups, emerged in the late 1970s and early 1980s as discrete groups generated by automata.
This approach proved to be an effective method for constructing finitely generated groups with specific properties,
leading to the resolution of several open problems in geometric group theory that were not approachable by other methods,
see for instance \cite{Nekrashevych:Self-similar}. In \cite{Nekrashevych:Cuntz--Pimsner} and \cite{Nekrashevych:Cstar_selfsimilar},  Nekrashyevich developed a theory  of $C^*$-algebras that serve as ``noncommutative spaces'' that underlie the given self-similar action $(\Gamma, X)$.
In this approach, the set $X$ is represented by a Cuntz algebra, 
the group $\Gamma$  by a unitary representation, and the  whole system is represented by a Pimsner algebra \cite{Pimsner:Generalizing_Cuntz--Krieger}, which also has an ample groupoid model.
Exel and Pardo \cite{Exel-Pardo:Self-similar} generalised this theory to $C^*$-algebras of self-similar actions of groups on finite directed graphs, where the directed graph is represented by the graph $C^*$-algebra.
Since then, this field of research has taken off, with numerous related papers appearing regularly, see \cite{Enrique's_notes} for a recent summary.

For self-similar actions on graphs, it is more natural to consider actions of groupoids rather than groups.
A self-similar action of a groupoid $\Gr$ on a graph $E$ can be defined in several, essentially equivalent, ways.
We distinguish the following four different approaches that we discuss in detail in Section~\ref{Sec:Self-similar groupoid actions}: 
\begin{enumerate}
\item\label{enu:viewpoint1} an action of $G$ on $E$ equipped with a $1$-cocycle, cf. \cites{Exel-Pardo:Self-similar, Miller_Steinberg};
\item\label{enu:viewpoint2} a  groupoid homomorphism from $G$ to the groupoid of isomorphisms between trees in the forest $T_E$ generated by $E$
, cf. \cites{Laca-Raeburn-Ramagge-Whittaker:Self-similar-groupoids, Deaconu, BBGHSW24, Aakre};
\item\label{enu:viewpoint3} a matched pair of categories $G$ and $E^*$, 
cf. \cites{Yusnitha:Self-similar_actions, Mundey_Sims}; and
\item\label{enu:viewpoint4} a discrete groupoid correspondence, \cites{Anutnes_Ko_Meyer, Miller_Steinberg}.
\end{enumerate}

Each of these points of view has its advantages. Viewpoint \ref{enu:viewpoint2} is closest to the original
combinatorial description of self-similar groups. It provides a very efficient method for constructing groups (or groupoids) with various properties using automata. 

Viewpoint \ref{enu:viewpoint3} 
provides a perspective that puts both the groupoid and path space actions on equal footing by treating them as actions of categories.
This allows one to study self-similar actions from the point of view of their Zappa--Sz\'ep product categories, see \cites{Ortega_Pardo:tight_groupoid_of_inverse_semigroup, Ortega_Pardo:Zappa--Szep_products}, and thus immerse them in the realm explored by Spielberg \cites{Spielberg1, Spielberg2}. Moreover, as shown by the first named author and Sims \cite{Mundey_Sims}, general matched pairs of categories form
the right framework to construct and analyse the relevant (co)homology theory. The twist that we consider comes from \cite{Mundey_Sims}.

There are also many reasons to adopt viewpoint \ref{enu:viewpoint4}. For one, it explains the asymmetry built into self-similar actions, which arises from the difference between groupoids and graphs. More importantly, as exploited by Miller and Steinberg in \cite{Miller_Steinberg}, the associated constructions depend on
the groupoid correspondence rather than the groupoid action itself. Moreover, the dependence on the groupoid correspondence is functorial, which is 
an inherent feature of such objects, see \cite{Anutnes_Ko_Meyer}.

In this paper, however, our starting point will always be viewpoint \ref{enu:viewpoint1}. It has the advantage of defining objects using the smallest number of generators and relations, in the spirit of the good old theory of graph $C^*$-algebras \cites{Cuntz--Krieger, Bates_Pask_Raeburn_Szymanski, FLR, Drinen_Tomforde, Raeburn:Graph_algebras}.

This paper makes the following novel contributions to the field:

\begin{enumerate}[label={(\alph*)}]
\item\label{enu:newpoint1} We consider \emph{arbitrary self-similar groupoid actions} $(\Gr,E)$ \emph{on arbitrary graphs}. 
We do not assume row-finiteness  of the graph, or any condition on the action that implies that the associated groupoids are Hausdorff.
\item\label{enu:newpoint2} We
study \emph{twisted algebras of self-similar actions}. In particular, all the related objects such as inverse semigroups and transformation groupoids need to be twisted.  

\item\label{enu:newpoint3} We introduce and analyse \emph{a number of different algebras} associated to $(\Gr,E)$. This includes reduced and essential algebras, their Toeplitz counterparts, and also the associated core subalgebras. 
\item\label{enu:newpoint4} We associate \emph{$L^P$-operator algebras} to self-similar actions. These include not only C*-algebras but also   $L^p$-operator algebras, and their symmetrised Banacha $*$-versions.

\end{enumerate}
Ad \ref{enu:newpoint1}. All previous articles examining the structure of $C^*$-algebras related to self-similar actions
assume that the underlying graph is row-finite, see \cites{Exel-Pardo:Self-similar, Deaconu, Li_Yang1, Li_Yang2, Larki21, Yusnitha:Self-similar_actions, Larki25}. The preprint \cite{Exel-Pardo-Starling:Self-similar} claims to consider arbitrary graphs (and omits a number of proofs) but adds several assumptions in their final results, and row-finiteness is one of them, see also \cite{Enrique's_notes}. Moreover, in the structural results of these articles, it is usually assumed
that the underlying tight groupoid is Hausdorff. In fact, a number of sources assume a much stronger condition called pseudo-freeness, see 
\cites{Deaconu, Li_Yang1, Li_Yang2, Larki21}, that from our point of view ``trivialises'' the associated structures.  
We believe that the non-Hausdorff case is the most interesting, and just as self-similar groups were invented to produce
revealing examples of groups, we view self-similar actions as an indispensable tool to produce examples that help to understand the structure of non-Hausdorff groupoids and their algebras, a topic that has recently attracted considerable attention, cf. examples in \cites{Clark-Exel-Pardo-Sims-Starling:Simplicity_non-Hausdorff, Steinberg_Szakacs, Brix-Gonzales-Hume-Li:Hausdorff_covers, Martinez_Szakacs}.

Ad \ref{enu:newpoint2}. Twisted $C^*$-algebras of self-similar actions on (row-finite higher rank) graphs were defined in \cite{Mundey_Sims}. 
So far they have only been studied from the point of view of their cohomology theory \cite{Mundey_Sims} or $K$-theory  \cite{Mundey_Sims2}.
To study other properties we need to develop the theory of twisted inverse semigroups and their associated twisted ample groupoids.
It also seems that the theory of twisted groupoid correspondences has not yet been developed. 
The purely algebraic twists considered in \cite{Cortinas} are related but not exactly compatible with the ones we consider 
 (see Remark~\ref{rem:trivial_bowtie} below). Considering  twisted algebras is important for several reasons, one of them is the theory of Cartan subalgebras
\cite{Renault:Cartan}.

Ad \ref{enu:newpoint3}.
So far, the main focus has been on the universal $C^*$-algebra $\OO(\Gr,E)$ of  $(\Gr,E)$ or its purely algebraic counterpart. 
The Toeplitz $C^*$-algebra $\TT(\Gr,E)$ has been occasionally used, often as an auxiliary object, in the analysis of  $K$-theory or KMS-states, cf. \cites{Laca-Raeburn-Ramagge-Whittaker:Self-similar-groupoids, Mundey_Sims, Miller_Steinberg}. However, as indicated in \cite{Nekrashevych:Cstar_selfsimilar}, the study of the fixed-point core subalgebra $\OO(\Gr,E)_{0}$ of $\OO(\Gr,E)$ has some merit and geometric motivation. The nuclearity characterisation of  $\OO(\Gr,E)$ in \cite{Miller_Steinberg} shows that even a much smaller subalgebra, that we denote by $\OO(\Gr,E)_{00}$,
carries important information about the system. We take a closer look at these subalgebras and their Toeplitz algebra counterparts.
Moreover,  using concrete representations of the system, we define the reduced and essential analogues of the aforementioned algebras.
In addition, we do this in the twisted setting and consider representations on $L^p$-spaces, for $p\in[1,\infty]$, rather than just Hilbert spaces.

Ad \ref{enu:newpoint4}. Phillips' program of generalising $C^*$-algebraic constructions to the $L^p$-operator algebra context, which was initiated in 
\cites{Phi12, Phillips}, see \cite{Gardella},   has sparked in recent years. 
In particular, the graph $L^p$-operator-algebras, for $p\in [1,\infty)$, are quite well studied now \cites{Phi12, Phillips, Gardella, Gardella_Lupini17, cortinas_rodrogiez, cgt, cortinas_Montero_rodrogiez, BKM2}. 
We generalise this is to twisted algebras of self-similar actions of graphs.  Moreover, we phrase our results in terms   $L^P$-operator algebras, where $P\subseteq [1,\infty]$ is a set of parameters. These  were introduced in the context of twisted crossed products in \cite{BK},
and then used in the study of twisted groupoid algebras in \cite{BKM2}. They are a useful framework that unifies both the $L^p$-operator algebras of and
their symmetrised $*$-Banach algebra versions \cites{Austad_Ortega, Elkiaer}.  A general theory of Banach groupoid algebras developed in \cite{BKM} and \cite{BKM2} shows that a reasonable theory of $L^P$-operator algebras defined in terms of 
generators and relations is possible. 
When leaving the comfort zone of $C^*$-algebras, one sometimes needs to look for non-standard arguments and more  concrete and spatial constructions.  The reward is a broader perspective and deeper insight, but this is not all.
A concrete advantage is, for instance, that the groupoid $L^p$-operator algebras for $p\neq 2$ are much more rigid and carry the full information about the underlying groupoid, see \cites{Gardella_Palmstrom_Thiel, Hetland_Ortega, cgt}.

We now pass to a more detailed description of the main objects and results of the paper.

\subsection*{An inverse semigroup and groupoid toolkit}
Let $\Gr$ be a discrete groupoid whose unit space $\Gr^0$ is also the set of vertices $E^0$ in a directed graph $E=(E^{0},E^{1},\rg,\sr)$.
We call the pair $(\Gr,E)$ a \emph{self-similar action} if there is a left action of $\Gr$ on  $E^1$ and a $\Gr$-valued $1$-cocycle \(\Gr*E^{1}\ni(g,e)\mapsto g|_e\in \Gr\) for this action (see Definition~\ref{defn:self_similar_action}). We assume that both $\Gr$ and $E$ are countable, though this is only used when we talk about essential algebras or amenability of non-Hausdorff groupoids.

We start by discussing our findings on groupoids that come from inverse semigroups associated to  $(\Gr,E)$.
Besides providing powerful tools for studying associated algebras,  they are interesting in their own right.
The inverse semigroup $S(\Gr,E)$ of a self-similar action $(\Gr,E)$ is the (unique up to isomorphism)
universal inverse semigroup with zero that is generated by the groupoid $\Gr$ and the set of edges $E^1$ modulo the relations
$$
g \cdot h=[\sr(g)=\rg(h)] gh,\qquad e\cdot f=0 \text{ if }\sr(e)\neq \rg(f), \qquad  g \cdot e=[\sr(g)=\rg(e)] (ge)\cdot g|_{e}
$$
for $g,h\in \Gr$, $e,f \in E^1$, where \([sentence]\) is zero if
the \(sentence\) is false and 1 otherwise. See \cite{Miller_Steinberg}*{Subsection 2.2} for the groupoid correspondence construction and \cite{Deaconu}*{Section 5}
 (or Definition~\ref{def:inverse_semigroup_ssa} below) for the explicit description of $S(\Gr,E)$.
We denote by  $S_{0}(\Gr,E)$  the kernel of a natural $\Z$-valued $1$-cocyle on  $S(\Gr,E)$, and 
we let $S_{00}(\Gr,E)$ be the smallest wide inverse subsemigroup of $S(\Gr,E)$ containing $\Gr$.
This yields a sequence of wide inverse subsemigroups  
\[
S_{00}(\Gr,E)\subseteq S_{0}(\Gr,E)\subseteq S(\Gr,E).
\]

We denote by  $\Gg(\Gr,E)$  the tight  groupoid of $S(\Gr,E)$. Its unit space is the boundary path space $\partial E$ of the graph.
The  tight groupoids of  $S_{00}(\Gr,E)$ and $S_{0}(\Gr,E)$ can be naturally identified  with open subgroupoids of $\Gg(\Gr,E)$, and so 
we have a corresponding sequence of wide open subgroupoids 
\[
\Gg_{00}(\Gr,E)\subseteq \Gg_{0}(\Gr,E)\subseteq \Gg(\Gr,E).
\]
In fact, $\Gg_{00}(\Gr,E)$ coincides with the groupoid denoted by $\mathcal{H}_0$ in \cite{Miller_Steinberg}*{Subsection 2.4}.
The universal groupoids associated to the inverse semigroups $S_{00}(\Gr,E)\subseteq S_{0}(\Gr,E)\subseteq S(\Gr,E)$ also form a sequence of open wide subgroupoids 
\[
\widetilde{\Gg}_{00}(\Gr,E)\subseteq \widetilde{\Gg}_0(\Gr,E)\subseteq  \widetilde{\Gg}(\Gr,E).
\]
Their common unit space is the path space $E^{\leq \infty}$ of the graph.

The analysis of tight groupoids is facilitated by the results of \cite{Exel-Pardo:Tight}, in which the key properties of the groupoids are explained in terms of the inverse semigroups.
Following \cite{BKM2}, we give names to the corresponding conditions for inverse semigroups, and we extend the list with \emph{topological freeness} (see Definition~\ref{def:InverseSemigroupSimplicityProperties}).
Topological freeness for   \'etale groupoids was introduced in \cite{Kwasniewski-Meyer:Essential}, and for non-Hausdorff  groupoids  it is a weaker and more suitable condition than effectiveness (cf. Remark \ref{rem:effectiveness_no_good}). The inverse semigroup characterisation of Hausdorffness of the universal groupoid was given in  \cite{Steinberg0} and the corresponding inverse semigroup was called \emph{Hausdorff}  in \cite{Steinberg_Szakacs}. We were not able to find a reasonable condition on the inverse semigroup characterising topological freeness of the universal groupoid.

Combining both inverse semigroup analysis and direct study of groupoids (Sections~\ref{sect:inverse_semigroup} and \ref{sec:self-similar_groupoid}, respectively)  
we establish conditions on  $(\Gr,E)$ that characterise a number of   fundamental properties of the associated groupoids.
They generalise and improve upon a number of results in this direction in various papers. 
The following conditions are phrased in standard terminology established in \cites{Exel-Pardo:Self-similar, Laca-Raeburn-Ramagge-Whittaker:Self-similar-groupoids} and in the theory of graph $C^*$-algebras:
\begin{enumerate}\label{enumerate:conditions}
\item[$\Fin$] Every $g \in \Gr$ admits at most  finitely many minimal strongly $g$-fixed paths.
\item[$\Evr$] If $g\in \Gr$ fixes every  path in $\sr(g)E^*$, then $g$ strongly fixes some path in $\sr(g)E^*$.
\item[$\Cyc$] For all $g\in \Gr$, every $g$-cycle  has an entrance.
\item[$\Rec$] If $g\in \Gr$ fixes a path $\mu$ whose source is an finite receiver, then $\mu$ is strongly $g$-fixed.
\item[$\Min$] There are no nontrivial hereditary, saturated, $\Gr$-invariant subsets of $E^0$.
\item[$\Con$] For every $v\in E^0$ there exists $\mu\in vE^*$ such that $\sr(\mu)$ lies on a $g$-cycle with an entrance.
\end{enumerate}

Their relationship with properties of the associated groupoids is summarised  in  Table \ref{table:tight groupoid}, where  $*=\Space,0,00$ and $\Space$ means \emph{no symbol} (an empty space).

\begin{table*}[hbtp]
\centering

\begin{tabular}{lll}
\toprule
\multicolumn{1}{c}{\textbf{Self-similar action}} & \multicolumn{1}{c}{\textbf{Inverse semigroup}} & \multicolumn{1}{c}{\textbf{Groupoid}} \\
\midrule
$\Fin$&  \begin{tabular}{l} $S_{*}(\Gr,E)$ is closed 
\\
\midrule[0.1pt]
$S_{*}(\Gr,E)$ is Hausdorff
\end{tabular}
&  \begin{tabular}{l} $\Gg_*(\Gr,E)$ is Hausdorff \qquad
\\
\midrule[0.1pt]
$\widetilde{\Gg}_*(\Gr,E)$ is Hausdorff
\end{tabular}  \\ 
\midrule 
$\Evr$ + $\Cyc$ & \, $S(\Gr,E)$ is topologically free &\, $\Gg(\Gr,E)$ is topologically free \\
\midrule 
$\Evr$ + $\Rec$ & \qquad \qquad --------- & \, $\widetilde{\Gg}_*(\Gr,E)$ is topologically free \\
\midrule 
$\Evr$& \begin{tabular}{l}
 $S(\Gr,E)$ is quasi-fundamental
\\
\midrule[0.1pt]
 $S_{0}(\Gr,E)$ is topologically free
\\
\midrule[0.1pt]
$S_{00}(\Gr,E)$ is topologically free
\end{tabular}  & 
\begin{tabular}{l} $\Gg_0(\Gr,E)$ is topologically free
\\
\midrule[0.1pt]
$\Gg_{00}(\Gr,E)$ is topologically free
\end{tabular}
 \\
\midrule 
$\Min$ & \, $S(\Gr,E)$ is minimal & \, $\Gg(\Gr,E)$ is minimal \\
\midrule 
\multirow{2}{*}{$\Con$} & \, $S(\Gr,E)$ is (strongly) & \, $\Gg(\Gr,E)$ is locally contracting \\
 & \, locally contracting & \, with respect to $S$ \\
\bottomrule
\end{tabular}
\vspace{5pt}
\caption{Properties of self-similar actions, their inverse semigroups and  groupoids.}
\label{table:tight groupoid}
\end{table*}

Table \ref{table:tight groupoid} shows that
Hausdorffness of any of the considered groupoids  is equivalent to Hausdorffness of all of them (see Proposition \ref{prop:Hausdorff_extended_groupoid} and Corollary \ref{cor:Hausdorff_equiv_closed}) and holds if and only if $(\Gr,E)$ satisfies $\Fin$. This means that closedness of any inverse semigroup $S_*(\Gr,E)$ implies Hausdorfness of all of them, which is a strong property in that all their  actions with clopen domains yield Hausdorff transformation groupoids  (see Remark \ref{rem:Hausdorff_vs_closed}).
As far as we are aware conditions $\Evr$ and $\Rec$ have not appeared in previous papers; in fact, condition  $\Evr$ appears in disguise in \cite{Exel-Pardo:Self-similar}*{Theorem 18.8(2)(b)} and this is the only instance we know of. 
There might be several reasons for that. Firstly, for faithful self-similar actions, condition $\Evr$ is always fulfilled. Secondly,  most authors have only studied the tight groupoid $\Gg(\Gr,E)$. Thirdly, they tried to characterise its effectiveness rather than topological freeness.
We  characterise effectiveness of $\Gg(\Gr,E)$ in full generality (see Theorem~\ref{thm:effective_tight_groupoid}) and show  that, in general, it is strictly stronger than the related inverse semigroup condition introduced in \cite{Exel-Pardo:Tight} (see Example~\ref{ex:infinitely_many_edges} and Corollary~\ref{cor:effectiveness_fails}).  
Characterisation of topological freeness is easier and more natural. We view it as   manifestation of the fact that, for non-Hausdorff groupoids, topological freeness is the relevant condition, not effectiveness.
Topological freeness for the universal groupoids is equivalent to topological freeness of any of its core subalgebras. For tight groupoids, topological freeness of core subgroupoids is much weaker and is equivalent to that $S(\Gr,E)$ is quasi-fundamental in the sense of \cite{Steinberg_Szakacs}  (see Theorem \ref{thm:topological_freeness_self_similar_transformation_groupoids}). Condition $\Min$, equivalent to minimality of $\Gg(\Gr,E)$, can be also phrased as cofinality of $(\Gr,E)$ (see Definition \ref{def:Zappa--Szep_preorder} and Proposition \ref{prop:groupoid_minimality}).

By \cite{Miller_Steinberg}*{Theorem 2.18}, if any of the groupoids $\Gg_{00}(\Gr,E)\subseteq \Gg_{0}(\Gr,E)\subseteq \Gg(\Gr,E)$ is amenable,
then all of them are. This happens, for instance, when $\Gr$ is amenable or, more generally, when the canonical action $\Gr \curvearrowright \partial E$ is amenable.
It also happens when $(\Gr,E)$ is a contracting self-similar action. We add to this by showing that 
$$
\Gr \text{ is amenable  }\,\, \Longleftrightarrow\,\,  \widetilde{\Gg}(\Gr,E) \text{ is amenable}.
$$
We deduce it from  our $C^*$-correspondence analysis and  results of \cites{Brix-Gonzales-Hume-Li:Hausdorff_covers, Buss_Martinez} (see Theorem~\ref{thm:Toeplitz_nuclearity}), though it probably could be proved directly as in \cite{Miller_Steinberg}.

We show that all non-Hausdorff points of the universal groupoid $\widetilde{\Gg}(\Gr,E)$ belong, in fact,  to the tight groupoid $\Gg(\Gr,E)$, 
and we describe all elements of $\Gg(\Gr,E)$ that cannot be separated from a given point in the unit space 
(see Proposition~\ref{prop:non-Hausdorff_points_description}). This, in particular, characterises when the considered groupoids are \emph{finitely non-Hausdorff}: each point cannot be separated from at most finitely many other points. For an \'etale  groupoid $\Gg$ this is equivalent to saying that the source map restricted to the closure of $\Gg^0$ is finite-to-one. 
As a consequence we show  that
if $(\Gr,E)$ is contracting, as defined in \cite{BBGHSW24}, then $\widetilde{\Gg}(\Gr,E)$ and $\Gg(\Gr,E)$ are finitely non-Hausdorff (see Corollary~\ref{cor:contracting_non_Hausdorff}). 

These results are important since \cite{Brix-Gonzales-Hume-Li:Hausdorff_covers}*{Theorem C} says that for every finitely non-Hausdorff 
\'etale groupoid $\Gg$, the $C^*$-algebraic singular ideal vanishes if and only if the algebraic singular ideal vanishes.
 Efficient characterisations of the latter are now known, see \cite{Hume}. 
We  note that the result of \cite{Brix-Gonzales-Hume-Li:Hausdorff_covers}
 holds in a much more general setting that can be applied to the reduced and essential groupoid Banach algebras introduced in \cite{BKM}. 
The corresponding result can be formulated as follows  
 (see Theorem~\ref{thm:vanishing_of_the_singular_ideal} and Corollary~\ref{cor:vanishing_of_singulars}):
\begin{thmx}[Vanishing of singular ideals]\label{thmx:singular_ideal}
Let $\Gg$ be finitely non-Hausdorff \'etale groupoid with locally compact Hausdorff unit space.  
Let $\mathfrak{C}_c(\Gg)$ be the set of quasi-continuous functions on $\Gg$ and let $\mathfrak{M}_0(\Gg )$ consist
of functions with meager strict support that are in the uniform closure of  $\mathfrak{C}_c(\Gg)$.
The following conditions are equivalent:
\begin{enumerate}
\item $\mathfrak{M}_0(\Gg )=\{0\}$;
\item\label{enux:singular_ideal2} $\mathfrak{C}_c(\Gg)\cap \mathfrak{M}_0(\Gg )=\{0\}$;
\item every reduced groupoid Banach algebra of $\Gg$ is essential; 
\item some reduced groupoid Banach algebra of $\Gg$ is essential.
\end{enumerate}
Condition \ref{enux:singular_ideal2}  is characterised in terms of bisections in \cite{Brix-Gonzales-Hume-Li:Hausdorff_covers} and 
in terms of isotropy groups of $\Gg$ in \cite{Hume}, see condition $\Hum$ on page~\pageref{Condition:Hume} below.
\end{thmx}

\subsection*{Twisted \texorpdfstring{$C^*$}{C*}-algebras of self-similar actions} 
Let $(\Gr,E)$ be a self-similar action.
We define a \emph{twist} of  $(\Gr,E)$ to be a pair $\sigma = (\sigma_{\Gr},\sigma_{\bowtie})$, where $\sigma_{\Gr}\colon \Gr^2 \to \T$
 is a standard groupoid $2$-cocycle on $\Gr$ in the usual sense, 
and $\sigma_{\bowtie} \colon \Gr * E^{1}\to \T$ is a map that relates the twist $\sigma_{\Gr}$ with the action  (see Definition~\ref{dfn:ss_cocycle}). 
We call the triple  $(\Gr,E,\sigma)$ a \emph{twisted self-similar action}.
As we explain, such twists exhaust the corresponding cohomology classes of the associated Zappa--Sz\'ep product category.

A \emph{representation} of  $(\Gr,E,\sigma)$ in a $C^*$-algebra $B$ is a pair $(W,T)$ where
$W \colon \Gr \to B$ is a $\sigma_{\Gr}$-twisted unitary representation of $\Gr$, and $T\colon E^1\to B$ is such that 
$T_e^*T_e=W_{\sr(e)}$ and $\sum_{f\in F}T_f T_f^*\leq W_{\rg(e)}$ for all $e\in E^{1}$ and finite $F\subseteq E^1$, and 
$$
W_g T_e = \sigma_{\bowtie}(g,e)T_{g e} W_{g|_e}, \qquad  \text{ for all } (g,e) \in \Gr * E^1. 
$$
We say that a representation is   \emph{(Cuntz--Krieger) covariant} if 
\begin{equation*}\label{eq:CK}
W_v=\sum_{e \in \rg^{-1}(v)} T_e T_e^*,\qquad  \text{ for all }v\in E^{0}_{\reg}, 
\end{equation*}
where $E^{0}_{\reg}$ is the set of regular vertices.
We propose to call a representation $(W,T)$ \emph{Coburn--Toeplitz covariant} if it satisfies
the extreme opposite of the above condition:
$$
W_v\neq\sum_{e \in \rg^{-1}(v)} T_e T_e^*,\qquad  \text{ for all }v\in E^{0}_{\reg}.
$$
We say that $(W,T)$ is \emph{nonzero} if all projections $W_{v}$, $v\in V$, are nonzero.

There is a universal representation  of $(\Gr,E,\sigma)$ and it is Coburn--Toeplitz covariant. 
We also concretely define the reduced and essential Coburn--Toeplitz covariant representations of $(\Gr,E,\sigma)$.
We denote the $C^*$-algebras generated by them by $\TT(\Gr,E,\sigma)$, $\TT_{\red}(\Gr,E,\sigma)$ and $\TT_{\ess}(\Gr,E,\sigma)$, respectively.
Similarly, we define universal, reduced, and essential Cuntz--Krieger covariant representations of $(\Gr,E,\sigma)$, and we denote the associated $C^*$-algebras by $\OO(\Gr,E,\sigma)$, $\OO_{\red}(\Gr,E,\sigma)$ and $\OO_{\ess}(\Gr,E,\sigma)$.
Thus, we have a commuting diagram of surjective $*$-homomorphisms
$$
\begin{tikzcd}  [row sep=small]
\TT(\Gr,E,\sigma)\arrow[r, ->>]\arrow[d, ->>] & \TT_{\red}(\Gr,E,\sigma)\arrow[r, ->>]\arrow[d, ->>] & \TT_{\ess}(\Gr,E,\sigma)
\arrow[d, ->>] 
\\
\OO(\Gr,E,\sigma)\arrow[r, ->>] & \OO_{\red}(\Gr,E,\sigma)\arrow[r, ->>]& \OO_{\ess}(\Gr,E,\sigma) 
\end{tikzcd}
$$
Let $*=\Space,\red,\ess$ where $\Space$ means \emph{no symbol} (an empty space). 
Each algebra $\TT_{*}(\Gr,E,\sigma)$ naturally contains  a copy of the Toeplitz graph $C^*$-algebra $\TT(E)$ of the graph $E$ and so it also contains the  diagonal subalgebra  $\Cont_0(E^{\leq \infty})$.
Similarly, each $\OO_{*}(\Gr,E,\sigma)$ contains a copy of the graph $C^*$-algebra $\OO(E)$ and so also it contains the diagonal subalgebra 
$\Cont_0(\partial E)$. 

The algebras $\TT_{*}(\Gr,E,\sigma)$ and $\OO_{*}(\Gr,E,\sigma)$ come equipped with natural gauge $\T$-actions,
and we denote their fixed-point subalgebras by   $\TT_{*}(\Gr,E,\sigma)_0$ and $\OO_{*}(\Gr,E,\sigma)_0$. 
The fixed-point subalgebras contain important smaller algebras $\TT_{*}(\Gr,E,\sigma)_{00}$ and $\OO_{*}(\Gr,E,\sigma)_{00}$ that are generated by  representations of $(\Gr,\sigma_{G})$ and the diagonal $C^*$-algebras. The introduced subalgebras form the commuting diagram
$$
\begin{tikzcd}  [row sep=14pt, column sep = small]
\Cont_0(E^{\leq \infty})\arrow[r, phantom, sloped, "\subseteq"] \arrow[d, ->>] & \TT_{*}(\Gr,E,\sigma)_{00}\arrow[r, phantom, sloped, "\subseteq"]\arrow[d, ->>] 
& \TT_{*}(\Gr,E,\sigma)_{0}
\arrow[r, phantom, sloped, "\subseteq"]\arrow[d, ->>] & \TT_{*}(\Gr,E,\sigma)\arrow[d, ->>] 
\\
\Cont_0(\partial E)\arrow[r, phantom, sloped, "\subseteq"]  & \OO_{*}(\Gr,E,\sigma)_{00}\arrow[r, phantom, sloped, "\subseteq"]
& \OO_{*}(\Gr,E,\sigma)_{0}
\arrow[r, phantom, sloped, "\subseteq"] & \OO_{*}(\Gr,E,\sigma)
\end{tikzcd},
$$
where the vertical homomorphisms are induced by the Cuntz--Krieger relations.

It is known that $C^*$-algebras associated to groupoid correspondences are modelled by relative Cuntz--Pimsner algebras, see \cite{Meyer}. For
discrete groupoid correspondences this is described in detail \cite{Miller_Steinberg}*{Subsection 2.2} where in fact 
the (untwisted) algebras $\TT(\Gr,E)$ and $\OO(\Gr,E)$ are defined as the Toeplitz and a relative Cuntz--Pimsner algebra of 
a $C^*$-correspondence $X(\Gr,E)$. We generalise this to the twisted case, and analyse Cuntz--Pimsner models for the reduced and essential algebras (see Corollary~\ref{cor:Cuntz--Pimsner_picture}).
The most interesting consequence of this analysis is the  following result (see Theorem~\ref{thm:Toeplitz_nuclearity}).
\begin{thmx}[Nuclearity I]\label{thmx:nuclearity}
Let $(\Gr,E)$ be a  self-similar groupoid action with a twist $\sigma = (\sigma_{\Gr},\sigma_{\bowtie})$. 
The following are equivalent:
\begin{enumerate}
	\item the discrete groupoid $\Gr$ is amenable;
	\item $\TT(\Gr,E,\sigma)$ is nuclear;
	\item $\TT_{\red}(\Gr,E,\sigma)$ is nuclear;
	\item $C^*(\Gr,\sigma_{\Gr})$ is nuclear;
	\item $C^*_{\red}(\Gr,\sigma_{\Gr})$ is nuclear.
\end{enumerate}
Assume that these equivalent conditions hold. Then  $\TT(\Gr,E,\sigma)=\TT_{\red}(\Gr,E,\sigma)$ and $C^*(\Gr,\sigma_{\Gr})=C^*_{\red}(\Gr,\sigma_{\Gr})$, 
and these algebras $KK$-equivalent. For all $*=\Space,\, \red,\,\ess$ and $**=\Space,\,0$ the 
algebras $\TT_{*}(\Gr,E,\sigma)_{**}$ and $\OO_{*}(\Gr,E,\sigma)_{**}$  
are nuclear. If, in addition, $E$ is row-finite, then 
$$
\TT(\Gr,E,\sigma)=\TT_{\red}(\Gr,E,\sigma)=\TT_{\ess}(\Gr,E,\sigma).
$$
\end{thmx}
The automatic equality $\TT(\Gr,E,\sigma)=\TT_{\ess}(\Gr,E,\sigma)$   for $\Gr$ amenable and row-finite $E$  in Theorem~\ref{thmx:nuclearity} is quite unexpected: we  show by example that it may fail  when  $E$ is not row-finite.   It can be interpreted as automatic vanishing of the $C^*$-algebraic singular ideal in the universal groupoid associated to $(\Gr,E)$, and could possibly be explained using the machinery developed recently in \cite{Hume}.

  It seems difficult to  obtain a version of the above theorem for $\OO(\Gr,E,\sigma)$ using relative Cuntz--Pimsner picture,
as we have little knowledge about the corresponding covariance ideal (cf. Problem~\ref{problem:embed_into_OO}). The existing groupoid tools allow us  to cover 
only the untwisted case or Hausdorff case (cf. Problem \ref{problem:nuclearity}). Combining recent results from  \cites{Miller_Steinberg, Brix-Gonzales-Hume-Li:Hausdorff_covers, Buss_Martinez}  and our groupoid models we get the following result (see Theorem~\ref{thm:Cuntz--Pimsner_nuclearity} and Remark~\ref{rem:Cuntz--Pimsner_nuclearity}).
\begin{thmx}[Nuclearity II]\label{thmx:nuclearity2}
Let $(\Gr,E, \sigma)$ be a twisted self-similar groupoid action. Assume that either $\Fin$ holds or the twist is trivial. 
The following are equivalent:
\begin{enumerate}
	\item\label{itm:nuclearity2_first} the ample groupoid $\Gg_{00}(\Gr,E)$ is amenable; 
	\item $\OO(\Gr,E,\sigma)$ is nuclear;
	\item $\OO_{\red}(\Gr,E,\sigma)$ is nuclear;
	\item $\OO(\Gr,E,\sigma)_{00}$ is nuclear;
	\item\label{itm:nuclearity2_last}  $\OO_{\red}(\Gr,E,\sigma)_{00}$ is nuclear.
\end{enumerate}
If any of these conditions hold, then for all $*=\Space,\,0,\,00$  we have  $\OO(\Gr,E,\sigma)_{*}=\OO_{\red}(\Gr,E,\sigma)_{*}$, and these algebras,
as well as  $\OO_{\ess}(\Gr,E,\sigma)_{*}$, are nuclear. 
If the self-similar action $(\Gr,E)$ is contracting,
then \ref{itm:nuclearity2_first}--\ref{itm:nuclearity2_last} hold. 
\end{thmx}

We now explain the role of the conditions listed on page~\pageref{enumerate:conditions}. Conditions $\Evr$ and $\Cyc$ are the right generalisation of condition (L) for graphs, which guarantees 
the celebrated Cuntz--Krieger Uniqueness Theorem, see \cites{Cuntz--Krieger,Bates_Pask_Raeburn_Szymanski, FLR, Drinen_Tomforde}.
We obtain the following generalisation of this result to the twisted self-similar setting, 
which gives  true uniqueness  when $\OO_{}(\Gr,E,\sigma)=\OO_{\ess}(\Gr,E,\sigma)$ (see Theorem~\ref{thm:uniqueness_results}\ref{enu:Cuntz--Krieger uniqueness1'}).
\begin{thmx}[Cuntz--Krieger Uniqueness]\label{thmx:Cuntz--Krieger_Uniqueness}
Let $(\Gr,E,\sigma)$ be a twisted self-similar groupoid action. 
If $\Evr$ and $\Cyc$ hold, then for  every $C^*$-algebra $C^*(T,W)$ generated by a nonzero Cuntz--Krieger representation  $(W,T)$ of $(\Gr,E,\sigma)$
we have canonical epimorphisms 
$$
\OO_{}(\Gr,E,\sigma)\onto C^*(T,W)\onto \OO_{\ess}(\Gr,E,\sigma).
$$
Moreover, if either
\begin{enumerate}
\item the action of $\Gr$ on $\partial E$ is amenable and $(\Gr,E)$ satisfies $\Fin$; or
\item $(\Gr,E)$ is contracting, the twist is trivial, and $\Gg(\Gr,E)$ satisfies $\Hum$;
\end{enumerate}
then $\OO_{}(\Gr,E,\sigma)=\OO_{\ess}(\Gr,E,\sigma)$.
\end{thmx}

Fowler and Raeburn \cite{Fowler_Raeburn:Toeplitz}*{Theorem 4.1} proved a version of the uniqueness theorem for the Toeplitz graph $C^*$-algebra $\TT(E)$, which has since been generalised in different directions (see \cites{Raeburn_Sims:Product_systems, Carlsen_Kwasniewski_Ortega, Paper_with_Nadia,Underrated_paper_with_Nadia}).
We generalise this uniqueness theorem to our setting as follows (see Theorem~\ref{thm:uniqueness_results}\ref{enu:Cuntz--Krieger uniqueness2'}).
\begin{thmx}[Coburn--Toeplitz Uniqueness]\label{thmx:Coburn_Uniqueness}
Let $(\Gr,E,\sigma)$ be a  twisted self-similar groupoid action. 
If $\Evr$ and $\Rec$ hold, then for every $C^*$-algebra $C^*(T,W)$ generated by a nonzero Coburn--Toeplitz covariant  representation  $(W,T)$ of $(\Gr,E,\sigma)$
we have  canonical epimorphisms 
$$
\TT(\Gr,E,\sigma)\onto C^*(T,W)\onto \TT_{\ess}(\Gr,E,\sigma).
$$
If the twist is trivial,  the converse implication holds. Moreover, if $\Gr$ is amenable and $E$ is row-finite, or  $(\Gr,E)$ satisfies $\Fin$, then
$\TT_{}(\Gr,E,\sigma)=\TT_{\ess}(\Gr,E,\sigma).
$ 
\end{thmx}

Let  $*=\Space,\,\red,\,\ess$ and $**=\Space,\,0,\,00$. 
Condition $\Fin$ is equivalent to the existence of a canonical conditional expectation from $\TT_{*}(\Gr,E,\sigma)_{**}$ onto $\Cont_0(E^{\leq \infty})$ and from $\OO_{*}(\Gr,E,\sigma)_{**}$ onto  $\Cont_0(\partial E)$. This leads to the following characterisations of when the corresponding $C^*$-inclusion is 
Cartan \cites{Renault:Cartan, Kwasniewski-Meyer:Cartan} (see Theorem~\ref{thm:Cartan_algebras}  and Corollary~\ref{cor:Theorem C}).

\begin{thmx}[Cartan subalgebras]\label{thmx:Cartan}
Let $(\Gr,E,\sigma)$ be a  twisted self-similar groupoid action. 
If any of the inclusions 
$
C_0(\partial E)\subseteq \OO_{\red}(\Gr,E,\sigma)_{*}$   or $C_0( E^{\leq \infty})\subseteq \TT_{\red}(\Gr,E,\sigma)_{*}$, 
for some $*=\Space,\,0,\,00$, 
is Cartan, then $\Fin$ and $\Evr$ hold. Conversely, assume that $\Fin$  holds. Then 
$$
\OO_{\red}(\Gr,E,\sigma)_{*}=\OO_{\ess}(\Gr,E,\sigma)_{*}\quad \text{ and } \quad  \TT_{\red}(\Gr,E,\sigma)_{*}=\TT_{\ess}(\Gr,E,\sigma)_{*},
$$
for every $*=\Space,\,0,\,00$. Moreover,
\begin{enumerate}
	
	\item $C_0(\partial E)\subseteq \OO_{\red}(\Gr,E,\sigma)$ is Cartan  if and only if $\Cyc$ and $\Evr$ hold;
	\item  any of the inclusions $C_0( E^{\le \infty})\subseteq  \TT_{\red}(\Gr,E,\sigma), \TT_{\red}(\Gr,E,\sigma)_{0},  \TT_{\red}(\Gr,E,\sigma)_{00}$  is Cartan if and only if  $\Rec$ and $\Evr$ hold; and
	\item   any of the inclusions $C_0(\partial E)\subseteq \OO_{\red}(\Gr,E,\sigma)_{0}, \OO_{\red}(\Gr,E,\sigma)_{00}$ is Cartan  if and only if  $\Evr$ holds.
\end{enumerate}
\end{thmx}

We phrase the remaining results in the more general setting of $L^P$-operator algebras.
\subsection*{\texorpdfstring{$L^P$}{LP}-operator algebras and twisted inverse semigroups}
It seems to us that the $C^*$-algebras of twisted inverse semigroups have not previously been thoroughly studied in the literature. 
Since we needed to develop parts of this theory  from scratch, we decided to work in the broader framework of $L^P$-operator algebras. This may  be useful in  future applications, and may attract the attention of the ever-increasing number of researchers interested in $L^p$-operator algebras. Taking $P=\{2\}$ one recovers the $C^*$-algebras we have already discussed.

Firstly, we show that any twisted inverse semigroup $(S,\omega)$  induces a twist $\LL_{\omega}$ on any transformation groupoid for an action of $S$, and in particular on the universal groupoid $\widetilde{\Gg}(S)$ and the tight groupoid $\Gg(S)$ of $S$.
Secondly, for any $\emptyset \neq P\subseteq [1,\infty]$, we define the algebras $\OO^P(S,\omega)$  and  $\TT^P(S,\omega)$  as Banach algebras that are universal for naturally defined representations and covariant representations, respectively, on  $L^p$-spaces for all $p\in P$. 
We represent inverse semigroups  by Moore-Penrose partial isometries \cite{Mbekhta}, which for $C^*$-algebras are the usual partial isometries.
We show  that we have natural isometric isomorphisms  (see Corollary~\ref{cor:groupoid_presentation_inverse_semigroup_algebras})
$$
\TT^P(S,\omega)\cong F^P(\widetilde{\Gg}(S),\LL_{\omega}) \quad \text{and} \quad \OO^P(S,\omega)\cong F^P( \Gg(S),\LL_{\omega}),
$$
where $F^P(\Gg,\LL)$ denotes the universal $L^P$-operator algebra of a twisted groupoid $(\Gg,\LL)$ defined in \cite{BKM2}.
This allows us to also define reduced and essential versions of the above algebras that still have a corresponding groupoid model.\

En passant, we prove a general result that any  continuous groupoid homomorphism $c:\Gg\to \Gamma$ into a discrete abelian group $\Gamma$ induces gauge actions on
the universal $F^P(\Gg,\LL)$, the reduced $F^P_{\red}(\Gg,\LL)$, and the essential algebra $F^P_{\ess}(\Gg,\LL)$. However, we only apply this result to twisted inverse semigroup algebras, so we included the more general result as an appendix  (see Theorem~\ref{thm:gauge_actions_on_groupoids}).
For full algebras and essential algebras, this result is new even for $C^*$-algebras, cf. \cite{Brown- Fuller-Pitts-Reznikof:Graded}.

We show that a twist $\sigma$ of a self-similar action $(\Gr, E)$ defines a twist $\omega_\sigma$ of the inverse semigroup $S(\Gr,E)$ 
and therefore induces a twist $\LL_{\sigma}$ on the groupoids $\widetilde{\Gg}(\Gr,E)$ and $\Gg(\Gr,E)$ (see Proposition~\ref{prop:twist_self_similar_inverse_semigroup}).
It is important to note that even though these are ample groupoids, the twist   $\LL_{\sigma}$ is usually topologically nontrivial (see Example \ref{ex:topologically_nontrivial_twist_from_self-similar}). We define the algebras $\TT^P(\Gr,E,\sigma)$ and $\OO^P(\Gr,E,\sigma)$, which are
universal for representations and Cuntz--Krieger covariant representations, respectively, of $(\Gr,E,\sigma)$ on $L^p$-spaces for $p\in P$. 
We prove the isometric isomorphisms (see Theorem~\ref{thm:presentations_of_twisted_self_similar_algebras}):
\begin{align*}
\TT^P(\Gr,E,\sigma)&\cong \TT^P(S(\Gr,E),\omega_{\sigma})\cong F^P(\widetilde{\Gg}(\Gr,E),\LL_{\sigma}),\\
\OO^P(\Gr,E,\sigma)&\cong \OO^P(S(\Gr,E),\omega_{\sigma})\cong F^P(\Gg(\Gr,E),\LL_{\sigma}).
\end{align*}
These isomorphisms allow us to define reduced and essential versions of these algebras as well as their  core subalgebras.
By analysing extended representations using results of \cite{BKM} on Banach algebras generated by inverse semigroups,
 and combining this with structural results on groupoid Banach algebras from  \cite{BKM2}, we prove $L^P$-analogues of Theorems \ref{thmx:Cuntz--Krieger_Uniqueness}--\ref{thmx:Cartan}
(see Theorems \ref{thm:uniqueness_results}, 
\ref{thm:Cartan_algebras}), where uniqueness is expressed in terms of the ideal intersection property. 
Namely, we say that a Banach subalgebra $A$ \emph{detects ideals} in an ambient Banach algebra $B$ if 
every nonzero (closed two-sided) ideal $I$ in $B$ has a nonzero intersection with $A$.

Using the machinery we have developed, we may produce a myriad of results.  We  present a sample. We start by considering the Hausdorff case  (combine Theorems~\ref{thm:uniqueness_results} and \ref{thm:Cartan_algebras}, 
and Remark \ref{rem:amenablity_of_1_and_infty}):

\begin{thmx}[Detection of ideals in the Hausdorff case]\label{thmx:Hausdorff_detection} 
Let $(\Gr,E)$ be a self-similar groupoid action satisfying $\Fin$. The following are equivalent:
\begin{enumerate}[label={\textup{(o\arabic*)}}]
\item $(\Gr,E)$  satisfies $\Evr$ and $\Cyc$;
\item  $C_{0}(\partial E)$ is maximal abelian in   $\OO^P_{\red}(\Gr,E,\sigma)$ for every $\sigma$ and every $\emptyset \neq P\subseteq [1,\infty]$;
\item  $C_{0}(\partial E)$ is maximal abelian in   $\OO^P_{\red}(\Gr,E,\sigma)$ for some $\sigma$ and  some $\emptyset \neq P\subseteq [1,\infty]$;
\item   $C_{0}(E^0)$ detects ideals in  $\OO^P_{\red}(\Gr,E,\sigma)$ for every $\sigma$ and every $\emptyset \neq P\subseteq [1,\infty]$;
\item   $C_{0}(E^0)$ detects ideals in  $\OO^P(\Gr,E)$ for some $\emptyset \neq P\subseteq \{1,\infty\}$.
\end{enumerate}
 For any   $*=0,\,00$  the following are equivalent:
\begin{enumerate}[label={\textup{(c\arabic*)}}]
\item$(\Gr,E)$  satisfies $\Evr$; 
\item  $C_{0}(\partial E)$ is maximal abelian in   $\OO^P_{\red}(\Gr,E,\sigma)_{*}$ for every $\sigma$ and every $\emptyset \neq P\subseteq [1,\infty]$;
\item  $C_{0}(\partial E)$ is maximal abelian in   $\OO^P_{\red}(\Gr,E,\sigma)_{*}$ for some $\sigma$ and  some  $\emptyset \neq P\subseteq [1,\infty]$;
\item   $C_{0}(\partial E)$ detects ideals in  $\OO^P_{\red}(\Gr,E,\sigma)_{*}$ every $\sigma$ and every $\emptyset \neq P\subseteq [1,\infty]$;
\item   $C_{0}(\partial E)$ detects ideals in  $\OO^P(\Gr,E)_{*}$ for some $\emptyset \neq P\subseteq \{1,\infty\}$.
\end{enumerate}
Finally for every $*=\Space,\,0,\,00$,  the following are equivalent:
\begin{enumerate}[label={\textup{(t\arabic*)}}]
\item$(\Gr,E)$  satisfies $\Evr$ and $\Rec$; 
\item  $C_{0}(E^{\leq\infty})$ is maximal abelian in   $\TT^P_{\red}(\Gr,E,\sigma)_{*}$ for every $\sigma$ and every $\emptyset \neq P\subseteq [1,\infty]$;
\item  $C_{0}(E^{\leq\infty})$ is maximal abelian in   $\TT^P_{\red}(\Gr,E,\sigma)_{*}$ for some $\sigma$ and  $\emptyset \neq P\subseteq [1,\infty]$;
\item  $C_{0}(E^{\leq\infty})$ detects ideals in  $\TT^P_{\red}(\Gr,E,\sigma)_{*}$ for every $\sigma$ and every $\emptyset \neq P\subseteq [1,\infty]$;
\item  $C_{0}(E^{\leq\infty})$ detects ideals in  $\TT^P(\Gr,E)_{*}$ for   some  $\emptyset \neq P\subseteq \{1,\infty\}$;
\item   $C_{0}(E^{\leq 1})$ detects ideals in  $\TT^P_{\red}(\Gr,E,\sigma)$ for every $\sigma$ and every $\emptyset \neq P\subseteq [1,\infty]$;
\item   $C_{0}(E^{\leq 1})$ detects ideals in  $\TT^P(\Gr,E)$ for some $\emptyset \neq P\subseteq \{1,\infty\}$.
\end{enumerate}
\end{thmx}

In the amenable, untwisted, and locally Hausdorff case we obtain the following characterisations of detection of ideals (see Corollary \ref{cor:detection_of_ideals_amenable}).
\begin{thmx}[Detection of ideals in the amenable case]\label{thmx:Detection of ideals II} 
Let $(\Gr,E)$ be a self-similar groupoid action such that each path $\mu\in \partial E$ admits at most finitely many inequivalent singular decompositions. If $\Gg_{00}(\Gr,E)$ is amenable, then the following conditions are equivalent: 
\begin{enumerate}[label={\textup{(o\arabic*)}}]
\item\label{enux:Cuntz--Krieger uniqueness0} $(\Gr,E)$  satisfies $\Evr$ and $\Cyc$, and $\Gg(\Gr,E)$ satisfies $\Hum$;

\item\label{enux:Cuntz--Krieger uniqueness1} $C_{0}(E^0)$ detects ideals in  $\OO^P(\Gr,E)$ for every $\emptyset \neq P\subseteq [1,\infty]$;
\item\label{enux:Cuntz--Krieger uniqueness2}   $C_{0}(E^0)$ detects ideals in  $\OO^P(\Gr,E)$ for some $\emptyset \neq P\subseteq [1,\infty]$;
\end{enumerate}
If $\Gr$ is amenable, then the following conditions are equivalent: 
\begin{enumerate}[label={\textup{(t\arabic*)}}]
\item$(\Gr,E)$  satisfies $\Evr$ and $\Rec$, and $\widetilde{\Gg}(\Gr,E)$ satisfies $\Hum$; 
\item  $C_{0}(E^{\leq\infty})$ detects ideals in  $\TT^P(\Gr,E,\sigma)$ for every $\emptyset \neq P\subseteq [1,\infty]$;
\item  $C_{0}(E^{\leq\infty})$ detects ideals in  $\TT^P(\Gr,E,\sigma)$ for some $\emptyset \neq P\subseteq [1,\infty]$;
\end{enumerate}
\end{thmx}
The above results can be readily used to obtain simplicity criteria, as for any  $*=\Space,\,\red,\,\ess$ the algebra $\OO^P_{*}(\Gr,E,\sigma)$ is simple if and only if 
$C_0(E^0)$ detects ideals in $\OO^P_{*}(\Gr,E,\sigma)$ and $\Min$ holds. 
Moreover, using  condition $\Con$ we get the following simplicity and  pure infiniteness criteria that generalise similar results in 
\cites{FLR, Drinen_Tomforde, Exel-Pardo:Self-similar, Exel-Pardo-Starling:Self-similar, Larki21, cortinas_Montero_rodrogiez, BKM2}  (see Theorem~\ref{thm:simplicity}):

\begin{thmx}[Simplicity and pure infiniteness]\label{thmx:simplicity} 
Let $(\Gr,E)$ be a   self-similar groupoid action with a twist $\sigma$ and  let $\emptyset\neq  P\subseteq [1,\infty]$. 
If $\OO^P(\Gr,E)$ is simple, then   $\Evr$, $\Cyc$, and $\Min$ hold. Conversely, if $(\Gr,E)$ satisfies  $\Evr$, $\Cyc$, and $\Min$, then  $\OO_{\ess}^P(\Gr,E,\sigma)$ is simple, and in this case $\Con$ implies that $\OO_{\ess}^P(\Gr,E,\sigma)$ is also purely infinite in the sense of \cite{BKM2}.
\end{thmx}

Finally, we note that for contracting   self-similar actions $(\Gr,E)$ the assumptions in the first part of Theorem \ref{thmx:Detection of ideals II} are satisfied. 
Thus, we get the following  generalisation of the recent observation from \cite{Aakre}, that 
simplicity of the reduced $C^*$-algebra $\OO_{\red}(\Gr,E)$ is equivalent to the simplicity of the associated Steinberg algebra $A_{\mathbb{C}}(\Gg(\Gr,E))$ (see Corollary~\ref{cor:contracting_results}): 
\begin{corx}\label{corx:contracting_results}
Let  $(\Gr,E)$ be a contracting self-similar action (or more generally  a self-similar action as in the first part of Theorem \ref{thmx:Detection of ideals II}).  Let $\emptyset\neq  P\subseteq [1,\infty]$ and 
consider the condition $\Hum$, defined on page~\pageref{Condition:Hume}, for the groupoid $\Gg(\Gr,E)$. The following are equivalent: 
\begin{enumerate}
\item\label{item:contracting_results0} 
$\Evr$, $\Cyc$, $\Min$,  and $\Hum$ hold;

\item\label{item:contracting_results3} The algebra $\OO^P(\Gr,E)$ is simple;
\item\label{item:contracting_results3.5}  The Steinberg algebra $A_{\mathbb{C}}(\Gg(\Gr,E))$ is simple.
\end{enumerate}
\end{corx}

\subsection*{Organization of the paper}
The first four sections are largely introductory, but each one contains new elements. In Section \ref{twisted_inverse_semigroups}
we explain how to construct twisted groupoids from twisted inverse semigroups. Section \ref{sect:LP-groupoid_algebras}
 briefly discusses the main
result of \cite{BKM2} on twisted groupoid $L^p$-operator algebras  and extends them by the recent criteria for vanishing singular ideals.
We introduce twisted inverse semigroup $L^P$-operator algebras, and present the relevant conditions on inverse semigroups, in Section \ref{sect:Twisted inverse semigroup}. In Section \ref{sect:self-similar} we discuss in detail the four different viewpoints on self-similar actions listed on page \pageref{enu:viewpoint1}.
We start our analysis from the associated inverse semigroups in Section \ref{sect:inverse_semigroup}. Then we pass to our groupoid analysis in Section \ref{sec:self-similar_groupoid}.
In Section \ref{sect:pseudo_freeness} we give a brief summary of characterisations of pseudo-freeness, before moving on to present twists for self-similar actions in Section \ref{sect:twist}. In Section \ref{sect:self-similar_LP-operator} 
we introduce the title objects of the paper and prove our main structural results. For $C^*$-algebras 
we complement these results by exploiting their relative Cuntz-Pimsner picture in Section \ref{sect:C-star}.
We close the paper with an appendix in which we prove that groupoid homomorphism into abelian groups induce gauge actions on twisted groupoid $L^P$-algebras full, reduced and essential.

\section{Twisted inverse semigroups and  groupoids}\label{twisted_inverse_semigroups}
\subsection{Twisted inverse semigroups and their actions}
An inverse semigroup is a semigroup $S$ such that for every $t\in S$ there is a unique  $t^*\in S$  satisfying $t = t t^* t$ and $t^* = t^* t t^*$. In this paper we adopt the convention that \emph{all inverse semigroups have zero}, that is an element $0\in S$ satisfying $0t=t0=0$ for every $t\in S$.
If $S$ has no zero, one can always add it. Recall that  idempotents in an inverse semigroup $S$ commute and we have
\[
\EE(S) \coloneqq \{ e \in S : e^2 = e \} = \{ tt^* : t \in S \} = \{ t^*t : t \in S \}.
\]
Thus,  $\EE(S)$ is a semilattice with minimal element $0$.
There is a partial order on $S$ defined by $t \le u$ if and only if $t = u t^* t$ (or $t=tt^*u$), if and only if  $t = u e$  
(or $t=eu$) for some $e\in \EE(S)$,  see, for instance, \cite{Lawson}*{Theorem 1.4.6}. 

\begin{definition}\label{def:2-cocycles_inverse_semigroups}
A   \emph{$2$-cocycle over an inverse semigroup} $S$ with coefficients in $\T=\{z\in \mathbb{C}: |z|=1\}$ is a family $\omega=\{\omega(s,t)\}_{s,t\in S, st\neq 0}\subseteq \T$ 
such that   for all $r,s,t\in S$ with $rst\neq 0$ we have
\begin{equation}\label{eq:inverse_semigroup_cocycle}
	\omega(s,t) \omega(r,s t) =  \omega(r,s) \omega(rs,t).
\end{equation}
If, in addition, $\omega(e,e)=1$ for all $e\in \EE(S)\setminus\{0\}$, then we say $\omega$ is \emph{normalised}, and we call the pair $(S,\omega)$ a \emph{twisted inverse semigroup}. 
%
\end{definition}
\begin{remark}\label{rem:Lausch_cocycles} 
Lausch \cite{Lausch} associated $2$-cocycles to inverse semigroup extensions (using them to classify such extensions). A $2$-cocycle of Definition~\ref{def:2-cocycles_inverse_semigroups} can be identified with a Lausch $2$-cocycle associated to the trivial  
extension of $S$ by the abelian inverse semigroup 
$$
K \coloneqq (\EE(S)\times \T)/(\{0\}\times \T)\cong (\EE(S)\setminus\{0\} \times \T)\cup \{0\},
$$ where we identify $\{0\}\times \T$ with the zero element.
Indeed, see \cite{Steinberg}*{Subsection 3.2}, this extension is equivalent to the $S$-module structure on $K$ given by the obvious projection $p:K\to E(S)$ and the left action of $S$ where
$
s \cdot(e,\lambda) \coloneqq (ses^*,\lambda) 
$ if $se\neq 0$, and $s \cdot(e,\lambda) \coloneqq 0$ otherwise. 
By definition $2$-cocycles associated to the $S$-module  $K$ are 
maps $c:S\times S\to K$ satisfying $p(c(s,t))=st(st)^*$ and 
\[
r \cdot c(s,t) c(r,s t) =  c(r,s) c(rs,t), \qquad r,s,t \in S.
\]
We have a bijective correspondence between $2$-cocycles $\omega:S\times S\to \T$ over $S$ and 
$2$-cocycles $c:S\times S\to K$ associated to the $S$-module $K$, which is given by the formula 
$$
c(s,t)=\begin{cases}
	(st (st)^*,\omega(s,t)) & \text{ if }st\neq 0,
	\\
	0 &\text{ otherwise}.
\end{cases}
$$
A $2$-cocycle $c:S \times S \to K$ is normalised if $c(e,e)\in \EE(K)=\EE(S)\times\{1\}$ for all $e\in \EE(S)$. The bijection between $2$-cocycles restricts to one between normalised $2$-cocycles.
By \cite{Steinberg}*{Proposition 3.10}, any $2$-cocycle is cohomologous to a normalised one.
\end{remark}
\begin{lemma}\label{lem:cocycles_properties} Let $\omega$ be a normalised $2$-cocycle $\omega$ over $S$. For all $s\in S$, $e, f\in \EE(S)$ we have
\begin{enumerate}
	\item\label{ite:cocycles_properties1} $\omega(s,s^*)=\omega(s^*,s)$;
	\item\label{ite:cocycles_properties2} $ef\neq 0$ $\Longrightarrow$ $\omega(e,f)=1$;
	\item\label{ite:cocycles_properties3} $0\neq s^*s \leq e$ $\Longrightarrow$ $\omega(s,e)=\omega(e,s^*)=1$; and
	\item\label{ite:cocycles_properties4} $se\neq 0$  $\Longrightarrow$ $\omega(s,e)=\omega(ses^*,s)$.
\end{enumerate}  
\end{lemma}
\begin{proof}
Using the correspondence from Remark~\ref{rem:Lausch_cocycles}, the assertion follows from \cite{Steinberg}*{Proposition 3.11} (see items (2), (4), (7), (9) therein).
One can also simply follow the proof of \cite{Steinberg}*{Proposition 3.11}.
\end{proof}

Throughout this paper $X$ will stand for a locally compact Hausdorff space. The set  $\PHomeo(X)$ of \emph{partial homeomorphisms of $X$} (homeomorphisms between open
subsets of $X$) is an inverse semigroup under composition of partial maps, with the empty map as a zero. 

An \emph{action of an inverse semigroup} $S$ on $X$ is a zero-preserving semigroup homomorphism $h:S\to\PHomeo(X)$ that is nondegenerate in the sense that $X$ is the union of the
domains of the elements of $h(S)$. Thus,  if $X_{t^*}$ denotes the domain of $h_t$, then the action $S$ consists of  a family  of  homeomorphisms $h_t : X_{t^*} \to X_{t}$ such that $h_t \circ h_s = h_{t s}$ (as partial maps) for all $s,t\in S$, $h_0=\varnothing$, and  $\bigcup_{t\in S} X_t = X$.

Let $\Cont_0(X)$ be the $\Cst$-algebra of continuous functions that vanish at infinity on $X$. 
Then the continuous bounded functions $\Contb(X)$ form its multiplier algebra, and $\Contu(X)  \coloneqq  \{a\in \Contb(X): |a|\equiv 1\}$ is its (abelian) group of unitary multipliers. An inverse semigroup action $h : S \to \PHomeo(X)$  is equivalent to an action $\alpha : S \to \PAut(\Cont_0(X))$ of $S$ by partial automorphisms of $\Cont_0(X)$, where for each $t\in S$, $\alpha_t:\Cont_0(X_{t^*})\to \Cont_0(X_t)$ is a partial automorphism of $\Cont_0(X)$ given by $\alpha_t(a) = a\circ h_{t^*}$, $a\in \Cont_0(X_{t^*})$.  Thus,  the following is a special  (commutative) case of a twisted inverse semigroup action introduced by Buss and Exel.
\begin{definition}[cf. \cite{Buss_Exel}*{Definition 4.1}]
\label{def:twist_semigroup_action}
A \emph{twist}  of an action $h : S \to \PHomeo(X)$
is a family $u=\{u(s, t)\}_{s,t\in S}$ of continuous unitary multipliers such that $u(s, t)\in \Contu(X_{(s t)^*})$ for all $s,t \in S$, and such that for all $r,s,t\in S$ and $e,f\in \EE(S)$:
\begin{enumerate}[label={(A\arabic*)}]
	\item\label{enu:twisted actions2} $u(s,t)(h_{r^*}(x)) u(r,s t)(x) = u(r,s)(x) u(rs,t)(x)$ for all $x\in X_{r} \cap X_{(s t)^*}$; 
	\item\label{enu:twisted actions3} $u(e,f) = 1_{X_{e f}}$ and $u(t,t^*t) = u(t t^*,t) = 1_{X_t}$, where $1_{X_t}$ is the unit of $\Contu(X_t)$; and
	\item\label{enu:twisted actions4} $u(t^*,e)(x) u(t^*e,t)(x)= u(t^*,t)(x)$ for all $x\in X_{t^* e t}$. 
\end{enumerate}
Since $X_{0}=\varnothing$, by convention we  have  $1_{X_{0}}=0$ and $u(s,t)=0\in \Contu(\varnothing)=\{0\}$ if $st=0$.
We call the pair $(h,u)$ a\emph{ twisted inverse semigroup action}.
\end{definition}
\begin{remark}\label{twists_over_transformation_groupoids}
Initially
Sieben \cite{Sieben} introduced  twisted inverse semigroup actions using stronger conditions. Namely, instead of \ref{enu:twisted actions3} and \ref{enu:twisted actions4} Sieben requires 
that $u(s,t) = 1_{X_{s t}}$ whenever $s$ or $t$ is an idempotent. As noticed in \cite{Buss_Exel}, see  \cite{Buss_Exel}*{Theorem~7.2} or \cite{BKM}*{Subsection~3.4}, 
the  twists of $h$ of Definition~\ref{def:twist_semigroup_action} correspond to Kumjian twists of the transformation groupoid $S\ltimes_{h} X$, defined via Fell line bundles.   Sieben twists of $h$  give rise to topologically trivial twists, or equivalently to groupoid $2$-cocycles on $S\ltimes_{h} X$ (we give definitions below), see \cite{Buss_Exel}*{Proposition 7.4}. 
\end{remark}

A twist over $S$ can be used to twist all its actions.

\begin{proposition}\label{prop:semigroup_twists_to_action_twists}
A normalised $2$-cocycle $\omega$ over $S$ induces a twist $u$ for an action $h:S\to \PHomeo(X)$ 
by setting $u(s,t)  \coloneqq  \omega(s,t)1_{X_{(st)^*}}$ for all $s,t\in S$ with $st\neq 0$.
\end{proposition}
\begin{proof}Axioms~\ref{enu:twisted actions2} and \ref{enu:twisted actions3} follow from \eqref{eq:inverse_semigroup_cocycle} and 
Lemma~\ref{lem:cocycles_properties}, items \ref{ite:cocycles_properties2}, \ref{ite:cocycles_properties3}.
To show \ref{enu:twisted actions4}, let $s\in S$ and $e\in \EE(S)$ be such that $es\neq 0$. Then $s^*es\neq 0$ and
\begin{align*}
	\omega(s^*,e)\omega(s^*e, s)&\stackrel{\ref{lem:cocycles_properties}\ref{ite:cocycles_properties4}}{=}\omega(s^*es,s^*)\omega(s^*e, s)\stackrel{\eqref{eq:inverse_semigroup_cocycle}}{=}\omega(s,s^*)\omega(s^*e,ss^*)
	\\
	&\stackrel{\ref{lem:cocycles_properties}\ref{ite:cocycles_properties3}}{=}\omega(s,s^*)\stackrel{\ref{lem:cocycles_properties}\ref{ite:cocycles_properties1}}{=}\omega(s^*,s). \qedhere
\end{align*}
\end{proof}
\subsection{Twisted \'etale groupoids from twisted inverse semigroups}
An  \emph{\'etale groupoid} is a topological groupoid $\Gg$  with a locally compact Hausdorff unit space $X=\Gg^{0}$ and such that  
the range and source maps $\rg,\sr:\Gg\to X$ are locally injective open maps. 
This implies that $\Gg$ is a locally compact, locally Hausdorff space and $X$ is an open subspace of $\Gg$. 
The groupoid $\Gg$ is  Hausdorff if and only if $X$ is closed in $\Gg$.
We denote the composable pairs in $\Gg$ by $\Gg^2 = \{(\gamma,\eta) \in \Gg \times \Gg \mid \sr(\gamma) = \rg(\eta)\}$. 
A  \emph{bisection} of $\Gg$ is an open set on which $\rg$ and $\sr$ are injective. 
The family $\Bis(\Gg)$ of bisections of $\Gg$
form a unital inverse semigroup where composition is given by  
$$UV  \coloneqq  \{ \gamma \eta \mid (\gamma,\eta) \in \Gg^2 \cap (U \times V)\}$$
 for all $U,V \in \Bis(\Gg)$, the generalised inverse of $U$ is  $U^* = U^{-1} \coloneqq \{\gamma^{-1}\mid \gamma\in \Gg\}$, $X$ is the unit, and $\emptyset$ is the zero element in $\Bis(\Gg)$. \'Etalness of $\Gg$ implies that $\Bis(\Gg)$ covers $\Gg$.

By a \emph{twist} over \(\Gg\) we mean a  Fell line bundle $\LL$ over $\Gg$ in
the sense of \cite{Kumjian}. Thus,  \(\LL=(L_\gamma)_{\gamma\in \Gg}\) is a
locally trivial bundle of one-dimensional complex Banach spaces, together with
multiplication maps \(L_{\gamma}\times L_{\eta}\ni (z_{\gamma},z_{\eta})\mapsto
z_{\gamma}\cdot z_{\eta}\in L_{\gamma\eta}\), \((\gamma,\eta)\in \Gg^2\), and
involution maps \(L_{\gamma}\ni z\mapsto  \overline{z}\in L_{\gamma^{-1}}\),
$\gamma\in \Gg$, that are continuous and consistent with each other in a natural
way. In particular, we assume that $\LL|_X = X \times \mathbb{C}$ is trivial.
Equipping \(\Sigma\defeq\setgiven{z\in \LL}{|z|=1}\)  with the topology and
multiplication  from \(\LL\) yields a topological groupoid, with a natural exact
sequence \(X\times \T\into \Sigma \onto \Gg\) that turns \(\Sigma\) into a central
\(\T\)-extension of \(\Gg\). Every central \(\T\)-extension of \(\Gg\) arises
this way, see \cite{Kumjian}*{Example 2.5.iv}. This gives an equivalence between
Fell line bundles and  twists in the sense of Kumjian-Renault
\cites{Kumjian:Diagonals, Renault:Cartan}.

We say that a Fell line bundle $\LL$ is \emph{topologically trivial} if it is trivial as  a bundle, that is if $\LL\cong \Gg\times \mathbb{C}$ as a vector bundle. 
Such Fell bundles are equivalent to $2$-cocycles.
\begin{example}
A (normalized) \emph{groupoid $2$-cocycle}  is a continuous function \(\sigma\colon \Gg^2 \to \T\) such that
\begin{equation}\label{eq:groupoid_cocycle_identities}
\sigma(\alpha,\beta)\sigma(\alpha\beta,\gamma) = \sigma(\beta,\gamma)\sigma(\alpha,\beta\gamma) \quad \text{and} \quad    \sigma \big( \rg(\gamma),\gamma \big) = 1 = \sigma \big( \gamma , \sr(\gamma) \big),
\end{equation}
for every composable triple $\alpha,\beta,\gamma \in \Gg$, see  \cite{Renault}. For any such $\sigma$
the trivial bundle $\Gg\times \mathbb{C}$ becomes a Fell line bundle with operations given by 
$
(\alpha,w) \cdot (\beta,z)  \coloneqq  \big( \alpha \beta , \sigma(\alpha,\beta) wz \big)$ 
and  
$(\alpha,w)^{*}  \coloneqq  \big( \alpha^{-1} , \overline{\sigma(\alpha^{-1},\alpha) w } \big) .
$ Every topologically trivial Fell bundle comes from a $2$-cocycle.
\end{example}
To each inverse semigroup action
$h:S\to\PHomeo(X)$ we associate the \emph{transformation groupoid} $S \ltimes_{h} X$, defined as follows (see \cite{Paterson}*{p.\ 140} or \cite{Exel}*{Section~4}).
The arrows of 
$$
S\ltimes_{h} X=\{[t,x]: x\in X_{t^*},\, t\in S\}
$$ are equivalence classes of pairs $(t,x)$ for $x\in X_{t^*} \subseteq X$; where two pairs $(t,x)$ and $(t',x')$ are equivalent if $x = x'$ and there is $v\in S$ with $v\le t, t'$ and $x \in X_{v^*}$. The unit space of $S\ltimes_{h} X$ is $\{[e,x]:x\in X_e, e\in \EE(S)\}$, which is naturally identified with $X$ via the map $x\mapsto [e,x]$ for any $e\in \EE(S)$ with $x\in X_e$.
The range and source maps $\rg,\sr: S\ltimes_{h} X \rightrightarrows X$, and multiplication are defined by 
\[
\rg([t,x])  \coloneqq  h_t(x), \quad \sr([t,x])  \coloneqq  x, \quad \text{and} \quad [s,h_t(x)] \cdot [t,x] = [s\cdot t,x],\, x \in X_{(st)^*}.
\] 
We give $S\ltimes_{h} X$ the unique topology such that the unit space $X$ is equipped with the original topology and the sets $U_t  \coloneqq  \{ [t,x] : x \in X_{t^*} \}$ are open bisections of $S\ltimes_{h} X$. 
Then $S\ltimes_{h} X$ is an \'etale groupoid and  the map $S\ni t \mapsto U_t \in \Bis(S\ltimes_{h} X)$ is a zero preserving semigroup homomorphism.

\begin{remark}\label{rem:transformation_subgroupoids}
A subset $Y\subseteq X$ is\emph{ $h$-invariant} if $h_t(Y\cap X_{t^*})=Y\cap X_{t}$ for every $t\in S$. It 
follows from the above construction that for any $h$-invariant set $Y$ we may view $S\ltimes_{h} Y$ as a subgroupoid of $S\ltimes_{h} X$.
Moreover, this subgroupoid is open (resp. closed) in $S\ltimes_{h} X$ if and only if  $Y$ is open (resp. closed) in $X$.
Similarly, if $S_0\subseteq S$ is an inverse subsemigroup is \emph{wide}, that is it  contains all idempotents $\EE(S)$ in $S$, then the preorders in $S_0$ and $S$ are compatible and 
therefore, we may treat  $S_0\ltimes_{h} X$ as an open subgroupoid of $S\ltimes_{h} X$.
\end{remark}
By a  \emph{$1$-cocycle} (or a \emph{prehomomorphism} \cite{Lawson}) with values in a discrete group  $\Gamma$  we mean a map $c:S\setminus \{0\}\to \Gamma$ satisfying $c(st)=c(s) c(t)$ whenever $st\neq 0$.
\begin{lemma}\label{lem:from_cocycles_to_homomorphisms}
For any $1$-cocycle $c:S\setminus \{0\}\to \Gamma$ and any action $h:S\to\PHomeo(X)$, the formula 
$\widetilde{c}[t,x]=c(t)$ defines a continuous groupoid homomorphism $\widetilde{c}: S\ltimes_{h} X \to \Gamma$. Let $S_0 \coloneqq  c^{-1}(1)\cup \{0\}$. Then there is an 
isomorphism of groupoids
$
S_0\ltimes_{h} X \cong \widetilde{c}^{-1}(1),
$
so we may treat $S_0\ltimes_{h} X$ as a clopen wide subgroupoid of $S\ltimes_{h} X$. 
\end{lemma}
\begin{proof}
Assume $[t,x]=[t',x]$, so that there is $v\in S$ with $v\le t, t'$ and $x \in X_{v^*}$. In particular, $v\neq 0$
and there are idempotents $e,f\in E(S)\setminus \{0\}$ such that $v=te=t'f$. 
This implies that $c(v)=c(t)=c(t')$ as we  have $c(e)=c(f)=1$. Hence, $\widetilde{c}$ is well-defined and $\widetilde{c}|_{U_t}=c(t)$ for $t\in S$.
As $\widetilde{c}$ is locally constant, it is continuous, and since $U_tU_s=U_{st}$ it is a groupoid homomorphism. 
Moreover, $\widetilde{c}^{-1}(1)=\bigcup_{t\in c^{-1}(1)}U_t$ is a clopen wide subgroupoid of $S\ltimes_{h} X$, which can be identified with $S_0\ltimes_{h} X$.
\end{proof}

A twist $u=\{u(s, t)\}_{s,t\in S}$ of $h$, as in Definition~\ref{def:twist_semigroup_action}, 
induces a twist $\LL_{u}$ over $S\ltimes_{h} X$, see \cite{BKM2}*{Subsection 4.3}. Elements of $\LL_u$ are equivalence classes of triples $(a,t,x)$ for
$a\in \Cont_0(X_{t})$, $x\in X_{t^*}$, $t\in S$, where two triples $(a,t,x)$ and $(a',t',x')$ are
\emph{equivalent }\label{page:germ_relations} if 
\begin{enumerate}
\item $x=x'$ and there is $v\in S$ with $v\leq t,t'$, where $x\in X_{v^*}$; and
\item $(a \cdot u(vv^*,t))(h_{v}(x)) = (a' \cdot u(vv^*,t'))(h_{v}(x))$.  
\end{enumerate}
In particular, there is a canonical surjection $\LL_u\ni  [a,t,x] \mapsto [t,x]\in S\ltimes_{h} X$.
We equip $\LL_u$ with the unique topology  such that the local sections $[t,x] \mapsto [a,t,x]$, for $x\in X_{t^*}$, $a\in \Cont_0(X_t)$, $t\in S$, are continuous.
This makes $\LL_u$ a Fell line bundle with operations defined by 
\begin{align*}
\alpha[a,t,x] + \beta[b,t,x] & \coloneqq  [\alpha a+\beta b,t,x],
&\big| [a,t,x] \big| & \coloneqq  \big| a \big( h_{t}(x) \big) \big|,\\
[a,s,h_t(x)] \cdot [b,t,x] & \coloneqq  [a (b\circ h_{s^*}) u(s,t),st,x],
&\overline{[b,t,x]}& \coloneqq  [\overline{b\circ h_{t}}\cdot \overline{u(t,t^*)} , t^* , h_{t}(x)],
\end{align*}
for all $a\in C_0(X_s)$, $b\in C_0(X_t)$, $\alpha,\beta \in \mathbb{C}$, $x\in X_{(st)^*}$, and $s,t\in S$. 
\begin{definition}\label{defn:groupoid_twists_from_semigroup_twists}
Given an inverse semigroup action $h:S\to \PHomeo(X)$ and a twist $u$ of $h$, we call the pair $(S\ltimes_{h} X,\LL_{u})$ described above the \emph{twisted transformation  groupoid} of $(h,u)$.
If  $u$ comes from a  twist $\omega$ of $S$ as in Proposition~\ref{prop:semigroup_twists_to_action_twists}, then we write $\LL_{\omega} \coloneqq \LL_u$.
\end{definition}

\begin{remark}\label{rem:2-cocycle_groupoid} Every twisted \'etale groupoid is the transformation groupoid for some twisted inverse semigroup action, see, for instance \cite{BKM}*{Lemma 4.21}. However, the above twists $\LL_{\omega}$  are rather special as they come from
twisted inverse semigroups.
In the present paper we will be mostly interested in ample groupoids. It is well known, see for instance \cite{Kwasniewski-Meyer-Prasad}*{Remark 4.28},  that for ample Hausdorff and $\sigma$-compact groupoids  every line bundle is topologically trivial, so the groupoid twist might
be identified with a groupoid $2$-cocycle. The problem, however, is that there is no canonical choice of such a cocycle.
In addition, if the ample groupoid is non-Hausdorff, then one can easily construct topologically nontrivial twists, see Example~\ref{ex:nontrivial_twist} below.
Therefore, we are forced to work with Fell line bundles, rather than $2$-cocycles.
\end{remark}
\subsection{Groupoids from inverse semigroups}
Inverse semigroups act naturally on the spectra of semilattices of their idempotents.  
We  recall these standard constructions. 
Let $\EE \coloneqq \EE(S)$ denote  the semilattice of  idempotents in a fixed inverse semigroup $S$.
A \emph{cover} of an idempotent $e\in \EE$ is a finite set $F\subseteq \EE$ such that for every nonzero $z\leq e$ we have $z\cdot f\neq 0$ for some $f\in F$.
A \emph{homomorphism} from $\EE$ to a Boolean ring $R$ is a map $\phi:\EE\to R$ that preserves $0$ and the meet operations. A map  $\phi : \EE\to R$  is \emph{tight}  (\emph{cover preserving} or \emph{cover-to-join}) if the element $\phi_e - \bigvee_{f\in F} \phi_f = \bigwedge_{f\in F} (\phi_e\setminus \phi_f)$ is zero whenever $F$ covers $e\in \EE$.  By definition $\varnothing$ covers $0$ and so tight maps always preserve zero.

A (proper) \emph{filter} in $\EE$ is an upward-closed and downward directed  subset of $\EE\setminus \{0\}$. 
There is a bijective correspondence between filters and nonzero homomorphisms $\phi:\EE\to \{0,1\}$ (\emph{characters} of $\EE$), given by 
$\phi \leftrightarrow  \supp(\phi )  \coloneqq  \{ e\in \EE : \phi(e) = 1 \}$. The \emph{spectrum} of the semilattice $\EE$ is the set
\[
\widehat{\EE}
 \coloneqq \{ \xi \subseteq \EE : \xi \text{ is a  filter} \}=\{ \supp(\phi): \phi : \EE \to \{0,1\} \ \text{is a  homomorphism} \}
\]
equipped with topology inherited from $\{0,1\}^{\EE}$. Namely, putting  $Z(e)  \coloneqq  \{ \xi \in \widehat{\EE} : e\in \xi \}$ for  $e \in \EE$, the sets $Z(e) \setminus \bigcup_{f \in F} Z(f)$, ranging over all $e\in \EE$ and finite $F\subseteq e\EE$, constitute a basis of compact open sets of $\widehat{\EE}$, turning it into a totally disconnected locally compact Hausdorff space. The  inverse semigroup $S$ acts naturally on $\widehat{\EE}$ by  partial homeomorphisms 
$\widetilde{h}_{t} : Z(t^*t)\to Z(tt^*)$ given by the formulae
\begin{equation}\label{eq:inverse_semigroup_action_spectrum}
\widetilde{h}_{t}(\xi)  \coloneqq \{e\in \EE: t^*et\in \xi\}= \{ e \in \EE : \text{$f \geq t ft^*$  for some $f\in \xi$} \} ,
\end{equation}
for all $t\in S$, $\xi\in \widehat{\EE}$. In terms of characters \eqref{eq:inverse_semigroup_action_spectrum} reads as $\widetilde{h}_{t}(\phi)(e)  \coloneqq  \phi ( t^*et)$ for $e\in \EE$, see \cite{Exel}*{Proposition 10.3}.
A  character $\phi\in\widehat{\EE}$ is tight if and only if the corresponding filter $\xi \subseteq \EE$ is \emph{tight} in the sense that for every $e \in \xi$ and every cover $F$ of $e$ we have $F\cap \xi\neq \varnothing$.  
The \emph{tight} \emph{spectrum} of the semilattice $\EE$ is the following  subspace of $\EE$
\[
\partial\widehat{\EE}
 \coloneqq \{ \xi \subseteq \EE : \xi \text{ is a tight filter} \}=\{ \supp(\phi): \phi : \EE \to \{0,1\} \ \text{is a tight homomorphism} \}.
\]
It can be shown that 
$\partial\widehat{\EE}$  is the closure of \emph{ultrafilters} (maximal filters) in  
$\widehat{\EE}$, see  \cite{Exel}*{Theorem 12.9}, and $\partial\widehat{\EE}$ is invariant under the semigroup action $\widetilde{h}$, see \cite{Exel}*{12.8}.
Hence, we have  two actions  $\widetilde{h}:S\to\PHomeo( \widehat{\EE})$ and $h:S\to\PHomeo( \partial\widehat{\EE})$ and two transformation groupoids
\[
\widetilde{\Gg}(S) \coloneqq S \ltimes_{\widetilde{h}} \widehat{\EE}  \qquad \text{ and } \qquad \Gg(S) \coloneqq S \ltimes_{h} \partial\widehat{\EE}. 
\]

\begin{definition}\label{def:universal_and_tight_groupoids}
The groupoid $\widetilde{\Gg}(S)$ is called \emph{universal} or (contracted) \emph{Paterson groupoid} for $S$. 
The groupoid $\Gg(S)$, which is the restriction of $\widetilde{\Gg}(S)$ to $\partial\widehat{\EE}$,  is called  the \emph{tight groupoid} of $S$, see \cites{Exel, Steinberg0, Steinberg_Szakacs}.
\end{definition}

The preceding concepts are well-illustrated in the context of directed graphs. We establish the relevant definitions and notation next, to be used in the subsequent sections.

\begin{example}[Directed graphs]\label{ex:spectrum_graph_inverse_semigroup}
Let $E = (E^0,E^1,\rg,\sr)$ be a \emph{directed graph} where $E^0$ is the set of \emph{vertices}, $E^1$ is the set of \emph{edges}, and $\rg,\sr \colon E^1 \to E^0$ are the \emph{range} and \emph{source} maps. For $n\geq 1$ we denote by  
$E^n  \coloneqq  \{\mu = \mu_1\cdots \mu_n \mid \mu_i \in E^1,\, \sr(\mu_i) = \rg(\mu_{i+1})\}$, the collection of \emph{paths of length $n$} and adopt the convention that a vertex is a path of length $0$.
The range and source maps extend to  $\rg,\sr \colon E^n \to E^0$ by $\rg(\mu_1 \cdots \mu_n) = \rg(\mu_1)$ and $\sr(\mu_1 \cdots \mu_n) = \sr(\mu_n)$, and the range map extends to 
the set $E^{\infty}  \coloneqq  \{\mu = \mu_1 \mu_2 \cdots \colon \mu_i \in E^1, \sr(\mu_i) = \rg(\mu_{i+1})\}$ of \emph{infinite paths}. On $E^0$ we set $\rg(v) = \sr(v) = v$ for all $v \in E^0$.

We treat  $E^* \coloneqq \bigcup_{n=0}^\infty E^n$ as a small category---the \emph{path category} of $E$---with objects $E^0$ and composition given by the concatenation of paths: 
if $\mu \in E^m$ and $\nu \in E^n$ with $\sr(\mu) = \rg(\nu)$, then $\mu\nu  \coloneqq  \mu_1 \cdots \mu_m \nu_1 \cdots \nu_n \in E^{m+n}$.
Let $v,w \in E^0$. For $m \in \N \cup \{\infty\}$ we write $vE^m = \{\mu \in E^m \mid \rg(\mu) = v\}$ and for $m \in \N$ we write $E^m w = \{ \mu \in E^m \mid \sr(\mu) = w\}$  and $v E^m w = vE^m \cap E^m w$.  
The \emph{path space} or \emph{spectrum 
	of the graph $E$} is  the set 
\[
E^{\le \infty} \coloneqq E^*\cup E^{\infty}
\]
equipped with topology generated by relative complements of  cylinder sets  $Z(\alpha)  \coloneqq  \{ \alpha\mu \in  E^{\le \infty} \mid \mu \in  E^{\le \infty}\}$ 
where  $\alpha\in E^*$. Namely, the sets $Z(\alpha) \setminus \bigcup_{\beta\in F} Z(\alpha\beta)$, where $\alpha\in E^*$ and $F\subseteq \sr(\alpha)E^*$ is finite, from a basis for the topology on
$E^{\le \infty}$. 

We call a vertex $v\in E^0$ a \emph{finite receiver} if $vE^1$ is finite and an \emph{infinite receiver} otherwise.
If $vE^1 = \varnothing$ we call $v$ a \emph{source}. 
We say that a vertex is \emph{singular} if it is a source or infinite receiver. Otherwise we say that it is \emph{regular}.
We call an element in $E^*$ a \emph{boundary path} if its source is a singular vertex and we denote the collection of boundary paths in $E^*$ by $E^*_{\sing}$.
Infinite paths in $E^\infty$ are also considered to be boundary paths.  
The subspace
\[
\partial E \coloneqq  E^*_{\sing} \cup E^{\infty} \subseteq E^{\le \infty}
\]  
is called the \emph{tight spectrum} or a \emph{boundary path space} of $E$. 
It is the closure in $E^{\le \infty}$ of the set of all ``maximal paths'' $E^*_{\src} \cup E^{\infty}$, where $E^*_{\src}$ is the set of all paths that start in a source vertex (and hence cannot be extended).

The spaces $E^{\le \infty}$ and $\partial E$ naturally arise from the standard inverse semigroup associated to $E$. This  inverse semigroup is $S(E) \coloneqq E^*\fibre{\sr}{\sr}E^*\cup \{0\}$
with multiplication 
\[
(\alpha,\beta) (\gamma,\delta) \coloneqq 
\begin{cases} (\alpha \beta', \delta) & \text{if } \gamma=\beta\beta'
	\\
	(\alpha,  \delta \gamma') & \text{if } \beta=\gamma\gamma'
	\\
	0 & \text{otherwise}
\end{cases}
\]
and involution $(\alpha,\beta)^* =(\beta,\alpha)$. The semilattice of idempotents  $\EE(S(E))=\{(\alpha,\alpha):\alpha\in E^*\}\cup \{0\}$ is homeomorphic, 
via  $(\alpha,\alpha)  \mapsto \alpha$, to the semilattice $E^*\cup \{0\}$ with the partial order determined  by 
\[
\alpha\leq \beta\,\, \stackrel{def}{\Longleftrightarrow}\,\, \alpha=\beta\beta' \text{ for some }\beta'\in \sr(\beta)E^*.
\] 
If $\alpha \le \beta$, then we say that $\alpha$ \emph{extends} $\beta$ or that $\beta$ is a \emph{prefix} of $\alpha$. We say that $\alpha$ and $\beta$ are \emph{comparable} if they are comparable in the partial order, that is either $\alpha \le \beta$ or $\beta \le \alpha$. 

A set $F\subseteq E^*$ \emph{covers} $\alpha\in E^*$, in the semigroup $E^*\cup\{0\}\cong \EE(S(E))$, if and only if
each extension of $\alpha$ is comparable with some element of $F$.
Thus,  the set $vE^1$ is a finite cover of  $v\in E^0$ if and only if $v$ is regular.
Moreover, every filter in $E^*\cup \{0\}$ is either of the form
$\phi_{\mu} \coloneqq \{\alpha\in E^*: \mu \leq \alpha\}$ for  $\mu \in E^*$, or  $\phi_{\mu} \coloneqq \{\alpha\in E^*: \mu_1\cdots\mu_n \leq \alpha \text{ for some } n\in \N\}$ for   $\mu=\mu_1\mu_2\cdots \in E^\infty$. 
Thus,  the spectrum of $E^*\cup \{0\}$ consists of the sets 
\[
\phi_{\mu}=\{\alpha\in E^*: \bone_{Z(\alpha)}(\mu)=1\}, \qquad \mu\in E^{\le \infty}.
\]
Furthermore, $\phi_{\mu}$ is an ultrafilter if and only if $\mu\in E^*_{\src} \cup E^{\infty}$, and $\phi_{\mu}$ is  tight if and only if $\mu\in \partial E$.
It follows that we have natural homeomorphisms
$$
E^{\le \infty}\cong \widehat{E^*\cup \{0\}}\cong \widehat{\EE(S(E))} \qquad \text{and} \qquad \partial E\cong \partial \widehat{E^*\cup \{0\}}\cong \partial\widehat{\EE(S(E))},
$$
given by
$E^{\le \infty}\ni \mu\mapsto \phi_{\mu}\in \widehat{E^*\cup \{0\}}$  and its restriction.
Moreover, we have 
$[\alpha,\beta; \phi_{\xi}]=[\gamma,  \delta; \phi_{\eta}]$ in $\widetilde{\Gg}(S(E))$ if and only if  there are $\beta',\delta' \in E^*$ such that
$\alpha \beta'=\gamma \delta'$, $\beta\beta'=\delta\delta',
$ and  $\xi= \eta\in Z(\beta\beta')=Z(\delta\delta')$. Using this, one sees that 
$
\widetilde{\Gg}(S(E))\ni [(\alpha,\beta), \phi_{\beta \xi}]\longmapsto (\alpha\xi,|\alpha|-|\beta|, \beta\xi)\in \widetilde{\Gg}(E)
$
is an isomorphism of topological groupoids, where 
$$
\widetilde{\Gg}(E) \coloneqq \{(\alpha\xi,|\alpha|-|\beta|, \beta\xi):  \alpha,\beta \in E^*, \xi\in E^{\le \infty},\ \sr(\alpha) = \sr(\beta) = \rg(\xi) \}
$$ 
is equipped with the topology generated by the `cylinders' $Z(\alpha,\beta)  \coloneqq  \{ (\alpha \xi, |\alpha|-|\beta|, \beta \xi) \in \widetilde{\Gg}(E) \}$, for $(\alpha,\beta)\in S(E)$, and their relative complements.  The algebraic structure is given by 
$    (\alpha,n,\beta) (\beta,m,\gamma)  \coloneqq  (\alpha,n+m,\gamma)$ and  $(\alpha,n,\beta)^{-1}  \coloneqq  (\beta,-n,\alpha)$.
The groupoid $\widetilde{\Gg}(E)$ contains 
$\Gg(E)  \coloneqq   \{(\alpha\xi,|\alpha|-|\beta|, \beta\xi) \in \widetilde{\Gg}(E) \colon  \xi \in \partial {E}\}$ as a  closed subgroupoid, and we have
$$
\widetilde{\Gg}(E) \cong \widetilde{\Gg}(S(E)), \qquad \Gg(E) \cong \Gg(S(E))
$$
These are the Deaconu--Renault groupoids associated to the one-sided shift map $\sigma_E \colon E^{\le \infty}\setminus E^0\to E^{\le \infty}$ 
and its restriction to $\partial E$, cf. \cite{Renault2}.
Consequently, these ample groupoids are amenable and Hausdorff.
\end{example}
\section{\texorpdfstring{$L^P$}{LP}-operator algebras associated to twisted groupoids}\label{sect:LP-groupoid_algebras}
Phillips \cites{Phi12, Phillips} defined
\emph{$L^p$-operator algebras}, for $p\in[1,\infty)$, as Banach algebras that can be isometrically represented (via bounded operators) on $L^p$-spaces. 
Here we will consider more general $L^P$-operator algebras parametrised by a subset $P\subseteq [1,\infty]$. Since we are interested only in the Banach space structure, it is useful to adopt the following notation. 
\begin{definition} A complex Banach space $Y$ is an  \emph{$L^p$-space}, for $p\in [1,\infty)$, if it is isometrically isomorphic to 
the Lebesgue space $L^p(\mu)$ associated to some measure $\mu$. We call $Y$ an \emph{$L^\infty$-space} if it is isometrically isomorphic to  
$\Cont_0(\Omega)$ for some locally compact Hausdorff space $\Omega$.  
\end{definition}
\begin{remark} $L^2$-spaces are nothing but Hilbert spaces.
The class of $L^\infty$-spaces as defined above is significantly larger than the class of spaces $L^{\infty}(\mu)$ for some measure $\mu$
(as, for instance, it contains the space $c_0$), and  it is easier to find nondegenerate representations on this larger class, see \cite{BKM}*{Remark~2.16} and Lemma~\ref{lem:making_representations_nondegenerate} below. This is one of the reasons why we consider such a more general class of spaces.
\end{remark}
\begin{definition}
By  a \emph{representation} $\pi:A\to B$ between two Banach algebras we mean a contractive homomorphism. If $B=\Bound(Y)$ for a Banach space $Y$ and $\overline{\pi(A)Y}=Y$ we say that $\pi$ is a \emph{nondegenerate representation} on $Y$.
\end{definition}
\begin{remark}
A homomorphism  between $C^*$-algebras is a representation (is contractive)  if and only if it is  $*$-preserving, cf. \cite{BKM}*{Remark 2.9}.
\end{remark}
\subsection{Twisted groupoid \texorpdfstring{$L^P$}{LP}-operator algebras}\label{sec:twisted_groupoid_algebras}
In this section we fix a twisted groupoid $(\Gg,\LL)$, where $\Gg$ is an \'etale groupoid  with locally compact Hausdorff unit space $X$ and  $\LL$ is a Fell line bundle over $\Gg$.
For each open \(U\subseteq \Gg\) we denote by  $\Contc(U,\LL)$ the space of continuous compactly supported sections of 
the restriction $\LL|_U$ of the bundle $\LL$  to the set $U$. 
The associated $*$-algebra is  defined on the set of \emph{quasi-continuous} compactly supported sections, which by definition are elements of
\[
\mathfrak{C}_c(\Gg,\LL)  \coloneqq  \linspan \{ f \in \Contc(U,\LL): U\in  \Bis(\Gg) \} , 
\]
where we treat sections of $\LL|_U$ as sections of $\LL$ that vanish outside $U$. If $\Gg$ is Hausdorff, then  $\mathfrak{C}_c(\Gg,\LL)=\Contc(\Gg,\LL)$ are usual continuous compactly supported sections. 
We equip $\mathfrak{C}_c(\Gg,\LL)$ with multiplication and involution given by 
$$ 
(f*g)(\gamma)  \coloneqq  \sum_{\rg(\eta) = \rg(\gamma)} f(\eta) \cdot g(\eta^{-1} \gamma), \qquad (f^*)(\gamma)  \coloneqq  f(\gamma^{-1})^* ,
$$
where $f,g \in \mathfrak{C}_c(\Gg, \LL)$. Since $X\in \Bis(\Gg)$ and  $\LL|_X=X\times \mathbb{C}$ is trivial, we get that $\Contc(X)$ is a $*$-subalgebra of $\mathfrak{C}_c(\Gg,\LL)$.
If the bundle $\LL$ is topologically trivial, so that the twist comes from a $2$-cocycle $\sigma: \Gg^2 \to \T$, then the algebra $\mathfrak{C}_c(\Gg,\LL)$ may be identified with the algebra 
$\mathfrak{C}_c(\Gg)   \coloneqq  \linspan \{ f \in \Contc(U): U\in  \Bis(\Gg) \}$ of quasi-continuous functions, and under this identification  
for $f,g\in \mathfrak{C}_c(\Gg)$ we have
\begin{equation}\label{eq:convolution_and_involution1}
(f*g)(\gamma)=\sum_{\alpha\beta=\gamma} f(\alpha)g(\beta)\sigma(\alpha,\beta) \quad \text{and} \quad  f^*(\gamma)= \sigma(\gamma,\gamma^{-1})^*f(\gamma^{-1})^*.
\end{equation}
In general, the  range and source $\rg,\sr : \Gg \to X$ induce the following  norms on $\mathfrak{C}_c(\Gg,\LL)$   
\[
\| f \|_{*s}  \coloneqq  \max_{x \in X} \sum_{\sr(\gamma) = x} \big| f(\gamma) \big| , \qquad \| f \|_{*r}  \coloneqq  \max_{x \in X} \sum_{\rg(\gamma) = x} \big| f(\gamma) \big| , \qquad \| f \|_I  \coloneqq  \max \{ \| f \|_{*s} , \| f \|_{*r} \} .
\]
We denote by $F_{*s}(\Gg,\LL)$, $F_{*r}(\Gg,\LL)$ and $F_{I}(\Gg,\LL)$  the corresponding completions of $\mathfrak{C}_c(\Gg,\LL)$ in the respective norms. 

\begin{definition} A \emph{representation} of $(\Gg,\LL)$ on an $L^p$-space $Y$, for some $p\in[1,\infty]$,  is a $\|\cdot\|_{I}$-contractive algebra homomorphism $\psi : \mathfrak{C}_c(\Gg,\LL)\to \Bound(Y)$. We say that $\psi$ is \emph{nondegenerate} if $\overline{\psi(\mathfrak{C}_c(\Gg,\LL))Y}=Y$ (which is equivalent to $\overline{\psi(\Cont_c(X))Y}=Y$).
\end{definition}
\begin{remark}For more general Banach spaces one may want to modify the above definition by replacing $\|\cdot\|_{I}$-contractiveness with some other norm condition, see \cite{BKM}.
\end{remark}
\begin{example}[Regular representations]\label{ex:regular_representation}
For $p\in [1,\infty)$, the Banach space $\ell^p(\Gg,\LL)$  of all sections of $\LL$ for which the norm $\|f\|_{p} = (\sum_{\gamma\in \Gg} |f(\gamma)|^{p})^{1/p}$  is finite is an $L^p$-space.  Similarly, the space of all bounded sections $\ell^{\infty}(\Gg,\LL)$ and its subspace of sections vanishing at infinity $c_0(\Gg,\LL)$ together with the norm 
$\|f\|_{\infty} = \sup_{\gamma\in \Gg}|f(\gamma)|$ are $L^\infty$-spaces. For any $p \in [1 , \infty]$ the convolution formula  
\[
\Lambda_p(f) \xi  \coloneqq  f * \xi, \qquad  \text{ for }f \in \mathfrak{C}_c(\Gg,\LL),\,\, \xi \in \ell^{p}(\Gg,\LL),
\] 
defines an injective representation  $\Lambda_p : \mathfrak{C}_c(\Gg,\LL) \to \Bound(\ell^{p}(\Gg,\LL))$, 
see \cite{BKM}*{Proposition 5.1}. For $p<\infty$  this representation is nondegenerate, while for $p=\infty$ it is not unless $X$ is compact.
In general, $\Lambda_{\infty}$ compresses to an injective nondegenerate representation  $\Lambda_{\infty} : \mathfrak{C}_c(\Gg,\LL) \to \Bound(c_0(\Gg,\LL))$.
\end{example}

\begin{definition}
For any nonempty $P\subseteq [1,\infty]$, we denote by $F^P(\Gg,\LL)$ and $F^P_{\red}(\Gg , \LL)$    completions of $\mathfrak{C}_c(\Gg,\LL)$ in  the norms
\[
\|f\|_{L^P}  \coloneqq  \sup \{ \| \psi(f) \|: \psi \text{ is a representation of }  (\Gg,\LL) \text{ on some }L^p\text{-space for }  p\in P\}. 
\]
and $\|f\|_{L^P,\red} \coloneqq \sup_{p\in P}\|\Lambda_p(f)\|$,  respectively. We call $F^P(\Gg,\LL)$ and $F^P_{\red}(\Gg , \LL)$   the \emph{full} and
the \emph{reduced  $L^P$-twisted groupoid algebra} of $(\Gg , \LL)$, 
respectively. When $P=\{p\}$ we write $F^p(\Gg,\LL) \coloneqq F^P(\Gg,\LL)$ and $F^p_{\red}(\Gg,\LL) \coloneqq F^P_{\red}(\Gg,\LL)$. 
\end{definition}
\begin{remark}\label{rem:general_comments_on_Lp}
For $p=2$, a map  $\psi : \mathfrak{C}_c(\Gg,\LL)\to \Bound(L^2(\mu))$ is a representation  if and only if $\psi$ is a $*$-homomorphism,
\cite{BKM}*{Corollary 4.27, Theorem 5.13}. 
In particular, $F^2(\Gg,\LL)=\Cst(\Gg,\LL)$  and $F^2_{\red}(\Gg,\LL)=\Cst_{\red}(\Gg,\LL)$ are the usual $\Cst$-algebras associated to $(\Gg,\LL)$. 
Reduced $L^p$-groupoid algebras were studied by many authors, see \cites{Gardella_Lupini17,cgt,Austad_Ortega,Hetland_Ortega,BKM}.
By \cite{BKM}*{Theorem 5.13} we always have
\[
F^1(\Gg,\LL) = F^1_{\red}(\Gg,\LL) = F_{*s}(\Gg,\LL) \quad \text{and} \quad F^{\infty}(\Gg,\LL) = F^{\infty}_{\red}(\Gg,\LL) = F_{*r}(\Gg,\LL)
\]
(and for $p=\infty$ one may consider the nondegenerate version of $\Lambda_{\infty}$ compressed to $c_0(\Gg,\LL)$). 
Thus,  if $P\subseteq [1,\infty]$ contains $\{1,\infty\}$, then $F^P(\Gg,\LL) = F^P_{\red}(\Gg,\LL) = F_{I}(\Gg,\LL)$, and if $P\subseteq \{1,\infty\}$, then $F^P(\Gg,\LL) = F^P_{\red}(\Gg,\LL)$.
If $P = P^* \coloneqq  \{ q : 1/p + 1/q = 1, p\in P \}$, then $F^P(\Gg,\LL)$ and $F^P_{\red}(\Gg,\LL)$ are Banach $*$-algebras with the standard involution, and the canonical homomorphism $F^P(\Gg,\LL)\to F^P_{\red}(\Gg,\LL)$ is  $*$-preserving. 
Such Banach $*$-algebras $F^P_{\red}(\Gg,\LL)$ were studied in \cites{Austad_Ortega, Elkiaer}.
The $L^P$-operator algebra for a set of parameters $P\subseteq [1,\infty]$ were first introduced  in  \cite{BK}*{Definition 4.12} for twisted group actions, and for
twisted groupoids in \cite{BKM2}. 
\end{remark}
\begin{remark}\label{rem:amenability}
Recall that an \'etale groupoid $\Gg$ is (topologically) \emph{amenable} if  there is a net $(\xi_i)_{i\in I}$ in $\mathfrak{C}_c(\Gg)$ such that
$\sup_{x\in X} \sum_{\gamma\in \rg^{-1}(x)}|\xi_i(\gamma)|\leq 1$ and the net of functions $\Gr\ni \gamma\mapsto \sum_{\rg(\eta)=\rg(\gamma)}\overline{\xi_i} (\eta)\xi_i(\gamma^{-1} \eta)$ 
converges compactly to $1$ on $\Gg$. For any non-empty $P\subseteq [1,\infty]$, it follows from
\cite{Gardella_Lupini17}*{Theorem 6.19}  that 
\begin{center}
	\emph{if $\Gg$ is amenable and second countable, then $F^P(\Gg)=F^P_{\red}(\Gg)$.}
\end{center}
In fact, we always have $F^P(\Gg,\LL) = F^P_{\red}(\Gg,\LL)$ if either $\{1,\infty\}\subseteq P$ or $P\subseteq \{1,\infty\}$, see Remark~\ref{rem:general_comments_on_Lp}.  
When  $\Gg$ is a transformation groupoid for a group action  amenability of $\Gg$ implies $F^P(\Gg,\LL) = F^P_{\red}(\Gg,\LL)$ for all  $P$ and $\LL$, see \cite{BK}.
It is also now known that $C^*(\Gg) = C^*_{\red}(\Gg)$ for any amenable groupoid $\Gg$, see \cites{Buss_Martinez_approx, Brix-Gonzales-Hume-Li:Hausdorff_covers}, and it could be deduced  from  results of \cite{Buss_Martinez_approx}
that $C^*(\Gg,\LL) = C^*_{\red}(\Gg,\LL)$ when $\Gg$ is amenable and $\LL$ is topologically trivial. 
But,  in general, the expected positive answer to the following problem still needs to be verified, even for $C^*$-algebras.
\end{remark}
\begin{question}\label{problem:amenability_implies_twisted_coincide}
Let $(\Gg,\LL)$ be a twisted  \'etale groupoid where  $\Gg$ is amenable. 
Does it follow that $F^{P}(\Gg,\LL)=F^{P}_{\red}(\Gg,\LL)$ for any non-empty $P\subseteq [1,\infty]$?
\end{question}

A bounded linear map $E:A\to B$ between two Banach algebras is \emph{faithful} if the only (closed two-sided) ideal in $A$ that is contained in $\ker E$, is the zero ideal.
We denote by  $\mathcal{B}(X)$  the  Banach algebra of all bounded Borel functions on $X$. 
\begin{definition}[\cite{BKM2}*{Definition 3.1}]  A Banach algebra completion $B$ of $\mathfrak{C}_c(\Gg,\LL)$ is called a \emph{reduced groupoid Banach algebra} if  the map $\mathfrak{C}_c(\Gg,\LL)\ni f \to f|_{X}$ 
extends to a faithful contractive linear map $B\to \mathcal{B}(X)$, which is isometric on $\Cont_c(X)$. 
\end{definition}
\begin{remark}\label{rem:reduced_Lp_algebras}  Let $\mathfrak{C}_0(\Gg , \LL)$ be the completion  of $\mathfrak{C}_c(\Gg , \LL)$ in the supremum norm $\| \cdot \|_{\infty}$.
A Banach algebra completion $B$ of $\mathfrak{C}_c(\Gg,\LL)$ is a reduced groupoid Banach algebra if and only if
the inclusion  $\mathfrak{C}_c(\Gg,\LL) \subseteq \mathfrak{C}_0(\Gg , \LL)$ extends to an injective contractive linear map $j:B\to \mathfrak{C}_0(\Gg , \LL)$,
 which is isometric on $\Cont_c(X)$, 
see \cite{BKM2}*{Remark 3.7}.  The algebras $F^P_{\red}(\Gg,\LL)$ 
are examples of reduced groupoid Banach algebras, see \cite{BKM2}*{Proposition 3.15}. 
\end{remark}
As we are mainly interested in non-Hausdorff groupoids we need $L^p$-versions of  essential $C^*$-algebras introduced in \cite{Kwasniewski-Meyer:Essential}.
To this end, we use the ``spatial construction'' from  \cite{BKM2}*{Proposition 4.16}.
As in \cite{BKM2}, following Dixmier's terminology,  we say that  $x\in Y$ is a \emph{Hausdorff point} in a topological space $Y$ 
if $x$ and  every
  $y \in Y \setminus \{ x \}$  can be separated by disjoint open sets.
\begin{example}[Essential representations]
By  \cite{BKM2}*{Lemma 4.4},  the set $\Gg_{\Hau}$ of all Hausdorff  points in $\Gg$ is a full subgroupoid of $\Gg$,  and  $\Gg_{\Hau}$ is comeager in $\Gg$ when $\Gg$ has a  countable cover by  open bisections or when $\Gg$ is topologically principal.
It follows that for any $p\in[1,\infty]$, the $L^p$-space $\ell^{p}(\Gg_{\Hau},\LL|_{\Gg_{\Hau}})$ is an invariant subspace for the regular representation  $\Lambda_p$ and hence the  $\Lambda_p$ restricts to a subrepresentation 
$\Lambda_p^{\ess}:F^p(\Gg , \LL) \to \Bound(\ell^{p}(\Gg_{\Hau},\LL|_{\Gg_{\Hau}}))$.
\end{example}
\begin{definition}
Assume $\Gg_{\Hau}$ is comeager in $\Gg$.
For any nonempty subset $P\subseteq [1,\infty]$ we denote by $F^P_{\ess}(\Gg , \LL)$ the   Hausdorff completion of $\mathfrak{C}_c(\Gg,\LL)$ in the seminorm
$$
\|f\|_{L^P,\ess} \coloneqq \sup_{p\in P}\|\Lambda_p^{\ess}(f)\|.
$$
We call $F^P_{\ess}(\Gg , \LL)$ the \emph{essential $L^P$-twisted groupoid algebra} for $(\Gg,\LL)$. When $P=\{p\}$  we  write $F^p_{\ess}(\Gg,\LL) \coloneqq  F^P_{\ess}(\Gg,\LL)$. 
\end{definition}
\begin{convention} Each time we mention $F^P_{\ess}(\Gg , \LL)$ we implicitly assume that  $\Gg_{\Hau}$ is comeager (which holds for instance if $\Gg$ can be covered by  a sequence of  bisections or when $\Gg$ is Hausdorff or when $\Gg$ is topologically principal, see \cite{BKM2}*{Lemma 4.4}).
\end{convention}
By a \emph{strict support} of a section $f:\Gg\to \LL$ we mean the set  $\supp(f)  \coloneqq  \{ \gamma \in \Gg : f(\gamma)\neq 0 \}$.
We denote by $\M(X)$ the ideal in  $\B(X)$  consisting of bounded Borel functions  with meager  strict support. The quotient 
$\mathcal{D}(X) \coloneqq \mathcal{B}(X)/\M(X)$ is sometimes called the Dixmier algebra of $X$. It can be also viewed as local multiplier algebra of injective hull of $\Cont_0(X)$, cf. for instance \cite{Kwasniewski-Meyer:Essential}*{Subsection 4.4}.
\begin{definition}[\cite{BKM2}*{Definition 4.3}]  A Banach algebra Hausdorff completion $B$ of $\mathfrak{C}_c(\Gg,\LL)$ is  an \emph{essential groupoid Banach algebra} if  the map $\mathfrak{C}_c(\Gg,\LL)\ni f \to q(f|_{X})$ induces  a faithful contractive linear map $B\to \mathcal{D}(X) =\mathcal{B}(X)/\M(X)$, which is isometric on $\Cont_c(X)$.
\end{definition}
\begin{remark}\label{rem:essential_groupoid_algebras}
Let $\mathfrak{M}_0(\Gg , \LL)$ be the subspace of $\mathfrak{C}_0(\Gg , \LL)$ consisting of sections with meager strict support. In fact, by \cite{BKM2}*{Proposition 4.6} we have
\begin{align*}
	\mathfrak{M}_0(\Gg , \LL)&=\{f\in \mathfrak{C}_0(\Gg , \LL):\supp(f) \text{ is meager}\}
	\\
	&=\{f\in \mathfrak{C}_0(\Gg , \LL):\supp(f) \text{ has empty interior}\}
\end{align*}
and if $\Gg_{\Hau}$ is comeager then also $\mathfrak{M}_0(\Gg , \LL)=\{f\in \mathfrak{C}_0(\Gg , \LL):\supp(f) \subseteq \Gg\setminus \Gg_{\Hau}\}$.
A Banach algebra Hausdorff completion $B$ of $\mathfrak{C}_c(\Gg,\LL)$ is an essential groupoid Banach algebra if and only if the quotient map
$\mathfrak{C}_c(\Gg,\LL) \to \mathfrak{C}_0(\Gg , \LL)/\mathfrak{M}_0(\Gg , \LL)$ induces an injective contractive linear map
$j^{\es}:B\to  \mathfrak{C}_0(\Gg , \LL)/\mathfrak{M}_0(\Gg , \LL)$, which is isometric on $\Cont_c(X)$, see \cite{BKM2}*{Proposition 4.12}. 
\end{remark}
\begin{remark}
The algebra $F^P_{\ess}(\Gg , \LL)$ is  an essential groupoid Banach algebra,  see \cite{BKM2}*{Corollary 4.19} and   $F^2_{\ess}(\Gg , \LL)=\Cst_{\ess}(\Gg,\LL)$ is the essential $\Cst$-algebra introduced in \cite{Kwasniewski-Meyer:Essential}. 
\end{remark}
\begin{lemma}\label{lem:making_representations_nondegenerate}
For each $p\in [1,\infty]$ there are isometric nondegenerate representations of $F^p(\Gg,\LL)$, $F^p_{\red}(\Gg,\LL)$ and $F^p_{\ess}(\Gg,\LL)$ on some $L^p$-spaces.
\end{lemma}
\begin{proof}
For reduced and essential algebras,  and for $p<\infty$ it is clear, as $\Lambda_{p}$ and $\Lambda_p^{\ess}$  are nondegenerate. For 
$p=\infty$ one may consider  compressions of  $\Lambda_{p}$ and $\Lambda_p^{\ess}$ to $c_0(\Gg,\LL)$ and  $c_0(\Gg_{\Hau},\LL)$, respectively.
For the universal algebra $F^p(\Gg,\LL)$ and $p<\infty$ this follows from \cite{BKM}*{Theorem 5.19(1)}. 
The same result gives also an isometric representation $\psi:F^p(\Gg,\LL)\to \B(L^{\infty}(\mu))$ where 
$\psi|_{\Cont_0(X)}$ is given by multiplication operators by functions in $L^{\infty}(\mu)$. This  implies that 
$\overline{\psi(\Cont_0(X))L^{\infty}(\mu)}$ is an ideal in $L^{\infty}(\mu)$  treated as a commutative $C^*$-algebra.
Hence,  $\overline{\psi(\Cont_0(X))L^{\infty}(\mu)}=\Cont_0(\Omega)$ for some locally compact Hausdorff space $\Omega$.
Since $\overline{\psi(F^p(\Gg,\LL))L^{\infty}(\mu)}=\overline{\psi(\Cont_0(X))L^{\infty}(\mu)}=\Cont_0(\Omega)$, we conclude 
that $\psi$ compresses to the isometric nondegenerate representation on the $L^\infty$-space $\Cont_0(\Omega)$.
\end{proof}
\subsection{Structural results}
We recall conditions for simplicity and pure infiniteness of the essential algebras. 
A Banach algebra is \emph{simple} if it does not contain non-trivial (closed two-sided) ideals.
Following \cite{BKM2}*{Definition 6.1}, we say that a simple Banacha algebra $B$ is \emph{purely infinite} if for  every $y\in B\setminus\{0\}$,
 there are $x, z\in B$ such that $xyz$ is an infinite idempotent in $B$.
When $B$ has a unit, this is equivalent to assuming that  $B\not\cong \mathbb{C}$ and for every $y\in B\setminus\{0\}$ there are $x,z\in B$ such that $xyz=1$;  
when $B$ has an approximate unit consisting of idempotents, this is equivalent to the definition given in \cite{cortinas_Montero_rodrogiez}; and when $B$ is a $C^*$-algebra, this agrees with the standard $C^*$-algebraic notion, see  \cite{BKM2}*{Proposition 6.3}.
\begin{definition}\label{defn:groupoid_properties}
Let $\text{Iso}(\Gg) \coloneqq  \{\gamma\in \Gg:\rg(\gamma)=\sr(\gamma)\}$ be the isotropy bundle in $\Gg$.  A set $U\subseteq X$ is \emph{$\Gg$-invariant} if $\sr(\gamma)\in U$ implies $\rg(\gamma)\in U$ for all $\gamma \in G$. We call the groupoid $\Gg$ 
\begin{itemize}
	\item \emph{minimal} if there are no nontrivial $\Gg$-invariant open sets in $X$;
	\item \emph{topologically free} if the interior of $\text{Iso}(\Gg) \setminus X$ in $\Gg$ is empty;
	\item \emph{effective}  if the interior of $\text{Iso}(\Gg)$ in  $\Gg$ is the unit space $X$;
		\item \emph{locally contracting with respect to a family $S\subseteq \Bis(\Gg)$} if for every non-empty open $U\subseteq X$ there are open $V \subseteq U$ and $W\in S$ with $\overline{V} \subseteq \sr(W)$ and $\rg(W\overline{V})\subsetneq V$; 
	\item \emph{locally contracting} if $\Gg$ is locally contracting with respect to $S=\Bis(\Gg)$.
\end{itemize} 
\end{definition}
Local contractiveness is a classical notion, see \cite{A-D}*{Definition 2.1}, and its relative version is taken from \cite{BKM2}*{Definition 6.8}.
Topological freeness of $\Gg$, introduced in \cite{Kwasniewski-Meyer:Essential}*{Definition 2.20}, is equivalent to effectiveness when $\Gg$ is Hausdorff or more generally when the algebraic singular ideal vanishes, see Lemma \ref{lem:effectiveness_is_about_singular_ideal} below.
Topological freeness is equivalent to the generalised intersection property in the spirit of Archbold-Spielberg \cite{Arch_Spiel}.

\begin{theorem}\label{thm:groupoid_simplicity_pure_infiniteness}
 Let $\Gg$ be an \'etale groupoid with a locally compact Hausdorff space $X$.
Let $P\subseteq [1,\infty]$ be a nonempty set of parameters and let $\mathcal{N}$ be the kernel of the canonical homomorphism $F^P(\Gg,\LL)\to F^P_{\ess}(\Gg,\LL)$.
\begin{enumerate}
	\item\label{enu:groupoid_simplicity_pure_infiniteness0} If $\Gg$ is topologically free, then 
  every ideal $I$ in $F^P(\Gg,\LL)$ with $I\cap\Cont_0(X)=\{0\}$ is contained in $\mathcal{N}$.
	When the twist is trivial the converse implication holds.
	\item\label{enu:groupoid_simplicity_pure_infiniteness1} If $\Gg$ is topologically free, then 
	every representation of $F^P_{\ess}(\Gg,\LL)$ which is injective on $\Cont_0(X)$ is injective on  $F^P_{\ess}(\Gg,\LL)$.
	Conversely,  if  every representation of $F^P(\Gg)$ which is injective on $\Cont_0(X)$ is injective on  $F^P(\Gg)$,
	then $\Gg$ is topologically free.

	\item\label{enu:groupoid_simplicity_pure_infiniteness1.5} If $\Gg$ is Hausdorff, then $F^P_{\red}(\Gg,\LL)=F^P_{\ess}(\Gg,\LL)$ and $\Gg$ is topologically free if and only if $C_0(X)\subseteq F^P_{\red}(\Gg,\LL)$ is a maximal commutative subalgebra.
	\item\label{enu:groupoid_simplicity_pure_infiniteness2} If $\Gg$ is topologically free, then  $F^P_{\ess}(\Gg,\LL)$ is simple if and only if  $\Gg$ is minimal.
	The  algebra $F^P(\Gg)$ is simple if and only if $\Gg$ is topologically free and minimal and the canonical map  $F^P(\Gg)\to F^P_{\ess}(\Gg)$ is injective. 
	\item\label{enu:groupoid_simplicity_pure_infiniteness3} If $\Gg$  is topologically free, minimal, and locally contractive with respect to some unital inverse subsemigroup $S\subseteq \Bis(\Gg)$ covering $\Gg$ and such that $\LL|_{U}$ is topologically trivial for every $U\in S$,  then $F^P_{\ess}(\Gg,\LL)$ is  purely infinite simple. 
\end{enumerate}
If the canonical map $F^P_{\red}(\Gg,\LL)\to F^P_{\ess}(\Gg,\LL)$ or $F^P(\Gg,\LL)\to F^P_{\ess}(\Gg,\LL)$ is injective, then the above statements hold with  $F^P_{\ess}(\Gg,\LL)$ replaced by 
$F^P_{\red}(\Gg,\LL)$ or  $F^P(\Gg,\LL)$, respectively.
\end{theorem}
\begin{proof} Statement \ref{enu:groupoid_simplicity_pure_infiniteness0} follows from   \cite{BKM2}*{Theorems 5.10(1) and 5.13}.
Item \ref{enu:groupoid_simplicity_pure_infiniteness1} is ``almost a special case of \ref{enu:groupoid_simplicity_pure_infiniteness0}''.
Formally, as we do not know whether the map $F^P(\Gg,\LL)\to F^P_{\ess}(\Gg,\LL)$ is surjective, to get the first implication one needs to apply \cite{BKM2}*{Theorem 5.10(2) or Theorem 5.16}.
For the converse implication, note that if every representation of $F^P(\Gg)$ that is injective on $\Cont_0(X)$ is injective in $F^P(\Gg)$, then in particular $\mathcal{N}=0$. Hence,  
$\Gg$ is topologically free by \cite{BKM2}*{Theorem 5.13}.
Item \ref{enu:groupoid_simplicity_pure_infiniteness1.5} follows from \cite{BKM2}*{Proposition 5.11}. 
Item 
\ref{enu:groupoid_simplicity_pure_infiniteness2} follows from  \cite{BKM2}*{Theorems 5.10(3) and 5.13}.
While \ref{enu:groupoid_simplicity_pure_infiniteness3}  holds by \cite{BKM2}*{Theorem D}.
The last part follows from the fact that the algebras $F^P_{\red}(\Gg,\LL)$ or $F^P(\Gg,\LL)$ are essential groupoid Banach algebras, when $F^P_{\red}(\Gg,\LL)\to F^P_{\ess}(\Gg,\LL)$ or $F^P(\Gg,\LL)\to F^P_{\ess}(\Gg,\LL)$ is injective, respectively.
\end{proof}

Finally, we recall the relationship between representations of twisted inverse semigroup actions and the associated groupoids in the case where 
the domains of the action are compact open. The theory simplifies significantly in this case.
\begin{definition}\label{defn:covariant_representation_in_algebra}
Let $(h,u)$ be a twisted inverse semigroup action of $S$  on $X$ such that the domains $X_t\subseteq X$  are compact open for every $t\in S$, cf. Definition~\ref{def:twist_semigroup_action}.  
A \emph{covariant representation  of $(h,u)$ on a Banach space $Y$} is a pair $(\pi,v)$, where $\pi :\Cont_0(X) \to \Bound(Y)$ is a contractive homomorphism and $v : S \to \Bound(Y)_1$ takes values in the semigroup of contractive operators on $Y$,  such that: 
\begin{enumerate}[labelindent=40pt,label={(CR\arabic*)},itemindent=1em] 
	\item\label{item:covariant_representation1} $v_s v_t = \pi(u(s,t)) v_{st}$ for all  $s,t\in S$;
	\item \label{item:covariant_representation2} $v_t \pi(a)  = \pi(a\circ h_{t^*})v_{t}$ for all $a\in \Cont_0(X_{t^*})$; and
	\item\label{item:covariant_representation3} $v_e=\pi(1_{X_e}) $ for every idempotent  $e\in \EE(S)$.
\end{enumerate}
Then  $B(\pi,v)  \coloneqq  \clsp\{ \pi(a_t)v_t : a_t\in I_{t}, \, t\in S \}$ is a Banach algebra generated by the ranges of  $\pi$ and $v$.
We say that $(\pi,v)$ is nondegenerate, isometric, etc.\ if $\pi$ has that property. 
\end{definition}
The following  integration-disintegration result holds by \cite{BKM}*{Theorems 4.25, 5.13}. 
\begin{proposition}\label{prop:disintegration_theorem}
Let $(h,u)$ be a twisted inverse semigroup action of $S$  on $X$ with all domains compact open, and let  $(S\ltimes_{h} X,\LL_{u})$ be  the  associated twisted transformation  groupoid.
For any $L^p$-space $Y$, $p\in [1,\infty]$, we have a bijective correspondence between covariant representations $(\pi , v)$ of $(h,u)$ on $Y$ 
and representations  $\pi\rtimes v$ of $\mathfrak{C}_c(\Gg,\LL)$ on $Y$ where 
$$
\pi\rtimes v(a)=\pi(a),\qquad \pi\rtimes v(1_{U_t})=v_t, \qquad a\in \Cont_0(X),\,\, t\in S.
$$
Moreover, we have $B(\pi,v)=\overline{\pi\rtimes v(\mathfrak{C}_c(\Gg,\LL))}$.
\end{proposition}
\subsection{Comments on singular ideals and effectiveness}
Recall from Remark \ref{rem:reduced_Lp_algebras} that a reduced groupoid Banach algebra $B$ of $(\Gg , \LL)$ comes equipped with an injective contractive map $j:B\to \mathfrak{C}_0(\Gg , \LL)$.  
By \cite{BKM2}*{Proposition 4.12}, $J^{B}_{\sing} \coloneqq j^{-1}(\mathfrak{M}_0(\Gg , \LL))$ is an ideal in $B$ such that
the quotient $B/J^{B}_{\sing} $ naturally  becomes an essential groupoid Banach algebra of $(\Gg , \LL)$, 
and so the reduced algebra $B$ is essential if and only if  $J^{B}_{\sing}=\{0\}$.
\begin{definition}
We call $J^{B}_{\sing}$ defined above the \emph{singular ideal} for the reduced groupoid Banach algebra $B$ of $(\Gg , \LL)$.
We denote by $J_{\sing}^P$ the singular ideal for $F^P_{\red}(\Gg , \LL)$. When $P=\{p\}$  we  write  $J_{\sing}^p \coloneqq  J_{\sing}^P$. 
\end{definition}
\begin{remark} 
The ideal $J_{\sing}^P$ coincides with the kernel of the canonical representation
$F^P_{\red}(\Gg , \LL)\to F^P_{\ess}(\Gg , \LL)$, whenever $F^P_{\ess}(\Gg , \LL)$ is defined ($\Gg_{\Hau}$ is comeager), see comments before \cite{BKM2}*{Proposition 4.16}. 
\end{remark}
Embedding the reduced algebra $B$ via the injective $j$-map into  $\mathfrak{C}_0(\Gg , \LL)$  we have 
$$
\mathfrak{C}_c(\Gg , \LL)\cap  \mathfrak{M}_0(\Gg , \LL)\subseteq J^{B}_{\sing} \subseteq   \mathfrak{M}_0(\Gg , \LL).
$$
Clearly, $\mathfrak{M}_0(\Gg , \LL)=\{0\}$ when $\Gg$ is Hausdorff. 
We also know that  $\mathfrak{M}_0(\Gg , \LL)=\{0\}$ when
$\Gg$ is ample and every compact open set in $\Gg$ is regular open, see  \cite{BKM2}*{Lemma 4.8}.
In the untwisted case, vanishing of $\mathfrak{C}_c(\Gg)\cap \mathfrak{M}_0(\Gg)$ has been completely characterised
 in \cite{Brix-Gonzales-Hume-Li:Hausdorff_covers} and then in \cite{Hume}  by  different conditions.
Moreover, the authors of \cite{Brix-Gonzales-Hume-Li:Hausdorff_covers} found a natural condition under  which vanishing of $\mathfrak{C}_c(\Gg)\cap \mathfrak{M}_0(\Gg)$ is equivalent
to  the equality $\Cst_{\ess}(\Gg)=\Cst_{\red}(\Gg)$. 
We claim that their proof actually shows vanishing of  $\mathfrak{M}_0(\Gg)$. 
To explain this we establish the relevant terminology.
\begin{definition}\label{defn:finite_non_Hausdorffness}
Let $Y$ be a topological space. For each $y\in Y$ we denote by $\{y\}_{Y}$ the set of $z\in Y$ such that  $y$ and $z$ cannot be separated by disjoint open sets (equivalently, $y$ and $z$
are limit points of a net in $Y$). We say that  $Y$ is \emph{finitely non-Hausdorff} if  $\{y\}_{Y}$ is finite for every $y \in Y$.
\end{definition}
\begin{remark}
A point $y\in Y$ is Hausdorff if and only if $\{y\}_{Y}=\{y\}$. The relation of being non-separated in general is not transitive and hence  is not an equivalence relation.
\end{remark}
Finite  non-Hausdorffness  is exactly the condition assumed in \cite{Brix-Gonzales-Hume-Li:Hausdorff_covers}*{Theorem C}, which was phrased using  sets $\overline{X}(x)=\Gg_{x}^{x}\cap \overline{X}$ \label{page:closure_over_point}, where $\overline{X}$ is the closure of $X$ in $\Gg$ and $\Gg_{x}^{x}=\rg^{-1}(x)\cap \sr^{-1}(x)$ is the isotropy group over $x\in X$.
\begin{lemma}\label{lem:finite_non_Hausdorff_groupoid}
An \'etale groupoid $\Gg$ is finitely non-Hausdorff if and only if each $\overline{X}(x)$, $x\in X$, is finite
if and only if the source (equivalently the range map) $\overline{X}\onto X$ is finite-to-one.
\end{lemma}
\begin{proof}
For any $\gamma\in \Gg$ we have a well-defined injective map
$[\gamma]_{\Gg}\ni \eta\mapsto \eta^{-1}\gamma\in \overline{X}(\sr(\gamma))$. Indeed, injectivity is clear, and if $(\gamma_n)_n$ is a net in $\Gg$ that has 
$\eta$ and $\gamma$ as limits then by continuity of  multiplication and taking inverses the net    $(\gamma_n^{-1}\gamma_n)_n=(\sr(\gamma_n))_n$ in  $X$ 
has $\eta^{-1}\gamma$ and $\gamma^{-1}\gamma=\sr(\gamma)$ as limits. Moreover, $\eta^{-1}\gamma$ is in the isotropy group by continuity of   $\sr$ and $\rg$.   When $\gamma=x$ is in $X$, this injection becomes the equality $\{x\}_{\Gg}=\overline{X}(x)$.  This gives the first equivalence, and the second one readily follows.
\end{proof}

For any net $(x_n)_n$ in $X$ the set $\Gamma$ of  its limit points in $\Gg$ is either empty or forms a subgroup of the isotropy group $\Gr_{x}^{x}$ for some $x\in X$.
In particular, $\overline{X}(x)$ the union of such subgroups in $\Gr_{x}^{x}$. Recall that a net is \emph{primitive} (or maximal) if  each of its subnets has the same set of limit points. 
Following \cite{Hume}, for each $x\in X$ we denote by $\mathcal{X}(x)$ the family of subgroups of $\Gr_{x}^{x}$ that arise as sets of limit points for primitive nets
in  $X\setminus \Gg_{\Hau}$.  The following condition, as well as other equivalent variants, appears in \cite{Hume}:
\begin{enumerate}\label{Condition:Hume}
\item[$\Hum$]  for each $x\in X$  there is no nonzero  $a\in \Cont_c(\Gg_{x}^{x})$ that satisfies $\sum_{\eta \in \Gamma} a(\gamma \eta)=0$ for all $\Gamma\in \mathcal{X}(x)$ and $\gamma \in \Gg_x^x$. 
\end{enumerate}
This serves as an alternative for the (negation of) condition $\mathcal{S}_{0}$ of  \cite{Brix-Gonzales-Hume-Li:Hausdorff_covers}:
\begin{theorem}
\label{thm:vanishing_of_the_singular_ideal}
For any \'etale groupoid the following conditions are equivalent:
\begin{enumerate}
	\item\label{enu:vanishing_of_the_singular_ideal1} $\mathfrak{C}_c(\Gg)\cap \mathfrak{M}_0(\Gg )=\{0\}$;
	\item\label{enu:vanishing_of_the_singular_ideal2} $\Gg$ satisfies  $\Hum$;
	\item\label{enu:vanishing_of_the_singular_ideal3} $\Gg$ does not satisfy condition $\mathcal{S}_{0}$ of  \cite{Brix-Gonzales-Hume-Li:Hausdorff_covers};

\end{enumerate}
If  $\Gg$ is finitely non-Hausdorff, then the above conditions are  equivalent to $\mathfrak{M}_0(\Gg  )=\{0\}$.
\end{theorem}
\begin{proof} Note that $\mathfrak{C}_c(\Gg)\cap \mathfrak{M}_0(\Gg )=J_{\sing}^{2}\cap \mathfrak{C}_c(\Gg)$. 
Hence,  \ref{enu:vanishing_of_the_singular_ideal1}$\Longleftrightarrow$\ref{enu:vanishing_of_the_singular_ideal3} by 
\cite{Brix-Gonzales-Hume-Li:Hausdorff_covers}*{Theorem 4.1(i)}, and \ref{enu:vanishing_of_the_singular_ideal1}$\Longleftrightarrow$\ref{enu:vanishing_of_the_singular_ideal2}
by \cite{Hume}*{Theorem F}, see also \cite{Hume}*{Lemma 5.3}.
Assume $\Gg$ is finitely non-Hausdorff. 
Then \cite{Brix-Gonzales-Hume-Li:Hausdorff_covers}*{Theorem 4.7} states that $J^{2}_{\sing}\neq \{0\}$ implies  $\mathfrak{C}_c(\Gg)\cap \mathfrak{M}_0(\Gg )\neq \{0\}$,
but the proof works when $J^{2}_{\sing}$ is replaced by $\mathfrak{M}_0(\Gg )$. Indeed, the proof uses only that $\mathfrak{M}_0(\Gg )$ is a $\Cont_0(X)$-module or more generally 
that a convolution of $f\in \mathfrak{M}_0(\Gg )$ with any $a\in \Cont_c(U)$, for $U\in \Bis(\Gg)$,  is an element of $\mathfrak{M}_0(\Gg )$.
Also \cite{Brix-Gonzales-Hume-Li:Hausdorff_covers}*{Lemma 3.12} can be stated in a slightly more general way by replacing the image of $C^*_r(G)$ in $C^*_{r}(\widetilde{G})$ with the image 
of $\mathfrak{C}_0(\Gg)$ in $\Cont_0(\widetilde{G})$.
\end{proof}
\begin{corollary}\label{cor:vanishing_of_singulars}
Let $B$ be a reduced Banach algebra of $(\Gg,\LL)$ where the twist $\LL$ is topologically trivial (given by a cocycle).
Then $\Hum$ is a necessary condition for vanishing of the singular ideal $J^{B}_{\sing}$. If 
$\Gg$ is finitely non-Hausdorff it is also sufficient.

In particular, if  $\Gg$ is finitely non-Hausdorff, then $B$ is essential if and only if $\Hum$ holds.
\end{corollary}
\begin{proof}
When the twist $\LL$ is topologically trivial we may view $B$ as a completion of $\mathfrak{C}_c(\Gg)$, with operations given by \eqref{eq:convolution_and_involution1}. 
Then $\mathfrak{C}_c(\Gg )\cap  \mathfrak{M}_0(\Gg )\subseteq J^{B}_{\sing} \subseteq   \mathfrak{M}_0(\Gg )$, and so 
Theorem~\ref{thm:vanishing_of_the_singular_ideal} applies.
\end{proof}
\begin{remark}
When the twist is topologically nontrivial the statement in Corollary~\ref{cor:vanishing_of_singulars} fails, see Example~\ref{ex:nontrivial_twist} below. In particular, conditions in Theorem~\ref{thm:vanishing_of_the_singular_ideal}
are not necessary for $\mathfrak{C}_c(\Gg , \LL)\cap  \mathfrak{M}_0(\Gg , \LL)=\{0\}$. 
To this day, it is not known whether vanishing of $\mathfrak{C}_c(\Gg)\cap  \mathfrak{M}_0(\Gg)$  implies vanishing of $\mathfrak{M}_0(\Gg)$ or even $J_{\sing}^{2}$, but the progress in this research is rapid and examples such as those in \cite{Martinez_Szakacs} (which are given by groupoids coming from self-similar actions) suggest that this fails in general. 
\end{remark}

Finally, we explain how effectiveness is related to vanishing of singular ideals.

\begin{lemma}\label{lem:effectiveness_is_about_singular_ideal}
Assume that $\mathfrak{C}_c(\Gg)\cap \mathfrak{M}_0(\Gg)=\{0\}$, which is equivalent to $\Hum$ by Theorem~\ref{thm:vanishing_of_the_singular_ideal}.  In particular, this holds whenever the singular ideal in $F^P_{\red}(\Gg,\LL)$ vanishes, for some $P\subseteq [1,\infty]$ and some topologically trivial twist $\LL$. Then $\Gg$ is effective if and only if it is topologically free.
\end{lemma}
\begin{proof}
Assume that $\Gg$ is topologically free but not effective. So there is a non-empty (open) bisection  $U\subseteq \text{Iso}(\Gg)$
such that $U\setminus X$ is non-empty, but its interior is empty.  Then also $\rg(U)\setminus U=\rg(U\setminus X)$ is non-empty but has an empty interior. 
The union $U\setminus X \sqcup  \rg(U)\setminus U$ also has empty interior. Indeed, for any  open $V\subseteq U\setminus X \sqcup  \rg(U)\setminus U$ using that $X$ is open and 
$V\cap X \subseteq \rg(U)\setminus U$ we get that $V\cap X=\emptyset$. Hence, $V\subseteq U\setminus X$ which implies that $V=\emptyset$.
Now take any $a\in \Cont_c(\rg(U))$ such that 
$a(\rg(U\setminus X))\neq 0$. Then $a\circ r|_{U}^{-1}\in  \Cont_c(U)$ and $b \coloneqq a- a\circ r|_{U}^{-1}\in \mathfrak{C}_c(\Gg)$.
For any $\gamma\in \Gg$, we have $b(\gamma)=-a(\rg(\gamma))$ if $\gamma\in U\setminus X$,  $b(\gamma)=a(\gamma)$ if $\gamma \in \rg(U)\setminus U$ and $b(\gamma)=0$ otherwise.
Hence, $b$ is nonzero and $\supp(b)\subseteq  U\setminus X \sqcup  \rg(U)\setminus U$.
Thus,  $0\neq b\in \mathfrak{C}_c(\Gg)\cap \mathfrak{M}_0(\Gg)$.
\end{proof}
\begin{remark}\label{rem:effectiveness_no_good}
In the study of ideal structure of $\Cst_{\red}(\Gg)$ one usually independently checks vanishing of the singular ideal  and  
effectiveness, cf. \cites{Clark-Exel-Pardo-Sims-Starling:Simplicity_non-Hausdorff, Steinberg_Szakacs}. 
The above lemma shows that these procedures are not independent. We believe that  it is much more natural and (nomen omen) efficient to replace effectiveness with  topological freeness here. 
Topological freeness has a number of natural characterisations, while effectiveness is  almost only useful when it coincides with topological freeness. 
One (negative) application is that, when the twist $\LL$ is topologically trivial and $\Gg$ is topologically free, then the lack of effectiveness   implies
non-vanishing of all singular ideals. However, as the following example shows, this criterion does not work for general twists.
\end{remark}
\begin{example}\label{ex:nontrivial_twist}
Let  $\Gg \coloneqq X\cup \{\star\}$ be the groupoid with the
unit space $X=\{\frac{1}{n}:n\in \Z\}\cup \{0\}$ and $\star$ being a nontrivial element in the isotropy group of $0$. 
We view $\Gg$ as an ample groupoid where $X$ is equipped with the usual topology inherited from $\R$, and 
the bisection $X\setminus \{0\}\cup\{\star\}$ is compact open and homeomorphic to $X$ via the map that sends $\star$ to $0$ and
acts as the identity elsewhere. Then $\Gg$ is an ample groupoid with $\Gg_{\Hau}=\{\frac{1}{n}:n\in \Z\}$.
Since $\Gg=\text{Iso}(\Gg)$ it is not effective, but it is topologically free, as the singleton $\Gg\setminus X=\{\star\}$ is not open in $\Gg$.
Moreover, for any $P\subseteq [1,\infty]$ we have
$$
F^P(\Gg)=F^P_{\red}(\Gg)=\big\{a\oplus b\in \Contc(\Gg_{\Hau})\oplus \mathbb{C}\Z_2: \lim_{n\to \pm \infty} a(1/n)=b(0)+b(1)\big\}
$$
and $F^P_{\ess}(\Gg)=\Cont(X)$. By Lemma~\ref{lem:effectiveness_is_about_singular_ideal}, the singular ideal
$J^P_{\sing}\cong \mathbb{C}$ does not vanish. This agrees with Theorem~\ref{thm:vanishing_of_the_singular_ideal}: it is easy to see that setting $a(0)=1$ and $a(\star)=-1$ gives $a\in \Cont_c(\{0,\star\})$, showing that $\Hum$ fails as $\mathcal{X}(0)$ consists only of $\{0,\star\}$.

The situation changes when we add a nontrivial twist.
Following \cite{DEP25}*{Section 3}, we ``twist the product topology'' on $\LL \coloneqq \Gg\times \mathbb{C}$  by
declaring that the product space $X\times \mathbb{C}$ sits as an open subset of $\LL$ while for  any $z\in \mathbb{C}$ we define
the  neighbourhood subbasis $\{U_{V,N}\}_{V,N}$ for $(\star,z)$, parametrised by an open neighbourhood $V\subseteq \mathbb{C}$ of $z$ and a natural number $N\in \N$,  by the formula
$$
U_{V,N} \coloneqq  \left(\left\{-\frac{1}{n}: n\geq N\right\}\cup \{\star\}\right)\times (-V) \cup \left(\left\{\frac{1}{n}: n\geq N\right\}\cup \{\star\}\right)\times V,
$$
where $-V\coloneq\{-z:z\in V\}$. Then 
$$
F^P(\Gg,\LL)=F^P_{\red}(\Gg,\LL)=F^P_{\ess}(\Gg,\LL)=\big\{a\oplus b\in \Contc(\Gg_{\Hau})\oplus \mathbb{C}\Z_2: \lim_{n\to \pm \infty} a(1/n)=b(0)\pm b(1)\big\}.
$$
\end{example}

\section{Twisted inverse semigroup \texorpdfstring{$L^P$}{LP}-operator algebras}\label{sect:Twisted inverse semigroup}

We now generalise the definition of an inverse semigroup representation \cite{BKM}*{Definition 6.12} to the twisted case.
Recall that an element $b\in B$ in a unital Banach algebra $B$ is \emph{hermitian}  if and only if 
$\|e^{itb}\|\leq 1$ for all $t\in \R$. If $B$ is approximately unital (has a two-sided contractive approximate unit), then following \cite{BP19}*{Definition 2.8} we say that an element in $B$ is hermitian if it is hermitian in the minimal unitisation $\widetilde{B}$ of $B$. 
We say that $v\in B$ has \emph{Moore-Penrose generalised inverse} if there is $v^*\in B$ such that 
$v v^*v=v$ and $v^*v v^*=v^*$, and $v^*v$ and $v^*v$ are hermitian. 
Then $v^*$ is uniquely determined by $v$, and if in addition $\|v\|, \|v^*\|\leq 1$, then $v$ is a \emph{Moore-Penrose partial isometry}, as defined in \cite{Mbekhta}. 
When $B$ is a $C^*$-algebra, then  $b\in B$ is hermitian if and only if  $b=b^*$, and $v\in B$ is a Moore-Penrose partial isometry if and only if it is a partial isometry in the usual sense, and then $v^*$ is its hermitian adjoint.
For a Banach space $Y$ we denote by $\MPIso(Y)$  the set of Moore-Penrose partial isometries in the Banach algebra $\Bound(Y)$.
\begin{proposition}\label{prop:inverse_semigroup_of_partial_isos}
For any $p\in [1,\infty]\setminus\{2\}$, and any $L^p$-space $Y$, the  Moore-Penrose partial isometries $\MPIso(Y)\subseteq \Bound(Y)$ form an inverse semigroup.
\end{proposition}
\begin{proof}
By \cite{BKM}*{Theorem 2.28} the assertion follows for  $p<\infty$  and for $p=\infty$ if the space $Y$ is isomorphic to $L^{\infty}(\mu)$ for some measure $\mu$.
Since we adopted the convention that $L^\infty$-spaces are more general, we need to consider the case when $Y=\Cont_0(\Omega) $  for a locally compact Hausdorff space $\Omega$. 
By \cite{BKM}*{Theorem 2.13}, hermitian projections in $\Bound(Y)$ are multiplication operators by characteristic function of some closed open subsets of $\Omega$.
This implies that  Moore-Penrose partial isometries are equivalent to invertible isometries $\Cont_0(D)\to\Cont_0(R)$ for some  clopen subsets $D,R\subseteq \Omega$.
Hence, using the Stone-Banach theorem we get that $v\in \Cont_0(\Omega) $ is a Moore-Penrose partial isometry if and only if $v$ is a weighted composition operator of the form
$v\xi=\omega (\xi\circ \varphi)$ where $\omega\in \Cont_u(D)$ and  $\varphi\colon D\to \varphi(D)$ is a homeomorphism between clopen subsets of $\Omega$.
With this presentation it is now straightforward to see that these operators form an inverse semigroup.
\end{proof}
\begin{remark} 
The  elements of $\MPIso(Y)$ considered in Proposition~\ref{prop:inverse_semigroup_of_partial_isos} coincide with $L^p$-partial isometries  on $Y$ in the sense of 
\cite{BKM}*{Definition 2.22},
and have ultrahermitian idempotents in $\Bound(Y)$, see \cite{Gardella_Palmstrom_Thiel}*{Definition 2.15}.  Thus,  one could  conclude  that $\MPIso(Y)$ is an inverse semigroup using  either \cite{BKM}*{Proposition 2.23}  or \cite{Gardella_Palmstrom_Thiel}*{Theorem 2.17}.
When  $Y=L^{p}(\mu)$ for a localisable measure $\mu$, then $\MPIso(Y)$ is the inverse semigroup of  spatial partial isometries, as defined by Phillips, see \cite{BKM}*{Theorem 2.28}.
\end{remark}
\subsection{Representations of twisted inverse semigroups  and  groupoid models}
\begin{definition}\label{def:semigroup_representations} 
A \emph{representation of a twisted inverse semigroup} $(S,\omega)$ on  a Banach space $Y$ is a zero preserving map
$v: S\to \Bound(Y)_1$ into the semigroup of contractive operators on $Y$ such that
\begin{enumerate}[label={(SR\arabic*)}]
	\item\label{enu:semigroup_representation1} $v_s v_t =\omega(s,t) v_{st}$ if $st\neq 0$ and $v_s v_t=0$ otherwise, for all  $s,t\in S$; and
	\item\label{enu:semigroup_representation2} $v|_{\EE}$ takes values in hermitian idempotents.
\end{enumerate}
We define the \emph{range of the representation} $v$ to be 
$$
B(v)  \coloneqq  \clsp\{v_t : \, t\in S \}.
$$
It is the Banach algebra generated by the range of $v$ as a map.
We say that a representation $v$ is \emph{covariant} if, in addition,
\begin{enumerate}[label={(SR\arabic*)}]
	\setcounter{enumi}{2}
	\item\label{enu:semigroup_representation3} $v|_{\EE}$ is tight in the sense that $\prod_{f\in F} (v_e -v_f) = 0$ for every cover $F$ of $e\in \EE$ 
	(equivalently $v|_{\EE}$ is tight as a map into a Boolean ring of idempotents in $\linspan \{ v_e : e\in \EE \}$). 
\end{enumerate}
For nonempty $P\subseteq [1,\infty]$, we denote by   $\OO^P(S,\omega)$ the  Banach algebra which is a range of a direct sum  of universal    covariant representations
of $(S,\omega)$ on $L^p$-spaces for $p\in P$ (we prove its existence by giving a groupoid model in Corollary~\ref{cor:groupoid_presentation_inverse_semigroup_algebras}). Similarly, we let $\TT^P(S,\omega)$ be  the universal Banach algebra  for all (not necessarily covariant) representations on $L^p$-spaces for  $p\in P$. 
We call $\OO^P(S,\omega)$  and  $\TT^P(S,\omega)$  the \emph{(universal) $L^P$-operator algebra} and \emph{(universal) Toeplitz $L^P$-operator algebra} of the twisted inverse semigroup $(S,\omega)$, respectively.

\end{definition}
\begin{remark}\label{rem:inverse_semigroup_representations}
Condition \ref{enu:semigroup_representation1} implies that $v|_{\EE}$ takes values in idempotents, and that for every $t\in S$ the operator $v_t^* \coloneqq \omega(t^*,t)^*v_{t^*}$   is a generalised inverse for $v_t$, i.e. $v_t v_t^*v_t=v_t$ and $v_t^*v_t v_t^*=v_t^*$. Thus,   \ref{enu:semigroup_representation1} and \ref{enu:semigroup_representation2} imply that $v_t^*$ is the (necessarily unique) Moore-Penrose generalised inverse of $v_t$, and hence   $v$ takes values in Moore-Penrose partial isometries.
In particular,  when $Y$ is a Hilbert space,  $v_t^*$ is the hermitian adjoint of $v_t$  (since the $v_t$ are assumed to be contractive,  condition \ref{enu:semigroup_representation2} in this case is automatic).
When $Y$ is an $L^p$-space for $p\in [1,\infty]\setminus\{2\}$,  $v_t^*$ is the unique generalised inverse in the inverse semigroup $\MPIso(Y)$, see Proposition~\ref{prop:inverse_semigroup_of_partial_isos}.\end{remark}

\begin{remark}\label{rem:about_spatial_and_Hilbert_space_partial_isometries}
In the untwisted case, that is when   $\omega(s,t)=1$ for all $s,t\in S$, we will write $\OO^P(S)$ and $\TT^P(S)$ instead of 
$\OO^P(S,\omega)$ and $\TT^P(S,\omega)$.  
In particular,  for $P=\{2\}$ the algebra $\OO^P(S)$  coincides with the Exel's tight $C^*$-algebra $C^*_{\tight}(S)$, see \cite{Exel},
while $\TT^P(S)$  is a contracted version of Paterson's universal  $C^*$-algebra $C^*(S)$, see \cite{Paterson}*{2.1} (contracted here means 
that we identify the zero in $S$ with the zero operator in $C^*(S)$).
\end{remark}
\begin{remark}
We believe the above definition of a representation of $(S,\omega)$ is correct when $Y$ is an $L^p$-space. For more a general Banach space one would want to add 
more restrictions such as \emph{joint-contractiveness} from \cite{BKM}*{Definition 6.6} (which for $L^p$-spaces is automatic).
\end{remark}
To formulate the next fact we introduce an auxiliary terminology.
\begin{definition}
We say that a  representation $\pi:A\to B$ between two approximately unital Banach algebras is \emph{hermitian} if 
it maps hermitian operators to hermitian operators. Accordingly, we say that  a covariant representation $(\pi,v)$, as in Definition~\ref{defn:covariant_representation_in_algebra}, is \emph{hermitian} if  $\pi :\Cont_0(X) \to \Bound(Y)$ is hermitian.
\end{definition}
\begin{remark}
If a representation $\pi:A\to B$ extends to  unital representation $\widetilde{\pi}:\widetilde{A}\to \widetilde{B}$ (so for instance, if  $\pi:A\to \Bound(Y)$ is
nondegenerate), then
$\pi$ is automatically hermitian, see \cite{cgt}*{Lemma 2.4}. In particular, representations in Lemma~\ref{lem:making_representations_nondegenerate} are hermitian. Since representations between $C^*$-algebras are necessarily  $*$-preserving they are always hermitian.
\end{remark}

\begin{proposition}\label{prop:representations_inverse_semigroups_vs_actions}
Let $(S,\omega)$ be a  twisted inverse semigroup  and let $(\widetilde{h},\widetilde{u})$ and  $(h,u)$ be the associated twisted partial actions of $S$ on 
the spectrum $\widehat{\EE}$ and the tight spectrum $\partial\widehat{\EE}$ respectively, cf. Proposition~\ref{prop:semigroup_twists_to_action_twists}. Consider an $L^p$-space $Y$,  for some $p\in [1,\infty]$. 

\begin{enumerate}
	\item\label{item:representations_inverse_semigroups_vs_actions1} 
	The equations $\widetilde{\pi}(1_{Z(e)})=v_e$, $e\in \EE$, yield a bijective correspondence between representations $v$ of 
	$(S,\omega)$ on $Y$ and hermitian covariant representations $(\widetilde{\pi}, v)$ of  $(\widetilde{h},\widetilde{u})$ on $Y$.
	\item\label{item:representations_inverse_semigroups_vs_actions2} 
	The equations $\pi(1_{Z(e)\cap \partial\widehat{\EE}})=v_e$, $e\in \EE$, yield a  bijective correspondence between covariant representations $v$ of 
	$(S,\omega)$ on $Y$ and hermitian covariant representations $(\pi, v)$ of  $(h,u)$  on $Y$.
\end{enumerate}
For any representation $\widetilde{\pi}:\Cont_0(\widehat{\EE})\to B$ in a Banach algebra $B$, putting $v_e \coloneqq \widetilde{\pi}(1_{Z(e)})$, $e\in \EE$, we have that $\widetilde{\pi}$ is isometric on $\Cont_0(\widehat{\EE})$ if and only if $\prod_{f\in F} (v_e -v_f) \neq 0$
for every   finite $F\subseteq e\EE\setminus\{e\}$, $e\in \EE$. Similarly, a representation $\pi:\Cont_0(\partial\widehat{\EE})\to B$ is isometric if and only if for the corresponding operators we have $\prod_{f\in F} (v_e -v_f) \neq 0$ for every $e\in E$ and finite $F\subseteq eE$ that does not cover $e$.

\end{proposition}
\begin{proof} If $(\widetilde{\pi}, v)$ is  a hermitian covariant representation  of  $(\widetilde{h},\widetilde{u})$, then
$v$ is a representation of $(S,\omega)$. Indeed,  \ref{item:covariant_representation1}  for $(\widetilde{\pi}, v)$, in our setting, is equivalent to  \ref{enu:semigroup_representation1} for $v$,
and by \ref{item:covariant_representation3} for every $e\in \EE$  we have $v_e=\widetilde{\pi}(1_{Z(e)})$, which is hermitian as $1_{Z(e)}$ is hermitian in 
$\Cont_0(\widehat{\EE})$. Conversely, let $v$ be a representation of 
$(S,\omega)$ on $Y$.  By \cite{BKM}*{Lemma 6.7}, there is a unique representation $\widetilde{\pi}:\Cont_0(\widehat{\EE})\to \Bound(L^p(\mu))$
such that $\widetilde{\pi}(1_{Z(e)})=v_e$, $e\in \EE$. Then $\widetilde{\pi}$ is hermitian and  \ref{item:covariant_representation1}, \ref{item:covariant_representation3} for  $(\widetilde{\pi}, v)$ hold.    To check  \ref{item:covariant_representation2}  it suffices to consider $a=1_{Z(e)}$ for $e\in t^*t\EE$, as such functions generate the $\Cst$-algebra $\Cont_0(Z(t^*)) =\Cont_0(Z(t^*t))$. 
We have
\begin{align*}
	v_t \widetilde{\pi}(a)  &=v_t v_{e}= \omega(t,e)v_{te}=\omega(t,e)v_{tet^*t}\stackrel{\ref{lem:cocycles_properties}\ref{ite:cocycles_properties4}}{=} 
		v_{tet^*} v_{t} =\widetilde{\pi}(1_{Z(tet^*)})v_{t}= \widetilde{\pi}(a\circ h_{t^*})v_{t}. 
\end{align*}
Hence, $(\widetilde{\pi}, v)$ is  a hermitian covariant representation $(\widetilde{\pi}, v)$ of  $(\widetilde{h},\widetilde{u})$. 
This proves \ref{item:representations_inverse_semigroups_vs_actions1}.
Note that $\Cont_0(\partial\widehat{\EE})\cong \Cont_0(\widehat{\EE})/\Cont_0(\widehat{\EE}\setminus \partial\widehat{\EE})$ where  $\Cont_0(\widehat{\EE}\setminus \partial\widehat{\EE})=\clsp\{\prod_{f\in F} 1_{Z(e)} - 1_{Z(f)} : \text{$F$ covers $e\in \EE$} \}$, 
by  \cite{Steinberg_Szakacs}*{Corollary 2.14}. Hence, the bijective correspondence in \ref{item:representations_inverse_semigroups_vs_actions1} descends to the bijective correspondence \ref{item:representations_inverse_semigroups_vs_actions2}. 

The last part of the assertion follows from  the last parts of \cite{BKM}*{Theorem 6.9}
and \cite{BKM}*{Theorem 6.15}, respectively.
\end{proof}
\begin{corollary}\label{cor:groupoid_presentation_inverse_semigroup_algebras} 
Let $(S,\omega)$ be a  twisted inverse semigroup. Equip the associated groupoid
$\widetilde{\Gg}(S)=S \ltimes_{h} \widehat{\EE}$, and its restriction $\Gg(S)=S \ltimes_{h} \partial\widehat{\EE}$, with the twist $\LL_{\omega}$  coming from $\omega$, see Definition~\ref{defn:groupoid_twists_from_semigroup_twists}. For any nonempty $P\subseteq [1,\infty]$ we have canonical isometric isomorphisms 
$$
\TT^P(S,\omega)\cong F^P(\widetilde{\Gg}(S),\LL_{\omega}), \qquad \OO^P(S,\omega)\cong F^P( \Gg(S),\LL_{\omega}).
$$
For $p\in [1,\infty]$ we  have a  representation  ${\tv}^{\rd,p} : S \to \Bound\big(\ell^p(\widetilde{\Gg}(S))\big)$ of $(S,\omega)$ given by 
\[
\tv^{\rd,p}_t\xi[s,\phi] = \begin{cases}
	\omega(t,t^*s) \xi \big( [t^*s,\phi] \big), & s^*tt^*s\in \phi, \\
	0 , & \text{otherwise}, 
\end{cases}
\]
where  $\xi \in \ell^p(\widetilde{\Gg}(S))$, $s^*s\in \phi\in \widehat{\EE}$, $t,s \in S$. It compresses to a representation $\tv^{\es,p} : S \to \Bound\big(\ell^p(\widetilde{\Gg}(S)_{\Hau})\big)$ and  covariant representations 
${\vv}^{\rd,p} : S \to \Bound\big(\ell^p(\Gg(S))\big)$, ${\vv}^{\es,p} : S \to \Bound\big(\ell^p(\Gg(S)_{\Hau})\big)$, and 
we have canonical isometric isomorphisms  
\begin{align*}
	B( \tv^{\rd,p})\cong F^p_{\red}(\widetilde{\Gg}(S),\LL_{\omega}),& \qquad B (\tv^{\es,p})\cong F^p_{\ess}(\widetilde{\Gg}(S),\LL_{\omega})
	\\
	B(\vv^{\rd,p})\cong F^p_{\red}(\Gg(S),\LL_{\omega}),&\qquad
	B(\vv^{\es,p})\cong F^p_{\ess}(\Gg(S),\LL_{\omega}).
\end{align*}
\end{corollary}
\begin{proof}
By Proposition~\ref{prop:disintegration_theorem} we have bijective correspondences between representations of groupoids and covariant representations of the 
corresponding actions. By Lemma~\ref{lem:making_representations_nondegenerate} we may consider only nondegenerate representations, which are then hermitian. 
Combining this with Proposition~\ref{prop:representations_inverse_semigroups_vs_actions} we get  bijective correspondence between nondegenerate representations of groupoids and the relevant representations of $S$. 
This yields isomorphisms for the universal algebras. 
Using the disintegrated form of  a regular representation from \cite{BKM}*{Remark 5.2}, it follows that the maps ${\tv}^{\rd,p}$ and ${\vv}^{\rd,p}$ are representations of $S$ corresponding to 
regular representations of $\widetilde{\Gg}(S)$ and $\Gg(S)$. By construction the same is true for essential representations.
\end{proof}
\begin{definition} For  $\emptyset\neq P\subseteq [1,\infty]$, using the notation of Corollary~\ref{cor:groupoid_presentation_inverse_semigroup_algebras}, we call 
$$
\TT^P_{\red}(S,\omega) \coloneqq  B\Big(\bigoplus_{p\in P} \tv^{\rd,p}\Big) 
\quad\text{and}\quad \TT^P_{\ess}(S,\omega) \coloneqq  B\Big(\bigoplus_{p\in P} \tv^{\es,p}\Big)
$$
the \emph{reduced Toeplitz} and the \emph{essential Toeplitz $L^P$-operator algebra} of $(S,\omega)$, respectively. Similarly, we call
$$
\OO^P_{\red}(S,\omega) \coloneqq  B\Big(\bigoplus_{p\in P} \vv^{\rd,p}\Big), 
\quad\text{and}\quad \OO^P_{\ess}(S,\omega) \coloneqq  B\Big(\bigoplus_{p\in P} \vv^{\es,p}\Big)
$$
the \emph{reduced} and  the \emph{essential $L^P$-operator algebra} of $(S,\omega)$, respectively.
These algebras are canonically isometrically isomorphic to $F^P_{\red}(\widetilde{\Gg}(S),\LL_{\omega})$, 
$F^P_{\ess}(\widetilde{\Gg}(S),\LL_{\omega})$, $F^P_{\red}(\Gg(S),\LL_{\omega})$, and  $F^P_{\ess}(\Gg(S),\LL_{\omega})$,
respectively 
\end{definition}
The above algebras are related to each other by  the  canonical representations, making the following diagram commute:  
\begin{equation}\label{eq:inverse_semigroup_algebras}
\xymatrixcolsep{3pc} \xymatrixrowsep{1pc} 
\xymatrix{
\TT^P(S,\omega) \ar@{->}[r] \ar@{->}[d]&  \TT^P_{\red}(S,\omega) \ar@{->}[r] \ar@{->}[d] &  \TT^P_{\ess}(S,\omega) \ar@{->}[d]
\\
\OO^P(S,\omega) \ar@{->}[r] &  \OO^P_{\red}(S,\omega) \ar@{->}[r] &  \OO^P_{\ess}(S,\omega).
}
\end{equation}
\subsection{Inverse semigroup characterisations of groupoid properties}
The properties of the tight groupoid modelling the algebras in the bottom row of 
\eqref{eq:inverse_semigroup_algebras} can be effectively expressed in terms of $S$. 
Hausdorffness of the universal groupoid is also understood.
\begin{definition}\label{def:InverseSemigroupSimplicityProperties}
Let $S$ be an inverse semigroup with zero and denote by $\EE$ the set of idempotents in $S$. 
We say that $e\in \EE$ is \emph{fixed} by $t\in S$ (or is \emph{$t$-fixed}) if we have $(tft^*)\cdot f\neq 0$, for every nonzero idempotent $f\leq e$. 
Elements of 
\[
F_t \coloneqq \{ e\in \EE\setminus\{0\} : e\leq t \} 
\]
 are fixed by $t$, and we call them \emph{trivially fixed} by $t$. 
Recall that $F \subseteq  \EE$ is a \emph{cover} of  $e \in \EE$ if for every  $0\neq z \leq e$ there is $f \in F$ with $z f \neq 0$.
We  say that the inverse semigroup $S$ is:
\begin{itemize}
\item \emph{Hausdorff} if for every $s,t\in S$ there is a finite $F\subseteq S$ such that $r\leq s,t$ if and only if $r\leq f$ for some $f\in F$, i.e.  $s^{\downarrow}\cap t^{\downarrow}=F^{\downarrow}$ where $F^{\downarrow}$ is the order ideal generated by $F$.  
	\item \emph{closed} if for every $t\in S$ there is a finite $F\subseteq F_t$ that covers every $e\in F_t$; 
	\item  \emph{minimal} if for every $e, f \in \EE\setminus\{0\}$, there is a finite $T\subseteq S$, such that $\{ t f t^* \}_{t\in T}$  covers  $e$;
	\item \emph{topologically free} if for every $e\in \EE$ fixed by $t\in S$ there is  $f\in F_t$ with $f  e\neq 0$; 
	\item \emph{effective} if for every $e\in \EE$ fixed by $t\in S$ there is a finite  $F\in F_t$ that covers $e$; 
	\item  \emph{locally contracting} if for every $e \in \EE\setminus\{0\}$ there exists $s\in S$, a finite set $F\subseteq e s^*s\EE\setminus\{0\}$, and $f_0\in F$, such that for every $f\in F$ the set $F$ is a cover of $sfs^*$ and $f_0 (s fs^*)=0$;
	\item  \emph{strongly locally contracting} if for every $e \in \EE\setminus\{0\}$ there are $s\in S$ and $f_0,f_1 \in \EE\setminus\{0\}$ such that  $ f_0\leq f_1 \leq e s^*s$, $sf_1 s^*\leq f_1$, and   $f_0 (s f_1 s^*)=0$. 
\end{itemize}

\end{definition}
\begin{remark}\label{rem:fixed_idepomtents_inverse_semigroup}
For any  $e\in \EE$ and $t\in S$ we have 
\(
(te t^*) e\neq 0 \, \Longleftrightarrow\, e t^* e\neq 0 \, \Longleftrightarrow\, e t e\neq 0.
\)
Hence, $e$ is fixed by $t$ if and only if it is fixed by $t^*$, and  every trivially fixed idempotent is fixed.
\end{remark}
\begin{remark} 
In Definition~\ref{def:InverseSemigroupSimplicityProperties} we follow the naming from \cite{BKM2} which differs a bit from that in  \cite{Exel-Pardo:Tight}, cf. \cite{BKM2}*{Remark 7.29}. Topological freeness  was introduced in  \cite{BKM2}*{Definition 7.28}. We gave the name effective to a condition appearing in \cite{Exel-Pardo:Tight}*{Theorem 4.10}, and  we gave the name strongly locally contracting to the conditions of \cite{Exel-Pardo:Tight}*{Proposition 6.7}, which imply (by taking $F = \{f_1\}$) that the semigroup is locally contracting. 
 Hausdorff inverse semigroups were introduced in \cite{Steinberg0} under the name ``weak semilattices''. This was changed to ``Hausdorff'' in \cite{Steinberg_Szakacs}.
\end{remark}
\begin{remark}\label{rmk:terminology_changes}
We have the following correspondences between properties of an inverse semigroup $S$ and the associated groupoids:
\begin{align*}
	S \text{ is Hausdorff } &\Longleftrightarrow \,\, \widetilde{\Gg}(S) \text{ is Hausdorff } 
	\\
	S \text{ is closed } &\Longleftrightarrow \,\, \Gg(S) \text{ is Hausdorff } 
	\\
	S \text{ is minimal } &\Longleftrightarrow \,\, \Gg(S) \text{ is minimal } 
	\\
	S \text{ is topologically free } &\Longleftrightarrow \,\, \Gg(S) \text{ is topologically free } 
	\\
	S \text{ is locally contracting } &\Longrightarrow  \,\, \Gg(S) \text{ is locally contracting with respect to $\overline{S}$}
\end{align*}
where $\overline{S}$ is the canonical image of $S$ in $\Bis(\Gg(S))$ extended by  the unit $X$. Moreover, the last
implication is an equivalence if  every tight filter in $\EE$ is an ultrafilter.
The above relationships follow from  \cite{Steinberg0}*{Theorem 5.17}, \cite{Exel-Pardo:Tight}*{Theorems 3.16 and 5.5}, 
\cite{BKM2}*{Proposition 7.31}  and  \cite{Exel-Pardo:Tight}*{Theorem 6.5},
respectively. 
Also 
by \cite{Exel-Pardo:Tight}*{Theorem 4.10} assuming either that  every tight filter in $\EE$ is an ultrafilter 
or that $S$ is closed, we have 
\[
S \text{ is effective } \Longleftrightarrow \,\, \Gg(S) \text{ is effective}. 
\]
In general effectiveness of $\Gg(S)$ implies effectiveness of $S$, but the converse fails, see Example \ref{ex:infinitely_many_edges} below.
In particular, if $S$ is closed, then effectiveness of  $S$ is equivalent to topological freeness of  $S$, but in general it is stritly stronger.
The groupoid $\widetilde{\Gg}(S)$ is not minimal unless it is equal to  $\Gg(S)$ (that is  all filters are tight) and local contractiveness is only useful in the minimal case.  
\end{remark}
\begin{remark}\label{rem:Hausdorff_vs_closed} In general, Hausdorffness  of $S$ is strictly stronger than closedness. In particular, if $S$ is any non-Hausdorff inverse semigroup, 
then adding  a new zero $0_{\text{new}}$,  the inverse semigroup $S\cup \{0_{\text{new}}\}$ remains non-Hausdorff but is trivially closed, as the old zero in $S$
covers any non-zero idempotent in $S\cup \{0_{\text{new}}\}$. Moreover, by \cite{Steinberg0}*{Theorem 5.17}, an inverse semigroup $S$ is  Hausdorff if and only if it has a strong universal property that for every action $h:S\to \PHomeo(X)$ such that the domains $X_t$, $t\in S$ are clopen in $X$,
the transformation groupoid $S\ltimes_{h} X$ is Hausdorff. In particular,  Hausdorffness of $S$ is close to, but still weaker than, being $E^*$-unitary, cf. Section \ref{sect:pseudo_freeness} below.
\end{remark}
\begin{definition}
An inverse semigroup $S$ is \emph{fundamental}, if for all $s,t\in S$ such that $ses^*=tet^*$ for all $e\in \EE$ implies that
$s=t$. We say that $S$ is \emph{quasi-fundamental}, if for all $s,t\in S\setminus\{0\}$ such that $ses^*=tet^*$ for all $e\in \EE$ there is $u\in S\setminus\{0\}$ 
with  $u\leq s,t$.
\end{definition}
\begin{remark}
The notion of fundamental inverse semigroup is standard, cf. \cite{Lawson}*{Section 5.2}. The  quasi-fundamental version was introduced in \cite{Steinberg_Szakacs}, see 
\cite{Steinberg_Szakacs}*{Lemma 2.1}. If $S$ is the inverse semigroup of all compact open bisection of an ample groupoid $\Gg$, then $S$ is fundamental if and only if $\Gg$ is effective,
and 
$S$ is quasi-fundamental if and only if $\Gg$ is topologically free, see \cite{Steinberg_Szakacs}*{Proposition 2.10}. 
However, for more general inverse semigroups  the above equivalences break down, cf. Proposition~\ref{prop:fundamentalness} below.
\end{remark}
Finally we discuss two basic examples that we will unify in the Section \ref{sect:inverse_semigroup}.
 \begin{example}\label{ex:groupoid_inverse_semigroup}
Let $\Gr$ be a discrete groupoid. We turn it into  the inverse semigroup $S(\Gr)=\Gr\cup \{0\}$ with zero,
by   declaring that $gh=0$ whenever $g$ and $h$ are not composable in $\Gr$ (that is if $\sr(g)\neq \rg(h)$). 
Then $\EE(S(\Gr))=\Gr^{0}\cup \{0\}$ and for every $g, h\in \Gr$ we have $g\leq h$ if and only if $g=h$.
Thus, we have
$$
\widetilde{\Gg}(S(\Gr))=\Gg(S(\Gr))=\Gr
$$
and   the correspondences in Remark \ref{rmk:terminology_changes} are clearly visible. In particular,    $S(\Gr)$ is Hausdorff; $S(\Gr)$ is topologically free  if and only if $\Gr$ is principal, i.e. all stabiliser groups are trivial; and $S(\Gr)$  is never locally contracting (``singletons can not be contracted'').
\end{example}

\begin{example}\label{ex:graph_inverse_semigroup}
Let  $S(E)$ be the inverse semigroup of a directed graph $E = (E^0,E^1,\rg,\sr)$, as described in Example~\ref{ex:spectrum_graph_inverse_semigroup}.
Recall that $\EE(S(E))\cong E^*\cup \{0\}$ and $F\subseteq E^*$ covers $\alpha\in E^*$ if and only if
every extension of $\alpha$ is comparable with an element in $F$.
Note that   $(\alpha,\beta) \le (\gamma,\delta)$ in $S(E)$ if and only if 
there is $\eta\in E^*$ such that $\alpha=\gamma\eta$ and $\beta=\delta\eta$. 
Hence $(\gamma,\delta)^{\downarrow}\cap (\gamma',\delta')^{\downarrow}\neq \{0\}$ implies that 
either $(\gamma,\delta)\leq (\gamma',\delta')$ or  $(\gamma',\delta')\leq (\gamma,\delta)$ and so  $(\gamma,\delta)^{\downarrow}\cap (\gamma',\delta')^{\downarrow}=[(\gamma,\delta)\wedge (\gamma',\delta')]^{\downarrow}$.
Hence
the inverse semigroup $S(E)$ is always Hausdorff.

To characterise the other properties we need some more (standard) terminology. 
The \emph{base vertices} of a path $\alpha=\alpha_1\cdots\alpha_n\cdots \in E^*\cup E^\infty$, $\alpha_i\in E^1$, $i=1,\dots $, are the vertices $\rg(\alpha_i)$, $i=1,\dots, n$, and $\sr(\alpha_n)=\sr(\alpha)$.
The path $\alpha$ \emph{has an entrance} if at least one of its base vertices is the range of two edges.
We say that $\alpha\in E^*\setminus E^0$ is a \emph{cycle} if $\sr(\alpha)=\rg(\alpha)$. 
A direct calculation gives 
that all $(\alpha, \alpha)$-fixed idempotents are trivially fixed, and if $\alpha\neq \beta$, then  $(\alpha,\beta)$ fixes a nonzero idempotent  if and only if $\alpha$ and $\beta$   are comparable and their difference is a subpath of a cycle without entrances.
Consequently, 
\begin{align*}
	S(E) \text{ is topologically free}\,\,\Longleftrightarrow\,\, 
	&\text{every cycle in $E$ has an entrance}.
\end{align*}
For two vertices $v,w\in E^0$ we write $v\leftarrow w$ if  $vE^*w\neq \varnothing$, that is if there is a path $\mu$ in $E$ that ends in $v$ and starts in $w$.
The graph $E$ is \emph{cofinal} 
if the set of vertices of every boundary path $\mu = \mu_1 \mu_2 \cdots $ in $E$ is cofinal in the preordered set $(E^0,\leftarrow)$, that is  for every $v\in E^0$ there is $i$ such
that $v\leftarrow \sr(\mu_i)$. That is, if for every vertex $v$ and every singular vertex $w$, there is a path from $w$ to $v$, and for every infinite path $\mu = \mu_1 \mu_2 \cdots$ there is a path from $\sr(\mu_i)$ to $v$. If $E$ is cofinal, then $E$ has at most one source. 
Cofinality of $E$ can also be characterised using hereditary and saturated sets.
A subset $V\subseteq E^0$ is \emph{hereditary} if $V\ni v\leftarrow w \in E^0$ implies that $w\in V$, and
$V$ is \emph{saturated} if for every regular $w \in E^0$ with $\sr(\rg^{-1}(w))\subseteq V$ we have  $w\in V$.

It should be no surprise to experts, and it follows from Propositions~\ref{prop:minimality_inverse_semigroup} and~\ref{prop:cofinal_equiv_no_invariants} which we prove below, that
\begin{align*}
	\text{$S(E)$ is minimal}\,\Longleftrightarrow\,\text{$E^0$ has no nontrivial hereditary, saturated subsets}
	\,\Longleftrightarrow\,\text{$E$ is cofinal.}
\end{align*}
Additionally, $S(E)$  is   locally contracting if and only if it is strongly locally contracting and this holds if and only if every vertex is the range of a path whose source lies in a cycle with an entrance,
see  Proposition~\ref{prop:locally_contracting_semigroup} below.
\end{example}

\section{Self-similar groupoid actions}\label{Sec:Self-similar groupoid actions}\label{sect:self-similar}

Recall that a groupoid $G$ is a small category  in which every morphism is invertible, and a directed graph $E$ generates a path category $E^*$ in which all non-identity morphisms are not invertible, see Example~\ref{ex:spectrum_graph_inverse_semigroup}. The pair $(G,E)$ is self-similar if the categories  $G$ and $E^*$ act on each other in a consistent way.
Therefore, we start by recalling the relevant notation and definitions, cf. \cite{Mundey_Sims}.

We identify a small category $\CC$ with its set of morphisms and write $\CC^0\subseteq \CC$ for its set of objects. We denote by $\rg,\sr:\CC\to \CC^0$ the range (codomain) and source (domain) maps.  
If $\DD$ is another small category with the same set of objects $\DD^0=\CC^0$, then for any subsets $C\subseteq \CC$, $D\subseteq \DD$ we write
\[
C*D  \coloneqq  C\fibre{\sr}{\rg}D=\{(c,d)\in C\times D: \sr(c)=\rg(d)\}.
\] 
In the sequel, the maps $\rg,\sr$ will be clear from the context.
When declaring that some relations concerning morphisms hold, we implicitly assume that they make sense (so that the sources and ranges of morphisms match). 
A \emph{left action} of a small category $\CC$ on a set $X$ consists of maps $\rg:X\to \CC^0$ and 
$\cdot: 
\CC *X \to  X$
such that 
\[
\rg(x)\cdot x=x\,\,\, \text{ and }\,\,\, (c_1c_2)\cdot x=c_1\cdot (c_2\cdot x) \quad \text{ for all }(c_1,c_2,x)\in \CC^2*X
\] 
(which in particular forces $\rg(c\cdot x)=\rg(c)$ for all $(c,x)\in \CC*X$). We will also usually suppress writing $\cdot$ and say $cx  \coloneqq  c \cdot x$. 
Similarly,  a \emph{right action} of $\CC$ on $X$ consists of maps $\sr:X\to \CC^0$ and 
$\cdot: %
X*\CC  \to  X$
such that $x \cdot \sr(x) =x$ and $ x\cdot (c_1c_2)=(x \cdot c_1)\cdot c_2$ for all $(x, c_1,c_2)\in X*\CC^2$.
A left action $\cdot: 
\CC *X \to  X$ is \emph{faithful} if $\sr(c_1)=\sr(c_2)$ and 
$c_1 x=c_2x$ for every $x\in \sr(c_1)X$ implies that $c_1=c_2$. Faithfulness of a right action is defined similarly.
 
Some authors, \cite{Duwenig_Li}, prefer to work with  matched  pairs  without the assumption  $\DD^0=\CC^0$, but  this can  
always be arranged by passing to ``transformation categories'', cf. \cite{Duwenig_Li}*{Proposition~2.23} and Example~\ref{exm:Exel_Pardo_vs_groupoid_self-similar} below.
\subsection{Self-similar actions as  groupoid actions with a cocycle}

\begin{definition}\label{defn:self_similar_action}
A \emph{self-similar action} of a  groupoid $\Gr$ on a directed graph 
$E = (E^0,E^1,\rg,\sr)$ with $\Gr^0 = E^0$ is a left $\Gr$-action $\Gr* E^1\to  E^1$ on the set of edges $E^1$ 
equipped with a $1$-cocycle \(\Gr* E^{1}\ni(g,e)\mapsto g|_e\in \Gr\), which in this context means that for all $(h,g,e) \in \Gr^2* E^{1}$,
\[
(hg)|_e  = (h|_{ge})  (g|_e), \qquad \sr(g|_e)=\sr(e),  \qquad \rg(g|_e)=\sr(ge).
\]
\end{definition}
\begin{example}
A \emph{self-similar action of a group $\Gamma$ on a set} $X$ is an action of $\Gamma$ by bijections of $X$ equipped with  the section map \(\Gamma \times X\ni(g,e)\mapsto g|_e\in \Gamma\) satisfying the $1$-cocycle identity  $(hg)|_e  = (h|_{ge})  (g|_e)$ for all $h,g\in \Gamma$ and $e\in X$, see \cite{Nekrashevych:Self-similar}. Such actions are nothing but groupoid actions on  the graph with single vertex and the set of edges $E^1=X$. 
\end{example}

\begin{example}\label{exm:Exel_Pardo_vs_groupoid_self-similar}
A  \emph{left automorphism} of the graph $E = (E^0,E^1,\rg,\sr)$ is a bijection $\sigma:E^{0}\sqcup E^{1}\to E^{0}\sqcup E^{1}$ such that $\sigma(E^{i})=E^{i}$ for $i=0,1$ 
and $\rg\circ \sigma= \sigma \circ \rg$ on $E^1$. 
It is an  \emph{automorphism} of $E$ if in addition $\sr\circ \sigma=\sigma \circ \sr$ on $E^1$.
 A \emph{self-similar group action} of $\Gamma$ on $E$ is an action of $\Gamma$ by left automorphisms of $E$  equipped with a $1$-cocycle \(\Gamma \times E^{1}\ni(g,e)\mapsto g|_e\in \Gamma\), which in this context means that  for  all $h,g\in \Gamma$  and $e\in  E^{1}$
\begin{equation}\label{eq:cocycle_for_self_similar}
(hg)|_e  = (h|_{ge})  (g|_e), \qquad \sr(ge)=g|_{e}\sr(e),  \qquad \rg(g e)=g\rg(e).
\end{equation}
Such an action can be naturally treated as a self-similar action of the transformation groupoid $\Gr=\Gamma\times E^0$ on $E$. 
Moreover, every self-similar groupoid action by a transformation groupoid $\Gr=\Gamma\times E^0$ arises this way,
see \cite{Anutnes_Ko_Meyer}*{Proposition~4.3}.
When $\Gamma$ acts by graph automorphisms, not just left automorphisms, then 
 conditions \eqref{eq:cocycle_for_self_similar} reduce to 
\[
(hg)|_e  = (h|_{ge})  (g|_e)\qquad g|_e v=g v,
\]
 for all $h,g\in \Gamma$, $e\in  E^{1}$ and $v\in E^0$, cf. \cite{Laca-Raeburn-Ramagge-Whittaker:Self-similar-groupoids}*{Appendix A}.
Originally,  Exel and Pardo \cite{Exel-Pardo:Self-similar} introduced self-similar actions of groups on graphs  in the latter sense,
and they considered finite graphs without sources.
\end{example}
Self-similar actions on a graph $E$ extend uniquely to actions on the path category $E^*$, as explained in the following proposition, cf.
\cite{Exel-Pardo:Self-similar}*{Proposition~2.4}. 

\begin{proposition}\label{prop:extension_of_self-similar_action_on_graph}
For any  self-similar groupoid action $(\Gr,E)$ 
the left action $\Gr  * E^1 \to E^1$  and  $1$-cocycle $\Gr  * E^1 \to \Gr$ extend uniquely 
to a left action $ \Gr  * E^* \to E^*$ of $\Gr$  and a right action $ \Gr  * E^*\to \Gr$ of $E^*$ such that
\begin{equation}\label{eq:matching_for_self-similar}
	g \cdot (\mu\nu) = (g \cdot \mu) (g|_{\mu} \cdot \nu),\qquad  \text{for }(g,\mu,\nu) \in \Gr*E^**E^*
\end{equation}
(in particular $ \sr(g|_\mu)=\sr(\mu)$, $\rg(g|_\mu)=\sr(g\mu)$).
These extended actions necessarily satisfy 
\begin{equation}\label{eq:matching_for_self-similar2}
	(hg)|_\mu  = (h|_{g\mu})  (g|_\mu),\qquad  \qquad \text{for } (h,g,\mu)\in \Gr^2* E^{*},
\end{equation}
and the left action preserves the length of paths, so $g \cdot (\sr(g)E^n) = \rg(g) E^n$ for  $n \in \N$,  $g \in \Gr$.
\end{proposition}
\begin{proof}
For each $n \ge 0$ let $E^{\le n}$ denote the collection of paths of length at most $n$. 
By \eqref{eq:matching_for_self-similar} the extended left action has to preserve the length of paths. 
For $(g,v)\in \Gr*E^0$, the action axioms  force us to put $g \cdot v \coloneqq  \rg(g)$ and  $g|_v= g$. Thus,  we have defined extensions $ \Gr  * E^{\leq 1} \to E^{\leq 1}$  and $ \Gr  * E^{\leq 1} \to \Gr$. By \eqref{eq:matching_for_self-similar} further extensions have to  satisfy  the following recursive formula:
once  we have defined maps $ \Gr  * E^{\leq n} \to E^{\leq n}$  and $ \Gr  * E^{\leq n} \to \Gr$ for $n\in\N$ then  $ \Gr  * E^{\leq n+1} \to E^{\leq n+1}$  and $ \Gr  * E^{\leq n+1} \to \Gr$ are given by
\begin{equation}\label{eq:recursive_extension_rules}
	g\cdot(e\mu) \coloneqq  (g\cdot e)(g|_{e}\mu),\qquad g|_{e\mu} \coloneqq (g|_{e})|_{\mu}, \qquad (g,e,\mu)\in \Gr*E^{1}*E^{\leq n}.
\end{equation}
One just needs to check that this recursive recipe works and  produces the desired actions. This can  be proved by induction on $n$. Namely, assume that for $n\in\N$ we have well-defined maps $ \Gr  * E^{\leq n} \to E^{\leq n}$  and $ \Gr  * E^{\leq n} \to \Gr$  satisfying
\begin{enumerate}[label=(\alph*)]
	\item\label{ite:extended_actions1}   $g \cdot (\mu\nu) = (g \cdot \mu) (g|_{\mu} \cdot \nu)$ for  $(g,\mu\nu) \in \Gr*E^{\leq n }$, 
	in particular $\sr(g|_\mu)=\sr(\mu)$ and $\rg(g|_\mu)=\sr(g\mu)$;
	\item\label{ite:extended_actions2} $g|_{\mu\nu} = (g|_\mu)|_{\nu}$ for all $(g,\mu\nu) \in \Gr*E^{\leq n }$;
	\item\label{ite:extended_actions3} $(hg)|_\mu  = (h|_{g\mu})  (g|_\mu)$  for $(h,g,\mu)\in \Gr^2* E^{\leq n}$; 
	\item\label{ite:extended_actions4} $(h g)\mu=h (g\mu)$ and $\rg(\mu) \cdot \mu=\mu$ for all $(h,g,\mu)\in \Gr^2* E^{\leq n}$.
\end{enumerate}
For $n=1$ these properties hold, so suppose that $n > 1$.
Assumption  \ref{ite:extended_actions1} implies that  $g \cdot (\sr(g)E^n) = \rg(g) E^n$.
In particular, since $\sr(g\cdot e)=\rg(g|_{e})$ and $\sr(g|_e)=\sr(e)$ for $(g,e)\in \Gr*E^{1}$, the formulas in \eqref{eq:recursive_extension_rules}
make sense. 	To show the inductive step for \ref{ite:extended_actions1} and \ref{ite:extended_actions2} note that for 
$(g,e,\mu\nu)\in \Gr*E^1*E^{\leq n}$,
\begin{align*}
g \cdot (e\mu\nu) &\stackrel{\eqref{eq:recursive_extension_rules}}{=}(ge)(g|_e \mu\nu)
\stackrel{\ref{ite:extended_actions1}}{=}(ge)(g|_e \mu) (g|_e|_{\mu}\nu)\stackrel{\eqref{eq:recursive_extension_rules}}{=} (g \cdot e\mu) (g|_{e\mu} \cdot \nu), 
\\
g|_{e\mu\nu} &\stackrel{\eqref{eq:recursive_extension_rules}}{=} (g|_e) |_{\mu\nu}\stackrel{\ref{ite:extended_actions2}}{=}[(g|_e)|_\mu]|_{\nu}\stackrel{\eqref{eq:recursive_extension_rules}}{=} (g|_{e\mu})|_{\nu}.
\end{align*}
For the inductive step for \ref{ite:extended_actions3} and \ref{ite:extended_actions4} let 
$(h,g,e,\mu)\in \Gr^2*E^1*E^{\leq n}$ and note that 
\begin{align*}
	(hg)|_{e\mu}\stackrel{\eqref{eq:recursive_extension_rules}}{=} & [(hg)|_{e}]|_{\mu}\stackrel{\ref{ite:extended_actions3}}{=}(h|_{ge} \cdot g|_e)|_{\mu}
	\stackrel{\ref{ite:extended_actions3}}{=}(h|_{ge})|_{g|_e \mu} \cdot (g|_e)|_{\mu}\\
	\stackrel{\eqref{eq:recursive_extension_rules}}{=} &h|_{(ge) (g|_e \mu)}  \cdot g|_{e\mu}
	\stackrel{\eqref{eq:recursive_extension_rules}}{=}h|_{g (e\mu)} g|_{e \mu},
\end{align*}
\begin{align*}
	h(g (e\mu))&\stackrel{\eqref{eq:recursive_extension_rules}}{=} h [(g e)(g|_{e}\mu)] 
	\stackrel{\eqref{eq:recursive_extension_rules}}{=} (h g e)
	h|_{ge} (g|_{e}\mu)\stackrel{\ref{ite:extended_actions4}}{=}(h g e)
	(h|_{ge} g|_{e})\mu\\	
	&\stackrel{\ref{ite:extended_actions3}}{=}(h g e) (hg)|_{e} \mu\stackrel{\eqref{eq:recursive_extension_rules}}{=} (hg) e\mu.
\end{align*}
This finishes the proof.
\end{proof}
\begin{corollary}\label{cor:self-similar_actions_on_graphes_path_categories}
Let $\Gr$ be a groupoid and $E$ be a directed graph with $\Gr^0=E^0$.
There is a bijective correspondence between self-similar actions of $\Gr$ on $E$ and pairs of actions $\Gr  * E^* \to E^*$   and $ \Gr  * E^*\to \Gr$ 
satisfying \eqref{eq:matching_for_self-similar} and \eqref{eq:matching_for_self-similar2} and such that $\Gr  * E^* \to E^*$  restricts to $\Gr  * E^{1} \to E^{1}$ (equivalently the left action preserves the length of paths).
\end{corollary}
\begin{proof}
By Proposition~\ref{prop:extension_of_self-similar_action_on_graph} any self-similar action extends uniquely to the desired pair of actions. Conversely, any such pair of actions  $\Gr  * E^* \to E^*$   and $ \Gr  * E^*\to \Gr$ comes from the self-similar action given by the restricted actions 
$\Gr  * E^{1} \to E^{1}$ and $\Gr  * E^{1} \to \Gr$.
\end{proof}

\begin{remark} We will often identify self-similar actions with the corresponding extended actions $\Gr  * E^* \to E^*$   and $ \Gr  * E^*\to \Gr$.
We note that for $g \in G$ and $\mu \in \sr(g)E^*$,~\eqref{eq:matching_for_self-similar2} forces
\begin{equation}\label{equ:inverse_vs_restriction}
	(g|_\mu)^{-1} = g^{-1}|_{g\mu},
\end{equation}
cf. \cite{Laca-Raeburn-Ramagge-Whittaker:Self-similar-groupoids}*{Proposition 3.6(4)}.
\end{remark}
The left action above extends to a continuous action on the path space. Recall that a  left action of a discrete groupoid $\Gr$ on a topological space  $X$ is \emph{continuous} if both the  action map $\Gr\times X\supseteq \Gr *X \to  X$ and the anchor map $\rg:X\to \Gr^0$ 
are continuous. Equivalently,  the maps $\rg^{-1}(\sr(g))\ni x \mapsto  g\cdot x\in \rg^{-1}(\rg(g))$ are partial homeomorphisms of $X$,  for all $g\in \Gr$. 
\begin{lemma}\label{lem:actions_on_infinite_paths}
For any self-similar action $(\Gr, E)$  the left action  $\Gr  * E^* \to E^*$  extends uniquely to  a continuous action $\Gr  * E^{\le \infty} \to  E^{\le \infty} $ 
of the groupoid $\Gr$ on the path space $E^{\le \infty}=E^*\cup E^\infty$, with the 
range map $\rg:E^{\le \infty}\to E^0=\Gr^0$ as the anchor map. This action descends to a continuous action $\Gr  * \partial E \to  \partial E$ on the boundary path space $\partial E\subseteq E^{\le \infty}$.
\end{lemma}
\begin{proof}
Since  $E^*$ is dense in  $E^{\le \infty}$ the range map $\rg:E^{\le \infty}\to E^0$ is the unique continuous 
extension of its restriction to $E^*$. For any $\mu=\mu_1\mu_2\cdots\in E^\infty$ its finite prefixes $\mu_1\cdots\mu_n\in E^n$ converge to $\mu$ in $E^{\le \infty}$.
For any $g\in \Gr \rg(\mu)$ we have $g(\mu_1\cdots\mu_n)= (g\mu_1)(g|_{\mu_1}\mu_2)\cdots (g|_{\mu_1\cdots\mu_{n-1}}\mu_n)$. Thus, if the 
extended action exists, it has to be given by 
\[
g(\mu_1\mu_2\cdots) \coloneqq  (g\mu_1)(g|_{\mu_1}\mu_2)\cdots (g|_{\mu_1\cdots\mu_{n-1}}\mu_n)\cdots \in E^\infty.
\] 
This defines a homomorphism $\sr(g)E^\infty\mu\mapsto g\mu\in  \rg(g)E^\infty$ uniquely determined by the property that $g(\mu_1\cdots \mu_n)$ is a prefix of $g(\mu)$ for all $n\in \N$. 
It is immediate to see that the constructed  map $\Gr  * E^{\le \infty} \to  E^{\le \infty} $ is a continuous action.

The left action of $g \in \Gr$ establishes a bijection $\sr(g)E^1\cong \rg(g)E^1$, so $\sr(g)$ is a source or infinite receiver if and only if $\rg(g)$ is a source or infinite receiver, respectively. Since $\partial E$ is closed in $E^{\le \infty}$, the action restricts to a continuous action $\Gr  * \partial E \to  \partial E$.
\end{proof}
\begin{remark}\label{rem:transformation_groupoids_from_groupoid_actions}
We may treat $\Gr* E^{\le \infty}$ and  $\Gr  * \partial E$  as transformation groupoids $\Gr\rtimes E^{\le \infty}$ and $\Gr\rtimes \partial E$, respectively. Here two arrows  $(g,\xi)$ and $(h,\eta)$ are composable if and only if $\xi=h\eta$, in which case $(g,\xi)\cdot (h,\eta)=(gh,\eta)$.
\end{remark}
\subsection{Other pictures of self-similar actions}
\begin{definition}\label{def:self_similar_faithful}
A self-similar action $(\Gr,E)$ is \emph{faithful}, if the corresponding left action of $\Gr$ on  $E^*$  is faithful, that is $g \cdot \mu =\mu$ for all $\mu \in \sr(g) E^*$ implies $g\in \Gr^0$. 
\end{definition}
\begin{remark}\label{rem:kernel_of_the_action}
Every self-similar action  $(\Gr,E)$ factors through a faithful self-similar action. Namely,
$$
N \coloneqq \{g\in \Gr: g\mu=\mu \text{ for every }\mu \in \sr(g)E^*\}
$$
is a wide  subgroupoid of $\Gr$, in fact of $\Iso(\Gr)$. We call it the \emph{kernel} of the action of $\Gr$ on $E^*$. It is a normal subgroupoid in the sense that $g N g^{-1}\subseteq N$ for all $g\in \Gr$, and thus $\Gr/N\coloneqq 
\{g N:g\in \Gr\}$ is naturally a groupoid with the same unit space $\Gr^0$, cf. \cite{PT}. Clearly, 
the left action of $\Gr$ on $E^*$ factors through to the faithful action of $\Gr/N$ on $E^*$.
Moreover,  $N$ is closed under sections, because if $g$ fixes all paths, then all their sections also fix all paths. 
Therefore, the right action of $E^*$ on $\Gr$ descends to an action of $E^*$ on $\Gr/N$. 
The pair $(G/N, E^*)$ with these actions satisfy analogues of \eqref{eq:matching_for_self-similar} and \eqref{eq:matching_for_self-similar2}.
Hence, it is a faithful self-similar action. See \cite{Miller_Steinberg}*{2.3} for a correspondence approach to this construction. 
\end{remark}

For a faithful self-similar action, the right action $ \Gr  * E^* \to \Gr$  satisfying \eqref{eq:matching_for_self-similar} is 
uniquely determined by the left action $ \Gr  * E^* \to E^*$, and also  existence of this right action can be easily characterised.  
This leads to the definition of a self-similar action from \cite{Laca-Raeburn-Ramagge-Whittaker:Self-similar-groupoids}, where only faithful self-similar actions were considered.
The authors	 of \cite{Laca-Raeburn-Ramagge-Whittaker:Self-similar-groupoids} introduced such actions using 
the groupoid of partial automorphisms of the forest associated to $E$. The \emph{forest} in question is defined as the disjoint union $T_{E} \coloneqq  \bigsqcup_{v\in E^0} vE^*$ and the subsets $vE^*$, $v\in E^0$, are viewed as  \emph{trees} (oriented trees with the underlying edges $\mu \to \mu e$ directed from the root in $E^0$). An isomorphism between two such trees is a bijection $\Phi:vE^*\to wE^*$ that respects concatenation in the sense that $\Phi(\mu e)\in \Phi(\mu)E^1$ for all $\mu\in vE^*$ and $e\in \sr(\mu)E^1$. By induction this last assumption is equivalent to the condition that 
for all $\mu\in vE^*$ and $\eta \in \sr(\mu)E^n$ there exists a (necessarily unique) $\Phi|_{\mu}(\eta)\in \sr(\Phi(\mu)) E^n$ such that 
\begin{equation}\label{equ:section_of_iso}
\Phi(\mu \eta)=\Phi(\mu)\Phi|_{\mu}(\eta).
\end{equation}

We denote by $\PIso(E^*)$ the set of all  isomorphism between trees, which we view as partial automorphisms of the forest $T_E$. Their composition, whenever non-empty, is again an isomorphism of trees and an inverse of an isomorphism is an isomorphism, see \cite{Laca-Raeburn-Ramagge-Whittaker:Self-similar-groupoids}*{Proposition 3.2}. So $\PIso(E^*)$ forms  a groupoid  whose unit space may be identified with the set of vertices $E^0$. Then the range and source of 
an isomorphism $\Phi:vE^*\to wE^*$ is $w$ and $v$, respectively.

\begin{proposition}
For every directed graph $E$, the pair $(\PIso(E^*), E)$ is   a faithful self-similar action     where  
\[
\PIso(E^*)*E^*\ni (\Phi,\mu)\mapsto \Phi(\mu)\in E^*\,\, \text{ and }\,\,\PIso(E^*)*E^*\ni (\Phi,\mu)\mapsto \Phi|_{\mu}\in \PIso(E^*).
\]
It is universal in the sense that every self-similar groupoid action $(G,E)$ factors through $(\PIso(E^*), E^*)$: there is a unique
groupoid homomorphism $\varphi:G\to \PIso(E^*)$ satisfying
$$
\varphi(g)(\mu)=g\mu \qquad\text{and}\qquad \varphi(g|_{\mu})=\varphi(g)|_{\mu}
$$
	for all $(g,\mu)\in  \Gr  *E^*$. Such  $\varphi$ is injective if and only if $(G,E)$ is faithful. Therefore, faithful self-similar actions on $E$ can be identified with wide subgroupoids $G\subseteq  \PIso(E^*)$ 
that are closed under sections in the sense that $\Phi|_{\mu}\in G$ for all $\Phi\in G$ and $\mu \in \sr(\Phi)E^*$.
\end{proposition}
\begin{proof}
It is immediate that the evaluation defines a faithful left action of $\PIso(E^*)$ on $E^*$. 
Let us consider  sections of  $\Phi\in\PIso(E^*)$. For any $\mu\in \sr(\Phi)E^*$ relation \eqref{equ:section_of_iso} defines
a map $\Phi|_{\mu}\colon \sr(\mu)E^*\to \sr(\Phi(\mu))E^*$.  If $\Phi|_{\mu}(\eta)=\Phi|_{\mu}(\eta')$, then 
$\Phi(\mu \eta)=\Phi(\mu \eta')$ which forces $\eta=\eta'$ because $\Phi$ is injective. 
Also, for any $\nu\in \sr(\Phi(\mu))E^*$ there is $\eta$ such that $\Phi(\mu\eta)=\Phi(\mu)\nu$ by surjectivity of $\Phi$ and the ``factorisation property'' \eqref{equ:section_of_iso}. This implies $\Phi|_{\mu}(\eta)=\nu$. Hence, $\Phi|_{\mu}\colon \sr(\mu)E^*\to \sr(\Phi(\mu))E^*$ is a bijection. 
Moreover, for any $\eta\nu\in \sr(\mu)E^*$ we have 
$$
\Phi(\mu)\Phi|_{\mu}(\eta) \Phi|_{\mu}|_{\eta} (\nu)=\Phi(\mu)\Phi|_{\mu}(\eta\nu)=\Phi(\mu\eta\nu)=\Phi(\mu)\Phi|_{\mu\eta}(\nu),
$$
which implies that $\Phi|_{\mu}|_{\eta}=\Phi|_{\mu\eta}$. It follows that $\Phi|_{\mu}$ is an isomorphism of trees and that the map 
$\PIso(E^*)*E^*\ni (\Phi,\mu)\mapsto \Phi|_{\mu}\in \PIso(E^*)$ is a well-defined right action. 
By construction the two actions satisfy \eqref{eq:matching_for_self-similar}. To check  \eqref{eq:matching_for_self-similar2} let 
$(\Phi,\Psi)$ be composable elements of $\PIso(E^*)$ and let $\mu
\eta\in \sr(\Psi)E^*$. Then
\[
(\Phi\circ \Psi)(\mu) (\Phi\circ \Psi)|_{\mu}(\eta)=(\Phi\circ \Psi)(\mu\eta)=\Phi (\Psi(\mu)\Psi|_{\mu} (\eta))=(\Phi\circ \Psi)(\mu) (\Phi|_{\Psi(\mu)}\circ \Psi|_{\mu}) (\eta).
\]
Hence, $(\Phi\circ \Psi)|_{\mu}=\Phi|_{\Psi(\mu)}\circ \Psi|_{\mu}$. This shows that $(\PIso(E^*), E)$ is   a faithful self-similar action.

Now fix a self-similar action $(G,E)$. 
Every $g\in G$ maps $\sr(g)E^*$ bijectively onto $\rg(g)E^*$. It follows from \eqref{eq:matching_for_self-similar}
that for every $g\in G$  the formula $\varphi(g)(\mu)=g\mu$ defines a tree isomorphism $\varphi(g):\sr(g)E^*\to \rg(g)E^*$ and moreover $\varphi(g|_{\mu})=\varphi(g)|_{\mu}$ for $(g,\mu)\in  \Gr  *E^*$.   By the axioms of the left action, $\varphi:G\to \PIso(E^*)$ is a groupoid homomorphism.
If $\varphi$ is injective, we may identify $G$ with $\varphi(G)$, and then
the right action of $E^*$ on $G$ has to be the restriction of the right action of $E^*$ on  $\PIso(E^*)$.
This gives the last part of the assertion.  \end{proof}

\begin{remark} Using the above picture one could define self-similar groupoid actions as pairs $(G,E)$  equipped with a groupoid homomorphism 
$\varphi:G\to \PIso(E^*)$, which acts as the identity on $\Gr^0 = E^0$, and a right action $\cdot \colon \Gr  *E^* \to \Gr$ such that 
$\varphi(g|_{\mu})=\varphi(g)|_{\mu}$ for $(g,\mu)\in  \Gr  *E^*$.
When $\varphi$ is injective,  existence of this right action (which is then necessarily unique) is equivalent to assuming that 	
$\varphi(g)|_{e}\in \varphi(G)$ for all  $(g,e)\in  \Gr *E^1$.
\end{remark}

\begin{remark} 
Faithful self-similar actions are usually constructed from a finite data called automata. An \emph{automaton} over  a finite graph
$E$, see \cite{Laca-Raeburn-Ramagge-Whittaker:Self-similar-groupoids}*{Definition 3.7}, is a finite set $A$ containing $E^0$ together with anchor maps $\rg,\sr:A\to E^{0}$, which are identities on $E^0$, and  an input-output function
\[
A\fibre{\sr}{\rg}E^1\ni (a,e)\longmapsto (a\cdot e, a|_e)\in  E^1\fibre{\sr}{\rg}A
\]
such that for every $a\in A$, $e\mapsto a\cdot e$ is a bijection $\sr(a)E^1 \stackrel{\cong}{\to} \rg(a)E^1$, and for any $(a,e)\in A\fibre{\sr}{\rg}E^1$ we have $\sr(a|_e)=\sr(e)$ and $\rg(a|_e)=\sr(a\cdot e)$. 
By \cite{Laca-Raeburn-Ramagge-Whittaker:Self-similar-groupoids}*{Theorem 3.9}, any such automaton $A$ generates a subgroupoid $\Gr_A$ in $\PIso(E^*)$ that is closed under sections, and so it acts in a faithful and self-similar  way on $E$.
\end{remark}

Another equivalent description of a self-similar action uses the notion of a matched pair of categories from \cite{Mundey_Sims}. It treats $G$ and the path category generated by $E$, as well as relations \eqref{eq:matching_for_self-similar} and \eqref{eq:matching_for_self-similar2}, on  equal footing. 

\begin{definition}[\cite{Mundey_Sims}*{Definition~3.1}]\label{
	Def:matched_pair}
A pair of small categories $(\CC,\DD)$  is \emph{matched} if $\CC^0=\DD^0$ and $\CC$ and $\DD$  act on each other via  left   and  right actions
$\rhd: \CC *\DD \to  \DD$ and $\lhd: \CC *\DD \to  \CC$ such that 
\[
c_2 \la (d_1d_2) = (c_2 \la d_1) ((c_2 \ra d_1) \la d_2) \quad \text{and} \quad (c_1c_2) \ra d_1 = (c_1 \ra (c_2 \la d_1)) (c_2 \ra d_1)
\]
for all
$(c_1,c_2,d_1,d_2) \in \CC^2 *\DD^2$. In particular,  $\sr(c \la d)=\rg(c \ra d)$ for  $(c,d)\in \CC*\DD$.
\end{definition}
\begin{proposition}[{cf. \cite{Mundey_Sims}*{Proposition~3.32}}]
Let $\Gr$ be a  groupoid and let  $E$ be a directed graph   with $\Gr^0 = E^0$.
The formulae 
\[
g\cdot\mu =g\la \mu \quad \text{and} \quad g|_{\mu}=g\ra \mu, \quad (g,\mu) \in \Gr * E^*,
\]
establish a bijective correspondence between self-similar actions 
of $\Gr$ on $E$ and matched pairs $(\Gr,E^*,\la,\ra)$ such that 
$|g \la \mu| = |\mu|$ for all $(g,\mu) \in \Gr *E^*$ (that is the left action preserves the length of paths). 
\end{proposition}
\begin{proof} For any pair of actions $ \Gr  * E^* \to E^*$   and  $ \Gr  * E^*\to \Gr$,  under the suggested notation, 
the relations \eqref{eq:matching_for_self-similar} and \eqref{eq:matching_for_self-similar2} become matching relations from Definition~\ref{
	Def:matched_pair}. Hence, the assertion follows from Corollary~\ref{cor:self-similar_actions_on_graphes_path_categories}.
	\end{proof}
	\begin{remark}\label{rem:Zappa--Szep_product}
To any matched pair of categories we may associate its Zappa--Sz\'ep product category, see \cite{Mundey_Sims}*{Definition~3.6}.
In the case of the self-similar action of $\Gr$ on $E$, the corresponding  \emph{Zappa--Sz\'ep product category} is $E^*\bowtie \Gr  \coloneqq  E^**\Gr$ with the composition law
\[
(\mu,g)\bowtie (\nu,h) \coloneqq (\mu(g\nu),g|_{\nu}h) \qquad \text{for }  (\mu,g,\nu,h)\in E^**\Gr*E^**\Gr.
\]
We identify $(E^*  \bowtie \Gr)^0$ with $G^0=E^0$ and the maps $\Gr\ni g\mapsto (\rg(g),g)\in E^* \bowtie \Gr$ and $E^*\ni \mu\mapsto (\mu,\sr(\mu))\in E^* \bowtie \Gr$
are faithful functors. The category $E^*\bowtie \Gr$ is left cancellative, cf. \cite{Mundey_Sims}*{Example 7.3}.
\end{remark}

Yet another description explains an asymmetry between the left and right actions in the above considerations. Namely, it can be viewed as an analogue of an asymmetry we can see in $\Cst$-correspondences where only the right actions are equipped with an inner product.

\begin{definition}[\cite{Anutnes_Ko_Meyer}*{Definition~3.1}]
A \emph{self-correspondence over  a discrete groupoid}  \(\Gr\)  is a set $X$ equipped with commuting 
left and right $\Gr$-actions and such that the right $\Gr$-action is free, which means that
the map $X*\Gr \ni (x, g)\mapsto (x\cdot g, x)  \in  X\times X$ is injective.
\end{definition}

\begin{proposition}[\cite{Anutnes_Ko_Meyer}*{Example~4.4}]\label{prop:correspondence_from_self_similar}
Let $\Gr$ be a  groupoid. For any self-similar action of $G$ on a directed graph  $E = (E^0,E^1,\rg,\sr)$ the formulas 
\[
X=E^1* G,\qquad g \cdot (e,h)=(g \cdot e,g|_eh),\qquad (e,h)\cdot k=(e,hk),
\]
$\sr(e,h)=\sr(h)$, and $\rg(e,h)=\rg(e)$, for $(g,e,h,k) \in \Gr * E^1  *\Gr^2$, define a self-correspondence over $\Gr$.
Moreover, up to an isomorphism every self-correspondence over $\Gr$ arises in this way, and the associated directed graph is determined up to isomorphism by the 
correspondence.
\end{proposition}
\begin{proof}
 For the first part, the right action on $X=E^1* G$ is well-defined and free as it is given by the right action of $G$ on itself. The composition law for the left action on $X$ is exactly the $1$-cocycle identity for the self-similar action. 

Now, consider  a self-correspondence over $\Gr$ on a set $X$. Put $E^0=G^0$ and let $E^{1}\subseteq X$ be a fundamental domain $E^{1}\subseteq X$ for the orbit space $X/\Gr$ of the right $\Gr$-action (that is $E^{1}$ contains exactly one point from each of these orbits). 
Then, together with anchor maps restricted from $X$ to $E^1$, we get a directed graph $E = (E^0,E^1,\rg,\sr)$.  Since the right $\Gr$-action is free we get that
$E^1*\Gr\ni (e,g)\mapsto e\cdot g\in X$ is an isomorphism of right $\Gr$-sets. Under this isomorphism the right $\Gr$-action on $X\cong E^1*\Gr$ is  as described in the assertion. 
For any $(g,e)\in \Gr*E^{1}$ there is  a unique element in $\sr(ge)\Gr$ that we denote by $g|_e$ such that  $(g e) g|_e\in E^{1}$. Since  the actions commute, the map $ \Gr*E^{1}\ni (g,e)\mapsto g e g|_e\in E^{1}$ is a left action of $\Gr$ on $E^{1}$. Using this left action on $E^1$,  the left $\Gr$-action on $X$, under the isomorphism $X\cong E^1*\Gr$, is  as described in the assertion. This  forces the map $\Gr*E^{1}\ni (g,e)\mapsto g|_e\in\Gr$ to satisfy the relations of Definition~\ref{defn:self_similar_action}. 
\end{proof}

\section{The inverse semigroup analysis} 
\label{sect:inverse_semigroup}
In this section we analyse the inverse semigroup associated to a self-similar action $(\Gr,E)$. We are mainly concerned with the properties described in Definition~\ref{def:InverseSemigroupSimplicityProperties}. In particular, we generalise and improve upon a number of results 
from \cites{Exel-Pardo:Self-similar, Exel-Pardo-Starling:Self-similar, Deaconu}.

\subsection{The inverse semigroups and their partial order}
\begin{definition}
\label{def:inverse_semigroup_ssa}
The \emph{inverse semigroup of the  self-similar action} $(\Gr,E)$ is 
$$
S(\Gr,E)=E^* \fibre{\sr}{\rg} \Gr\fibre{\sr}{\sr}E^*\cup \{0\}=\left\{(\alpha,g,\beta): \alpha,\beta\in E^*, g\in \sr(\alpha)\Gr{\sr(\beta)}\right\}\cup \{0\}
$$
with the multiplication given by
\[
(\alpha,g,\beta) (\gamma,h,\delta)=
\begin{cases} (\alpha (g\beta'), g|_{\beta'}h, \delta), & \text{if } \gamma=\beta\beta',
\\
(\alpha,  g (h^{-1}|_{\gamma'})^{-1}, \delta (h^{-1}\gamma')), & \text{if } \beta=\gamma\gamma',
\\
0, & \text{otherwise}.
\end{cases}
\]
Then $(\alpha,g,\beta)^*=(\beta,g^{-1},\alpha)$ and every nonzero idempotent in  $S(\Gr,E)$ is of the form
\(
(\alpha, \sr(\alpha),\alpha)
\) for some $\alpha\in E^*$. We denote by $\EE(\Gr,E) \coloneqq \EE(S(\Gr,E)) $ the associated semilattice of idempotents.
\end{definition}

The inverse semigroups $\Gr\cup \{0\}$ and $S(E)$ discussed in Examples \ref{ex:groupoid_inverse_semigroup} and \ref{ex:graph_inverse_semigroup} are the extreme cases of Definition~\ref{def:inverse_semigroup_ssa}. In the first one the graph $E$ has no edges, and in the second one the groupoid $\Gr$ consists only of units. 
In general, the map
\begin{equation}\label{eq:semilattice_isomorphism}
E^*\ni \alpha \longmapsto f_\alpha \coloneqq (\alpha,\sr(\alpha), \alpha)\in \EE(\Gr,E)
\end{equation}
yields a semigroup isomorphism  $E^*\cup\{0\}\cong  \EE(\Gr,E)$, cf. Example~\ref{ex:graph_inverse_semigroup}.
The maps $S(E)\ni (\alpha,\beta)\mapsto (\alpha,\sr(\alpha),\beta)\in S(\Gr,E)$ and
$\Gr\ni g\mapsto (\rg(g),g,\sr(g))\in S(\Gr,E)$ determine embeddings $S(E)\hookrightarrow S(\Gr,E)$ and $\Gr\cup\{0\}\hookrightarrow S(\Gr,E)$
of inverse semigroups from Examples \ref{ex:graph_inverse_semigroup} and \ref{ex:groupoid_inverse_semigroup} into $S(\Gr,E)$, and these subsemigroups generate $S(\Gr,E)$ as a semigroup.

There is a natural $1$-cocycle map $c \colon S(\Gr,E)\setminus\{0\} \to \Z$ given by 
\begin{equation}\label{eq:length_cocycle}
c((\alpha,g,\beta)) \coloneqq |\alpha|-|\beta|,\qquad (\alpha,g,\beta)\in  S(\Gr,E)
\end{equation}
that we call the \emph{length cocycle}. The associated kernel inverse subsemigroup is
\begin{equation}\label{eq:core_subsemigroup}
S_0(\Gr,E) \coloneqq  c^{-1}(0)\cup \{0\}= \{(\alpha,g,\beta)\in S(\Gr,E): |\alpha|=|\beta|\}\cup \{0\}.
\end{equation}
It contains  the inverse subsemigroup generated by the  idempotents $\EE(\Gr,E)$ and the image of $G\cup \{0\}$ in $S(\Gr,E)$. This inverse semigroup is
\[
S_{00}(\Gr,E) \coloneqq   \{(g\beta,g|_{\beta},\beta): (g,\beta)\in \Gr *E^*\}\cup \{0\},
\]
and it is isomorphic to the inverse semigroup $\Gr *E^*\cup \{0\}$ where 
\[
(g,\beta) (h,\alpha) \coloneqq 
\begin{cases} (gh, \alpha), & \text{if } h\alpha=\beta\beta',
\\
(gh,  h^{-1}\beta), & \text{if } \beta=(h\alpha)\gamma',
\\
0, & \text{otherwise}.
\end{cases}
\]
Note that  $\EE(\Gr,E)\subseteq S_{00}(\Gr,E) \subseteq S_{0}(\Gr,E)\subseteq S(\Gr,E)$, and so the partial order in $S_{00}(\Gr,E)$ and $S_{0}(\Gr,E)$ is the 
one inherited from $S(\Gr,E)$.

\begin{lemma}\label{lem:inverse_semigroup_preorder}
We have   $(\alpha,g,\beta) \le (\gamma,h,\delta)$ in $S(\Gr,E)$ if and only if  $\beta = \delta \delta'$, $\alpha = \gamma (h \cdot \delta')$, and $g = h|_{\delta'}$ for some $\delta' \in \sr(\delta)E^*$, and so the diagram 
\[
\begin{tikzcd}[ampersand replacement=\&]
	\bullet \& \bullet \& \bullet \\
	\bullet \& \bullet \& \bullet
	\arrow["\gamma"', from=1-2, to=1-1]
	\arrow["{h \cdot\delta'}"', from=1-3, to=1-2]
	\arrow["\delta"', from=2-2, to=2-1]
	\arrow["{\delta'}"', from=2-3, to=2-2]
	\arrow["{g = h|_{\delta'}}"', from=2-3, to=1-3, dashed]
	\arrow["h"', from=2-2, to=1-2, dashed]
\end{tikzcd}
\]
commutes. 
In particular, the order ideal generated by $(\gamma,h,\delta)$ is 
$$
(\gamma,h,\delta)^{\downarrow}=\{\gamma (h \cdot \delta'), h|_{\delta'},\delta \delta'): \delta' \in \sr(\delta)E^*\},
$$
and $f_\alpha \le (\gamma,h,\delta)$ if and only if $\gamma = \delta$, $\alpha = \gamma \gamma'$ for some $\gamma' \in \sr(\gamma) E^*$ such that $h$ strongly fixes $\gamma'$ in the sense that $h\gamma'=\gamma'$ and $h|_{\gamma'}=\sr(\gamma')$.
\end{lemma}
\begin{proof} 
Since $(\alpha,g,\beta)^*(\alpha,g,\beta)=(\beta,\sr(\beta),\beta)$, we have  $(\alpha,g,\beta) \le (\gamma,h,\delta)$ if and only if 
the product $
(\gamma, h ,\delta) (\beta,\sr(\beta),\beta)
$ is equal to $(\alpha,g,\beta)$.  The product is nonzero if either $\beta = \delta \delta'$ or  $\delta=\beta\beta'$.
If $\beta = \delta \delta'$,  the product evaluates to $(\gamma (h \cdot \delta'), h|_{\delta'}, \delta \delta')$, so it is equal to $(\alpha,g,\beta)$ if and only if 
the relations in the assertion hold. If $\delta=\beta\beta'$, 
the product evaluates to  $(\gamma, h ,  \delta)$,
so it is equal to   $(\alpha,g,\beta)$ if and only if the relations in the assertion hold 
for $\delta'=\sr(\beta)$.

The remaining statements follow immediately from the first. 
\end{proof}

The final two relations in Lemma~\ref{lem:inverse_semigroup_preorder} motivate the following definition.  
\begin{definition}[\cite{Exel-Pardo:Self-similar}*{Definition~5.2}]\label{def:strongly_fixed}
We say that $g \in \Gr$ \emph{strongly fixes} $\alpha \in \sr(g) E^*$ or that $\alpha$ is \emph{strongly $g$-fixed} if 
\(
g \cdot \alpha = \alpha\) and \(g|_{\alpha} = \sr(\alpha).
\)
If, in addition, no proper prefix of $\alpha$ is strongly fixed by $g$ we say that $\alpha$ is a \emph{minimal strongly $g$-fixed path}. 
\end{definition}
\begin{remark}\label{rem:about_strongly_fixed}
If  $\alpha$ is strongly $g$-fixed, then it is also strongly $g^{-1}$-fixed, because using  \eqref{equ:inverse_vs_restriction}
we then have $g^{-1}|_\alpha=(g^{-1}|_{g\alpha})=(g|_{\alpha})^{-1}=\sr(\alpha)$. 
A vertex $v\in E^0$ is strongly $g$-fixed if  and only if $g=v$, because we always have $g|_{v}=g$. 
A unit $x\in \Gr^0$ strongly fixes every path $\mu\in  xE^*$ as we always have $x\mu=\mu$ and 
$x|_{\mu}=\sr(\mu)$. 
\end{remark}
The property of being strongly fixed respects the partial order on $E^*$.
\begin{lemma}\label{lem:strongly_fix_respects_the_order}
Let $g\in \Gr$, $\alpha\in \sr(g)E^*$ and $\beta \in \sr(\alpha)E^*$. 
The composed path $\alpha\beta$  is strongly $g$-fixed if and only if $\alpha$ is $g$-fixed and $\beta$ is strongly $g|_\alpha$-fixed.
In particular, every extension of a strongly $g$-fixed  path is strongly $g$-fixed, and so a path is strongly $g$-fixed if and only if  it is  an extension of a minimal strongly $g$-fixed path. 
\end{lemma} 
\begin{proof}
Since $g(\alpha\beta)=(g\alpha) (g|_{\alpha}\beta)$, we see that $g$ fixes $\alpha\beta$ if and only if it fixes $\alpha$ and $g|_{\alpha}$ fixes $\beta$. 
Combining this with $g|_{\alpha\beta}=(g|_{\alpha})|_{\beta}$ gives the first part of the assertion. 
If $\alpha$ is strongly $g$-fixed, then $\beta$ is trivially strongly fixed by $g|_{\alpha}=\sr(\alpha)$. 
This implies the second part of the assertion.
\end{proof}

By Lemma~\ref{lem:inverse_semigroup_preorder}, the set of idempotents trivially fixed by  $t = (\alpha,g,\beta) \in S(\Gr,E)$ is given by
\begin{equation}\label{eq:F_t_for_self-similar}
F_t = \begin{cases}
\{f_{\alpha \alpha'} \colon g \text{ strongly fixes } \alpha' \in \sr(g)E^* \} & \text{if } \alpha = \beta\\
\varnothing & \text{if } \alpha \ne \beta.
\end{cases}
\end{equation}

\subsection{Closedness}

The following generalises \cite{Exel-Pardo:Self-similar}*{Theorem 12.2} and \cite{Exel-Pardo-Starling:Self-similar}*{Theorem 4.2}, 
and characterises when the inverse semigroup considered is closed.

\begin{proposition}\label{prop:closedness} 
The condition 
\begin{enumerate}
\item[$\Fin$] every $g \in \Gr$ admits at most  finitely many minimal strongly $g$-fixed paths;
\end{enumerate}
is equivalent to each of the following:
\begin{enumerate}
\item \label{itm:closed_1} the inverse semigroup $S(\Gr,E)$ is closed;
\item \label{itm:closed_2} the inverse semigroup $S_0(\Gr,E)$ is closed;
\item \label{itm:closed_2.5} the inverse semigroup $S_{00}(\Gr,E)$ is closed.
\end{enumerate} 
\end{proposition}
\begin{proof} 

Let us pick $t\in  S(\Gr,E)$ and note that we may assume that  $t = (\alpha,g,\alpha)$ for some $\alpha\in \sr(g)E^*$, as otherwise  $F_t= \emptyset$, by  \eqref{eq:F_t_for_self-similar}. Then denoting by  $M_g$ the finite set of all minimal strongly $g$-fixed paths, one sees
that the set
$
\{f_{\alpha \delta} \colon \delta \in M_g \}
$
is a  cover of every $f_{\alpha \alpha'} \in F_t$. Indeed, by \eqref{eq:F_t_for_self-similar}  the set $F_t$ is parametrised by  strongly $g$-fixed paths $\alpha' \in \sr(g)E^*$ and for any such path there  are $\delta \in M_g$ and $\delta' \in \sr(\delta)E^*$ such that $\alpha' = \delta \delta'$,
and so $f_{\alpha \alpha'}\leq f_{\alpha \delta}$. 
Hence, $\Fin$ implies \ref{itm:closed_1}. 

Implications \ref{itm:closed_1}$\Rightarrow$\ref{itm:closed_2}$\Rightarrow$\ref{itm:closed_2.5} are obvious. 
Assume \ref{itm:closed_2.5} and pick $g\in \Gr$. We may assume that the set $M_g$ of all minimal strongly $g$-fixed paths is non-empty. 
Then we necessarily have $\rg(g)=\sr(g)$ (it suffices that $g$ fixes some path) and so $t \coloneqq (\rg(g),g,\sr(g))$ is a valid element of $S_{00}(\Gr,E)$. 
By assumption there is a finite set $F \subseteq F_t$ that covers $f_{\delta}$ for every $\delta\in M_g$. 
So for every $\delta\in M_g$ there  is $f_{\alpha'} \in F$, where $\alpha' \le  \delta$ ($\alpha'$ is an extension of $\delta$). 
Since $F$ is finite so is $M_g$. Thus,  \ref{itm:closed_2.5} implies $\Fin$.
\end{proof}
\begin{remark}
Every $v \in \Gr^0$ is the unique minimal strongly fixed $v$-path. Hence, 
in condition $\Fin$  we only need to look at $g \in \Gr\setminus \Gr^0$. 
Also, by Lemma~\ref{lem:strongly_fix_respects_the_order}, two minimal strongly $g$-fixed paths are different if they are incomparable.  
Accordingly, $\Fin$ can be equivalently phrased as 
\begin{enumerate}
\item[$\Fin$] every $g \in \Gr$ admits finitely many mutually incomparable  strongly $g$-fixed paths.
\end{enumerate}
\end{remark}
\begin{remark}
Using Lemma \ref{lem:inverse_semigroup_preorder} one may show that $\Fin$ is also equivalent to Hausdorffness of any of the inverse semigroups $S(\Gr,E)$, $S_{0}(\Gr,E)$ or $S_{00}(\Gr,E)$. We will prove it on the groupoid level, see Corollary \ref{cor:Hausdorff_equiv_closed} below.
\end{remark}
\subsection{Quasi-fundamentalness and  core subsemigroups}
We begin by analysing the product $tf_{\gamma}t^*$ for $t\in S(\Gr,E)$ and $f_{\gamma}\in \EE(\Gr,E)$.
\begin{lemma}\label{lem:pre_fixed_characterisation}
Let $t = (\alpha,g,\beta)$ and $f_{\gamma} = (\gamma,\sr(\gamma),\gamma)$, $\gamma\in E^*$. Then 
\begin{equation}\label{eq:pre_fixed_characterisation}
tf_{\gamma}t^* = \begin{cases}
	f_{\alpha(g\beta')} & \text{if } \gamma = \beta \beta', \\
	f_{\alpha} & \text{if } \beta = \gamma \gamma', \\
	0 & \text{otherwise}.
\end{cases}
\end{equation}
In particular, $tf_{\gamma}t^* \cdot f_{\gamma} \ne 0$ if and only if either $\gamma = \beta \beta'$ and $\gamma$ is comparable with 	$\alpha(g \beta')$; or
$\beta = \gamma \gamma'$ and $\gamma$ is comparable with $\alpha$. 
\end{lemma}
\begin{proof}
For the first statement we calculate
\begin{align*}
tft^* 
&=
\begin{cases}
	(\alpha(g\beta'), g|_{\beta'}, \beta\beta') (\beta,g^{-1},\alpha) & \text{if } \gamma = \beta \beta'\\
	(\alpha,g,\gamma\gamma')(\gamma\gamma',g^{-1}, \alpha) & \text{if } \beta = \gamma\gamma'\\
	0 & \text{otherwise}
\end{cases}= 
\begin{cases}
	(\alpha(g\beta'), \sr(g\beta') ,\alpha(g\beta')) & \text{if } \gamma = \beta \beta'\\
	(\alpha,\sr(\alpha),\alpha) & \text{if } \beta = \gamma\gamma'\\
	0 & \text{otherwise}.
\end{cases}
\end{align*}
The second statement follows immediately from the first. 
\end{proof}
It seems, to us, that the condition introduced in the next proposition has not appeared in the literature previously.
It was noticed in \cite{Aakre} that for a faithful self-similar action the inverse semigroup $S(\Gr,E)$ is always fundamental.
\begin{proposition}\label{prop:fundamentalness} 
The condition 
\begin{enumerate}
\item[$\Evr$] if $g\in \Gr$ fixes every  path in $\sr(g)E^*$, then $g$ strongly fixes some path in $\sr(g)E^*$;
\end{enumerate}
is equivalent to each of the following:
\begin{enumerate}
\item \label{itm:funamentalness_1} the inverse semigroup $S(\Gr,E)$ is quasi-fundamental;
\item \label{itm:funamentalness_2} the inverse semigroup $S_0(\Gr,E)$ is topologically free;
\item \label{itm:funamentalness_3} the inverse semigroup $S_{00}(\Gr,E)$ is topologically free.
\end{enumerate} 
The inverse semigroup $S(\Gr,E)$ is fundamental if and only if the self-similar action is faithful.
\end{proposition}
\begin{proof} 
Assume that $\Evr$ fails, and so there is $g\in \Gr$ with $v:=\sr(g)$ that fixes every  path in $vE^*$, but none of them is strongly $g$-fixed. 
Put $t:=(v,g,v) \in S_{00}(\Gr,E)$. Using \eqref{eq:pre_fixed_characterisation}, and that all paths in $vE^*$  are fixed,   
for any $\gamma\in vE^*$ we get $tf_{\gamma}t^*= f_{\gamma}$. In particular, $f_v$ is $t$-fixed but $F_t=\emptyset$ by \eqref{eq:F_t_for_self-similar}. 
Hence $S_{00}(\Gr,E)$, and all the more $S_0(\Gr,E)$, are not topologically free. 
Moreover, putting $s:=f_v$ we see that $t f_{\gamma}t^{*}= s f_{\gamma}s^{*}$ is either zero or $f_{\gamma}$ if $\gamma\in vE^*$.
By Lemma~\ref{lem:inverse_semigroup_preorder},   $0\neq u\leq t=f_v$ implies that  $u=f_v$, and $f_v\leq s$ implies that $g$ strongly fixes $v$.
Hence $S(\Gr,E)$ is not quasi-fundamental.
This proves that any of the conditions \ref{itm:funamentalness_1}--\ref{itm:funamentalness_3} implies  $\Evr$.

Now assume $\Evr$. We show that $S_{0}(\Gr,E)$ is  topologically free. Let $t=(\alpha,g,\beta)\in S_{0}(\Gr,E)$ and fix $f_{\gamma}\in \EE(\Gr,E)$.
Thus we have $|\alpha|=|\beta|$ and $t f_{\gamma\gamma'} t^*\cdot  f_{\gamma\gamma'}\neq 0$ for every $\gamma'\in \sr(\gamma)E^*$. Therefore, the second part of Lemma \ref{lem:pre_fixed_characterisation} implies that 
$\alpha=\beta$ and $g$ fixes all paths in $\sr(g)E^*$. 
Hence there is a strongly $g$-fixed path  $\gamma'\in\sr(g)E^*$ by $\Evr$. If $|\gamma|\leq |\alpha|$, then  $\delta:=\alpha\gamma'\in F_{t}$, by \eqref{eq:F_t_for_self-similar}, and $f_{\delta} f_{\gamma}=f_{\gamma}\neq 0$. 
Assume then that $|\gamma|\geq |\alpha|$ so that $\alpha\alpha'=\gamma$. Then $g|_{\alpha'}$ fixes all paths in $\sr(\gamma)E^*$. Hence there is 
a strongly $g|_{\alpha'}$-fixed path $\gamma'\in\sr(g)E^*$ by $\Evr$. Then $\alpha'\gamma'$ is  $g$-strongly fixed by Lemma~\ref{lem:strongly_fix_respects_the_order}.
Thus for $\delta:=\gamma\gamma'=\alpha\alpha'\gamma'$ we have $\delta\in F_{t}$  and $f_{\delta} f_{\gamma}=f_{\delta}\neq 0$.
Accordingly, $S_{0}(\Gr,E)$ is  topologically free, and so   all the more $S_{00}(\Gr,E)$ is  topologically free. 

Finally, we show that $\Evr$ implies that $S(\Gr,E)$ is quasi-fundamental. Take two  elements $t=(\alpha,g,\beta)$ and $s=(\eta,h,\delta)$ in $S(\Gr,E)$ such that 
  $tf_\gamma t^*=sf_\gamma s^*$  for every $\gamma\in E^*$. 
In view of \eqref{eq:pre_fixed_characterisation}, the equalities $t f_{\beta} t^*=f_{\alpha}=s f_{\beta} s^*$ and $t f_{\eta} t^*=f_{\delta}=s f_{\eta} s^*$  imply
that  we have $\alpha=\eta$ and $\beta=\delta$. Therefore for any $\beta'\in \sr(g)E^*=\sr(h)E^*$ the equality 
$tf_{\beta\beta'} t^*=sf_{\beta\beta'} s^*$ means that  $g\beta'=h\beta'$. In other words, $g^{-1}h$ fixes all paths in $\sr(g)E^*$. Hence there is
a $g^{-1}h$-strongly fixed  $\beta'\in \sr(g)E^*$ by $\Evr$. Equality $g^{-1}h|_{\beta'}=\sr(\beta)$ is equivalent to $g|_{\beta'}=h|_{\beta'}$, and so 
 $(\alpha g\beta', g|_{\beta'}, \beta\beta')=(\alpha h\beta', h|_{\beta'}, \beta\beta'
)$. Denoting this element by $u$ we have $0\neq u\leq s,t$ by Lemma~\ref{lem:inverse_semigroup_preorder}. 
$S(\Gr,E)$ is quasi-fundamental. 

This shows that $\Evr$ implies the conditions \ref{itm:funamentalness_1}--\ref{itm:funamentalness_3}. The above argument also shows
that $S(\Gr,E)$ is fundamental if and only if every $g\in \Gr$ that fixes all   paths in $\sr(g)E^*$ has to be a unit, that is $(\Gr,E)$ is faithful.
\end{proof}
\begin{example}
Any action of a  group $\Gamma$ by automorphisms of a graph $E$  may be viewed as a self-similar action of $\Gamma$ on $E$ with
the trivial $1$-cocycle $\Gamma\times E^1\to \Gamma$ where $g|_{e}=g$ for all $g\in \Gamma$ and $e\in E^1$. 
This in turn may be treated as a self-similar action of the transformation groupoid $\Gr=\Gamma\times E^1$, see Example \ref{exm:Exel_Pardo_vs_groupoid_self-similar}. 
Such actions satisfy $\Fin$. They were used in \cite{Larki_Hasiri2} to model crossed products for group actions on graph algebras by
$C^*$-algebras associated to the self-similar action. 
For such an action condition $\Evr$ holds if and only if for every $g\in \Gamma\setminus\{1\}$ and every $v\in E^0$ 
there is a path $\gamma\in vE^{*}$ that contains an edge not fixed by $g$.
\end{example}
\subsection{Topological freeness and effectiveness}
We now turn to description of topological freeness of $S(\Gr,E)$, which generalises topological freeness of $S(E)$ described in Example~\ref{ex:graph_inverse_semigroup}.
\begin{definition}[{cf. \cite{Exel-Pardo:Self-similar}*{Definition~14.1}}] 
For any $g \in \Gr$, a \emph{$g$-cycle}  is a finite path  $\alpha \in E^* \setminus E^0$ such that $g\sr(\alpha)=\rg(\alpha)$. A \emph{$\Gr$-cycle} is a path that is a $g$-cycle for some $g\in \Gr$.
\end{definition}
\begin{remark}\label{rem:g_cycles}
Cycles in $E$ are nothing but $g$-cycles for $g\in \Gr^0$. In general,  $\rg(g)E^*\sr(g)\setminus  E^0$ is the set of $g$-cycles and $\sr(g)E^*\rg(g)\setminus E^0$ is the set of $g^{-1}$-cycles. 
\end{remark}
\begin{lemma}\label{lem:g_paths_entrances}
A path $\alpha\in \sr(g)E^*$ has an entrance if and only if 
$g\alpha$ has an entrance.
\end{lemma}
\begin{proof} Write $\alpha=\alpha_1\cdots\alpha_n$, where each $\alpha_i\in E^1$. 
As the action of $\Gr$ on $E^*$ respects the length of paths,  we get $g\alpha=\beta_1\cdots \beta_n$, where each $\beta_i\in E^1$, and $|\rg(\alpha_i)E^1|=|\rg(\beta_i)E^1|$  for all $i$.
\end{proof}
\begin{lemma}\label{lem:infinite_path_from_g_cycle}
Let $\alpha$ be a $g$-cycle. Let $\alpha_1  \coloneqq  \alpha$ and $g_1  \coloneqq  g^{-1}$, and recursively define 
\[
\alpha_{n+1} \coloneqq  g_n\alpha_n \qquad \text{and} \qquad g_{n+1}  \coloneqq  g_n|_{\alpha_n} \qquad \text{ for $n > 1$}.
\]
Then the concatenation $\alpha_{\infty} \coloneqq \alpha_1\alpha_2\alpha_3\cdots $ yields a well-defined infinite path.
Moreover, if  $\alpha$ has no entrance, then so does $\alpha_{\infty}$, and  
$\rg(\alpha)E^*$ consists of finite subpaths of $\alpha_{\infty}$.
\end{lemma}
\begin{proof}
The first part is straightforward. If $\alpha$ has no entrance, then every $\alpha_n$ has no entrance, by Lemma~\ref{lem:g_paths_entrances},
and so  $\alpha_{\infty}$ has no entrance. If there is a path $\gamma \in \rg(\alpha)E^*$ which is not a subpath of $\alpha_\infty$, then we may write $\gamma$ as $\gamma=\mu\gamma'$
where $\mu$ is the longest common subpath of $\gamma$ and $\alpha_{\infty}\in \rg(\alpha)E^{\infty}$. Then $\sr(\mu)$ is a base point of $\alpha_\infty$ that  receives at least two different edges, and so $\alpha_{\infty}$ has an entrance. 
\end{proof}

\begin{definition}[{cf. \cite{Exel-Pardo:Self-similar}*{Definition~14.9}}] 
We say that  $g \in \Gr$ \emph{has slack} if there 
is a finite set $F\subseteq  \sr(g)E^*$ consisting of strongly $g$-fixed paths and such that  every  path in $\sr(g)E^*$ is comparable with a path in $F$.
\end{definition}
\begin{remark}
Every unit $g\in \Gr^0$ has slack. Since extensions of strongly fixed paths are strongly fixed (see Lemma~\ref{lem:strongly_fix_respects_the_order}),
if $E^1$ is finite, then   $g \in \Gr$ has slack if and only if there 
is $n\in \N$ such that every $\gamma \in \sr(g)E^*$ with $|\gamma|\geq n$ is strongly fixed by $g$. Thus,  our  definition is consistent with \cite{Exel-Pardo:Self-similar}*{Definition~14.9} formulated for finite graphs.
\end{remark}


\begin{proposition}\label{prop:topological_freeness_semigroup}
In addition to condition $\Evr$ from Proposition \ref{prop:fundamentalness}, consider the following two more conditions:
\begin{enumerate}
\item[$\Cyc$]  for every $g\in \Gr$ every $g$-cycle  has an entrance;
\item[$\Sla$] if $g\in \Gr$ fixes every  path in $\sr(g)E^*$, then $g$ has slack.
\end{enumerate}
Then  
\begin{enumerate}
\item 
$S(\Gr,E)$ is topologically free if and only if $\Cyc$ and $\Evr$ hold; 
\item  $S(\Gr,E)$ is effective if and only if $\Cyc$ and $\Sla$ hold.
\end{enumerate}
\end{proposition}
\begin{proof}	
To show sufficiency of the above conditions, suppose that $t = (\alpha,g,\beta)\in S(\Gr,E)$ fixes $ f_{\varepsilon} = (\varepsilon, \sr(\varepsilon),\varepsilon)$. For any $\gamma\in E^*$ recall that $f_{\gamma}\leq f_{\varepsilon}$  if and only if $\gamma \le \varepsilon$.  For any $\beta'\in \sr(g)E^*=\sr(\beta)E^*$, 
putting $\gamma \coloneqq \beta\beta'$ the relation $(tf_{\gamma}t^*)\cdot  f_{\gamma} \ne 0$   is equivalent to saying 
that $\beta\beta'$ is comparable with $\alpha(g \beta')$ (see Lemma~\ref{lem:pre_fixed_characterisation}). By passing to $t^*$ if necessary   we may assume that $|\alpha|\leq |\beta|$ (see Remark~\ref{rem:fixed_idepomtents_inverse_semigroup}).

First suppose that $|\alpha|<|\beta|$. Then $\beta=\alpha\alpha'$ for some $\alpha'\in \sr(\alpha)E^*\setminus E^0$ and the comparability condition says that and 
$\beta\beta'=\alpha\alpha'\beta'$ is an extension  of $\alpha (g\beta')$  for every $\beta'\in \sr(g)E^*=\sr(\beta)E^*=\sr(\alpha')E^*$.
In particular, 
$
\rg(\alpha')=\rg(g\beta')=g \rg(\beta')=g \sr(\alpha'),
$
and so $\alpha'$ is a $g$-cycle. Thus,  assuming $\Cyc$, $\alpha'$ has an entrance and so there is  $\alpha''\in \rg(\alpha')E^*$ incomparable with $\alpha$.  Putting $\beta' \coloneqq g^{-1}\alpha''\in \sr(\beta)E^*$ we get that $\beta\beta'=\alpha\alpha'\beta'$ can not be an extension of $\alpha (g\beta')=\alpha\alpha''$, which is a contradiction. Thus,  condition $\Cyc$ excludes the case $|\alpha|<|\beta|$.

Suppose then that $|\alpha|=|\beta|$.  Then comparability of $\beta\beta'$ and $\alpha(g \beta')$  means that $\alpha=\beta$ and $\beta'=g \beta'$. 
In particular, $g$ fixes every $\beta'\in \sr(g)E^*$. 
Consider two subcases. 

Assume $\varepsilon \le \beta$. If  $\Evr$ holds there exists a strongly $g$-fixed $\ol{\beta} \in \sr(g)E^*$.
Putting $\gamma \coloneqq \beta\ol{\beta}$, we have $f_\gamma\in F_t$ by \eqref{eq:F_t_for_self-similar},  and  
$f_\gamma \cdot f_\varepsilon=f_\gamma\neq 0$. Similarly, if we assume $\Sla$, then there is a finite $F\subseteq \sr(g)E^*$ 
such that $\{f_{\beta\ol{\beta}}: \ol{\beta}\in F\}\subseteq F_t$ covers $f_{\varepsilon} $.

Assume $\beta \le \varepsilon$, then we have $\varepsilon=\beta\beta'$ for some $\beta'\in \sr(g)E^*$. Note that every $\beta''\in \sr(\beta')E^*$ is fixed by $g|_{\beta'}$ as 
we have $\beta'\beta''=g(\beta'\beta'')=g\beta' g|_{\beta'}\beta''=\beta' g|_{\beta'}\beta''$. Hence, if we assume $\Evr$ there is a strongly $g|_{\beta'}$-fixed $ \ol{\beta} \in \sr(\beta')E^*=\sr(\varepsilon)E^*$. Then  $\beta'\ol{\beta}$ is strongly $g$-fixed because $g|_{\beta'\ol{\beta}}=(g|_{\beta'})|_{\ol{\beta}}=\sr(\ol{\beta})$.  Putting $\gamma \coloneqq \varepsilon\ol{\beta}$, we get $f_\gamma\in F_t$ by \eqref{eq:F_t_for_self-similar},  and  
$f_\gamma \cdot f_\varepsilon=f_\gamma\neq 0$. 
Similarly, if we assume $\Sla$, then there is a finite $F\subseteq \sr(\varepsilon)E^*$ 
such that $\{f_{\varepsilon\ol{\beta}}: \ol{\beta}\in F\}\subseteq F_t$ covers $f_{\varepsilon} $.

This finishes the proof of sufficiency  of $\Cyc$ and $\Evr$ for topological freeness, and $\Cyc$ and $\Sla$ for effectiveness of $S(\Gr,E)$.
To check the necessity suppose first that $\Cyc$ fails, and so there is $g$-cycle $\alpha$ without an entrance. Putting  $t \coloneqq (\rg(g),g,\alpha)\in S(\Gr,E)$ we get from \eqref{eq:F_t_for_self-similar} that $F_t=\varnothing$. We claim that $t$ fixes $f_{\alpha} \coloneqq   (\alpha, \sr(\alpha),\alpha)$ which provides a contradiction to topological freeness (and hence all the more to effectivness) of $S(\Gr,E)$.
Indeed, by  Lemma  \ref{lem:infinite_path_from_g_cycle} every path in $\rg(g)E^*$ is a prefix  of a (unique) infinite path $\alpha_{\infty}\in \rg(g)E^{\infty}$.
In other words all paths in $\rg(g)E^*$ are comparable and so for $\gamma,\beta \in \rg(g)E^*$  we have $f_{\gamma}\leq f_{\beta}$ if and only if 
$|\gamma|\geq |\beta|$. Thus,  if $f_{\gamma}\leq f_{\alpha}$, that is if  $\gamma=\alpha\alpha'$ for some $\alpha'\in \sr(\alpha)E^*$, then  we get 
$tf_{\gamma} t^* \cdot f_{\gamma}\stackrel{\eqref{eq:pre_fixed_characterisation}}{=}f_{g\alpha'}  \cdot f_{\gamma} =f_{g\alpha'}\neq 0 $, because $|\gamma|\geq |g\alpha'|$.   This proves the claim.

Now suppose that $\Evr$ fails, so that there is $g\in\Gr$ that fixes all paths in  $\sr(g)E^*$ but does not strongly fix any of them.
This implies that  $t \coloneqq (\sr(g),g,\sr(g))\in S(\Gr,E)$ fixes all idempotents $f_{\alpha}$ for $\alpha\in \sr(g)E^*$ (see Lemma~\ref{lem:pre_fixed_characterisation}) and  $F_t=\varnothing$ (see \eqref{eq:F_t_for_self-similar}), so $S(\Gr,E)$ is not topologically free.
Similarly, if $\Sla$ fails, there is $g\in\Gr$ such that $t \coloneqq (\sr(g),g,\sr(g))\in S(\Gr,E)$ fixes idempotents $f_{\sr(g)}$ but for every finite $F \subseteq  \sr(g)E^*$ consisting of strongly $g$-fixed paths there is a path in $\alpha\in \sr(g)E^*$ which is not comparable with any of paths in $F$.
The latter says  that every finite $F\subseteq F_{t}$ does not cover $f_{\sr(g)}$, so  $S(\Gr,E)$ is not effective.
\end{proof}
\begin{remark} 
Proposition~\ref{prop:topological_freeness_semigroup} can be viewed as a far-reaching generalisation of \cite{Exel-Pardo:Self-similar}*{Theorem 14.10} 
which aimed at characterisation of effectiveness of the associated tight groupoid $\Gg(\Gr,E)$.
We show in Example \ref{ex:infinitely_many_edges} below that, in general, effectiveness of $S(\Gr,E)$ is strictly weaker than that of  $\Gg(\Gr,E)$.
\end{remark}
\subsection{Minimality}
To characterise minimality, we combine the path preorder relation $v\leftarrow w$ on $E^0=\Gr^0$ (that is $vE^*w\neq \varnothing$, see Example~\ref{ex:graph_inverse_semigroup}) with the orbit equivalence relation $v\sim w$ in the groupoid $\Gr$ (that is $v\Gr w\neq \varnothing$). 
\begin{definition}\label{def:Zappa--Szep_preorder}
For $v, w\in E^0$ we write $v\ll w$ if there exists $v'\in E^0$ such that $vE^*v'\neq \varnothing$  and $v'\Gr w\neq \varnothing$.
We say that $(\Gr,E)$ is \emph{cofinal} if the set of base points of any boundary path $\mu = \mu_1 \mu_2 \cdots \in \partial E$ is cofinal in $(E^0,\ll)$, that is for every $v\in E^0$ there is $i$ such that  $v\ll \sr(\mu_i)$.
\end{definition}
\begin{remark}  The relation $\ll$ is a generalisation of the relation introduced in \cite{Exel-Pardo:Self-similar}*{Definition 13.3}, and cofinality of $(\Gr,E)$ is a generalisation of 
a condition called \emph{weak $\Gamma$-transitivity} in \cite{Exel-Pardo:Self-similar}*{Definition 13.4}. 
\end{remark}
The relation $\ll$ is natural from the point of view of the Zappa--Sz\'ep product category
$E^* \bowtie \Gr$,  cf. Remark~\ref{rem:Zappa--Szep_product}. 
Indeed, we have $v\ll w$ if and only if   there is an arrow from $w$ to $v$ in $ E^* \bowtie \Gr$.
In particular, this immediately implies that $\ll$ is a preorder, which was an issue in \cite{Exel-Pardo:Self-similar}. More specifically, $v\ll w$ is equivalent to existence of a pair $(\alpha,g)\in v(E^* \bowtie \Gr)w$, which means that the relations $v=\rg(\alpha)$ and $\sr(\alpha)=gw$  make sense
and hold (here $\sr(\alpha)=\rg(g)$ plays the role of $v'$ in Definition~\ref{def:Zappa--Szep_preorder}). This is weaker than assuming existence of   a pair $(g,\alpha)\in v\Gr *E^*w$, which means that the relations $v=g \rg(\alpha)$ and $\sr(\alpha)=w$  make sense
and hold. Indeed,  if $(g,\alpha) \in v \Gr * E^* w$, then we  can consider it as a composable pair in $E^* \bowtie \Gr$, namely $((v,g),(\alpha,w))$, and the composition in this category gives  $(v,g) (\alpha,w) = (g\alpha, g|_{\alpha})\in E^* \bowtie \Gr= E^**\Gr$. 
Pictorially, for any vertices $v,w\in E^0$ we have the following implication
\begin{equation}\label{eq:transitivity_implication}
\exists_{(g,\alpha)}\,\,
\begin{tikzcd}[ampersand replacement=\&]
v	 \bullet \,\,\,\&  \\
\bullet \& \bullet w
\arrow["{\alpha}"', from=2-2, to=2-1]
\arrow["g"', from=2-1, to=1-1, dashed]
\end{tikzcd}\quad 
\Longrightarrow\quad 
\exists_{(\alpha,g)}\,\,
\begin{tikzcd}[ampersand replacement=\&]
v	 \bullet \&   \bullet  \\
\& \,\,\,\, \bullet w 
\arrow["{\alpha}"', from=1-2, to=1-1]
\arrow["g"', from=2-2, to=1-2, dashed]
\end{tikzcd}
\end{equation}
where the consequent means that $v\ll w$.	
In the context of Exel-Pardo  group  actions  the implication \eqref{eq:transitivity_implication} can be reversed, see \cite{Exel-Pardo:Self-similar}*{Proposition 13.2}, and so it does not matter which condition one uses. In our general context, it is important to use the weaker condition, and  implication \eqref{eq:transitivity_implication} is exactly what one needs to prove the following.
\begin{lemma}
The relation $\ll$ is the smallest preorder relation on $E^0$ containing the path preorder relation $\leftarrow$ and orbit equivalence $\sim$. 
\end{lemma}
\begin{proof}
Clearly, $\ll$ contains $\leftarrow $  and  $\sim$, and every transitive relation on $E^0$ containing   $\leftarrow $  and  $\sim$ necessarily contains $\ll$.
Thus,  it suffices to show that $\ll$ is transitive, which follows from the above description as coming from morphisms in the Zappa--Sz\'ep category. It can also be readily proved using \eqref{eq:transitivity_implication}. 
\end{proof}

We use the relation $\ll$ to describe the conjugacy classes of idempotents in $S(\Gr,E)$.
For instance, Lemma~\ref{lem:pre_fixed_characterisation} says that for every $v\in E^0$ we have
\begin{equation}\label{eq:conjugate_in_semigroup_characterisation}
\{tf_vt^*: t\in S(\Gr,E)\}\cup \{0\}=\{f_\alpha: \alpha\in E^*,\,\, v\ll \sr(\alpha)\}\cup \{0\}, 
\end{equation}
which is crucial in the proof of the following

\begin{proposition}\label{prop:minimality_inverse_semigroup} 
The inverse semigroup $S(\Gr,E)$ is minimal if and only if $(\Gr,E)$ is cofinal.
\end{proposition}
\begin{proof}  
Suppose that $S(\Gr,E)$ is minimal. Fix $v \in E^0$ and $\mu = \mu_1 \mu_2 \cdots \in \partial E$. We identify  $\EE(\Gr,E)$ with $E^* \cup \{0\}$.
If $\mu \in \partial E\setminus E^{\infty}$, then $\sr(\mu)$ is either a source or an infinite receiver. Thus,  a finite set $T_0\subseteq E^*$ covers $\mu$ in $E^*\cup \{0\}$ if and only if $T_0$ contains a prefix of $\mu$.   Similarly, if $\mu\in E^\infty$ and  $T_0$ covers a  prefix of  $\mu$, then $T_0$ has to contain a prefix of $\mu$.
By  \eqref{eq:pre_fixed_characterisation}, for every finite $T\subseteq S(\Gr,E)$ the set  $\{t f_{v} t^*: t\in T\}$ (modulo the zero element) is of the form $\{f_{\alpha}: \alpha\in T_0\}$
for a finite $T_0\subseteq  E^*$ such that  $v\ll \sr(\alpha)$ for every $\alpha\in T_0$.
Thus,  applying minimality of $S(\Gr,E)$ to $f_v$ and $f_{\sr(\mu)}$, if $\mu\notin E^{\infty}$, or to $f_v$ and  $f_{\mu'}$ for some prefix $\mu'$ of $\mu$, if $\mu\in E^{\infty}$, 
we conclude that  there is $i$ such that  $v\ll \sr(\mu_i)$.
This shows that minimality of $S(\Gr,E)$ implies cofinality of $(\Gr,E)$.

Conversely, suppose that $(\Gr,E)$ is cofinal. Fix $\gamma, \mu\in E^*$ and put $v \coloneqq \sr(\gamma)$ and $w \coloneqq \sr(\mu)$. 
We show that
for any $T\subseteq  wE^*$ such that $v\ll \sr(\alpha)$ for all $\alpha\in T$, we have 
\[\{f_{\mu\alpha}: \alpha\in T\}\subseteq  \{t f_{\gamma} t^*: t\in S(\Gr,E)\}.
\]
Indeed, if $u \in E^0$, $\beta\in vE^*u$, and $g\in u\Gr \sr(\alpha) $, then putting $t \coloneqq (\mu\alpha,g, \gamma\beta) $  we get $t f_{\gamma} t^*=f_{\mu\alpha}$ by \eqref{eq:pre_fixed_characterisation}. 

We identify a finite $T$ as above such that  $\{\mu\alpha: \alpha\in T\}$ covers $\mu$. If $v\ll w$  then $T \coloneqq \{w\}$ does the job.
So suppose that $v\not\ll w$.  By cofinality,  $w$ is regular and so the set $wE^1$ is finite non-empty. In particular, the sets $F_1 \coloneqq \{e\in wE^1: v\ll \sr(e) \}$ 
and $G_1 \coloneqq \{e\in wE^1: v\not\ll \sr(e) \}$ are finite. 
In general, for  $n\geq 1$ the sets 
\begin{align*}
F_{n}& \coloneqq \{\alpha_1\cdots\alpha_n\in wE^{n}: v\ll \sr(\alpha_n) \text{ and }v \not\ll \sr(\alpha_k) \text{ for all }k=1,\dots,n-1\}, \text{ and }\\
G_n& \coloneqq \{\alpha_1\cdots\alpha_n\in wE^n: v \not\ll \sr(\alpha_k) \text{ for all }k=1,\dots,n\}, 
\end{align*}
are finite, as $\sr(G_n)$ consists of regular elements  and every element in $G_{n+1}\cup F_{n+1}$ is an extension of an element in $G_n$. 
By  cofinality there exists $N\in \N$ such that $G_{N}=\varnothing$ as otherwise there would be an infinite path $\alpha=\alpha_1\alpha_2\cdots$ such that $v \not\ll \sr(\alpha_k)$ for all $k\in \N$.
For this $N$, every element in $wE^*$ is comparable with a path in the finite set $T \coloneqq \bigcup_{k=1}^N F_k$. Hence, 
$\{\mu\alpha: \alpha\in T\}$ covers $\mu$, and equivalently $\{f_{\mu\alpha}: \alpha\in T\}$ covers $\mu$. So $S(\Gr,E)$ is minimal. 
\end{proof}

The appropriate notion of an invariant subset of $E^0$ for the self-similar action $(\Gr,E)$ is a set which is $G$-invariant, hereditary and saturated (cf. Examples \ref{ex:graph_inverse_semigroup}).
\begin{definition} We say that a subset  $ V\subseteq E^0$ is \emph{$(\Gr,E)$-invariant} if it is $G$-invariant, hereditary and saturated.
\end{definition}
Note that $V\subseteq E^0$ is $G$-invariant and  hereditary if and only if it is $\ll$ upward closed, so this condition could be viewed as a positive invariance. 
Similarly, $V\subseteq E^0$ is $G$-invariant and  saturated if and only if $V$ contains every regular vertex $w\in E^0$ for which there exists $g\in \Gr w$ such  that $gw E^1\subseteq E^1V$,
which could be viewed as a negative invariance.  

\begin{lemma}\label{lem:H_n(v)}
Fix  $v\in E^0$ and let $H_0(v) \coloneqq \{w\in E^0: v\ll w\}$. For $n \ge 1$ inductively define $H_{n+1}(v)$ as the union of $H_n(v)$ and $\Gr$-orbits of all those regular vertices 
$w\in E^{0}$ such that  $\sr(e) \in H_{n}(v)$  for all $e \in wE^{1}$. Then every $H_n(v)$ is  $\ll$ upward closed and $H(v)  \coloneqq \bigcup_{n=0}^\infty H_n(v)$ is the smallest $(\Gr,E)$-invariant set containing $v$. 
\end{lemma}
\begin{proof} Since $H_0(v)$ is upwards closed, it suffices to prove that if $H_{n}(v)$ is $\ll$ upward closed, then so is $H_{n+1}(v)$. 
Take any $u\in H_{n+1}(v)$ and $w\in E^0$ such that $u\ll w$. If $u\in H_{n}(v)$, then $w\in  H_{n}(v)\subseteq  H_{n+1}(v)$ because $H_{n}(v)$ is upward closed.
Otherwise, $u$ is regular and $uE^1\subseteq E^1 H_{n}(v)$.   Take  $(\alpha,g)\in E^*\fibre{\sr}{\sr}G$ so that $u=\rg(\alpha)$ and $g\sr(\alpha)=w$.
If $|\alpha|\geq 1$, then $w \in  H_{n}(v)\subseteq  H_{n+1}(v)$. Otherwise $\alpha$ is a vertex, and so  $gu=w$, which gives $w \in   H_{n+1}(v)$ by construction.
\end{proof}

\begin{proposition}\label{prop:cofinal_equiv_no_invariants}
The self-similar action $(\Gr,E)$ is cofinal if and only if there are no nontrivial $(\Gr,E)$-invariant sets in $E^0$.
\end{proposition}

\begin{proof}
Assume that $(\Gr,E)$ is not cofinal. There two cases to consider.
Firstly,   there might be  a singular vertex $w\in E^0$ and a vertex $v\in E^0$ such that  $v\not\ll w$. By Lemma~\ref{lem:H_n(v)},  $w\not\in H_0(v)$, and since $w$ is singular, we inductively get 
that $w\not\in H_n(v)$ for every $n\in \N$. Hence, $w\not\in H(v)$, and so $H(v)$ is a nontrivial  $(\Gr,E)$-invariant set.
Secondly, it may happen that there is an infinite path $\mu\in E^\infty$ and a vertex $v\in E^0$ such that  $v\not\ll w$ for every  $w$  in the set $B_\mu=\{\sr(\mu_i)\}_{i=1}^\infty$ of base points  of $\mu=\mu_1\mu_2\cdots$.
This implies that  $H_0(v)\cap B_{\mu}=\varnothing$, and since $\sr(\mu_{i+1})\not\in H_0(v)$ implies $\sr(\mu_i)=\rg(\mu_{i+1})\not\in H_1(v)$, we also have $H_1(v)\cap B_{\mu}=\varnothing$. 
Proceeding inductively,  $H_n(v)\cap  B_{\mu}=\varnothing$ for every $n\in \N$. Hence, $H(v)$ is a nontrivial  $(\Gr,E)$-invariant set.

Conversely, assume that $(\Gr,E)$ is cofinal. Seeking a contradiction, suppose that there is a nontrivial $(\Gr,E)$-invariant set $V\subseteq E^0$.  By cofinality,
and since  $V$ is $\ll$-upward closed, $V$ contains all singular vertices and intersects the set of base points of every infinite path. Take any $w\in E^0\setminus V$. Then $w$ is necessarily  regular, and so the sets $F_1 \coloneqq wE^1V$ 
and $G_1 \coloneqq wE^1\setminus V$ are finite. The set $G_1$ is non-empty because $V$ is hereditary and $w\not\in V$. 
In general, for  $n\geq 1$ we let
\begin{align*}
F_{n}& \coloneqq \{\alpha_1\cdots\alpha_n\in wE^{n}: \sr(\alpha_n)\in V \text{ and } \sr(\alpha_k)\not\in V \text{ for all }k=1,\dots,n-1\}, \text{ and }\\
G_n& \coloneqq \{\alpha_1\cdots\alpha_n\in wE^n:  \sr(\alpha_k)\not\in V \text{ for all }k=1,\dots,n\}. 
\end{align*}
Then $\sr(G_n)$ consists of regular elements  and  $G_{n+1}\cup F_{n+1}$ is the set of one-edge  extensions of  elements in $G_n$.
In particular, every $G_{n+1}$ is non-empty. Indeed, if $G_{n+1}=\varnothing$, then $F_{n+1}$ is the set of  all one edge  extensions of  elements in $G_n$,
and so $\sr(G_n)E^1\subseteq E^1\sr(F_{n+1})\subseteq  E^1V$ which implies  $\sr(G_n)\subseteq V$ as $V$ is saturated.
But $\sr(G_n)\subseteq V$ only if $G_{n}=\varnothing$. Proceeding inductively, this implies that $G_1=\varnothing$ which is a contradiction.  

Now, since every $G_{n}$ is non-empty, there is an infinite path $\mu=\mu_1\mu_2\cdots\in E^\infty$ such that 
$\mu_1\cdots\mu_{n}\in G_n$ for every $n\in \N$. Thus,  , $V\cap \{\sr(\mu_k):k\in \N\}\neq \varnothing$ (a  property of $V$) and  
$\sr(\mu_k)\not\in V$ for every $k\in \N$ (by definition of the $G_n$), a contradiction.
\end{proof}

\subsection{Local contractiveness} 

We now pass to discussing local contractivity of the inverse semigroup $S(\Gr,E)$.
To this end we introduce one more condition that $(\Gr,E)$ may satisfy.
\begin{enumerate}
\item[\Con] Every vertex is reachable from a $\Gr$-cycle with an entrance. That is, for any $v\in E^0$ there is $\mu\in vE^*$ such that $\sr(\mu)$ is a base point of a $\Gr$-cycle with an entrance.
\end{enumerate}

\begin{proposition}\label{prop:locally_contracting_semigroup}
For any self-similar action $(\Gr,E)$, the following conditions are equivalent:
\begin{enumerate}
\item\label{enu:locally_contracting_semigroup1}  $S(\Gr,E)$ is locally contracting;
\item\label{enu:locally_contracting_semigroup2} $S(\Gr,E)$ is strongly locally contracting; and
\item\label{enu:locally_contracting_semigroup3} $(\Gr,E)$ satisfies $\Con$.
\end{enumerate}
If any of the above equivalent conditions hold, then every $\Gr$-cycle has an entrance.
\end{proposition}
\begin{proof}
\ref{enu:locally_contracting_semigroup3}$\Rightarrow$\ref{enu:locally_contracting_semigroup2}. 
Let $\mu \in E^*$ and put $v=\sr(\mu)$. 
By assumption there are $\alpha\in vE^*$, $g\in \sr(\alpha)\Gr$, and a $g$-cycle $\gamma\in \rg(g)E^*\sr(g)=\sr(\alpha)E^*\sr(g)$ with an entrance, so we may find $\gamma'\in \sr(\alpha)E^*$ such that
$\gamma$ and $\gamma'$ are not comparable.  Then $s \coloneqq (\mu\alpha\gamma, g^{-1} ,\mu\alpha)\in S(\Gr,E)$. Putting $f_1 \coloneqq f_{\mu\alpha}$, $f_0 \coloneqq f_{\mu\alpha\gamma'}$
and  using  \eqref{eq:pre_fixed_characterisation} we get
$sf_1s^*=f_{\mu\alpha\gamma}\leq f_1$, $0\neq f_0\leq f_1=s^*s\leq f_\mu$ and $f_0 \cdot s=0$. Hence, $ S(\Gr,E)$  is strongly locally contracting.

\ref{enu:locally_contracting_semigroup2}$\Rightarrow$\ref{enu:locally_contracting_semigroup1}. See Remark~\ref{rmk:terminology_changes}.

\ref{enu:locally_contracting_semigroup1}$\Rightarrow$\ref{enu:locally_contracting_semigroup3}. 
Fix $v\in E^0$. Since $S(\Gr,E)$ is locally contractive, there exists $s=(\alpha, g, \beta)\in S(\Gr,E)$ and $F=\{f_0,\dots, f_n\}\subseteq  f_v s^*s\EE(\Gr,E)\setminus\{0\}$ such that $F$ is a cover of $s f_i s^*$ and $f_0(sf_is^*) = 0$ for all $0 \le i \le n$. 
Since $s^*s=f_{\beta}$, $f_v s^*s\neq 0$ if and only if $\rg(\beta)=v$.  The relation  $f_i\in  f_vs^*s\EE(\Gr,E)\setminus\{0\}$ means that $f_{i}=f_{\beta\beta_i}$
for some $\beta_i\in \sr(\beta)E^*$, and we have 
$s f_i s^*=s f_{\beta\beta_i} s^*=f_{\alpha g\beta_i}$ by  \eqref{eq:pre_fixed_characterisation}. 
Hence, each $\alpha g\beta_i$ is covered by $\{\beta\beta_{j}\}_{j=1}^n$ and incomparable with $\beta\beta_0$.
In particular,  $\beta$ and $\alpha$ are comparable, so $\rg(\alpha) =  \rg(\beta) = v$. We consider three cases.

If $\alpha=\beta\beta'$ for some $\beta'\in \sr(\beta)E^*\setminus E^{0}$, then  $\beta'$ is a $g^{-1}$-cycle as  $\sr(\beta')=\sr(\alpha)=g\sr(\beta)= g\rg(\beta')$. 
Also, each $\beta' g\beta_i$ is incomparable with $\beta_0\in \sr(\beta)E^*=\sr(\beta')E^*$.  
By  Lemma  \ref{lem:infinite_path_from_g_cycle}, this cannot happen if  $\beta'$ has no entrance.
Hence, $\beta'\in \sr(\beta)E^*$ is a $g^{-1}$-cycle with an entrance.

If $\beta=\alpha\alpha'$ for some $\alpha'\in \sr(\alpha)E^*\setminus E^{0}$, then  $\alpha'$ is a $g$-cycle as $\rg(\alpha')=\sr(\alpha)=g\sr(\beta)= g \sr(\alpha')$. 
Also, each $ g\beta_i\in \sr(\alpha)E^*$ is  incomparable with $\alpha'\beta_0$. By  Lemma  \ref{lem:infinite_path_from_g_cycle} this cannot happen if  $\alpha'$ has no entrance.
Hence, $\alpha'\in \sr(\alpha)E^*$ is a $g$-cycle with an entrance.

If $\alpha=\beta$, then each $ g\beta_i $ is covered by $\{\beta_{j}\}_{j=1}^n$ and incomparable with $\beta_0$.
This leads to a contradiction.  Indeed, put $i_0=0$ and for $k\geq 1$ inductively choose $i_k\in \{1,\dots, n\}$ such that  $\beta_{i_k}$ is the shortest path amongst $\{\beta_{j}\}_{j=1}^n$ which is  comparable with  $g\beta_{k-1}$. 
Since $\{\beta_{j}\}_{j=1}^n$ is finite there exist  $k,l\geq 1 $ such that
$\beta_{i_k}=\beta_{i_{k+l}}$. Assume $k$ is the smallest with this property.  
Note that  $g\beta_{i_{k-1}}$ and $g\beta_{i_{k+l-1}}$ are comparable with $\beta_{i_k}$.
Since comparability of paths is transitive and preserved under $\Gr$-action, 
this means that $\beta_{i_{k-1}}$ and $\beta_{i_{k+l-1}}$ are comparable.
Since $\beta_{i_0}=\beta_0$ is incomparable with $\beta_{i_m}$ for all $m\geq 1$, we have $k-1\geq 1$. 
Thus,  $\beta_{i_{k-1}}$ is not longer than $\beta_{i_{k+l-1}}$ as it was chosen the shortest  comparable with 
$g\beta_{i_{k-2}}$, and $\beta_{i_{k+l-1}}$ is not longer than $\beta_{i_{k-1}}$ as it is the shortest comparable with $g\beta_{i_{k+l-2}}$.
In other words, $\beta_{i_{k-1}}=\beta_{i_{k+l-1}}$.
But this contradicts minimality of $k$.

For the final assertion, the last part of Lemma  \ref{lem:infinite_path_from_g_cycle}
implies that every base point of a $\Gr$-cycle without entrance fails to satisfy \ref{enu:locally_contracting_semigroup3}.
\end{proof}
\begin{remark}
If every vertex is a range of a path whose source lies on a cycle (which holds, for example, when $E^0$ is finite and there are no sources),
then the above equivalent conditions of Proposition~\ref{prop:locally_contracting_semigroup} hold if and only if every cycle in $E$ has an entrance.
This in particular gives \cite{Exel-Pardo:Self-similar}*{Theorem 15.1}.
\end{remark}

\begin{example}\label{ex:not_exel-pardo_top_free}
Here we give an example of a self-similar group action for which $S(\Gr,E)$ is topologically free, minimal, and locally contracting, but not effective and (hence, necessarily) not closed. 
Let $E$ be the directed graph
\[
\begin{tikzpicture}
[baseline=-0.25ex,
vertex/.style={
circle,
fill=black,
inner sep=1.5pt
},
edges/.style={
-stealth,
shorten >= 3pt,
shorten <= 3pt
},
scale =1]
\clip  (-1,-0.5) rectangle (1,0.5);

\node[vertex] (v) at (0,0) {};%

\node[anchor= south] at (v) {\scriptsize{$v$}};

\draw[edges] (v.east) to [out = -40, in = 40,loop, min distance=10mm] node[anchor=west, inner sep = 2pt]{\scriptsize{$f$}} (v.east);
\draw[edges] (v.west) to [out = 220, in = -220,loop,min distance=10mm] node[anchor=east, inner sep = 2pt]{\scriptsize{$e$}} (v.west);
\end{tikzpicture}
\]
and let $\Gr = \Z$. 
Identify $0 \in \Z$ with $v$ and define a self-similar action of $\Z$ on $E$, given on generators by
\begin{align*}
1 \cdot e &= e, & 1|_e &= 1, &
1 \cdot f &= f, \text{ and} & 1|_f &= 0.
\end{align*}
Observe that each $k \in \Z$ fixes every $\alpha \in E^*$. 
For each $k \in \Z$ and $n \in \N$, the path $e^n f$ is a minimal strongly $k$-fixed path, so by Proposition~ \ref{prop:closedness} the inverse semigroup $S( \Z,E)$ is not closed. All (nontrivial) paths are cycles with an entrance, and each $k \in \Z$ strongly fixes $f$, so  $S( \Z,E)$ is topologically free by Proposition~\ref{prop:topological_freeness_semigroup}. Since $E^0 = \{v\}$ there are no nontrivial $( \Z,E)$-invariant sets of vertices, so by Propositions~\ref{prop:cofinal_equiv_no_invariants} and \ref{prop:minimality_inverse_semigroup} the inverse semigroup $S( \Z,E)$ is minimal. Since $v$ lies on a $ \Z$-cycle with an entrance, $S( \Z,E)$ is locally contracting by Proposition~\ref{prop:locally_contracting_semigroup}. 

On the other hand, $S( \Z,E)$ does not satisfy condition $\Sla$ in Proposition \ref{prop:topological_freeness_semigroup}: each $k \in \Z$ fixes $Z(v)$, but for each $n \in \N$ the path $e^n$ is not strongly $k$-fixed, so $k$ is not slack at $v$.

\end{example}

\section{The groupoid analysis}{}\label{sec:self-similar_groupoid}
In this section we analyse the basic properties (described in Definition 
\ref{defn:groupoid_properties}) of groupoids  associated to a  self-similar groupoid action  $(\Gr,E)$.

\subsection{Groupoids associated to self-similar actions}
We start by describing the universal and tight groupoid (see Definition~\ref{def:universal_and_tight_groupoids}) associated to the inverse semigroup $S(\Gr,E)$.
The path  space $E^{\le \infty}$ and boundary path space $\partial E$  were introduced in Example~\ref{ex:spectrum_graph_inverse_semigroup}.
\begin{proposition}\label{prop:transformation_groupoids_self_similar_0}
Let $(\Gr,E)$ be a self-similar groupoid action. We have an action of   the inverse semigroup $S(\Gr,E)$   on the path space $E^{\le \infty}$ of $E$, where for $(\alpha,g,\beta) \in S(\Gr,E)$ the partial homeomorphisms 
$\widetilde{h}_{(\alpha,g,\beta)} \colon Z(\beta) \to Z(\alpha)$ are given by
\[
\widetilde{h}_{(\alpha,g,\beta)}(\beta \eta) = \alpha(g \cdot \eta), \qquad \eta\in E^{\le \infty},\, \rg(\eta)=\sr(\beta).
\]
The boundary path space $\partial E$ is a closed $\widetilde{h}$-invariant subset of $E^{\le \infty}$, and so $\widetilde{h}$ restricts to an action $h$ on  $\partial E$. 
In fact, the map $E^{\le \infty} \ni \mu\mapsto \xi_{\mu}\in\widehat{\EE(G,E)}$  
given by $\xi_{\mu} \coloneqq \{(\alpha,\sr(\alpha),\alpha):\alpha\in E^*,\, \bone_{Z(\alpha)}(\mu)=1\}$, 
intertwines $\widetilde{h}$ with the canonical action of $S(\Gr,E)$  on $\widehat{\EE(G,E)}$, and  restricts 
to the homeomorphism $\partial E\cong \partial\widehat{\EE(G,E)}$. Thus, we have natural isomorphisms
\[
\widetilde{\Gg}(S(\Gr,E))\cong S(\Gr,E) \ltimes_{\widehat{h}} E^{\le \infty} \quad \text{and} \quad \Gg(S(\Gr,E))\cong S(\Gr,E)\ltimes_{h} \partial E
\]
of topological groupoids.
\end{proposition}
\begin{proof} Since $\EE(G,E)\cong \EE(E)$, it follows from Example~\ref{ex:spectrum_graph_inverse_semigroup} that the map 
$E^{\le \infty}\ni \mu\mapsto \xi_\mu\in  \widehat{\EE(G,E)}$   where 
$\xi_{\mu} \coloneqq \{f_{\alpha}:\alpha\in E^*,\, \bone_{Z(\alpha)}(\mu)=1\}$ 
is a homeomorphism $
E^{\le \infty}\cong  \widehat{\EE(G,E)}$, which restricts to the homeomorphism $\partial E\cong  \partial\widehat{\EE(G,E)}$.
This homeomorphism intertwines the appropriate actions. Indeed, let $t=(\alpha,g,\beta)$. Then $\mu\in Z(\beta)$ 
if and only if $f_\beta \in \xi_{\mu}$. 
Let
$\mu=\beta\overline{\mu}\in Z(\beta)$ for some $\overline{\mu} \in E^{\le \infty}$. For the action $\widetilde{h}$ on  $E^{\le \infty}$ we have	$\widetilde{h}_t(\mu)=\widetilde{h}_{(\alpha,g,\beta)}(\beta \overline{\mu}) = \alpha(g \cdot \overline{\mu})$, and so 
for $\eta\in E^*$ we have
\[
f_\eta \in \xi_{\widetilde{h}_t(\mu)} \Longleftrightarrow 1=\bone_{Z(\eta)}(\alpha(g \cdot \overline{\mu})).
\]
On the other hand, using the same symbol  $\widetilde{h}$ for the canonical action on $\widehat{\EE(G,E)}$,
\begin{align*}
f_\eta \in  \widetilde{h}_t(\xi_{\mu})&\stackrel{\eqref{eq:inverse_semigroup_action_spectrum}}{\Longleftrightarrow} t^*f_{\eta}t\in \xi_{\mu} \stackrel{\eqref{eq:pre_fixed_characterisation}}{\Longleftrightarrow} 
(f_{\beta(g^{-1}\alpha')}\in\xi_{\mu}  \text{ and } \eta=\alpha\alpha') \text{ or } (f_{\beta}\in \xi_{\mu}  \text{ and } \alpha=\eta\eta')  
\\
&\stackrel{\mu=\beta\overline{\mu}}{\Longleftrightarrow} 1=\bone_{Z(g^{-1}\alpha')}(\overline{\mu}) \text{ and } \eta=\alpha\alpha' \text{ or }  \alpha=\eta\eta'.
\end{align*}
Now we have three cases. If $\eta$ and $\alpha$ are not comparable, then the above conditions imply that $f_\eta \not\in \xi_{\widetilde{h}_t(\mu)} $ and $f_\eta \not\in  \widetilde{h}_t(\xi_{\mu})$.
If $\alpha=\eta\eta'$, then $f_\eta \in \xi_{\widetilde{h}_t(\mu)} $ and $f_\eta \in  \widetilde{h}_t(\xi_{\mu})$.
If $\eta=\alpha\alpha'$, then
\[
\bone_{Z(\eta)}(\alpha(g \cdot \overline{\mu}))=\bone_{Z(\alpha')}(g \cdot \overline{\mu})=\bone_{Z(g^{-1}\alpha')}(\overline{\mu}),
\]
and so $f_\eta \in \xi_{\widetilde{h}_t(\mu)}$ if and only if $f_\eta \in  \widetilde{h}_t(\xi_{\mu})$. Thus,  $\widetilde{h}_t(\xi_{\mu})=\xi_{\widetilde{h}_t(\mu)}$ as claimed. 
This implies the remaining assertions. 
\end{proof}

To describe, somewhat more concretely, the  above transformation groupoids we consider the following equivalence relation on the set of quadruples $(\alpha, g ,\beta; \xi)$, where $(\alpha,g,\beta) \in S(\Gr,E) \text{ and } \xi \in Z(\beta)$:
\[
(\alpha, g ,\beta; \xi) \sim (\gamma, h, \delta; \eta) \iff
\begin{cases}
(\alpha,g,\beta) \text{ and } (\gamma, h, \delta) \text{ share a  lower bound }(\epsilon, k ,\zeta) \\
\text{ and }\xi =\eta \in Z(\zeta)
\end{cases}
\]
Denote the equivalence class of $(\alpha, g ,\beta;\xi)$ by $[\alpha,g,\beta; \xi]$.
We  equip 
\[
\widetilde{\Gg}(\Gr,E)  \coloneqq  \{[\alpha,g,\beta; \beta \xi] \colon (\alpha,g,\beta) \in S(\Gr,E) \text{ and } \xi \in Z(\sr(\beta))\}
\]
with the structure of an \'etale  groupoid \label{page:self-similar_groupoid} with multiplication  given by 
\[
[\alpha,g,\beta;  \beta\xi] \cdot [\gamma, h,\delta; \delta \eta]  \coloneqq  [(\alpha,g,\beta)\cdot (\gamma,h,\delta); \delta \eta] \qquad \text{whenever $\beta \xi = \gamma(h \cdot \eta)$,}
\]
inverse $[\alpha,g,\beta;\beta\xi]^{-1} = [\beta, g^{-1}, \alpha; \alpha(g \cdot \xi)]$, 
and a basis for the topology given by the  bisections
\[
U(\alpha,g,\beta;V)  \coloneqq  \{ [\alpha,g,\beta;\xi] \mid \xi \in V \}
\]
where $V \subseteq Z(\beta)$ is open. 
The unit space 
\[
\widetilde{\Gg}(\Gr,E)^0 = \{[\rg(\xi), \rg(\xi), \rg(\xi); \xi] \colon \xi \in E^{\le \infty}\} 
\]
is naturally homeomorphic to $E^{\le \infty}$, and we use this homeomorphism to identify  $\Gg(\Gr,E)^0$  with $E^{\le \infty}$. 
Then the source and range maps are given by
\begin{equation}\label{eq:source_and_range_maps}
\sr([\alpha,g,\beta;  \beta\xi])= \beta\xi, 
\qquad 
\rg([\alpha,g,\beta;  \beta\xi]) = \alpha(g \cdot \xi).
\end{equation} 
The boundary path space  $\partial E$ is a closed $\widetilde{\Gg}(\Gr,E)$-invariant subset of $E^{\le \infty}$ and hence
\[
\Gg(\Gr,E)  \coloneqq  \{[\alpha,g,\beta; \xi] \colon (\alpha,g,\beta) \in S(\Gr,E) \text{ and } \xi \in Z(\beta)\cap  \partial E\}
\]
is a closed full subgroupoid of $\widetilde{\Gg}(\Gr,E)$.
By construction these groupoids are isomorphic to the transformation groupoids  described in Proposition~\ref{prop:transformation_groupoids_self_similar_0}, and so
we have natural groupoid isomorphisms
\begin{equation}\label{eq:transformation_groupoids_self_similar}
\widetilde{\Gg}(\Gr,E)\cong \widetilde{\Gg}(S(\Gr,E)) \qquad   \Gg(\Gr,E) \cong \Gg(S(\Gr,E)).
\end{equation}
These groupoids are $\Z$-graded by the continuous groupoid homomorphism $c:\widetilde{\Gg}(\Gr,E)\to \Z$
given by \(
c([\alpha,g,\beta; \xi])=|\alpha|-|\beta|.
\)
We identify the two clopen subgroupoids
\begin{align*}
	\widetilde{\Gg}_0(\Gr,E)  &\coloneqq  c^{-1}(0)=\{[\alpha,g,\beta; \xi] \in \widetilde{\Gg}(\Gr,E)\colon |\alpha|=|\beta|\} \subseteq \widetilde{\Gg}(\Gr,E), \text{ and }\\
	\Gg_0(\Gr,E) &\coloneqq  \widetilde{\Gg}_0(\Gr,E)\cap \Gg(\Gr,E) \subseteq \Gg(\Gr,E).
\end{align*}
Here, we again have natural groupoid isomorphisms
\[
\widetilde{\Gg}_0(\Gr,E)\cong \widetilde{\Gg}(S_0(\Gr,E)), \qquad   \Gg_0(\Gr,E)\cong \Gg(S_0(\Gr,E)),
\]
where $S_0(\Gr,E)$ is the core subsemigroup of $S(\Gr,E)$, given by \eqref{eq:core_subsemigroup}.
Since $S_{00}(\Gr,E)\cong  \Gr *E^*\cup \{0\}$ is a wide inverse subsemigroup of $S_0(\Gr,E)$,  by the last part of Remark~\ref{rem:transformation_subgroupoids},
we also get the following open subgroupoids
\begin{align*}
\widetilde{\Gg}_{00}(\Gr,E)
& \coloneqq \{[g\beta, g|_{\beta}, \beta;\beta\xi]:(g,\beta)\in \Gr *E^* ,\, \xi\in \sr(\beta)E^{\le \infty}\}\cong \widetilde{\Gg}(S_{00}(\Gr,E)), 
\\
\Gg_{00}(\Gr,E )& \coloneqq  \widetilde{\Gg}_{00}(\Gr,E)\cap \Gg(\Gr,E)\cong \Gg(S_{00}(\Gr,E))
\end{align*}
of  $\widetilde{\Gg}_0(\Gr,E)$ and $\Gg_0(\Gr,E)$ respectively.
Thus, the discussed groupoids form the diagram:
\begin{equation}\label{eq:various_groupoids}
\begin{tikzcd}  [row sep=small]
\widetilde{\Gg}(\Gr,E)\arrow[r, phantom, sloped, "\supseteq"] & \widetilde{\Gg}_0(\Gr,E)\arrow[r, phantom, sloped, "\supseteq"] & \widetilde{\Gg}_{00}(\Gr,E)\\
\Gg(\Gr,E)\arrow[r, phantom, sloped, "\supseteq"] \arrow[u, phantom, sloped, "\subseteq"] & \Gg_0(\Gr,E)\arrow[r, phantom, sloped, "\supseteq"] \arrow[u, phantom, sloped, "\subseteq"] & \Gg_{00}(\Gr,E) \arrow[u, phantom, sloped, "\subseteq"] 
\end{tikzcd}
\end{equation}
where the horizontal inclusions describe open subgroupoids while the vertical ones are closed subgroupoid inclusions.
We may also use that the  inverse semigroups $\EE(E)\subseteq S(E)_0\subseteq S(E)$ sit naturally as wide subgroups in 
$S(G,E)$ to complement the diagram \eqref{eq:various_groupoids}   with the following diagrams of  inclusions of wide open subgroupoids:
\begin{equation}\label{eq:graph_groupoids_into_self-similar_groupoids}
\begin{tikzcd}  [row sep=small, column sep=small]
\widetilde{\Gg}(\Gr,E)\arrow[r, phantom, sloped, "\supseteq"] & \widetilde{\Gg}_0(\Gr,E)\arrow[r, phantom, sloped, "\supseteq"] & \widetilde{\Gg}_{00}(\Gr,E)
\\
\widetilde{\Gg}(E)\arrow[r, phantom, sloped, "\supseteq"] \arrow[u, phantom, sloped, "\subseteq"] & \widetilde{\Gg}_0(E)\arrow[r, phantom, sloped, "\supseteq"] \arrow[u, phantom, sloped, "\subseteq"] 
& E^{\le\infty} \arrow[u, phantom, sloped, "\subseteq"] 
\end{tikzcd}\quad\quad
\begin{tikzcd}  [row sep=small, column sep=small]
\Gg(\Gr,E)\arrow[r, phantom, sloped, "\supseteq"] & \Gg_0(\Gr,E)\arrow[r, phantom, sloped, "\supseteq"] & \Gg_{00}(\Gr,E)
\\
\Gg(E)\arrow[r, phantom, sloped, "\supseteq"] \arrow[u, phantom, sloped, "\subseteq"] & \Gg_0(E)\arrow[r, phantom, sloped, "\supseteq"] \arrow[u, phantom, sloped, "\subseteq"] 
& \partial E \arrow[u, phantom, sloped, "\subseteq"] 
\end{tikzcd}.
\end{equation}
\subsection{The germ relation}
To analyse  all these groupoids it is crucial to understand  the equivalence $\sim$ better and describe the equivalence classes that give units or isotropy arrows.
\begin{lemma}
\label{lem:lower_bound_description}
Let $(\alpha,g,\beta),(\gamma,h,\delta) \in S(\Gr,E)$ and $\xi,\eta \in E^{\le \infty}$ with $\xi \in Z(\beta)$ and $\eta \in Z(\delta)$. 
Then
$[\alpha,g,\beta; \xi]=[\gamma, h, \delta; \eta]$ if and only if $\xi = \eta$ and there are $\beta',\delta' \in E^*$ satisfying
\begin{equation}\label{eq:lower_bound_existence}
\alpha g\beta'=\gamma h\delta',\qquad  g|_{\beta'}=h|_{\delta'},  \qquad \beta\beta'=\delta\delta',
\end{equation}
and  $\xi\in Z(\beta\beta')=Z(\delta\delta')$. 	Further, this can only happen if and only if $\xi = \eta$ and either
\begin{enumerate}
\item \label{itm:one_way} 

$\alpha = \gamma \ol{\gamma}$ and $\beta=\delta\overline{\delta}$  for some $\ol{\gamma},\overline{\delta}\in E^*$ such that $h  \ol{\delta} = \ol{\gamma}$ and there is  a  strongly $g^{-1} (h|_{\ol{\delta}})$-fixed   path $\beta' \in \sr(\beta) E^*$ 
with $\xi \in Z(\beta\beta')$,  or
\item \label{itm:or_another} 
$\gamma = \alpha \ol{\alpha}$, $\delta = \beta \ol{\beta}$ and for some $\ol{\alpha},\ol{\beta }\in E^*$ such that $g  \ol{\beta} = \ol{\alpha}$ and 
there is a strongly $h^{-1}(g|_{\ol{\beta}})$-fixed path $\delta' \in \sr(\delta)E^*$ 	with $\xi \in Z(\delta\delta')$.
\end{enumerate}

\end{lemma}
\begin{proof}
That $\xi = \eta$ follows immediately from the definition of $\sim$.
Lemma~\ref{lem:inverse_semigroup_preorder} implies that 
$(\epsilon, k ,\zeta) \leq (\alpha,g,\beta),(\gamma,h,\delta)$ in $S(\Gr,E)$ if 
and only if there are $\beta',\delta' \in E^*$ satisfying
\eqref{eq:lower_bound_existence}
in which case $(\epsilon, k ,\zeta)=(\alpha g\beta', g|_{\beta'}, \beta\beta')=(\gamma h\delta', h|_{\delta'}, \delta\delta'
)$. The first part of the assertion follows.
Now note that the first and last equation of \eqref{eq:lower_bound_existence} show that either $\alpha \le \gamma$  and  $\beta \le \delta$, or $\gamma \le \alpha$ 
and $\delta \le \beta$. 	We restrict our attention to the first case, which corresponds to \ref{itm:one_way}. The other case, which corresponds to \ref{itm:or_another},  follows by a symmetric argument. 
Thus, let us assume that  $\alpha \le \gamma$  and  $\beta \le \delta$, that is $\alpha = \gamma \ol{\gamma}$ and $\beta = \delta \ol{\delta}$ for some  $\ol{\gamma}, \ol{\delta} \in E^*$. So assuming \eqref{eq:lower_bound_existence}, 
we get $\ol{\gamma} g\beta'=h\delta'$ and $\ol{\delta}\beta'=\delta'$. 
Thus,			\[
\ol{\gamma} g\beta'=  h \delta' = h (\ol{\delta}\beta') = (h \ol{\delta})(h|_{\ol{\delta}} \beta').  
\]
Since  $|g\beta'|=|\beta'|=|h|_{\ol{\delta}}  \beta'|$, we infer that  $\ol{\gamma} = h  \ol{\delta}$ 
and $g\beta'= h|_{\ol{\delta}} \cdot \beta'$. Hence,   $g^{-1}h|_{\ol{\delta}} \cdot \beta'=\beta'$. 
Moreover, by the middle identity of \eqref{eq:lower_bound_existence}, $g|_{\beta'} = h|_{\delta'} = h|_{\ol{\delta}\beta'}$, so
\[
(g^{-1} h|_{\ol{\delta}})|_{\beta'} 
= (g^{-1}|_{h|_{\ol{\delta}} \cdot \beta'})(h|_{\ol{\delta}\beta'})
= (g^{-1}|_{g \cdot \beta'})(g|_{\beta'})
= (g^{-1}g)|_{\beta'} = \sr(\beta').
\]
Thus, $\beta'$ is strongly fixed by $g^{-1} (h|_{\ol{\delta}})$. 
Therefore, $[\alpha,g,\beta; \xi]=[\gamma, h, \delta; \xi]$ implies that \ref{itm:one_way} holds. 

Conversely, suppose that \ref{itm:one_way} holds  and put   $\delta'  \coloneqq  \ol{\delta}\beta'$.
Then we immediately get $\beta\beta'=\delta\delta'$.	
Since $\beta'$ is strongly fixed by $g^{-1} (h|_{\ol{\delta}})$ we  have  $g\beta'=(h|_{\ol{\delta}})\beta'$
and  $\sr(\beta')=[g^{-1} (h|_{\ol{\delta}})]|_{\beta'}=(g^{-1} |_{h|_{\ol{\delta}}\beta'}) h|_{\ol{\delta}\beta'}=(g^{-1} |_{g\beta'}) h|_{\delta'}$. 
Therefore,  
\[
\gamma h \delta'=\gamma h (\ol{\delta}\beta')=\gamma (h\ol{\delta}) (h|_{\ol{\delta}}\beta')=\gamma \ol{\gamma} g\beta'=\alpha g\beta'
\]
and
$
g|_{\beta'}=(g^{-1} |_{g\beta'})^{-1}=h|_{\delta'}$. This proves \eqref{eq:lower_bound_existence} and so  $[\alpha,g,\beta; \xi]=[\gamma, h, \delta; \xi]$. 
\end{proof}
\begin{remark}\label{rem:groupoid_epimorphisms} 
We see that the typical situation in which an equality of germs
takes place is
$$
[\alpha,g,\beta; \beta\gamma\xi]= [ \alpha g\gamma,g|_{\gamma}, \beta\gamma;\beta\gamma\xi].
$$
In particular, $[\rg(g),g,\sr(g); \beta\xi]=[g\beta, g|_{\beta}, \beta;\beta\xi]$ for every $(g,\beta)\in \Gr *E^*$, and so
\begin{align*}
\widetilde{\Gg}_{00}(\Gr,E)
=\{[\rg(g),g,\sr(g); \xi]:g\in \Gr ,\, \xi\in \sr(g)E^{\le \infty}\}. 
\end{align*}
In particular, $\Gg_{00}(\Gr,E)$ coincides with the groupoid denoted by $\mathcal{H}_{0}$ in \cite{Miller_Steinberg}.
Note that the map  $(g,\xi)\mapsto [\rg(g),g,\sr(g); \xi]$ yields groupoid epimorphisms 
$\Gr\rtimes E^{\le \infty}\onto \widetilde{\Gg}_{00}(\Gr,E)$ and $\Gr\rtimes \partial E\onto\Gg_{00}(\Gr,E)$ 
for the transformation groupoids from 
Remark~\ref{rem:transformation_groupoids_from_groupoid_actions}. 

\end{remark}
By Lemma~\ref{lem:strongly_fix_respects_the_order} a finite path is strongly $g$-fixed if and only if it  contains a strongly $g$-fixed prefix.
This motivates us to extend this notion (Definition~\ref{def:strongly_fixed}) to infinite paths as follows.

\begin{definition}
An infinite path is \emph{strongly $g$-fixed} if it has a strongly $g$-fixed finite prefix.
\end{definition}

\begin{proposition}\label{prop:transformation_groupoids_self_similar}
Let $[\alpha,g,\beta; \beta\xi]\in \widetilde{\Gg}(\Gr,E)$. Then
\begin{enumerate}
\item\label{enu:transformation_groupoids_self_similar1} $[\alpha,g,\beta; \beta\xi]$ is a unit if and only if $\alpha=\beta$ and $\xi$ is strongly $g$-fixed; 
\item\label{enu:transformation_groupoids_self_similar2} 
$[\alpha,g,\beta; \beta\xi]$ is an isotropy arrow if and only if  one of the following hold:
\begin{enumerate}[label=\textup{(\alph*)}]
\item\label{enu:transformation_groupoids_self_similar2a}  $\alpha = \beta$ and $\xi$ is $g$-fixed;
\item \label{enu:transformation_groupoids_self_similar2b} $\beta=\alpha (g\alpha')$ where $\alpha'\in \sr(g)E^*$ is a  $g|_{\alpha'}$-cycle  and $\xi$ is the  infinite path  associated to this $g|_{\alpha'}$-cycle as in  Lemma~\ref{lem:infinite_path_from_g_cycle}; or
\item\label{enu:transformation_groupoids_self_similar2c}  $\alpha= \beta\beta'$ where $\beta'\in \sr(g)E^*$ is a  $g^{-1}$-cycle and $\xi$ is the infinite path  associated to this 	$g^{-1}$-cycle  as in  Lemma~\ref{lem:infinite_path_from_g_cycle}.
\end{enumerate}
\end{enumerate}
\end{proposition}
\begin{proof}
\ref{enu:transformation_groupoids_self_similar1}.  We have  
$[\alpha,g,\beta; \beta\xi]\in \widetilde{\Gg}(\Gr,E)^0$ if and only if 
$[\alpha,g,\beta; \beta\xi]=[\rg(\beta), \rg(\beta), \rg(\beta); \beta\xi]$, which by  Lemma~\ref{lem:lower_bound_description} holds  if and only if $\alpha=\beta$ and $\xi$ contains a strongly $g$-fixed prefix (equivalently $\xi$ is strongly $g$-fixed). 

\ref{enu:transformation_groupoids_self_similar2}.  By \eqref{eq:source_and_range_maps} we have 
$\sr([\alpha,g,\beta;  \beta\xi]) =\rg([\alpha,g,\beta; \beta\xi])$ if and only if $\beta\xi=\alpha(g \cdot \xi)$.
If $|\alpha|=|\beta|$, then $\beta\xi=\alpha(g \cdot \xi)$ is equivalent to  $\alpha=\beta$ and $g\xi=\xi$, which is \ref{enu:transformation_groupoids_self_similar2a}.
If  $|\beta|>|\alpha|$, then  $\beta\xi=\alpha(g \cdot \xi)$
 if and only if $\beta=\alpha(g\alpha')$ for  $\alpha'\in E^{|\beta|-|\alpha|}$ such that 
$\xi=\alpha'\xi'$ and $\xi'\in \sr(\alpha')E^{\infty}$ satisfies $g|_{\alpha'}(\xi')=\xi$.
The latter two relations are equivalent to that
$\alpha'\in \sr(g)E^*$ is a  $g|_{\alpha'}$-cycle  and $\xi$ is the associated infinite path.
Hence, $\beta\xi=\alpha(g \cdot \xi)$ is equivalent to \ref{enu:transformation_groupoids_self_similar2b} in this case.
Similarly, if 
$|\alpha|>|\beta|$, then  $\beta\xi=\alpha(g \cdot \xi)$
 if and only if $\alpha=\beta\beta'$ and $\beta'\in E^{|\alpha|-|\beta|}$ satisfies 
$\xi =\beta'g\xi$, which holds if and only if \ref{enu:transformation_groupoids_self_similar2c} holds.
\end{proof}
As a side remark we note that the open subgroupoid $\Gg_{00}(\Gr,E)$ of  $\Gg(\Gr,E)$ might not be closed, and hence all the more $\widetilde{\Gg}_{00}(\Gr,E)$ is not closed in   $\widetilde{\Gg}(\Gr,E)$ in general.  This was an issue in a preliminary version of \cite{Miller_Steinberg}.
\begin{lemma}\label{lem:H_0_characterisation}
We have $[\alpha,g,\beta; \beta\xi]\in \widetilde{\Gg}_{00}(\Gr,E)$ if and only if there is $h\in \Gr$  such that $h\beta =\alpha$ and $g^{-1}(h|_{\beta})$  strongly fixes $\xi$.
\end{lemma}
\begin{proof}
By Lemma~\ref{lem:lower_bound_description}\ref{itm:one_way},  $[\alpha,g,\beta; \beta\xi]=[\rg(h), h, \sr(h); \beta\xi]$ for some $h$
if and only if $h\beta=\alpha$ and $g^{-1}(h|_\beta)$ fixes strongly a prefix of $\xi$.
\end{proof}

\begin{example}
Let $E$ be the graph with a single vertex $v$ and four loops $E^{1}=\{\alpha,\beta,e,f\}$, with a self-similar action of $\Gr=\Z_2$ given by
$$
1\alpha=\beta,\qquad 1\beta=\alpha,\qquad  1e=e, \qquad 1f=f,
$$
$$
1|_{\alpha}=0,\qquad 1|_{\beta}=0, \qquad 1|_{e}=1, \qquad 1|_{f}=0.
$$
Let $e^{\infty}=ee\cdots \in E^{\infty}$ and $f^{\infty}=ff\cdots \in E^{\infty}$. By Lemma~\ref{lem:H_0_characterisation}, $[\alpha, 1, \beta; \beta e^{\infty}]$  is not in $\Gg_{00}(\Gr,E)$ because $e^{\infty}$ is not strongly fixed by $1$.
But $1$ strongly fixes $\beta e^n f^{\infty}$ and so $[\alpha,1,\beta;\beta e^n f^{\infty}]= [v,1,v;\beta e^n f^{\infty}]$ is $\Gg_{00}(\Gr,E)$ for $n\geq 1$. 
Since $[\alpha,1,\beta;\beta e^n f^{\infty}]$ tends to $[\alpha, 1, \beta; \beta e^{\infty}]$ we conclude that  $\Gg_{00}(\Gr,E)$ is not closed.
\end{example}

\subsection{Non-Hausdorffness}

We start by showing  that the groupoids in  the diagram \eqref{eq:various_groupoids} 
are simultaneously Hausdorff or not.

\begin{proposition}\label{prop:Hausdorff_extended_groupoid}
If one of the groupoids in  the diagram \eqref{eq:various_groupoids} is Hausdorff, then
all of them are. This happens if and only if $(\Gr,E)$ satisfies $\Fin$, that is if for every $g \in \Gr$ there are  finitely many minimal strongly $g$-fixed paths.
\end{proposition}
\begin{proof}	
If any of the groupoids in \eqref{eq:various_groupoids} is Hausdorff, then so is $\Gg_{00}(\Gr,E)$, as it is contained in
all of them. 
Hausdorffness of $\Gg_{00}(\Gr,E)\cong \Gg(S_{00}(\Gr,E))$ is equivalent to $\Fin$ by Proposition~\ref{prop:closedness} 
and Remark~\ref{rmk:terminology_changes}.
Thus, it suffices to  show that $\Fin$ implies that $\widetilde{\Gg}(\Gr,E)$ is Hausdorff, equivalently, that $\widetilde{\Gg}(\Gr,E)^0$ is closed in $\widetilde{\Gg}(\Gr,E)$. So let us assume $\Fin$ and fix $[\alpha,g,\beta;  \beta\xi] \in \widetilde{\Gg}(\Gr,E)\setminus \widetilde{\Gg}(\Gr,E)^0$.   
By Proposition~\ref{prop:transformation_groupoids_self_similar}\ref{enu:transformation_groupoids_self_similar1},
either  $\alpha=\beta$ and  $\xi$ is not  strongly fixed by $g$ or $\alpha \neq \beta$. 
If $\alpha\neq \beta$, then  $U(\alpha,g,\beta;Z(\beta))$ is a neighbourhood of $[\alpha,g,\beta;  \beta\xi]$ that is disjoint 
from   $\widetilde{\Gg}(\Gr,E)^0$. So assume that  $\alpha=\beta$ and  $\xi$ is not  strongly fixed by $g$. 
By $\Fin$ there are finitely many minimal strongly $g$-fixed paths $\gamma_1,\dots,\gamma_n\in \sr(g)E^*$. By Lemma~\ref{lem:strongly_fix_respects_the_order} none of these paths can be a subpath of $\xi$. Denoting by $\xi'$ the subpath of $\xi$ of length $\min\{ |\xi|, \max_{1 \le i \le n}|\gamma_i|\}$, 
the open set $V \coloneqq Z(\beta\xi')\setminus \bigcup_{i=1}^n Z(\beta\gamma_i)$ contains $\beta\xi$, and if $\beta\mu\in V$, then $\mu$ is not strongly $g$-fixed.
Hence, by Proposition~\ref{prop:transformation_groupoids_self_similar}\ref{enu:transformation_groupoids_self_similar1}, the open set $U(\alpha,g,\beta;V)$  
contains  $[\alpha,g,\beta;  \beta\xi]$ and is disjoint from $\widetilde{\Gg}(\Gr,E)^0$. So  $\widetilde{\Gg}(\Gr,E)^0$ is closed in $\widetilde{\Gg}(\Gr,E)$.
\end{proof}
\begin{corollary}\label{cor:Hausdorff_equiv_closed}
For any $*=\Space,0,00$, the following are equivalent:
\begin{enumerate}
\item $(\Gr,E)$ satisfies $\Fin$;
\item $S_{*}(\Gr,E)$ is closed;
\item $S_{*}(\Gr,E)$ is Hausdorff.
\end{enumerate}
\end{corollary}
\begin{proof}
Combine Proposition \ref{prop:Hausdorff_extended_groupoid} and Remark~\ref{rmk:terminology_changes}.
\end{proof}
We now describe the non-Hausdorff parts $\widetilde{\Gg}(\Gr,E)\setminus \widetilde{\Gg}(\Gr,E)_{\Hau}$ and $\Gg(\Gr,E)\setminus\Gg(\Gr,E)_{\Hau}$.
As these are full subgroupoids of  $\widetilde{\Gg}(\Gr,E)$ and $\Gg(\Gr,E)$ it suffices to determine
their unit spaces $E^{\leq \infty} \setminus  \widetilde{\Gg}(\Gr,E)_{\Hau}$ and $\partial E \setminus  \Gg(\Gr,E)_{\Hau}$.

\begin{definition}
A \emph{singular decomposition} of  $\mu\in E^\infty$ is a pair $(\alpha,g)\in E^**\Gr$ such that $\mu$ is not strongly $g$-fixed, but $\mu$ does have a decomposition
$\mu=\alpha\xi$ where every prefix of $\xi$ has a strongly $g$-fixed extension. 
Two singular decompositions $(\alpha,g)$ and $(\beta,h)$ of  $\mu$ are \emph{equivalent} 
if there is a prefix of $\gamma$ of $\mu$ such that $\gamma=\alpha\alpha'=\beta\beta'$
and $g|_{\alpha'}=h|_{\beta'}$. 
\end{definition}
\begin{remark}\label{rem:singular_infinity}
If $(\alpha,g)$ is a singular decomposition of $\mu$, then for any $\alpha'$ such that $\alpha\alpha'$ is a prefix of $\mu$ the pair
$(\alpha\alpha',g|_{\alpha'})$ is a singular decomposition of $\mu$ equivalent to $(\alpha,g)$.
Thus, within a fixed class of singular decompositions, we may pick arbitrarily long prefixes.
\end{remark}
\begin{definition}
A \emph{singular decomposition} of $\mu \in E^*$   is a pair $(\alpha,g)\in E^**\Gr$ 
such that  $\mu=\alpha\xi$ and $\xi$ is not strongly $g$-fixed, but there is an infinite set $I\subseteq \sr(\mu)E^1$ with the property that each $\xi e$, $e\in I$, has a strongly $g$-fixed extension.
Two singular decompositions  $(\alpha,g)$ and $(\beta,h)$ of  $\mu$ are \emph{equivalent} 
if there is  a common extension $\alpha\alpha'=\beta \beta'$  of  $\alpha$ and $\beta$, that  is compatible with $\mu$ and  $\alpha g\alpha'=\beta h\beta'$ and  $g|_{\alpha'}=h|_{\beta'}$. 
\end{definition}
\begin{remark}\label{rem:no_singular_decomposition}
When a finite path $\mu \in E^*$   has a singular decomposition its source $\sr(\mu)$ has to be an infinite receiver. 
In particular, we necessarily have $\mu \in \partial E$ and if  $E$ is row-finite there are no finite paths admitting
singular decompositions.
\end{remark}
In accordance with the notation from page~\pageref{page:closure_over_point}, for $\mu \in \partial E$ we denote by $\overline{X}(\mu)=\Gg(\Gr,E)_{\mu}^{\mu}\cap \overline{\partial E}$ those arrows in the isotropy group over $\mu$ that are also in the  closure of $\partial E$ in $\Gg(\Gr,E)$.
\begin{proposition}\label{prop:non-Hausdorff_points_description}
For any self-similar action $(\Gr,E)$ we have $E^{\leq \infty} \setminus  \widetilde{\Gg}(\Gr,E)_{\Hau}=\partial E \setminus  \Gg(\Gr,E)_{\Hau}$
and a path $\mu$ is in this set if and only if it admits a singular decomposition. More precisely, for every $\mu\in  \partial E$  we have
\begin{equation}\label{eq:singular_isotropy}
\overline{X}( \mu)=\{\mu\}\cup \{[\alpha, g, \alpha, \mu]: (\alpha,g)\in E^**\Gr \text{ is a singular decomposition of }\mu\}
\end{equation}
and the cardinality of $\overline{X}(\mu)$ is equal to the cardinality of
equivalence classes singular decompositions of $\mu$ plus one.
\end{proposition}
\begin{proof}
Let $\mu \in E^{\leq \infty}$. If $\sr(\mu)$ is a source or a finite receiver then the singleton $\{\mu\}=Z(\mu)\setminus\bigcup_{e\in \sr(\mu)E^{1}} Z(\mu e)$ is open. 
Then  $\mu \in  \widetilde{\Gg}(\Gr,E)_{\Hau}$ and so $\overline{X}(\mu)=\{\mu\}$. As  $\mu$ has no singular decompositions, by  Remark~\ref{rem:no_singular_decomposition}, we get \eqref{eq:singular_isotropy}. Thus, $E^{\leq \infty} \setminus  \widetilde{\Gg}(\Gr,E)_{\Hau}=\partial E \setminus  \Gg(\Gr,E)_{\Hau}$
and we may assume that $\mu \in \partial E$. An element $[\alpha,g,\beta; \beta\xi]\in \Gg(\Gr,E)$ is not in $\overline{\partial E}$ when 
$\alpha\neq \beta$ (consider $U(\alpha,g,\beta; Z(\beta))$). Thus, to determine  $\overline{X}(\mu)=\Gg(\Gr,E)_{\mu}^{\mu}\cap \overline{\partial E}$ we only need 
to consider elements $\gamma_{\alpha,g}=[\alpha,g,\alpha; \alpha\xi]$ where $\mu=\alpha\xi$ and $\xi$ is $g$-fixed, see Proposition~\ref{prop:transformation_groupoids_self_similar}\ref{enu:transformation_groupoids_self_similar2}. 
Any such element belongs to $\Gg(\Gr,E)_{\mu}^{\mu}$. We have $\gamma_{\alpha,g}\neq \mu$ if and only if $\xi$ is not strongly $g$-fixed, cf. Proposition~\ref{prop:transformation_groupoids_self_similar}\ref{enu:transformation_groupoids_self_similar1}.
Assuming this we see that $\gamma_{\alpha,g} \in \overline{\partial E}$ if and only if $(\alpha,g)$ is a singular decomposition of $\mu$.
This gives \eqref{eq:singular_isotropy}. 
Moreover, $\gamma_{\alpha,g}=\gamma_{\beta,h}$ if and only if the singular decompositions $(\alpha,g)$ and $(\beta,h)$ are 
equivalent, by Lemma~\ref{lem:lower_bound_description}.
\end{proof}
\begin{corollary}\label{cor:finite_Hausdorffness}
Finite Hausdorffness of any of the groupoids $\widetilde{\Gg}(\Gr,E)$, $\widetilde{\Gg}_0(\Gr,E)$, $\Gg(\Gr,E)$, $\Gg_0(\Gr,E)$ is equivalent 
to existence of at most finitely many inequivalent singular decompositions of each path $\mu\in \partial E$.
\end{corollary}

As an application we now generalise \cite{Brix-Gonzales-Hume-Li:Hausdorff_covers}*{Corollary 7.13} from group actions on sets to groupoid actions on graphs, which was also independently proved in \cite{Aakre}. 
This concerns contracting self-similar actions. For group actions on sets such actions were formalised in  \cite{Nekrashevych:Self-similar}, and
Nekrashevych's definition was  generalised to groupoid actions in \cite{BBGHSW24}, as follows.

\begin{definition}\label{def:contracting_actions} 
A self-similar groupoid action $(\Gr,E)$ on a finite graph $E$ without sources \emph{is contracting} 
if there is a finite subset $\NN\subseteq \Gr $ such that, for all $g\in \Gr$, there exists $n\geq 0$ such that $g|_{\mu}\in \NN$ for all $\mu \in \sr(g)E^*$ with $|\mu|\geq n$. The smallest such $\NN$ is called the \emph{nucleus}.
\end{definition}
For any contracting self-similar action the nucleus exists,  is unique and is closed under sections; that is,
if $n\in \NN$ and $\mu \in \sr(n)E^*$, then $n|_{\mu}\in\NN$, see \cite{BBGHSW24}.
\begin{corollary}\label{cor:contracting_non_Hausdorff}
Let $(\Gr,E)$ be a contracting self-similar groupoid action with nucleus $\NN$. Then for any 
$\mu\in \partial E= E^{\infty}$ and sufficiently long prefix $\alpha$ of $\mu$ we have
$$
\overline{X}( \mu)=\{\mu\}\cup \{[\alpha, n, \alpha, \mu]: n\in \NN \text{and $(\alpha,n)$ is a singular decomposition of }\mu\}.
$$
In particular,  $|\overline{X}( \mu)|\leq |\NN|$. Thus,  $\widetilde{\Gg}(\Gr,E)$ and $\Gg(\Gr,E)$ are finitely non-Hausdorff. 
\end{corollary}
\begin{proof}
Since $E$ has no sources, $\partial E= E^{\infty}$. Let $\mu\in  E^{\infty}$ and write $\mu_k$ for the prefix of $\mu$ with length $|\mu_k|=k$.
Denote by  $F_{k}$ the set of singular decompositions $(\mu_k,n)$ where  $n\in \N$. 
Remark~\ref{rem:singular_infinity} and contractiveness imply that any singular decomposition of $\mu$ falls into  $F_k$ for some $k$. Hence, $\overline{X}( \mu)\subseteq \bigcup_{k\in \N}F_k$. As nucleus is closed under sections,  Remark~\ref{rem:singular_infinity} 
implies that $F_{k}$, $k\in \N$ forms an ascending sequence of sets. However, for each $k\in \N$ we have $|F_k|< |\NN|$, and
so $|\bigcup_{k\in \N}F_k|< \NN.$ As a consequence there is $L\in  N$ such that $\bigcup_{k\in \N}F_k=F_{L}$. 
Now the assertion follows from Proposition~\ref{prop:non-Hausdorff_points_description}.
\end{proof}

\subsection{Dynamical properties}
We describe conditions equivalent to topological freeness for the groupoids of \eqref{eq:various_groupoids}.

\begin{theorem}\label{thm:topological_freeness_self_similar_transformation_groupoids}
Let $(\Gr,E)$ be a self-similar action and consider the following conditions
\begin{enumerate}
\item[\Cyc] \label{enu:cycles}  every $\Gr$-cycle  has an entrance. 

\item[\Evr] \label{enu:fixes_implies_strongly_fixes} if $g\in \Gr$ fixes every  path in $\sr(g)E^*$, then $g\in\Gr$ strongly fixes some path in $\sr(g)E^*$; 
\item[\Rec] \label{enu:receivers} if $g\in \Gr$ fixes some $\alpha \in E^*$ such that $\sr(\alpha)$ is a finite receiver, then $\alpha$ strongly $g$-fixed; 
\end{enumerate}
Then the following statements are true:
\begin{enumerate}
\item\label{enu:topological_freeness2} $\Gg(\Gr,E)$  is topologically  free if and only if $\Evr$ and $\Cyc$ hold.

\item\label{enu:topological_freeness3} Topological freeness of any of $\Gg_0(\Gr,E)$ and $\Gg_{00}(\Gr,E)$ is equivalent to $\Evr$.

\item\label{enu:topological_freeness1}  Topological freeness of any of the groupoids $\widetilde{\Gg}(\Gr,E)$, $\widetilde{\Gg}_0(\Gr,E)$ and $\widetilde{\Gg}_{00}(\Gr,E)$  is equivalent to $\Evr$ and $\Rec$.
\end{enumerate}
\end{theorem}
\begin{proof}
Items \ref{enu:topological_freeness2} and \ref{enu:topological_freeness3} follow from Propositions~\ref{prop:topological_freeness_semigroup} and~\ref{prop:fundamentalness}  and \cite{BKM2}*{Proposition 7.31}, see Remark~\ref{rmk:terminology_changes}.

\ref{enu:topological_freeness1}.
Since topological freeness passes to wide open subgroupoids, it suffices to show that $\Evr$ and $\Rec$ imply 
that $\widetilde{\Gg}(\Gr,E)$ is topologically free, and that topological freeness of $\widetilde{\Gg}_{00}(\Gr,E)$ implies $\Evr$ and $\Rec$.
Let us start with the latter.

Suppose first that $\Evr$ fails and so there is $g\in \Gr$ which fixes all paths in  $\sr(g)E^*$ but does not strongly fix any of them. 
Then  $U(\rg(\xi),g,\rg(\xi);\sr(g)E^*)$ is a non-empty open subset of $\widetilde{\Gg}_{00}(\Gr,E)$ contained in the isotropy bundle and  disjoint with the unit space $E^{\leq\infty}$, cf. Proposition~\ref{prop:transformation_groupoids_self_similar}. Hence, it witnesses the failure of topological freeness of $\widetilde{\Gg}_{00}(\Gr,E)$.
Now suppose that \textup{(Rec)} fails, so there is    $\xi\in \sr(g)E^*$ that starts in a finite receiver $\sr(\xi)$  and such that  $\xi$ is $g$-fixed but not strongly $g$-fixed. 
That $\sr(\xi)$ is a finite receiver implies that  the singleton
$\{\xi\}=Z(\xi)\setminus \bigcup_{e\in \sr(\xi)E^1} Z(\xi e)$ is open in $E^{\leq \infty}$ and so the singleton $\{[\rg(\xi),g,\rg(\xi), \xi]\}= U(\rg(\xi),g,\rg(\xi); \{\xi\})$ is open in $\widetilde{\Gg}_{00}(\Gr,E)$.
This singleton is in  $\Iso(\widetilde{\Gg}_{00}(\Gr,E))\setminus E^{\leq \infty}$, again by Proposition~\ref{prop:transformation_groupoids_self_similar}.
Hence, $\widetilde{\Gg}_{00}(\Gr,E)$ is not topologically free.

Now assume that $\Evr$ and  $\Rec$ hold. 
We need  to show that $\widetilde{\Gg}(\Gr,E)$ is topologically free,
equivalently, that every non-empty basic set  $U(\alpha,g,\beta;\beta V)$, 
which is disjoint with $\widetilde{\Gg}(\Gr,E)^0\cong E^{\le \infty}$
is not contained in the isotropy bundle. 
Let then  $V \subseteq Z(\sr(\beta))$ be a non-empty  open set and such that $U(\alpha,g,\beta;\beta V)\cap \widetilde{\Gg}(\Gr,E)^0=\varnothing$.
By Proposition~\ref{prop:transformation_groupoids_self_similar}\ref{enu:transformation_groupoids_self_similar1}, we   either have $\alpha \neq  \beta$ or $\alpha=\beta$ and  every path in $V$ is not strongly $g$-fixed. If $\alpha\neq \beta$, then taking any finite path $\xi\in V$ the element $[\alpha,g,\beta;\beta\xi]\in U(\alpha,g,\beta;\beta V)$
is not an isotropy arrow, by Proposition~\ref{prop:transformation_groupoids_self_similar}\ref{enu:transformation_groupoids_self_similar2}.  
Hence, we may assume that $\alpha=\beta$ and every path in $V$ is not strongly $g$-fixed. In view of  Proposition~\ref{prop:transformation_groupoids_self_similar}\ref{enu:transformation_groupoids_self_similar2}
we need to show that there is  $\xi\in V$ which is not $g$-fixed. Assume on the contrary, that for every $\xi\in V$ we have  $g\xi=\xi$.
By  \textup{(Rec)}, paths in  $V$ do  not  start in a finite receiver. 
As we may assume that $V$ is of the form $Z(\xi_0)\setminus \bigcup_{i=1}^n Z(\xi_0\alpha_i)$, 
this  implies that   sufficiently long paths $\xi$ in $V\cap E^*$ have the property 
that every extension  of $\xi$ is in $V$. Let us pick  $\xi\in V\cap E^*$ with such a property. As $g$ fixes everything in $V$ it follows that  $g|_{\xi}$ fixes all paths in $\sr(\xi)E^*$. 
Hence, by $\Evr$ there is $\mu\in \sr(\xi)E^*$ which is strongly fixed by $g|_{\xi}$.
But then $\xi\mu\in V$ is strongly $g$-fixed, by Lemma~\ref{lem:strongly_fix_respects_the_order}, which contradicts our assumption. 
\end{proof}
\begin{example} \label{ex:two_edges}
Let $E$ be the directed graph
\[
\begin{tikzpicture}
[baseline=-0.25ex,
vertex/.style={
circle,
fill=black,
inner sep=1.5pt
},
edge/.style={
-stealth,
shorten >= 3pt,
shorten <= 3pt
},
scale =1]

\node[vertex] (u) at (-1,0) {};%
\node[vertex] (v) at (0,0) {};%
\node[vertex] (w) at (1,0) {};%

\node[anchor= south] at (u) {\scriptsize{$u$}};
\node[anchor= south] at (v) {\scriptsize{$v$}};
\node[anchor= south] at (w) {\scriptsize{$w$}};

\draw[edge] (v.west) to node[anchor=south, inner sep = 2pt]{\scriptsize{$e$}} (u.east);
\draw[edge] (w.west) to node[anchor=south, inner sep = 2pt]{\scriptsize{$f$}} (v.east);
\end{tikzpicture}
\]
and let $\Gr$ be the group bundle over $E^0$ with fibres $G_u  \coloneqq  \rg^{-1}(u) \cong \Z$, $G_v  \coloneqq  \rg^{-1}(v) \cong \Z$, and $G_w  \coloneqq  \rg^{-1}(w) = \{0\}$. For each vertex $x \in E^0$ and $k \in \Z$ let $k_x$ denote the corresponding element of $G_x$. We define a self-similar action of $\Gr$ on $E$, given on generators by
\begin{align*}
1_u \cdot e &= e & 1_u|_e &= 1_v, &
1_v \cdot f &= f, \text{ and} & 1_v|_f &= 0_w.
\end{align*}
Every $k_u \in G_u$ fixes $uE^*$ and strongly fixes $ef$, every $k_v \in G_v$ fixes $vE^*$ and strongly fixes $f$, and $0_w$ strongly fixes $w$, so \textup{(Evr)} is satisfied. 
There are no $G$-cycles so \text{(Cyc)} is vacuously satisfied. 
 On the other hand $v$ is a finite receiver and $e$ is $1_u$-fixed but not strongly $1_u$-fixed, so \textup{(Rec)} is not satisfied. In particular, for $*=\Space,\,0,\,00$, the groupoids $\Gg_*(\Gr,E)$ are topologically free, but $\widetilde{\Gg}_*(\Gr,E)$ are not. 
\end{example}

\begin{example}
Let $E$ be the directed graph 
\[
\begin{tikzpicture}
[baseline=-0.25ex,
vertex/.style={
circle,
fill=black,
inner sep=1.5pt
},
edge/.style={
-stealth,
shorten >= 3pt,
shorten <= 3pt
},
scale =1]
\clip  (-.5,-0.5) rectangle (2,0.5);

\node[vertex] (v) at (0,0) {};%
\node[vertex] (w) at (1,0) {};%

\node[anchor= south] at (v) {\scriptsize{$v$}};
\node[anchor= south] at (w) {\scriptsize{$w$}};

\draw[edge] (w.west) to node[anchor=south, inner sep = 2pt]{\scriptsize{$e$}} (v.east);
\draw[edge] (w.north east) to [out = 40, in = -40,loop,min distance=10mm] node[anchor=west, inner sep = 2pt]{\scriptsize{$f$}} (w.south east);
\end{tikzpicture}
\]
and let $\Gr = E^0$ consist only of units. Then there is a unique self-similar action of $\Gr$ on $E$. Each $g \in \Gr$ strongly fixes every path in $\sr(g)E^*$, so $\Evr$ and $\Rec$ are satisfied.  The cycle $f$ does not have an entrance so $\Cyc$ is not satisfied. In particular, $\widetilde{\Gg}_*(\Gr,E)$ for $*=\Space,\,0,\,00$ and $\Gg_{**}(\Gr,E)$ for $**=0,\,00$ are all topologically free, but $\Gg(\Gr,E)$ is not. 
\end{example}

\begin{example}
Let $(\Gr,E)$ be the self-similar action of Example~\ref{ex:not_exel-pardo_top_free}. Since $E$ admits infinitely many strongly $1$-fixed paths, the associated groupoids of \eqref{eq:various_groupoids} are non-Hausdorff. Both $\Evr$ and $\Cyc$ hold as outlined in Example~\ref{ex:not_exel-pardo_top_free}.  On the other hand  $e$ is a path whose source is a finite receiver that is not strongly $1$-fixed, so \textup{(Rec)} does not hold. 
\end{example}

We now characterise effectiveness of
the tight groupoid $\Gg(\Gr,E)$, and as a consequence, we obtain that  in general it is not equivalent to effectiveness of the inverse semigroup $S(\Gr,E)$.
\begin{theorem}\label{thm:effective_tight_groupoid}
For any self-similar action $(\Gr,E)$, the tight groupoid $\Gg(\Gr,E)$  is effective
if and only if $(\Gr,E)$ satisfies $\Cyc$, $\Sla$ and there is no  $g\in \Gr\setminus \Gr^0$ such that $\sr(g)$ is an infinite receiver, and $g$ fixes all paths in $\sr(g)E^*$ 
except for those that are extensions of elements in some finite 
set $F\subseteq \sr(g)E^1$.
\end{theorem}
\begin{proof}
That effectiveness of $\Gg(\Gr,E)$  implies $\Cyc$  and  $\Sla$ follows from 
Proposition~\ref{prop:topological_freeness_semigroup} and \cite{Exel-Pardo:Tight}*{Theorem 4.10}, cf. Remark~\ref{rmk:terminology_changes}.
Let us assume that  there is $g\in \Gr\setminus \Gr^0$ such that $\sr(g)$ is an infinite receiver and $g$ fixes all paths in $\sr(g)E^*$ 
except for those that are extensions of elements in some finite 
set $F\subseteq \sr(g)E^1$. Then   $V \coloneqq Z(\sr(g))\setminus \bigcup_{e\in F} Z(e)\cap \partial E$ is an open subset of $\partial E$ containing $\sr(g)=\rg(g)$. By Proposition~\ref{prop:transformation_groupoids_self_similar}\ref{enu:transformation_groupoids_self_similar2},  the non-empty basic bisection
$U(\sr(g),g,\sr(g); V)$ is  contained  in $\Iso(\Gg(\Gr,E))$. However, it is not contained in the unit space. Namely, if we assume that the point $[\sr(g),g,\sr(g);\sr(g)] \in U(\sr(g),g,\sr(g); V)$ is in the unit space,  then by Proposition~\ref{prop:transformation_groupoids_self_similar}\ref{enu:transformation_groupoids_self_similar2}, $g$ strongly fixes $\sr(g)$. This forces $g$ to be the unit $\sr(g)$, cf. Remark~\ref{rem:about_strongly_fixed}, contradicting our choice of $g$. Hence, $\Gg(\Gr,E)$ is not effective.

Conversely,  let us assume $\Sla$ and $\Cyc$  and that there is no $g\in \Gr\setminus \Gr^0$ such that $\sr(g)$ is an infinite receiver and 
$g$ fixes all paths in $\sr(g)E^*$ 
except for those that are extensions of elements in some finite 
set $F\subseteq \sr(g)E^1$.
Let $U(\alpha,g,\beta;\beta V)$ be a non-empty bisection  contained in $\Iso(\Gg(\Gr,E))$.  
We may assume that $g\notin \Gr^0$, as otherwise $U(\alpha,g,\beta;\beta V)$ is contained in the unit space.
We may also assume that $V$ is a basic open set of the form $Z(\eta)\setminus \bigcup_{i=1}^k Z(\eta\eta_i)\cap \partial E$.
Let us consider two cases:

Assume that  $|\alpha|\neq |\beta|$. Then by Proposition~\ref{prop:transformation_groupoids_self_similar}\ref{enu:transformation_groupoids_self_similar2},
the open set $\beta V\subseteq \partial E$ is necessarily a  singleton consisting of an infinite path $\alpha_{\infty}$ coming from a $G$-cycle $\alpha_1$ as
described in Lemma~\ref{lem:infinite_path_from_g_cycle}. By the form of $V$, this means that every extension of $\beta\eta$ is a prefix of $\alpha_{\infty}$. But by $\Cyc$, $\alpha_1$ has to have an entrance and so every prefix of $\alpha_{\infty}$ will have infinitely many distinct extensions.
Hence, $\Cyc$ excludes this case.

Now suppose that $|\alpha|=|\beta|$. Then by Proposition~\ref{prop:transformation_groupoids_self_similar}\ref{enu:transformation_groupoids_self_similar2},
we have $\alpha=\beta$ and every $\xi\in V$ is $g$-fixed. By  Proposition~\ref{prop:transformation_groupoids_self_similar}\ref{enu:transformation_groupoids_self_similar1} we need to show that every $\xi\in V$ is strongly $g$-fixed, that is $g|_{\xi}=\sr(\xi)$. Let $\xi\in V$. Let us consider two  subcases.

Assume first that $\xi$ is infinite. Then writing $\xi=\xi_n\xi'$ for $\xi_n\in E^n$ and $\xi'\in \sr(\xi_n)E^\infty$ for  sufficiently large $n$ ($n\geq \max_i |\eta\eta_i|$), we get
that  every extension of $\xi_n$  in $\partial E$ is in $V$. Hence, for every $\eta\in \sr(\xi_n)\partial E$  the path $\xi_n\eta$ is $g$-fixed.
This implies that $g|_{\xi_n}$ fixes all paths in $\sr(\xi_n)E^{*}$. Hence, by $\Sla$ there is a strongly $g|_{\xi_n}$-fixed path which is a prefix of   $\xi'\in \sr(\xi_n)E^\infty$.  In other, words  $\xi'$ is strongly $g|_{\xi_n}$-fixed. Therefore, $\xi=\xi_n\xi'$ is strongly $g$-fixed, see Lemma~\ref{lem:strongly_fix_respects_the_order}.

Assume that $\xi$ is finite.  Since $\xi$ is $g$-fixed,   we have $g|_{\xi}$ fixes  $\sr(\xi)$. 
Thus, if $\sr(\xi)$ is a source, $g|_{\xi}$ fixes the unique element of $\sr(\xi)E^{*}=\{\sr(\xi)\}$, and so by  $\Sla$  we get that $g|_{\xi}=\sr(\xi)$, that is $\xi$ is strongly $g$-fixed.
Hence, we may assume that $\sr(\xi)$ is an infinite receiver. Let  $F$ be the set of edges $e\in \sr(\xi)E^1$ such that $\xi e$  is comparable with $\eta\eta_i $ for some $i$. Then $|F|\leq k<\infty$ and for every $\mu\in \sr(\xi)E^*$ which is not an extension of an  element in  $F$ we have that $\xi\mu\in V$. 
Thus, $\xi\mu$ is $g$-fixed and as a consequence $\mu$ is $g|_{\xi}$-fixed. In other words, $\sr(g|_{\xi})=\sr(\xi)$ is an infinite receiver and $g|_{\xi}$
fixes all $\sr(g|_{\xi})E^*$ except the paths that are comparable with  $F$. Hence, $g|_{\xi}=\sr(g|_{\xi})=\sr(\xi)$ by assumption.
\end{proof}
\begin{example} \label{ex:infinitely_many_edges}
Let $E$ be the directed graph
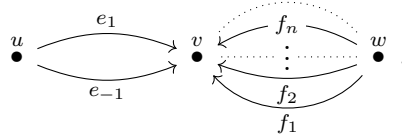
\[
\begin{tikzcd}
\stackrel{u}{\bullet} \arrow[rr, "e_1", bend left=22] \arrow[rr, "e_{-1}"', bend right=22] &  & 
\stackrel{v}{\bullet} &  & \stackrel{w}{\bullet}\arrow[ll, "f_1", bend left=49]
 \arrow[ll, dotted, no head, "\raisebox{5pt}{$\vdots$}" description] \arrow[ll, "f_2", bend left=20]  \arrow[ll, "f_n"  description, bend right]  \arrow[ll, no head, dotted, bend right=49]
\end{tikzcd}. 
\]
Then $v$ is the range of infinitely many edges $f_n$, $n\geq 1$, as well as $e_{\pm 1}$.
Let $\Gr$ be the group bundle over $E^0=\{u,v,w\}$ with fibres $G_u  \coloneqq  \{0\}$, $G_v  \coloneqq \Z_2=\{0,1\}$, and $G_w  \coloneqq  \{0\}$. 
We define a self-similar action of $\Gr$ on $E$,  determined by the relations 
\[
1\cdot e_{\pm} = e_{\mp 1}, \qquad		1 \cdot f_n = f_n, \text{ for }n\geq 1.
\]
Then $\Cyc$ and $\Sla$ are trivially satisfied as there are no $G$-cycles and the only nontrivial element in the group bundle, $1\in G_v$, does not fix 
all paths in $vE^*$. 
Hence, the associated inverse semigroup $S(\Gr,E)$ is effective by 
Proposition~\ref{prop:topological_freeness_semigroup}. But by Theorem~\ref{thm:effective_tight_groupoid}, the associated tight groupoid
$\Gg(\Gr,E)$ is not effective, as 
$1\in \Gr_{v}$ is not a unit, but it fixes all paths in $vE^{*}\setminus (e_1E^*\cup e_{-1}E^*)=\{f_n:n\geq 1\}$.
In fact, $\{f_n:n\geq 1\}$ are strongly $1$-fixed edges, and so $\Gg(\Gr,E)$ is non-Hausdorff by Proposition~\ref{prop:Hausdorff_extended_groupoid}.
As $\Gg(\Gr,E)$ is topologically free but not effective, we have 
an even stronger statement. Namely, by Remark~\ref{rem:effectiveness_no_good} 
the singular ideal in $F^P_{\red}(\Gg(\Gr,E))$ does not vanish for every nonempty  $P\subseteq [1,\infty]$.
\end{example}
\begin{corollary}\label{cor:effectiveness_fails} If the groupoid $\Gg(\Gr,E)$ is effective, 
then the inverse semigroup $S(\Gr,E)$ is effective. The converse implication holds 
when $E$ is row-finite or if $\Fin$ holds, but in general it fails
(and then the ``algebraic singular ideal'' $\mathfrak{C}_c(\Gg)\cap \mathfrak{M}_0(\Gg)\neq \{0\}$ is nonzero).
\end{corollary}
\begin{proof}
The assertion about implications follow from  \cite{Exel-Pardo:Tight}*{Theorem 4.10}, but also from 
Proposition~\ref{prop:topological_freeness_semigroup} and Theorem~\ref{thm:effective_tight_groupoid}.
Equivalence in general fails by Example~\ref{ex:infinitely_many_edges}, and then $\mathfrak{C}_c(\Gg)\cap \mathfrak{M}_0(\Gg)\neq \{0\}$ by Lemma~\ref{lem:effectiveness_is_about_singular_ideal}.
\end{proof}
Let us now pass to local contractiveness, see Definition 
\ref{defn:groupoid_properties} and Remark~\ref{rmk:terminology_changes}. 
The following generalises \cite{Exel-Pardo:Self-similar}*{Theorem 16.1} and \cite{Exel-Pardo-Starling:Self-similar}*{Theorem 4.6}.
\begin{proposition}\label{prop:contractiveness_of_groupoid}
Let \(
S
\) be the canonical image of $S(\Gr,E)$ in $\Gg(\Gr,E)$ and consider the following conditions
\begin{enumerate}
\item\label{enu:contractiveness_of_groupoid1}  $(\Gr,E)$ satisfies $\Con$ above;
\item\label{enu:contractiveness_of_groupoid2} the groupoid $\Gg(\Gr,E)$ is locally contracting with respect to $S$;
\item\label{enu:contractiveness_of_groupoid3}   the groupoid $\Gg(\Gr,E)$ is locally contracting; and
\item\label{enu:contractiveness_of_groupoid4} every $G$-cycle has an entrance (that is \Cyc~ holds).
\end{enumerate}
Then \ref{enu:contractiveness_of_groupoid1}$\Rightarrow$\ref{enu:contractiveness_of_groupoid2}$\Rightarrow$\ref{enu:contractiveness_of_groupoid3}$\Rightarrow$\ref{enu:contractiveness_of_groupoid4}.
If $E$ is row-finite then \ref{enu:contractiveness_of_groupoid1}$\Leftrightarrow$\ref{enu:contractiveness_of_groupoid2}.
If every vertex is the range of a path whose source is a base point of a $G$-cycle (which is automatic when $E$ is finite and has no sources), then all conditions \ref{enu:contractiveness_of_groupoid1}--\ref{enu:contractiveness_of_groupoid4}
are equivalent.
\end{proposition}
\begin{proof}
Implication \ref{enu:contractiveness_of_groupoid1}$\Rightarrow$\ref{enu:contractiveness_of_groupoid2} follows from Proposition~\ref{prop:locally_contracting_semigroup}
and \cite{Exel-Pardo:Tight}*{Theorem 6.5}, cf. Remark~\ref{rmk:terminology_changes}. It also gives \ref{enu:contractiveness_of_groupoid1}$\Leftrightarrow$\ref{enu:contractiveness_of_groupoid2}
when $E$ is row-finite (as row-finiteness of $E$ is equivalent to assuming that every tight filter in $\EE$ is an ultrafilter).
Implication \ref{enu:contractiveness_of_groupoid2}$\Rightarrow$\ref{enu:contractiveness_of_groupoid3} is trivial. 
If $\alpha$ is a $G$-cycle without entrance, then using 
Lemma~\ref{lem:infinite_path_from_g_cycle} we produce  $\alpha_{\infty}\in E^\infty$ such that the singleton $\{\alpha_{\infty}\}$ is open in $\partial E$, and so it cannot be contracted. This proves that \ref{enu:contractiveness_of_groupoid3}$\Rightarrow$\ref{enu:contractiveness_of_groupoid4}.
If every vertex is a range of a path whose source is a base point of a $G$-cycle, then clearly  \ref{enu:contractiveness_of_groupoid4} implies \ref{enu:contractiveness_of_groupoid1},  see Proposition~\ref{prop:locally_contracting_semigroup}\ref{enu:locally_contracting_semigroup3}.
\end{proof}
The following is an immediate consequence of Propositions~\ref{prop:minimality_inverse_semigroup} and~\ref{prop:cofinal_equiv_no_invariants}, 
and Remark~\ref{rmk:terminology_changes}.
\begin{proposition}\label{prop:groupoid_minimality}
The following conditions are equivalent:
\begin{enumerate}
\item\label{enu:groupoid_minimality1} the groupoid $\Gg(\Gr,E)$ is minimal;
\item\label{enu:groupoid_minimality2}  $(\Gr,E)$ is cofinal; and
\item\label{enu:groupoid_minimality3} there are no nontrivial $G$-invariant, hereditary and saturated sets in $E^0$.
\end{enumerate}
\end{proposition}
\begin{remark}
Since $\partial E$ is a closed $\tilde{\Gg}(\Gr,E)$-invariant subspace of $E^{\leq \infty}$, the groupoid $\tilde{\Gg}(\Gr,E)$ is not minimal unless $\tilde{\Gg}(\Gr,E)=\Gg(\Gr,E)$, which holds if and only if all vertices in $E^0$ are either sources or infinite receivers.
\end{remark}

\section{Digression on pseudo freeness}\label{sect:pseudo_freeness}
The following pseudo freeness condition was coined in \cite{Exel-Pardo:Self-similar}*{Definition 5.4}, cf. \cite{Deaconu}*{Definition 5.1}, 
and is assumed in a number of papers.   It can be viewed as a very strong form of Hausdorffness, which from our perspective is much too restrictive. Nevertheless, it can be rephrased using a number of natural regularity conditions that play a role in the literature. Therefore, for the sake of completeness, we briefly discuss these conditions here.

We recall that an inverse semigroup $S$ is called \emph{$E^*$-unitary} (or \emph{$0$-$E$-unitary}) if any element $t\in S$ that trivially fixes an idempotent is idempotent itself, i.e. if $0\neq e\leq t$ for some $e\in \EE(S)$ implies that $t\in \EE(S)$, cf. Definition~\ref{def:InverseSemigroupSimplicityProperties}.
This is a well-established notion, see for instance \cite{Lawson}*{Chapter 9}. If $S$ is $E^*$-unitary, then any action of $S$ yields a Hausdorff transformation groupoid $S\rtimes X$ (cf. \cite{Exel-Pardo:Tight}*{Theorem 3.15}  or \cite{Kwasniewski-Meyer:Essential}*{Lemma 2.6}). 
Recall also the epimorphisms $\Gr\rtimes E^{\le \infty}\onto \widetilde{\Gg}_{00}(\Gr,E)$ and $\Gr\rtimes \partial E\onto\Gg_{00}(\Gr,E)$ from Remark~\ref{rem:groupoid_epimorphisms}.
\begin{definition}
A self-similar groupoid action $(\Gr,E)$ is \emph{pseudo free} if every $g\in \Gr\setminus \Gr^0$ admits no strongly $g$-fixed edges.
\end{definition}
\begin{proposition}\label{prop:pseudo_freeness}
For any self-similar groupoid action $(\Gr,E)$ the following  are equivalent:
\begin{enumerate}
\item\label{enu:pseudo_freeness1} $(\Gr,E)$ is pseudo free;
\item\label{enu:pseudo_freeness2} every $g\in \Gr\setminus \Gr^0$ admits no strongly $g$-fixed paths;
\item\label{enu:pseudo_freeness2.5}  $g\alpha=h\alpha$ and $g|_{\alpha}=h|_{\alpha}$ implies $g=h$, for every $\alpha\in E^*$ and $g,h\in \Gr \rg(\alpha) $;
\item\label{enu:pseudo_freeness3} the canonical groupoid epimorphism $\Gr\rtimes E^{\le \infty}\onto \widetilde{\Gg}_{00}(\Gr,E)$ is an isomorphism;
\item\label{enu:pseudo_freeness4} the canonical groupoid epimorphism $\Gr\rtimes \partial E\onto\Gg_{00}(\Gr,E)$ is an isomorphism;
\item\label{enu:pseudo_freeness6} the inverse semigroup $S(\Gr,E)$ is  $E^*$-unitary;
\item\label{enu:pseudo_freeness7}  the left action of $\Gr$ on the associated groupoid correspondence $X=E^1* \Gr$ is free, cf.  Proposition~\ref{prop:correspondence_from_self_similar}; and
\item\label{enu:pseudo_freeness8} the associated Zappa--Sz\'ep product category  $E^*\bowtie \Gr = E^**\Gr$
is (right) cancellative, cf. Remark~\ref{rem:Zappa--Szep_product}.
\end{enumerate} 
\end{proposition}
\begin{proof}  Implication \ref{enu:pseudo_freeness1}$\Rightarrow$\ref{enu:pseudo_freeness2}  follows from Lemma~\ref{lem:strongly_fix_respects_the_order}.
Implication  \ref{enu:pseudo_freeness2}$\Rightarrow$\ref{enu:pseudo_freeness2.5} is straightforward, see \cite{Deaconu}*{Remark 5.2}. 
The converse implications are trivial and so \ref{enu:pseudo_freeness1}--\ref{enu:pseudo_freeness2.5} are equivalent.
By Lemma~\ref{lem:lower_bound_description} we have 	 $[\rg(g),g,\sr(g); \xi]=[\rg(h), h, \sr(h); \xi]$
if and only if there is $\alpha\in E^*$  such that  $g\alpha=h\alpha$ and $g|_{\alpha}=h|_{\alpha}$. 
Hence, injectivity of any of the maps in \ref{enu:pseudo_freeness3} and \ref{enu:pseudo_freeness4} is equivalent to \ref{enu:pseudo_freeness2.5}.
Thus, \ref{enu:pseudo_freeness1}--\ref{enu:pseudo_freeness4} are equivalent.

Now let $t\in S(\Gr,E)$. By \eqref{eq:F_t_for_self-similar} the existence of  a nonzero idempotent $e \in \EE(\Gr,E)$  with  $e\leq t$ is 
equivalent to having $t=(\alpha,g,\alpha)$ with  $g$ strongly fixing some $\alpha'\in \sr(g)E^*$. On the other hand,
$t=(\alpha,g,\alpha)$ is an idempotent if and only if $g$ is the unit. This shows that \ref{enu:pseudo_freeness6}$\Leftrightarrow$\ref{enu:pseudo_freeness2}.

Condition \ref{enu:pseudo_freeness7} means that the map 
$$
\Gr*X=\Gr*E^1* \Gr \ni ( g,e,h)\mapsto (ge, g|_{e}h, e,h)\in  E^1* \Gr\times E^1* \Gr=X\times X
$$ 
is injective.  Thus, the implications \ref{enu:pseudo_freeness7}$\Rightarrow$\ref{enu:pseudo_freeness1}
and \ref{enu:pseudo_freeness2.5}$\Rightarrow$\ref{enu:pseudo_freeness7} are clear. 

Finally, recall that the category $E^*\bowtie \Gr$  is always left cancellative and so \ref{enu:pseudo_freeness8} is equivalent to right cancellativity of $E^*\bowtie \Gr$. 
Hence, \ref{enu:pseudo_freeness8} reads as the following implication, cf. Remark~\ref{rem:Zappa--Szep_product}:
for all  $(\mu,g), (\mu'g'), (\nu,h)\in E^*\bowtie \Gr= E^**\Gr$ with $\sr(g)=\sr(g')=\rg(\nu)$ we have
\[
(\mu(g\nu),g|_{\nu}h)=(\mu'(g'\nu),g'|_{\nu}h) \,\, \Longrightarrow\,\, (\mu,g)= (\mu',g').
\]
But using that the  action of $\Gr$ on $E^*$ preserves length and that $h\in \Gr$ is invertible, 
the equality $(\mu(g\nu),g|_{\nu}h)=(\mu'(g'\nu),g'|_{\nu}h)$ is equivalent to 
the equalities $\mu=\mu'$,  $g\nu =g'\nu$ and    $g|_{\nu}=g'|_{\nu}$. 
Accordingly, the above implication is equivalent to \ref{enu:pseudo_freeness2.5}. Hence, all conditions \ref{enu:pseudo_freeness1}--\ref{enu:pseudo_freeness8} are equivalent.
\end{proof}
\begin{remark}
An inverse semigroup  $S$ is \emph{strongly $E^*$-unitary} if there is a $1$-cocycle $c:S\setminus \{0\}\to \Gamma$ into some group $\Gamma$ such that $c^{-1}(1)=\EE(S)\setminus \{0\}$. In general, this is a  strictly stronger condition than being $E^*$-unitary, see \cite{Bulman-Fleming_Fountain_Gould}. 
It is shown in \cite{Exel-Starling:Self-similar}*{Theorem 4.4}, see also \cite{Larki_Hasiri}*{Theorem 10.4}, that for group actions on (row-finite) graphs pseudo freeness is also equivalent to strong $E^*$-unitariness. It is not clear whether arguments of \cite{Exel-Starling:Self-similar}  can be generalised to 
groupoid actions.
\end{remark}

\begin{corollary}
If  $(\Gr,E)$ is a  pseudo free  self-similar action, then all the groupoids in  the diagram \eqref{eq:various_groupoids} are Hausdorff and
$\widetilde{\Gg}_{00}(\Gr,E)$ is clopen in $\widetilde{\Gg}(\Gr,E)$  (and so $\Gg_{00}(\Gr,E)$ is clopen in $\Gg(\Gr,E)$). 
\end{corollary}
\begin{proof}
Hausdorffness follows from Proposition~\ref{prop:Hausdorff_extended_groupoid}. 
By Lemma~\ref{lem:H_0_characterisation} and pseudo freeness,  $[\alpha,g,\beta; \beta\xi]\in \widetilde{\Gg}_{00}(\Gr,E)$ if and only if there is $h\in \Gr$  such that $h\beta =\alpha$ and $h|_{\beta}=g$. This condition does not depend on $\xi$. Hence, if $[\alpha,g,\beta; \beta\xi]\not\in \widetilde{\Gg}_{00}(\Gr,E)$, then
$U(\alpha,g,\beta;Z(\beta))\cap \widetilde{\Gg}_{00}(\Gr,E)=\emptyset$, and so $\widetilde{\Gg}_{00}(\Gr,E)$ is closed in $\widetilde{\Gg}(\Gr,E)$.
\end{proof}

\section{The twist}\label{sect:twist}

Fix  a self-similar action $(\Gr,E)$. We construct $2$-cocycles for self-similar actions using  as little data as possible. 
In particular, we may always assume that they are trivial on the graph.
\begin{definition}\label{dfn:ss_cocycle}
A   \emph{twist} or  a \emph{normalised $\T$-valued $2$-cocycle} for a self-similar action $(\Gr,E)$ is a pair $\sigma = (\sigma_{\Gr},\sigma_{\bowtie})$ where
\begin{enumerate}
\item $\sigma_{\Gr} \colon \Gr^2 \to \T$ satisfies \eqref{eq:groupoid_cocycle_identities}, i.e. it is a normalised  $\T$-valued groupoid 2-cocycle;
\item $\sigma_{\bowtie} \colon \Gr * E^1\to \T$ is such that for all $(g,h,e) \in \Gr * \Gr * E^1$ we have $\sigma_{\bowtie}(\rg(e),e) = 1$ and 
\begin{equation} 
\label{eq:ssc_edges_G} \sigma_{\bowtie}(h,e)\ol{\sigma_{\bowtie}(gh,e)} \sigma_{\bowtie} (g,h e) =\ol{ \sigma_{\Gr}(g|_{he}, h|_{e})} \sigma_{\Gr}(g,h).
\end{equation}
\end{enumerate}
\end{definition}
\begin{remark}\label{rem:trivial_bowtie}
By \eqref{eq:ssc_edges_G}, we can put $\sigma_{\bowtie}\equiv 1$ in the  cocycle  for  $(\Gr,E)$ if and only $\sigma_{\Gr}$ is invariant under sections  in the sense 
that $\sigma_{\Gr}(g,h)=\sigma_{\Gr}(g|_{he}, h|_{e})$ for all $(g,h,e) \in \Gr^2 * E^1$. On the other hand,
if we assume that $\sigma_{\Gr}\equiv 1$, then \eqref{eq:ssc_edges_G} reduces to 
$\sigma_{\bowtie}(gh,e)=\sigma_{\bowtie}(h,e) \sigma_{\bowtie} (g,h e)$, which means that  $\sigma_{\bowtie}$ is a $\T$ valued $1$-cocycle for the action of $\Gr$ on $E^1$.
Therefore, twists of the form $(1,\sigma_{\bowtie})$ correspond to twists considered in \cite{Cortinas}.
However,  the purely algebraic set up in \cite{Cortinas} is slightly different as Corti\~nas allows his  $1$-cocycles to take values in the multiplicative group of a ring.
\end{remark}
\begin{lemma}\label{lem:self-similar-cocycle-is-total-cocycle}
Let $\sigma$ be a  twist for $(G,E)$. The map $\sigma_{\bowtie}$ extends to a map $\sigma_{\bowtie} \colon \Gr * E^* \to \T$ determined inductively by
\begin{equation}\label{eq:cocycle_paths}
\sigma_{\bowtie}(h,e \mu)  \coloneqq  \sigma_{\bowtie}(h,e) \sigma_{\bowtie}(h|_{e},\mu), \qquad (h,e,\mu) \in \Gr * E^1 * E^*,
\end{equation}
and $\sigma_{\bowtie}(h,\sr(h))=1$ for any $h\in \Gr$. 	 Then for all $(g, h, \lambda, \mu) \in \Gr * \Gr * E^* *
E^*$ 
\begin{enumerate}
\item\label{itm:ssc_normalised} $\sigma_{\bowtie}(h,\sr(h))=1$ and $\sigma_{\bowtie}(\rg(\mu),\mu) = 1$;
\item\label{itm:ssc_bowtie} $\sigma_{\bowtie}(h,\lambda \mu) = \sigma_{\bowtie}(h,\lambda) \sigma_{\bowtie}(h|_{\lambda}, \mu)$;
\item\label{itm:ssc_G} $\sigma_{\bowtie}(h,\lambda)\ol{\sigma_{\bowtie}(gh,\lambda)} \sigma_{\bowtie} (g,h \lambda) = \ol{\sigma_{\Gr}(g|_{h\lambda}, h|_{\lambda})} \sigma_{\Gr}(g,h)$.
\end{enumerate}
\end{lemma}
\begin{proof}
\ref{itm:ssc_normalised} follows immediately from \eqref{eq:cocycle_paths} and that $\sigma_{\bowtie}(\rg(e),e) = 1$ for all $e \in E^1$.
For \ref{itm:ssc_bowtie} we induct on the length of $\lambda$. The base case is given by \eqref{eq:cocycle_paths}. Fix $k \ge 1$ and suppose that \ref{itm:ssc_bowtie} holds for all $(h,\lambda,\mu) \in \Gr * E^k *\Gr$. Fix $(h,\lambda,\mu) \in \Gr * E^{k+1} *\Gr$ and write $\lambda = e \lambda'$ for $e \in E^1$ and $\lambda' \in E^k$. Using the inductive hypothesis at the second equality we have
\begin{align*}
\sigma_{\bowtie}(h,\lambda \mu) 
&\stackrel{\eqref{eq:cocycle_paths}}{=}
\sigma_{\bowtie}(h,e) \sigma_{\bowtie}(h|_e,\lambda'\mu)
= \sigma_{\bowtie}(h,e) \sigma_{\bowtie}(h|_e,\lambda')
\sigma_{\bowtie} \big((h|_{e})|_{\lambda'}),\mu\big)\\
&\stackrel{\eqref{eq:cocycle_paths}}{=}
\sigma_{\bowtie}(h,e\lambda')
\sigma_{\bowtie} (h|_{e\lambda'},\mu)
=
\sigma_{\bowtie}(h,\lambda)
\sigma_{\bowtie} (h|_{\lambda},\mu).
\end{align*}

For \ref{itm:ssc_G} we also induct on the length of $\lambda$. The base case is \eqref{eq:ssc_edges_G}. Fix $k \ge 1$ and suppose that \ref{itm:ssc_G} holds for all $(g,h,\lambda) \in \Gr * \Gr * E^k$.  Fix $(g,h,\lambda) \in \Gr * \Gr * E^{k+1}$ and write $\lambda = e \lambda'$ for $e \in E^1$ and $\lambda' \in E^k$. Then $h\lambda = he h|_e \lambda'$, so using the inductive hypothesis at the third equality,
\begin{align*}
\sigma_{\bowtie}(h,\lambda) &\ol{\sigma_{\bowtie}(gh,\lambda)} \sigma_{\bowtie} (g,h \lambda) \\
&\stackrel{\eqref{eq:cocycle_paths}}{=}
\sigma_{\bowtie}(h,e) 	\sigma_{\bowtie}(h|_e,\lambda')
\ol{\sigma_{\bowtie}(gh,e)}
\ol{\sigma_{\bowtie}((gh)|_e,\lambda')}
\sigma_{\bowtie}(g,he)
\sigma_{\bowtie}(g|_{he},h|_e \lambda')\\
&\stackrel{\eqref{eq:ssc_edges_G}}{=}
\sigma_{\Gr}(g|_{he}, h|_{e}) \ol{\sigma_{\Gr}(g,h)}
\sigma_{\bowtie}(h|_e,\lambda')
\ol{\sigma_{\bowtie}(g|_{he} h|_e ,\lambda')}
\sigma_{\bowtie}(g|_{he},h|_e \lambda')\\
&\,\,=\ol{ \sigma_{\Gr}(g|_{he}, h|_{e})} \sigma_{\Gr}(g,h)
\ol{\sigma_{\Gr}\big((g|_{he})|_{h|_{e}\lambda'}  , (h|_{e})|_{\lambda'}\big)}
\sigma_{\Gr} (g|_{he},h|_e)\\
&\,\,= \sigma_{\Gr}(g,h)
\ol{\sigma_{\Gr}\big((g|_{he})|_{h|_{e}\lambda'}  , (h|_{e})|_{\lambda'}\big)}.
\end{align*}
Since $(g|_{he})|_{h|_{e}\lambda'} = g|_{he h|_{e}{\lambda'}} = g|_{h\lambda}$ and $(h|_e)|_{\lambda'} = h|_{e\lambda'} = h|_{\lambda}$ we are done. 
\end{proof}
Lemma~\ref{lem:self-similar-cocycle-is-total-cocycle} says that each $2$-cocycle for the self-similar aciton $(\Gr,E)$ extends  
to a $2$-cocycle for the matched pair of categories $(\Gr,E^*)$ defined as follows  \cite{Mundey_Sims}*{Definition~7.12}.

\begin{definition} We say that
$\varphi\colon \Gr^{2}
\sqcup (\Gr * E^*) \sqcup  E^{*2} \to \T$ is a \emph{total $2$-cocycle} on $(\Gr,
E^*)$, if $\varphi_{2,0}  \coloneqq  \varphi|_{\Gr^{2}}$,
$\varphi_{1,1}  \coloneqq  \varphi|_{\Gr * E^*}$ and $\varphi_{0,2}  \coloneqq 
\varphi|_{E^{*2}}$ satisfy
\begin{enumerate}
\item $\varphi_{2,0}\colon \Gr^{2} \to \T$ is a normalised $\T$-valued $2$-cocycle in the
sense of \cite{Renault};
\item $\varphi_{0,2}\colon E^{*2} \to \T$ is a normalised $\T$-valued categorical
$2$-cocycle in the sense of \cite{KPSiv};
\item $\varphi_{1,1}(h,\lambda) = 1$ whenever $h \in \Gr^0$, or $\lambda \in E^0$, and for $(g, h, \lambda, \mu) \in \Gr * \Gr * E^* *
E^*$,
\begin{align*}
\varphi_{1,1}(h|_\lambda,\mu)  \ol{\varphi_{1,1}(h, \lambda\mu)}  \varphi_{1,1}(h, \lambda)  \varphi_{0,2}(\lambda,\mu)  \ol{\varphi_{0,2}((h\lambda,h|_{\lambda}\mu))} &= 1\quad\text{and}\\
\varphi_{2,0}(g|_{h\lambda},h|_\lambda)  \ol{\varphi_{2,0}(g,h)}\,\ol{\varphi_{1,1}(h, \lambda)} \varphi_{1,1}(gh,\lambda)  \ol{\varphi_{1,1}(g, h\lambda)} &= 1.
\end{align*}
\end{enumerate}
\end{definition}
\begin{proposition}
Let $(\Gr,E)$ be a self-similar action. For every  twist $\sigma=(\sigma_{\bowtie},\sigma_{\Gr})$ of $(\Gr,E)$ gives a total $2$-cocycle   
$(1,\ol{\sigma}_{\bowtie},\sigma_{\Gr})$ on $(\Gr,E^*)$ where $\ol{\sigma}_{\bowtie}$ is the complex conjugate of the extended map from Lemma~\ref{lem:self-similar-cocycle-is-total-cocycle} and $1 \colon E^{*2} \to \T$ is the constant function with value $1$. Conversely, every  
total  $2$-cocycle $\varphi = (\varphi_{0,2},\varphi_{1,1},\varphi_{2,0})$  on $(\Gr,E^*)$ is cohomologous to one of the form $(1,\ol{\sigma}_{\bowtie},\sigma_{\Gr})$.
\end{proposition}

\begin{proof} The first part is clear by Lemma~\ref{lem:self-similar-cocycle-is-total-cocycle}. 
Fix a normalised total $2$-cocycle $\varphi$. The cohomology group $H^2(E^*;\T) = 0$, so there exists a cochain $\tau \colon E^* \to \T$ such that $d^{1}(\tau) = \varphi_{0,2}$, (see for example \cite{Mundey_Sims}*{Proposition~6.1})  where $d^{1}$ is the first differential in the categorical cochain complex associated to $E^*$ with coefficients in $\T$, \cite{Mundey_Sims}*{Definition~4.1}.
Using additive notation in the cochain groups, it follows that in the total cochain complex associated to $(G,E)$ (see \cite{Mundey_Sims}*{\S 4.4}) that
$\varphi + d^1_{Tot}(\tau,0)   =  (1,\varphi_{1,1} - d^{1,0}_h(\tau), \varphi_{2,0})$  is a normalised total $2$-cocycle that is cohomologous to $\varphi$. Since $\varphi + d^1_{Tot}(\tau,0)$ is a total $2$-cocycle, $\ol{\sigma}_{\bowtie}  \coloneqq  \varphi_{1,1} - d^{1,0}_h(\tau)$ and $\sigma_{\Gr}  \coloneqq  \varphi_{2,0}$ define a normalised self-similar $2$-cocycle. 
\end{proof}
\begin{remark}
We may associate to $(\Gr,E)$ the Zappa--Sz\'ep product category $E^* \bowtie \Gr$ described in Remark~\ref{rem:Zappa--Szep_product}, and we may apply to $E^* \bowtie \Gr$ the machinery developed in   \cite{Mundey_Sims}. In particular, a \emph{categorical $2$-cocycle} on $E^* \bowtie \Gr$  is a map $c:(E^* \bowtie \Gr)^2\to \T$ satisfying
\[
c((\mu,h),(\nu,k)) c((\lambda,g),(\mu (h\nu), h|_{\nu} k))
=
c((\lambda g \mu, g|_{\mu}h), (\nu,k)) c((\lambda,g),(\mu,h))
\]
for all $(\lambda,g,\mu,h,\nu,k) \in (E^* \bowtie \Gr)^{2}$. A categorical $2$-cocycle is normalised if 
$$c((\lambda,g),(\sr(g),\sr(g))) = c((\rg(\lambda),\rg(\lambda)),(\lambda,g)) =  1$$ for all $(\lambda,g) \in E^* \bowtie \Gr$. 
By the results of \cite{Mundey_Sims}*{p.53}   any normalised categorical $2$-cocycle  $c:(\Gr \bowtie E)^2\to \T$ on the category  $G\bowtie E$ is determined, up to cohomology class, by a total $2$-cocycle $\varphi = (\varphi_{2,0},\varphi_{1,1},\varphi_{0,2})$ via the formula
\[
c((\lambda, g),(\mu,h))=\varphi_{0,2}(\lambda, g\mu) \varphi_{1,1}(g,\mu)\varphi_{2,0}(g|_{\mu},h).
\]
Conversely, every total $2$-cocycle is determined, up to cohomology class, by a categorical $2$-cocycle. 
\end{remark}
The $2$-cocycle for self-similar action induces a $2$-cocycle for the associated inverse semigroup.

\begin{proposition}\label{prop:twist_self_similar_inverse_semigroup}
Suppose $\sigma = (\sigma_G,\sigma_{\bowtie})$ is a twist of a self-similar action $(\Gr,E)$. For $(\alpha,g,\beta), (\gamma,h,\delta) \in S(\Gr,E)$ the formula
\[
\omega_{\sigma}((\alpha,g,\beta), (\gamma,h,\delta))
 \coloneqq  \begin{cases}
\sigma_{\bowtie}(g,\beta') \sigma_{\Gr}(g|_{\beta'},h) & \text{if } \gamma = \beta \beta'\\
\sigma_{\Gr}(g, (h^{-1}|_{\gamma'})^{-1}) \sigma_{\bowtie} (h,h^{-1} \gamma')  & \text {if } \beta = \gamma \gamma' 
\end{cases}
\]
defines a normalised $2$-cocycle $\omega_{\sigma}=\{\omega_{\sigma}(s,t)\}_{s,t\in S(\Gr,E), st\neq 0}$ on $S(\Gr,E)$.
\end{proposition}
\begin{proof}
Normalisation follows from the fact that $\sigma_G$ is normalised and $\sigma_{\bowtie}(\rg(\alpha),\alpha) = 1$ for all $\alpha \in E^*$.
Fix $r = (\alpha,g,\beta)$, $s = (\gamma,h,\delta)$, and $t = (\zeta, k ,\eta)$ in $S(\Gr,E)$, so that $rst \ne 0$. Let $ M \coloneqq  \omega_{\sigma}(s,t) \overline{\omega_{\sigma}(rs,t)} \omega_{\sigma} (r,st) \overline{\omega_{\sigma}(r,s)}$. It suffices to show that $M = 1$. 	
Consider the case that $\gamma = \beta \beta'$, so $rs = (\alpha (g \beta'), g|_{\beta'} h, \delta)$. First, suppose that $\zeta = \delta \delta'$, so $st = (\gamma(h \delta'), h|_{\delta'}k, \eta)$. Then
\begin{align*}
M &= \sigma_{\bowtie}(h,\delta')
\sigma_G (h|_{\delta'}, k) 
\ol{\sigma_{\bowtie}(g|_{\beta'}h, \delta')} 
\ol{\sigma_G((g|_{\beta'} h)|_{\delta'},k)}\\
&\qquad \times 
\sigma_{\bowtie} (g, \beta'(h \delta'))
\sigma_{G} (g|_{\beta'(h \delta')}, h|_{\delta'}k ) 
\ol{\sigma_{\bowtie}(g,\beta')}
\ol{\sigma_G(g|_{\beta'},h)}.
\end{align*}
Since $\sigma_G$ is a $2$-cocycle on $\Gr$,
\[
\sigma_G(h|_{\delta'}, k) \ol{\sigma_G(g|_{\beta'(h\delta')} h|_{\delta'}, k)}
\sigma_G(g|_{\beta(h\delta')}, h|_{\delta'}k) = \sigma_G(g|_{\beta' (h\delta')}, h|_{\delta}'),
\]	
and as $(g|_{\beta'} h)|_{\delta'} = g|_{\beta' (h \delta')} h|_{\delta'}$, we have
\[
M = \sigma_{\bowtie}(h,\delta')\ol{\sigma_{\bowtie}(g|_{\beta'}h, \delta')} 
\sigma_{\bowtie}(g, \beta'(h \delta')) \ol{\sigma_{\bowtie}(g,\beta')}
\ol{\sigma_G(g|_{\beta'},h)}
\sigma_G(g|_{\beta' (h\delta')}, h|_{\delta}').
\]
By Lemma~\ref{lem:self-similar-cocycle-is-total-cocycle}~\ref{itm:ssc_bowtie}, $\sigma_{\bowtie} (g, \beta'(h \delta')) = \sigma_{\bowtie} (g,\beta') \sigma_{\bowtie} (g|_{\beta'}, h\delta')$, so
\[
M = \sigma_{\bowtie}(h,\delta') 
\ol{\sigma_{\bowtie}(g|_{\beta'}h, \delta')} 
\sigma_{\bowtie} (g|_{\beta'}, h\delta')\,
\ol{\sigma_G(g|_{\beta'},h)}
\sigma_G(g|_{\beta' (h\delta')}, h|_{\delta}') = 1
\]
with the final equality  given by Lemma~\ref{lem:self-similar-cocycle-is-total-cocycle}~~\ref{itm:ssc_G}.
Now suppose that $\delta = \zeta \zeta'$, so that $st = (\gamma, h(k^{-1}|_{\zeta'})^{-1}, \eta (k^{-1} \zeta'))$. Then
\begin{align*}
M &= 	
\sigma_{\Gr}(h, (k^{-1}|_{\zeta'})^{-1}) 
\sigma_{\bowtie} (k,k^{-1} \zeta')   
\ol{\sigma_{\Gr} (g|_{\beta'} h, (k^{-1}|_{\zeta'})^{-1})}
\ol{\sigma_{\bowtie} (k, k^{-1}\zeta' )}\\
&\qquad \times  
\sigma_{\bowtie} (g, \beta') 
\sigma_{\Gr}(g|_{\beta'}, h(k^{-1}|_{\zeta'})^{-1})
\ol{\sigma_{\bowtie}(g,\beta') }
\ol{\sigma_G(g|_{\beta'},h)}\\
&= 	\sigma_{\Gr}(h, (k^{-1}|_{\zeta'})^{-1}) 
\ol{\sigma_{\Gr} (g|_{\beta'} h, (k^{-1}|_{\zeta'})^{-1})}
\sigma_{\Gr}(g|_{\beta'}, h(k^{-1}|_{\zeta'})^{-1})
\ol{\sigma_G(g|_{\beta'},h)}.
\end{align*}
Since $\sigma_{\Gr}$ is a $2$-cocycle on $\Gr$,
\[
\ol{\sigma_{\Gr} (g|_{\beta'} h, (k^{-1}|_{\zeta'})^{-1})}
\sigma_{\Gr}(g|_{\beta'}, h(k^{-1}|_{\zeta'})^{-1})
\ol{\sigma_G(g|_{\beta'},h)} = \ol{\sigma_{\Gr}(h, (k^{-1}|_{\zeta'})^{-1})},
\]
and so $M = 1$. 

If $\beta = \gamma \gamma'$, then similar casewise arguments to the above show that $M =1$. 
So $\omega_{\sigma}$ is a $2$-cocycle on $S(\Gr,E)$. 
\end{proof}

Treating $\widetilde{\Gg}(\Gr,E)$ and $\Gg(\Gr,E)$ as transformation groupoids, via Proposition~\ref{prop:transformation_groupoids_self_similar_0}
 and \eqref{eq:transformation_groupoids_self_similar}, 
$\sigma$ induces, through  $\omega_{\sigma}$, see  Definition~\ref{defn:groupoid_twists_from_semigroup_twists},
 twists  on  $\widetilde{\Gg}(\Gr,E)$  and $\Gg(\Gr,E)$ that we denote by $\LL_{\sigma}$ (formally the twist on $\Gg(\Gr,E)$ is the restriction of that on $\widetilde{\Gg}(\Gr,E)$).

\begin{remark}
Let $\sigma = (\sigma_G,\sigma_{\bowtie})$ be a 2-cocycle for a self-similar action $(G,E)$, and let $\omega_{\sigma}$ be the associated  $2$-cocycle of 
Proposition~\ref{prop:twist_self_similar_inverse_semigroup}.
For all $(\alpha,g,\beta), (\gamma,h,\delta) \in S(\Gr,E)$ we have
\[
\omega_{\sigma}(f_{\alpha}, (\gamma,h,\delta)) = \begin{cases}
\sigma_{\bowtie}(h, h^{-1} \alpha') & \text{ if } \alpha = \gamma \gamma' \\
1 & \text{otherwise}
\end{cases}
\]	
and 
\[
\omega_{\sigma}((\alpha,g,\beta), f_{\gamma}) = 
\begin{cases}
\sigma_{\bowtie}(g,\beta') & \text{ if } \gamma = \beta \beta' \\
1 & \text{ otherwise}.	
\end{cases}
\]
Hence, the normalisation condition of Sieben \cite{Sieben} is  usually not satisfied,  which means that on the nose the associated groupoid twist $\LL_{\sigma}$ will look like a Kumjian twist rather than a twist  by a groupoid $2$-cocycle, cf. Remark~\ref{twists_over_transformation_groupoids}.
Since the universal and tight groupoids are ample, one may hope to turn it into cocycle by showing $\LL_{\sigma}$ is topologically trivial, cf. Remark  \ref{rem:2-cocycle_groupoid}. But in general $\LL_{\sigma}$ is topologically nontrivial even when the groupoid cocycle $\sigma_{G}$ is trivial.
\end{remark}

\begin{example}
Let $E$ be the graph consisting of a single vertex $v$ and a single edge $e$. Let $\Gr = \Z$ and define a self-similar action by $ne = e$ and $n|_{e} = n$ for all $n \in \Z$. Fix $z \in \T$ and put
\[
\sigma_{\bowtie} (n,\mu) = z^{n \cdot |\mu|} \quad \text{and} \quad \sigma_{\Gr}(n,m) = 1, \qquad n,m \in \Z,\, \mu \in E^*. 
\]
Then $\sigma  \coloneqq  (\sigma_{\bowtie},\sigma_{\Gr})$ is a $2$-cocycle for $(G,E)$. 
Since $\sigma_G = 1$, using the total homology of the double complex of \cite{Mundey_Sims}, $\sigma$ is homologous to the trivial $2$-cocycle if and only if there exist $f \colon \Gr \to \T$ and $g \colon E^* \to \T$ such that for all $n \in \Z$ and $\mu \in E^*$,
$
\sigma_{\bowtie}(n,\mu) = g(\mu)\ol{g(n \mu)}f(n)\ol{f(n|_{\mu})}. 
$
Since $n\mu = \mu$ and $\mu|_{n} = n$, this is equivalent to $\sigma_{\bowtie}(n,\mu) = 1$.	 
So for $z \ne 1$ the cocycle $\sigma$ is nontrivial. 
The corresponding $2$-cocycle on $S(G,E)$ is given by
\[
\omega_{\sigma}((\alpha,n,\beta),(\gamma,m,\delta)) = 
\begin{cases}
z^{n(|\gamma| - |\beta|)}	& \text{if } |\gamma| \ge |\beta|\\
z^{m(|\beta| - |\gamma|)}	& \text{if } |\beta| \ge |\gamma|,
\end{cases}
\]
The associated tight groupoid is isomorphic to the group $\Z^2$, via the map $[\alpha,n,\beta;e^{\infty}]\mapsto (|\alpha|-|\beta|, n)$, 
and so the induced twist must be topologically trivial. In fact the induced twist is given by the cocycle $\LL_{\sigma}((k,l),(n,m))= z^{ln}$ and  $(\Z^2, \LL_{\sigma})$ is 
the standard twisted group model for rotation algebras.
\end{example}

\begin{example}\label{ex:topologically_nontrivial_twist_from_self-similar}
Let $E$ be the directed graph 
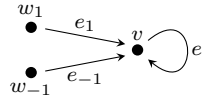
\[
\begin{tikzpicture}
[baseline=-0.25ex,
vertex/.style={
circle,
fill=black,
inner sep=1.5pt
},
edge/.style={
-stealth,
shorten >= 3pt,
shorten <= 3pt
},
scale =1]
\clip  (-.9,-0.9) rectangle (2,0.8);

\node[vertex] (w) at (-0.4,0.3) {};%
\node[vertex] (v) at (1,0) {};%
\node[vertex] (u) at (-0.4,-0.3) {};%

\node[anchor= south] at (w) {\scriptsize{$w_1$}};
\node[anchor= south] at (v) {\scriptsize{$v$}};
\node[anchor= north] at (u) {\scriptsize{$w_{-1}$}};

\draw[edge] (w.east) to node[anchor=south, inner sep = 2pt]{\scriptsize{$e_1$}} (v.north west);
\draw[edge] (u.east) to node[anchor=north, inner sep = 3pt]{\scriptsize{$e_{-1}$}} (v.south west);
\draw[edge] (v.north east) to [out = 40, in = -40,loop,min distance=10mm] node[anchor=west, inner sep = 2pt]{\scriptsize{$e$}} (v.south east);
\end{tikzpicture}
\]
and  let $\Gr$  the group bundle over $E^0=\{v,w_{\pm1}\}$ with fibres  $G_v  \coloneqq \Z_2=\{0,1\}$, $G_{w_{\pm 1}}  \coloneqq  \{0\}$. 
We define a self-similar action of $\Gr$ on $E$ by the relations 
\[
1\cdot e= e,\qquad 1|_{e}=1,  \qquad		1 \cdot e_{\pm 1}= e_{\pm},\qquad 1|_{e_{\pm 1}}=0.
\]
We  equip it with the twist $\sigma = (\sigma_{\Gr},\sigma_{\bowtie})$ where the only nontrivial value
is $\sigma_{\bowtie}(1,e_{-1})=-1$.
The associated twisted groupoid $(\Gg(\Gr,E), \LL_{\sigma})$ is isomorphic to the one considered in Example~\ref{ex:nontrivial_twist}, and
hence, $\LL_{\sigma}$ is topologically nontrivial.

More specifically, the unit space  $\partial E= \{e^\infty\}\cup \{e^n e_{\pm 1}: n\in \N\} \cup \{w_{\pm 1}\}$
is homeomorphic to $X=\{\frac{1}{n}:n\in \Z\}\cup \{0\}$, via the homeomorphism given by  $w_{\pm 1}\mapsto \pm 1$, $e^n e_{\pm 1} \mapsto \pm \frac{1}{n}$
and $e^{\infty}\mapsto 0$. Moreover,  $\star \coloneqq [v,1,v; e^{\infty}]$ is the only nontrivial element in $\Gg(\Gr,E)$
and $U=\{\star\} \cup \{e^n e_{\pm 1}: n\in \N\} \cup \{w_{\pm 1}\}$ is the largest bisection containing $\star$.
The only nontrivial values of the extended map $\sigma_{\bowtie} \colon \Gr * E^* \to \T$
are $\sigma_{\bowtie}(1,e^n e_{-1})=-1$ for $n\geq 0$.   
Recalling the equivalence relation on page~\pageref{page:germ_relations}, for $a\in C(U)$ and $a'\in C(\partial E)$ 
we have $[a,(v,1,v),e^n e_{\pm1}]= [a,(v,v,v),e^n e_{\pm 1}]$ if and only if $a(e^n e_{\pm 1})=\pm a'(e^n e_{\pm 1})$.
\end{example}

\section{Twisted \texorpdfstring{$L^P$}{LP}-operator algebras associated to self-similar actions}
\label{sect:self-similar_LP-operator}
\subsection{Universal \texorpdfstring{$L^P$}{LP}-algebras}
Let $(\Gr,E, \sigma)$ be a fixed twisted self-similar groupoid action. 
We start by introducing  representations of the graph $E = (E^0,E^1,\rg,\sr)$ on $L^p$-spaces.  
\begin{definition} 
Let $Y$ be an $L^p$-space for some $p\in[1,\infty]$. 
A \emph{Cuntz--Krieger $E$-family}  in  $Y$ is a pair $(W,T)$ where $W=\{W_v\}_{v\in E^0}\subseteq \Bound(Y)$ consists of pairwise orthogonal hermitian idempotents, and $T=\{T_e\}_{e\in E^1}\subseteq \Bound(Y)$ consists of  Moore-Penrose partial isometries with mutually orthogonal range projections that in addition satisfy 
\begin{enumerate}[label={(CK\arabic*)}]
\item\label{defn:Cuntz--Krieger1} $T_e^*T_e=W_{\sr(e)}$ and $T_e T_e^*\leq W_{\rg(e)}$ for all $e\in E^{1}$;
\item\label{defn:Cuntz--Krieger2} $W_v=\sum_{e \in \rg^{-1}(v)} T_e T_e^*$ for all $v\in E^{0}_{\reg}$.
\end{enumerate}
Here $T_e^*$ denotes the (unique) Moore-Penrose generalised inverse of $T_e$. A pair $(W,T)$  as above but not necessarily satisfying \ref{defn:Cuntz--Krieger2}  is called an \emph{$E$-family}.
\end{definition}
For each $p\in [1,\infty]$  a \emph{graph $L^p$-operator algebra} $F^p(E)$  is a universal for Cuntz--Krieger covariant $E$-families on $L^p$-spaces. It was studied in
\cites{cortinas_rodrogiez, cortinas_Montero_rodrogiez}, \cite{BKM}*{Subsection 7.3}, and \cite{BKM2}*{Subsection 7.4} where, in particular, simplicity and pure infiniteness criteria were established.
We generalise these results to twisted self-similar groupoid actions, while also considering Toeplitz versions of these algebras. 
We start by extending \cite{Mundey_Sims}*{Definition~7.14} (formulated for self-similar  actions on row-finite $k$-graphs)
from Hilbert spaces to $L^p$-spaces.
\begin{definition}\label{defn:reps_of_twisted_self_similar}
Let $\sigma = (\sigma_G,\sigma_{\bowtie})$ be a twist of a self-similar action $(G,E)$. 
A \emph{$\sigma$-twisted representation} of  $(\Gr,E)$ (or a \emph{representation} of  $(\Gr,E,\sigma)$) on an $L^p$-space $Y$ is a pair $(W,T)$ of maps  $W \colon \Gr \to \Bound(Y)_1$ and   $T\colon E^1\to \Bound(Y)_1$ into contractive operators such that
\begin{enumerate}[label={(EP\arabic*)}]
\item\label{enu:reps_of_twisted_self_similar1}  $W_g W_h = \sigma_{G}(g,h)W_{gh}$ for all $(g,h) \in \Gg^2$;
\item\label{enu:reps_of_twisted_self_similar2}  $\{W_v \mid v \in E^0\} \cup \{T_e \mid e \in E^1\}$ is an $E$-family on $Y$;
\item\label{enu:reps_of_twisted_self_similar3} 	$ W_g T_e = \sigma_{\bowtie}(g,e)T_{g e} W_{g|_e}$ for all $(g,e) \in \Gr \fibre{\sr}{\rg} E^1$.
\end{enumerate}
We say that  such a representation $(W,T)$ is \emph{(Cuntz--Krieger) covariant} if the 
$E$-family in \ref{enu:reps_of_twisted_self_similar2} is Cuntz--Krieger  (satisfies \ref{defn:Cuntz--Krieger2}).
By a \emph{Banach algebra generated by $(W,T)$} we mean the Banach subalgebra of $\Bound(Y)$  generated by $T_e$, $T_e^*$, $W_g$, for $e\in E^1$, $g\in \Gr$, 
and we denote it by $B(W,T)$.
\end{definition}
\begin{remark}\label{rem:reps_of_twisted_self_similar}
Let $(W,T)$ be a $\sigma$-twisted representation of  $(\Gr,E)$ on an $L^p$-space $Y$. Note that $(W,T)$ take values in the set $\MPIso(Y)$ of Moore-Penrose partial isometries on $Y$.
If $p\neq 2$, then  $\MPIso(Y)$ is an inverse semigroup, see  Proposition~\ref{prop:inverse_semigroup_of_partial_isos}, and   for  $e\in E^1$,  $T_e^*$ is the unique generalised inverse of $T_e$ in $\MPIso(Y)$, and for   $g\in \Gr$ the unique generalised inverse of $W_{g}$ in $\MPIso(Y)$ is given by
\begin{equation}\label{eq:inverse_of_W_g}
W_{g}^*=\ol{\sigma_{G}(g^{-1},g)}W_{g^{-1}}.
\end{equation}
For $p=2$, $\MPIso(Y)$ is the set of standard  partial isometries on the Hilbert space $Y$, and so $\MPIso(Y)$ is not a semigroup.
However,  the semigroup of operators generated by $T_e$, $T_e^*$, $W_{g}$, $e\in E^1$, $g\in \Gr$,  is an inverse semigroup, where the involution coincides with the hermitian adjoint (in particular \eqref{eq:inverse_of_W_g} is still valid), cf. Theorem~\ref{thm:presentations_of_twisted_self_similar_algebras} below.  
\end{remark} 
\begin{notation}\label{notation:Cuntz--Krieger_semigroup}
Let now $(W,T)$ be a $\sigma$-twisted representation of $(\Gr,E)$. We put 
$
T_v \coloneqq T_v^* \coloneqq W_v$ for $v\in E^0$ and  $T_{\mu} \coloneqq T_{\mu_1} T_{\mu_2}\cdots T_{\mu_n}$ and $T_{\mu}^* \coloneqq T_{\mu_n}^* T_{\mu_{n-1}}^*\cdots T_{\mu_1}^*$  for $\mu=\mu_1\cdots \mu_n\in E^n$, $n\geq 1$. By Remark~\ref{rem:reps_of_twisted_self_similar}, when $p\neq 2$ 
then operators $T_\alpha$, $T_\alpha^*$, $W_{g}$, $\alpha\in E^*$, $g\in \Gr$, belong to  the inverse semigroup $\MPIso(Y)$ and  $T_{\alpha},T_{\alpha}^*$ are  conjugate Moore-Penrose partial isometries for every $\alpha\in E^*$.
For $p=2$ the latter is also true, as then $T_{\alpha}^*$ is a hermitian adjoint of $T_{\alpha}$.
\end{notation}
\begin{definition}\label{defn:universal_self_similar_algebras}
Let  $P\subseteq [1,\infty]$ be non-empty.  The (universal) \emph{$L^P$-operator algebra of $(\Gr,E,\sigma)$},  denoted by  $\OO^P(\Gr,E,\sigma)$, is  the  Banach algebra  generated by  a $\sigma$-twisted covariant representation $(W^P,T^P)$ on an $\ell^{\infty}$-direct sum of $L^p$-spaces for $p\in P$, that is universal: for any   $\sigma$-twisted covariant representation $(W,T)$ on  an $L^p$-space with $p\in P$,  the maps $W^P_g\mapsto  W_g$ for $g \in \Gr$, and $T^P_e\mapsto T_e$ and $(T^P_e)^*\mapsto T_e^*$ for $e \in E^1$, extend to a representation from
$\OO^P(\Gr,E,\sigma)$ to the Banach algebra generated by  $(W,T)$.

Similarly, we define the \emph{Toeplitz $L^P$-operator algebra of $(\Gr,E,\sigma)$}, denoted by $\TT^P(\Gr,E,\sigma)$, as  the  Banach algebra generated  by a
$\sigma$-twisted representation $(\tw^P,\tt^P)$ which is universal for all $\sigma$-twisted representations of   $(\Gr,E)$ on $L^p$-spaces for $p\in P$.  
\end{definition}

The universal Banach algebras defined above exist by the following result.

\begin{theorem}\label{thm:presentations_of_twisted_self_similar_algebras}
Let $(\Gr,E, \sigma)$  be a twisted  self-similar action, and let 
$\omega_{\sigma}$ and $\LL_\sigma$  be the associated twists of the inverse semigroup $S(\Gr,E)$ and the universal groupoid $\widetilde{\Gg}(\Gr,E)$, respectively, see Proposition  \ref{prop:twist_self_similar_inverse_semigroup}. 
For any $L^p$-space $Y$, $p\in[1,\infty]$, we have a bijective correspondence between representations $(W,T)$ of $(\Gr,E, \sigma)$ and representations $v$ of $(S(\Gr,E),\omega_{\sigma})$ on $Y$. Moreover, $(W,T)$ is   covariant if and only if $v$ is covariant. 
Thus, for any non-empty $P\subseteq [1,\infty]$ we
have natural isometric isomorphisms 
$$
\TT^P(\Gr,E,\sigma)\cong \TT^P(S(\Gr,E),\omega_{\sigma})\cong F^P(\widetilde{\Gg}(\Gr,E),\LL_{\sigma}),
$$  
$$
\OO^P(\Gr,E,\sigma)\cong \OO^P(S(\Gr,E),\omega_{\sigma})\cong F^P(\Gg(\Gr,E),\LL_{\sigma}).
$$ 
\end{theorem}
\begin{proof} Recall that we have the embeddings $S(E)\ni (\alpha,\beta)\mapsto (\alpha,\sr(\alpha),\beta)\in S(\Gr,E)$ and
$\Gr\ni g\mapsto (\rg(g),g,\sr(g))\in S(\Gr,E)$.
Let $v:S(\Gr,E)\to \Bound(Y)_1$  be a representation of $(S(\Gr,E),\omega_{\sigma})$. Put
\[
W_{g} \coloneqq v_{(\rg(g),g,\sr(g))},\qquad T_{e} \coloneqq v_{(e,\sr(e),\sr(e))}, \qquad g\in \Gr,\, e\in E^1.
\]
For  $(g,h) \in \Gg^2$, using condition \ref{enu:semigroup_representation1} of Definition~\ref{def:semigroup_representations} and  that $	\sigma_{\bowtie}(g,\sr(g))=1$,  we have
\begin{align*}
W_{g}W_{h}&=v_{(\rg(g),g,\sr(g))} v_{(\rg(h),h,\sr(h))}=\omega_{\sigma}((\rg(g),g,\sr(g)),(\rg(h),h,\sr(h))) v_{(\rg(g),g,\sr(g))(\rg(h),h,\sr(h))}
\\
&=
\sigma_{\bowtie}(g,\sr(g)) \sigma_{\Gr}(g|_{\sr(g)},h) v_{(\rg(g),gh,\sr(h))}
=\sigma_{G}(g,h)W_{gh}.
\end{align*}
Similarly, for  $(g,e) \in E^1 \fibre{\sr}{\rg} \Gr$, using that $\sigma_{G}(g|_{e},\sr(e))=1$ we get
\[
W_g T_e 
=\omega_{\sigma}((\rg(g),g,\sr(g)),(e,\sr(e),\sr(e))) v_{(ge,g|_{e},\sr(e))}
=\sigma_{\bowtie}(g,e)v_{(ge,g|_{e},\sr(e))},
\]
while using that  $\sigma_{\bowtie}(\sr(ge),\sr(ge))=\sigma_{G}(\sr(ge),g|_e)=1$ and $\sr(e)=\sr(g|_e)$ we get
$$
T_{ge}\cdot W_{g|_e}=\omega_{\sigma}((ge,\sr(ge),\sr(ge)),(\rg(g|_e),g|_e,\sr(g|_e))) v_{(ge,g|_{e},\sr(g|_{e}))}=v_{(ge,g|_{e},\sr(e))}.
$$
Thus, conditions \ref{enu:reps_of_twisted_self_similar1}, \ref{enu:reps_of_twisted_self_similar3} in Definition~\ref{defn:reps_of_twisted_self_similar} hold.
For $e\in E^{1}$  put  $T^*_e \coloneqq v_{(\sr(e),\sr(e),e)}$ and note that 
\[
T_e^*T_e=v_{(\sr(e),\sr(e),\sr(e))}=W_{\sr(e)}, \qquad T_eT_e^*=v_{(e,\sr(e),e)}\leq v_{(\rg(e),\rg(e),\rg(e))}=W_{\rg(e)}.
\]
This, and  condition \ref{enu:semigroup_representation2} in Definition~\ref{def:semigroup_representations}, implies that  $T_e^*T_e$
and $T_eT_e^*$ are hermitian idempotents. It also implies that $T_e^*$ is a generalised inverse of $T_{e}$. 
Hence, $T_e^*$ is the unique Moore-Penrose generalised inverse of $T_e$ and so  $\{W_v \mid v \in E^0\} \cup \{T_e \mid e \in E^1\}$ is an $E$-family.
Thus, $(W,T)$ is a representation of $(\Gr,E, \sigma)$. It is known, see for instance   \cite{BKM}*{Theorem 6.24}, that $\{W_v \mid v \in E^0\} \cup \{T_e \mid e \in E^1\}$ is a Cuntz--Krieger $E$-family
if and only if $\EE(S(E))\cong E^*\cup \{0\}\ni \mu\mapsto v_{(\rg(\mu),\mu, \sr(\mu))}$ is a tight representation. 
Since $\EE(\Gr,E)\cong \EE(S(E))$, this in turn is  equivalent to $v|_{\EE(\Gr,E)}$ being tight. Hence, $(W,T)$ is covariant if and only if $v$ is covariant.

Now, let $(W,T)$ be a representation of $(\Gr,E, \sigma)$ on $Y$.
Extend the map $T$  to paths as described in Notation \ref{notation:Cuntz--Krieger_semigroup}.
It is routine to check that 
\[ 
T_{\beta}^* T_{\gamma} \coloneqq 
\begin{cases} T_{\beta'} & \text{if } \gamma=\beta\beta'
\\
T_{\gamma'}^* & \text{if } \beta=\gamma\gamma'
\\
0 & \text{otherwise},
\end{cases}
\] 
and  the map $S(E)\ni (\alpha,\beta)\mapsto T_{\alpha}T_{\beta}^*\in \Bound(Y)_1$ is a semigroup homomorphism (a  representation of $S(E)$ on $Y$).
We claim that \ref{enu:reps_of_twisted_self_similar3} in Definition~\ref{defn:reps_of_twisted_self_similar} 
generalises to paths of arbitrary length $n\in \N$, so that  
\begin{equation}\label{equ:reps_of_twisted_self_similar3}
W_g T_\alpha =\sigma_{\bowtie}(g,\alpha)T_{g\alpha} W_{g|_{\alpha}}\qquad \text{ for all } (g,\alpha) \in  \Gr\fibre{\sr}{\rg}E^n.
\end{equation}
Indeed, assume that \eqref{equ:reps_of_twisted_self_similar3} holds for some  $n$ and let $(g,\alpha) \in  \Gr\fibre{\sr}{\rg}E^{n+1}$. Write $\alpha=e\alpha'$ for $(e,\alpha')\in E^1\fibre{\sr}{\rg}E^n$.
Then
\begin{align*}
W_g T_\alpha &=W_g T_{e}T_{\alpha'} \stackrel{
\ref{enu:reps_of_twisted_self_similar3}}{=} 
\sigma_{\bowtie}(g,e)T_{g e} W_{g|_e}T_{\alpha'}\stackrel{\eqref{equ:reps_of_twisted_self_similar3}}{=}
\sigma_{\bowtie}(g,e)\sigma_{\bowtie}(g|_e,\alpha')T_{g e}T_{\alpha'}W_{g|_e|\alpha'}
\\
&\stackrel{\ref{lem:self-similar-cocycle-is-total-cocycle}\ref{itm:ssc_bowtie}}{=}
\sigma_{\bowtie}(g,e\alpha') T_{ge\alpha'} W_{g|_e|\alpha'}
\stackrel{\eqref{eq:recursive_extension_rules}}{=}
\sigma_{\bowtie}(g,\alpha)T_{g\alpha} W_{g|_{\alpha}}.
\end{align*}
Hence, \eqref{equ:reps_of_twisted_self_similar3}  holds by induction.
By passing in \eqref{equ:reps_of_twisted_self_similar3} to adjoints, either in the inverse semigroup
$\MPIso(Y)$ when $p\neq 2$ or in the hermitian sense when $p=2$,
and using \eqref{eq:inverse_of_W_g}
we get 
\begin{equation}\label{equ:reps_of_twisted_self_similar3*}
\sigma_{\bowtie}(g,\alpha) \sigma_{G}(g|_\alpha^{-1}, g|_\alpha) T_\alpha^*  W_{g^{-1}} =\sigma_G(g^{-1},g) W_{g|_\alpha^{-1}} T_{g\alpha}^*
\quad \text{ for all } (g,\alpha) \in  \Gr\fibre{\sr}{\rg}E^n.
\end{equation}
For $(\alpha,g,\beta)\in S(\Gr,E)$ we put 
\[
v_{(\alpha,g,\beta)} \coloneqq T_{\alpha} W_{g} T_{\beta}^*.
\]
Now let $s=(\alpha,g,\beta), t=(\gamma,h,\delta)\in S(\Gr,E)$ and consider three cases.


(1). Assume $\gamma=\beta\beta'$, so that $T_{\beta}^* T_{\gamma}=
T_{\beta'}$. Then
\begin{align*}
v_{s}\cdot v_{t}
&=T_{\alpha} W_{g} (T_{\beta}^* T_{\gamma}) W_{h} T_{\delta}^*
=T_{\alpha} (W_{g} T_{\beta'}) W_{h} T_{\delta}^* 
\\
&\stackrel{\eqref{equ:reps_of_twisted_self_similar3}}{=}\sigma_{\bowtie}(g,\beta')T_{\alpha} T_{g\beta'} W_{g|_{\beta'}}W_{h} T_{\delta}^*
\stackrel{
\ref{enu:reps_of_twisted_self_similar1}}{=}
\sigma_{\bowtie}(g,\beta')\sigma_{G}(g|_{\beta'},h) T_{\alpha g\beta'} W_{g|_{\beta'}h} T_{\delta}^*
\\
&
\stackrel{\ref{prop:twist_self_similar_inverse_semigroup}}{=}\omega_{\sigma}((\alpha,g,\beta), (\gamma,h,\delta)) v_{(\alpha (g\beta'), g|_{\beta'}h, \delta)}
=\omega_{\sigma}(s,t) v_{st}.
\end{align*}

(2). Assume that $\beta=\gamma\gamma'$, so that $T_{\beta}^* T_{\gamma}=T_{\gamma'}^*$. Then
\begin{align*}
v_{s}v_{t}
&=T_{\alpha} W_{g} (T_{\beta}^* T_{\gamma}) W_{h} T_{\delta}^*=T_{\alpha} W_{g} (T_{\gamma'}^* W_{h}) T_{\delta}^* 
\\
&
\stackrel{\eqref{equ:reps_of_twisted_self_similar3*}}{=}
\ol{\sigma_{\bowtie}(h^{-1},\gamma')}
\sigma_{G}(h,h^{-1})\ol{\sigma_G(h^{-1}|_{\gamma'}^{-1},h^{-1}|_{\gamma'})}
T_{\alpha} W_{g} W_{h^{-1}|_{\gamma'}^{-1}} T_{h^{-1}\gamma'}^* T_{\delta}^*
\\
&\stackrel{
\ref{enu:reps_of_twisted_self_similar1}}{=}
\ol{\sigma_{\bowtie}(h^{-1},\gamma')}
\sigma_{G}(h,h^{-1})\ol{\sigma_G(h^{-1}|_{\gamma'}^{-1},h^{-1}|_{\gamma'})}
\sigma_G(g, h^{-1}|_{\gamma'}^{-1})
T_{\alpha}W_{g (h^{-1}|_{\gamma'})^{-1}} T_{\delta (h^{-1}\gamma')}^*
\\
&\stackrel{
\ref{enu:reps_of_twisted_self_similar1}}{=}
\ol{\sigma_{\bowtie}(h^{-1},\gamma')}
\sigma_{G}(h,h^{-1})\ol{\sigma_G(h^{-1}|_{\gamma'}^{-1},h^{-1}|_{\gamma'})}
\sigma_G(g, h^{-1}|_{\gamma'}^{-1})
v_{(\alpha,g (h^{-1}|_{\gamma'})^{-1}, \delta (h^{-1}\gamma'))}
\\
&\stackrel{(*)}{=}\sigma_{\Gr}(g, (h^{-1}|_{\gamma'})^{-1}) \sigma_{\bowtie} (h,h^{-1} \gamma')v_{st}=\omega_{\sigma}(s,t)v_{st},
\end{align*}
where
$$
\sigma_{\bowtie} (h,h^{-1} \gamma')=\ol{\sigma_{\bowtie}(h^{-1},\gamma')}
\sigma_{G}(h,h^{-1})\ol{\sigma_G(h^{-1}|_{\gamma'}^{-1},h^{-1}|_{\gamma'})}.
\eqno{(*)}
$$
To see  $(*)$ note that by  Lemma~\ref{lem:self-similar-cocycle-is-total-cocycle}\ref{itm:ssc_G} we have
\[
\sigma_{\bowtie}(h^{-1},\gamma') \sigma_{\bowtie} (h,h^{-1} \gamma') =\sigma_{\bowtie} (h,h^{-1} \gamma')\ol{\sigma_{G}(h,h^{-1})}\sigma_{\bowtie}(h^{-1},\gamma')
=\ol{\sigma_G(h|_{h^{-1}\gamma'},h^{-1}|_{\gamma'})}\sigma_{G}(h,h^{-1}),
\]
and we have $h|_{h^{-1}\gamma'}=h|_{\gamma'}^{-1}$ by \eqref{equ:inverse_vs_restriction}.

(3). If  $\gamma$ and $\beta$ are incomparable, then $st=0$, $T_{\beta}^* T_{\gamma}=0$ and $v_s v_t=0$. 

Hence, $v$ is a representation of $(S(\Gr,E),\omega_{\sigma})$ and this finishes the proof of the first part of the assertion.
This readily gives the isometric isomorphisms $\TT^P(\Gr,E,\sigma)\cong \TT^P(S(\Gr,E),\omega_{\sigma})$ and $\OO^P(\Gr,E,\sigma)\cong \OO^P(S(\Gr,E),\omega_{\sigma})$.
Isomorphisms $\TT^P(S(\Gr,E),\omega_{\sigma})\cong F^P(\widetilde{\Gg}(\Gr,E),\LL_{\sigma})$ and $\OO^P(S(\Gr,E),\omega_{\sigma})\cong F^P(\Gg(\Gr,E),\LL_{\sigma})$ hold by Corollary~\ref{cor:groupoid_presentation_inverse_semigroup_algebras}. 
\end{proof}

Recall the subsemigroups $S_{00}(\Gr,E)\subseteq S_0(\Gr,E)\subseteq S(\Gr,E)$, see page~\pageref{eq:core_subsemigroup}.
\begin{corollary}\label{cor:subalgebras}
Let $(W,T)$ be  $\sigma$-twisted representation of  $(\Gr,E)$ on an $L^p$-space (or a direct sum of these). The Banach algebra $B(W,T)$ generated by $(W,T)$ is 
$$
B(W,T)=\clsp\{T_{\alpha}W_g T_\beta^*: (\alpha,g,\beta) \in S(\Gr,E)\},
$$
and it  contains the following Banach subalgebras
\begin{align*}
B(W,T)_0& \coloneqq \clsp\{T_{\alpha}W_g T_\beta^*: (\alpha,g,\beta) \in S_0(\Gr,E)\},
\\
B(W,T)_{00}& \coloneqq \clsp\{W_{g} T_{\beta}T_\beta^*: (g,\beta)\in \Gr *E^*\},
\\
B(T)& \coloneqq \clsp\{T_{\alpha}T_\beta^*: (\alpha,\beta) \in S(E)\}, 
\\
B(W)& \coloneqq \clsp\{W_g : g\in \Gr\},
\end{align*}
as well as the commutative $C^*$-algebra
$
D(T) \coloneqq \clsp\{T_{\beta}T_\beta^*: \beta\in E^*\}
$, which is a quotient of the algebra $\Cont_0( E^{\leq \infty})$. Moreover,
$$
D(T),B(W)\subseteq B(W,T)_{00}\subseteq B(W,T)_0\subseteq B(W,T).
$$
\end{corollary}
\begin{proof} By Theorem~\ref{thm:presentations_of_twisted_self_similar_algebras}, we may extend $(W,T)$ to a representation $v$ of $(S,\omega_{\sigma})$.  
Then the described linear spaces are Banach algebras generated by the range of $v$ and its restrictions to inverse subsemigroups of $S(\Gr,E)$; namely, $S_{0}(\Gr,E)$, $S_{00}(\Gr,E)$, ${S(E)}$,  ${\Gr \cup \{0\}}$ and $\EE(\Gr,E)$, cf. Section~\ref{sect:inverse_semigroup}. Also, $D(T)$ is the image of the  subalgebra $\Cont_0( E^{\leq \infty})\subseteq F^p(\widetilde{\Gg}(\Gr,E),\LL_{\sigma})\cong \TT^p(\Gr,E,\sigma)$. 
\end{proof}

\subsection{Reduced and essential algebras and their subalgebras}\label{sec:representations_of_twisted_self_similar_actions}
We will now examine in more detail algebras mentioned in Corollary~\ref{cor:subalgebras} in the case when they come from  regular and essential representations.
Let $\widetilde{\Gg}(\Gr,E)= \{[\alpha,g,\beta; \xi] \colon (\alpha,g,\beta) \in S(\Gr,E), \xi \in Z(\beta)\}$ be the groupoid described on page~\pageref{page:self-similar_groupoid}. Let $p\in [1,\infty]$.  Theorem~\ref{thm:presentations_of_twisted_self_similar_algebras} applied to the representation ${\tv}^{\rd,p} : S(\Gr,E) \to \Bound\big(\ell^p(\widetilde{\Gg}(\Gr,E))\big)$ described in Corollary~\ref{cor:groupoid_presentation_inverse_semigroup_algebras} yields  the representation $(\tw^{\rd,p},\ttt^{\rd,p})$ of $(\Gr,E, \sigma)$ 
on $\ell^p(\widetilde{\Gg}(\Gr,E))$ (for $p=\infty$ we may consider $c_0(\widetilde{\Gg}(\Gr,E))$. On the basic elements we have
\begin{align*}
\tw_{g}^{\rd,p} \bone_{[\alpha, h,\beta; \eta]}&=[\rg(\alpha)= \sr(g)]
\sigma_{\bowtie}(g,\alpha) \sigma_{\Gr}(g|_{\alpha},h)
\bone_{[g\alpha, g|_{\alpha} h,\beta; \eta]},
\\
\ttt_{e}^{\rd,p}\bone_{[\alpha, g,\beta; \eta]}&=[\sr(e)=\rg(\alpha)]\bone_{[e\alpha, g,\beta; \eta]}.
\end{align*}
Then the  Moore-Penrose generalised inverse of  $\ttt_{e}^{\rd,p}$ is given by
$$
(\ttt_{e}^{\rd,p})^*\bone_{[\alpha, g,\beta; \eta]}=[\alpha=e\alpha']\bone_{[\alpha', g,\beta; \eta]},
$$
where $[\alpha, g,\beta; \eta]\in \widetilde{\Gg}(\Gr,E)$, $e\in E^1$.
In particular,  we have
$$
B(\tw^{\rd,p},\ttt^{\rd,p})=\TT^p_{\red}(S(\Gr,E),\omega_{\sigma})\cong F^p_{\red}(\widetilde{\Gg}(\Gr,E),\LL_{\sigma}).
$$
The subspaces $\ell^p(\widetilde{\Gg}(\Gr,E)_{\Hau})$, $\ell^p(\Gg(\Gr,E))$, $\ell^p(\Gg(\Gr,E)_{\Hau})$ of $\ell^p(\widetilde{\Gg}(\Gr,E))$ 
are invariant for $(\tw^{\rd,p},\ttt^{\rd,p})$, and denoting by $(\tw^{\es,p},\ttt^{\es,p})$, $(\ww^{\rd,p},\tt^{\rd,p})$, $(\ww^{\es,p},\tt^{\es,p})$ the respective restrictions of
$(\tw^{\rd,p},\ttt^{\rd,p})$ we get representations that generate $\OO^p_{\ess}(S(\Gr,E),\omega_{\sigma})$, $\OO^p_{\red}(S(\Gr,E),\omega_{\sigma})$ and $\OO^p_{\ess}(S(\Gr,E),\omega_{\sigma})$, respectively. More generally, for any $\emptyset\neq P\subseteq [1,\infty]$ we denote by 
$(\tw^{\rd,P},\ttt^{\rd,P})$, $(\tw^{\es,P},\ttt^{\es,P})$, $(\ww^{\rd,P},\tt^{\rd,P})$ and $(\ww^{\es,P},\tt^{\es,P})$ direct sums of the corresponding representations   over all $p\in P$.

\begin{definition}\label{def:reduced_and_essential_self-similar-algebras} 
Let $(\Gr,E,\sigma)$ be a twisted self-similar action, and let 
$\omega_{\sigma}$  and $\LL_{\sigma}$ be the associated twist of  $S(\Gr,E)$ and $\widetilde{\Gg}(\Gr,E)$, see Proposition  \ref{prop:twist_self_similar_inverse_semigroup}. 
Let  $\emptyset\neq P\subseteq [1,\infty]$. We call 
$$
\TT^P_{\red}(\Gr,E,\sigma) \coloneqq  B( \tw^{\rd,P},\ttt^{\rd,P} )=\TT^P_{\red}(S(\Gr,E),\omega_{\sigma})\cong F^P_{\red}(\widetilde{\Gg}(\Gr,E),\LL_{\sigma}),
$$
$$
\TT^P_{\ess}(\Gr,E,\sigma) \coloneqq  B( \tw^{\es,P},\ttt^{\es,P} )=\TT^P_{\ess}(S(\Gr,E),\omega_{\sigma})\cong F^P_{\ess}(\widetilde{\Gg}(\Gr,E),\LL_{\sigma}),
$$
$$
\OO^P_{\red}(\Gr,E,\sigma) \coloneqq  B(  \ww^{\rd,P},\tt^{\rd,P} )=\OO^P_{\red}(S(\Gr,E),\omega_{\sigma})\cong F^P_{\red}(\Gg(\Gr,E),\LL_{\sigma}),
$$
$$
\OO^P_{\ess}(\Gr,E,\sigma) \coloneqq  B(  \ww^{\es,P},\tt^{\es,P})=\OO^P_{\ess}(S(\Gr,E),\omega_{\sigma})\cong F^P_{\ess}(\Gg(\Gr,E),\LL_{\sigma})
$$
the \emph{reduced Toeplitz}, the \emph{essential Toeplitz},
the \emph{reduced} and  the \emph{essential $L^P$-operator algebra} of $(\Gr,E,\sigma)$, respectively.
In addition, using the notation from Corollary~\ref{cor:subalgebras} and Definition~\ref{defn:universal_self_similar_algebras}, we define for $*=0,00$  the following \emph{core subalgebras}
$$
\OO^P(\Gr,E,\sigma)_{*} \coloneqq  B(\ww^{P}, \tt^{P})_{*},\qquad	\TT^P(\Gr,E,\sigma)_* \coloneqq  B( \tw^{P}, \ttt^{P}  )_*,
$$
$$
\OO^P_{\red}(\Gr,E,\sigma)_{*} \coloneqq  B( \ww^{\rd,P}, \tt^{\rd,P} )_{*},\qquad	\TT^P_{\red}(\Gr,E,\sigma)_* \coloneqq  B( \tw^{\rd,P}, \ttt^{\rd,P} )_*,
$$
$$
\OO^P_{\ess}(\Gr,E,\sigma)_{*} \coloneqq  B( \ww^{\es,P},\tt^{\es,P} )_{*},\qquad	\TT^P_{\ess}(\Gr,E,\sigma)_* \coloneqq  B( \tw^{\es,P}, \ttt^{\es,P} )_*.
$$
\end{definition}
\begin{remark}\label{rem:gauge_actions_on_reduced_essential} 
By Corollary~\ref{cor:gauge_action_inverse_semigroup}, for each $*=\Space\,,\red,\ess$, the algebras 
$\OO^P_{*}(\Gr,E,\sigma)$ and	$\TT^P_{*}(\Gr,E,\sigma)$ are equipped with gauge circle actions  coming from the length cocycle 
\eqref{eq:length_cocycle}.  Subalgebras $\OO^P_{*}(\Gr,E,\sigma)_0$ and	$\TT^P_{*}(\Gr,E,\sigma)_0$ are fixed-point
algebras for these actions.
The smaller core subalgebras $\OO^P_{*}(\Gr,E,\sigma)_{00}$ and $\TT^P_{*}(\Gr,E,\sigma)_{00}$ are generated by  diagonal subalgebras and 
representations of $(G,\sigma_G)$. They carry crucial information about Hausdorffness and amenability of the underlying groupoids.
\end{remark}
\begin{remark}\label{remembeddings_of_algebras}
By Lemma~\ref{lem:reduced_essential_passes_to_subgroupoids}, for each $*=0,00$ the core subalgebras $\OO^P_{\red}(\Gr,E,\sigma)_{*}$ and $\TT^P_{\red}(\Gr,E,\sigma)_*$ are reduced Banach algebras
of  
$(\Gg_{*}(\Gr,E), \LL_{\sigma})$  and $(\widetilde{\Gg}_{*}(\Gr,E), \LL_{\sigma})$, respectively. 
Similarly,
$\OO^P_{\ess}(\Gr,E,\sigma)_{*}$ and $\TT^P_{\ess}(\Gr,E,\sigma)_*$ are essential Banach algebras
of  
$(\Gg_{*}(\Gr,E), \LL_{\sigma})$  and $(\widetilde{\Gg}_{*}(\Gr,E), \LL_{\sigma})$, respectively.
\end{remark}
\begin{remark}\label{rem:amenablity_of_1_and_infty}
By  \cite{BKM}*{Theorem 5.13}, see Remark \ref{rem:general_comments_on_Lp}, if $P\subseteq \{1,\infty\}$, then
for any $*=\Space ,0,00$, we have
$$
\OO^P(\Gr,E,\sigma)_{*}=\OO^P_{\red}(\Gr,E,\sigma)_{*} \quad \text{and} \quad \TT^P(\Gr,E,\sigma)_{*}=\TT^P_{\red}(\Gr,E,\sigma)_{*}.
$$
In particular, the same holds if  $\{1,\infty\}\subseteq P$ as then the associated $L^P$-operator algebras coincide with $L^{\{1,\infty\}}$-operator algebras. If the corresponding groupoid is amenable,
we have $\OO^P(\Gr,E)_{*}=\OO^P_{\red}(\Gr,E)_{*}$ or  $\TT^P(\Gr,E)_{*}=\TT^P_{\red}(\Gr,E)_{*}$ for every $P\subseteq [1,\infty]$, see Remark \ref{rem:amenability}.
\end{remark}
\begin{lemma}\label{lem:embeddings_of_graph_algebras}
For any $\emptyset \neq P\subseteq [1,\infty]$ we have natural  injective homomorphisms, which are isometric into the full and reduced algebras,
$$
\Cont_0(\partial E)\subseteq \OO^P(E)\hookrightarrow \OO^P(\Gr,E,\sigma), \OO^P_{\red}(\Gr,E,\sigma), \OO^P_{\ess}(\Gr,E,\sigma),
$$
$$
\Cont_0( E^{\leq \infty})\subseteq\TT^P(E)\hookrightarrow \TT^P(\Gr,E,\sigma), \TT^P_{\red}(\Gr,E,\sigma), \TT^P_{\ess}(\Gr,E,\sigma).
$$
\end{lemma}
\begin{proof}
We  consider the case $P=\{p\}$. One gets general case by considering direct sums of representations.
We  use the embeddings 
$\widetilde{\Gg}(E)\subseteq \widetilde{\Gg}(\Gr,E)$ and $\Gg(E)\subseteq \Gg(\Gr,E)$ from \eqref{eq:graph_groupoids_into_self-similar_groupoids}.
The compression of the action of the algebra $B(\ttt^{\rd,p})$ to the (invariant) subspace $\ell^p(\widetilde{\Gg}(E))\subseteq \ell^p(\widetilde{\Gg}(\Gr,E))$
coincides with the canonical representation of  $\TT_{\red}^p(S(E))=F_{\red}^p(\widetilde{\Gg}(E))$. Since the groupoid $\widetilde{\Gg}(E)$ is amenable,  
we have  $F_{\red}^p(\widetilde{\Gg}(E))=F^p(\widetilde{\Gg}(E))$ by  \cite{Gardella_Lupini17}*{Theorem 6.19}. Thus,  $\TT_{\red}^p(S(E))=F^p(\widetilde{\Gg}(E))=\TT^p(S(E))=\TT^{p}(E)$. 
This implies that the canonical maps $\TT^{p}(E)\to B(\ttt^{\rd,p})\to \TT_{\red}^p(S(E))$ are isometric. Similar arguments give $\OO^p(E)\cong B(\tt^{\rd,p})$.

Now  composing the inclusion $\TT^p(E)\subseteq \TT^P_{\red}(\Gr,E,\sigma)\cong F^p_{\red}(\widetilde{\Gg}(\Gr,E),\LL_{\sigma})$ with the canonical  homomorphism $\TT^p_{\red}(\Gr,E,\sigma)\to \TT^p_{\ess}(\Gr,E,\sigma)\cong F^p_{\ess}(\widetilde{\Gg}(\Gr,E),\LL_{\sigma})$ we get a contractive homomorphism $\TT^p(E)\to \TT^p_{\ess}(\Gr,E,\sigma)$ that intertwines the canonical  expectations $\TT^p_{\ess}(\Gr,E,\sigma)\to \mathcal{D}(E^{\le \infty})$ and
$\TT^p(E)\to C_0(E^{\le \infty})\subseteq \mathcal{D}(E^{\le \infty})$, cf. Remark~\ref{rem:essential_groupoid_algebras}. Since these expectations are faithful we conclude that 
$\TT^p(E)\to \TT^p_{\ess}(\Gr,E,\sigma)$ is injective.
Similar arguments give $\OO^p(E)\hookrightarrow\OO^p_{\ess}(\Gr,E,\sigma)$.
Since $\OO^p(E)$  and $\TT^p(E)$ contain $\Cont_0(\partial E)$ 
and  $\Cont_0(E^{\leq\infty})$ as subalgebras, this finishes the proof.
\end{proof}

To say something about the canonical representation of $(\Gr,\sigma_{\Gr})$ in the above algebras
we need  more terminology.
Suppose that $\Gr$ acts self-similarly on $E$. Since the left action of $g \in \Gr$ establishes a bijection $\sr(g)E^1\cong \rg(g)E^1$
we see that $\sr(g)$ is a source or infinite receiver if and only if $\rg(g)$ has this property. 
In particular, the complementary sets  $E^0_{\reg}$ and $E^{0}_{\sing}$  are $\Gr$-invariant and $\Gr=\Gr_{\reg}\sqcup \Gr_{\sing}$ decomposes   \label{page:groupoid_decomposition_singular_regular} into a disjoint union
of the corresponding  restrictions
\[
\Gr_{\reg}  \coloneqq \Gr|_{E^0_{\reg}}=\{g \in \Gr \mid \rg(g), \sr(g)\in E^0_{\reg} \}, \quad \Gr_{\sing}  \coloneqq \Gr|_{E^0_{\sing}}=\{g \in \Gr \mid \rg(g), \sr(g)\in E^0_{\sing} \}
\]
is a subgroupoid of  $\Gr$.
We let $K_0 \coloneqq E^0$ and recursively define
$$
K_{n+1} \coloneqq E^{0}_{\sing}\sqcup \{g\in \Gr_{\reg}:  \text{$ge=e$ and $g|_{e}\in K_n$ for each $e\in \sr(g)E^1$}\}, \qquad n \geq 0.
$$
Then $K_0\subseteq K_1\subseteq K_2\subseteq \cdots$ and  $K \coloneqq \bigcup_{n=0}^\infty K_n$
is the smallest wide subgroupoid of $\Gr$ such that for any $g\in \Gr_{\reg}$
\begin{equation}\label{eq:tight_kernel_implication}
\text{$ge=e$ and $g|_{e}\in K$ for all $e\in \sr(g)E^1$ }\,\Longrightarrow\,\,\, g\in K.
\end{equation}
This groupoid $K$ is an algebraic analogue of a construction of a reduction ideal for $C^*$-correspondences, see 
\cite{Kwasniewski_Lebedev}*{Definition 5.2}. We follow the naming of Miller and Steinberg who called $K$ the tight kernel, see  \cite{Miller_Steinberg}*{Subsection 2.3}. 
\begin{definition}[\cite{Miller_Steinberg}]
We call $K$ constructed above the \emph{tight kernel} of $(\Gr,E)$. Explicitly, 
$$
K= \{g\in \Gr_{\reg}: \exists_{n\geq 1} \text{ $g$ strongly fixes  all paths in $\sr(g)E^n$ and in $\sr(g)E^{k}_{\sing}$ for $k\leq n$}\}.
$$
\end{definition}
\begin{remark}It is immediate that $K$ is a normal subgroupoid of $\Gr$ and that elements in $K$ fix all paths in $E^*$ (as extension of strongly fixed paths are strongly fixed). Thus, $K$ is a subgroupoid of the kernel $N$ of the action of $\Gr$, cf. Remark~\ref{rem:kernel_of_the_action}. By \eqref{eq:tight_kernel_implication}, $K$ is closed under taking sections.
Hence, the quotient $G/K$ groupoid acts self-similarily on $E$ by the formulas
$$
(gK) \mu =g\mu, \qquad (gK)|_{\mu}=g|_{\mu} K, \qquad (g,\mu)\in \Gr*E^*.
$$
By \cite{Miller_Steinberg}*{Theorem 2.12}, the surjective semigroup homomorphism $S(\Gr,E)\ni (\alpha,g,\beta)\mapsto (\alpha,gK,\beta)\in S(\Gr/K,E)$ induces 
isomorphism of tight groupoids 
$$
\Gg(\Gr, E)\cong \Gg(\Gr/K, E).
$$
\end{remark}
\begin{lemma}\label{lem:divide_by_K}
Suppose that $\sigma_{\Gr}$ is a twist of $\Gr$ which is invariant under sections, equivalently $\sigma \coloneqq (\sigma_{\Gr},1)$ is 
a twist of $(\Gr,E)$, cf Remark~\ref{rem:trivial_bowtie}. Then $\sigma/K \coloneqq (\sigma_{\Gr/K},1)$, where $\sigma_{\Gr/K}(gK,hK) \coloneqq \sigma_{\Gr}(g,h)$ for $(g,h)\in \Gr^2$,
is a well-defined twist of $(\Gr/K,E)$ and 
$$
(\Gg(\Gr, E),\LL_{\sigma})\cong (\Gg(\Gr/K, E),\LL_{\sigma/K}).
$$
Moreover, for every covariant representation $(W,T)$ of $(\Gr,E,\sigma)$, 
the $\sigma_{\Gr}$-twisted representation $\Gr\ni g\mapsto W_g$ of $\Gr$ descends to
a $\sigma/K$-twisted representation 
$\Gr/K\ni gK\mapsto  W_g$ of $\Gr/K$. 
\end{lemma}
\begin{proof}
Let $(g,h,k)\in \Gr^2*K$. Find $\mu$ which is strongly fixed by $k$. Then  using that $\sigma_{\Gr}$ is invariant under sections
we get
$
\sigma_{\Gr}(g,hk)= \sigma_{\Gr}(g|_{hk\mu},hk|_{\mu})=\sigma_{\Gr}(g|_{h},h|_{\mu})=\sigma_{\Gr}(g,k).
$
Similarly, for $(g,k,h)\in \Gr*K*\Gr$ finding $\mu$ which is (strongly) fixed by $k$ we get
$
\sigma_{\Gr}(gk,h)= \sigma_{\Gr}(gk|_{\mu},h|_{h^{-1}\mu})=\sigma_{\Gr}(g|_{\mu},h|_{h^{-1}\mu})=\sigma_{\Gr}(g,h).
$
Hence, $\sigma_{\Gr/K}$ is well defined. Clearly, it is an invariant under sections $2$-cocycle for $\Gr/K$.
Thus, $\sigma/K \coloneqq (\sigma_{\Gr/K},1)$ is a twist of $(\Gr/K,E)$. It is immediate that
the associated twists $\LL_{\sigma}$ and $\LL_{\sigma/K}$ are isomorphic through the isomorphism respecting $\Gg(\Gr, E)\cong \Gg(\Gr/K, E)$.

For the second part of the assertion it suffices to show that $W_{k}=W_{\sr(k)}$ for all $k\in K$.
This is trivial when $k\in K_0$. Assume this holds for all elements in $K_n$ and let $k\in K_{n+1}$. 
Thus, $k\in \Gr_{\reg}$ and  $ke=e$, $k|_{e}\in K_n$ for each $e\in \sr(k)E^1$.
Using this we get 
$$
W_{k}=W_{k}W_{\sr(k)} \stackrel{\ref{defn:Cuntz--Krieger2}}{=}W_{k}\sum_{e\in \sr(k)E^1} T_{e}T_{e}^*
\stackrel{\ref{enu:reps_of_twisted_self_similar3} }{=} \sum_{e\in \sr(k)E^1} T_{ke}W_{k|_{e}}T_{e}^*=\sum_{e\in \sr(k)E^1} T_{e}T_{e}^*
=W_{\sr(k)}.
$$
Hence, the assertion holds by induction.
\end{proof}
\begin{definition}
We say that $(\Gr,E)$ is \emph{tightly faithful} if $K=\Gr^0$, equivalently there is no $g\in \Gr_{\reg}\setminus \Gr^0$ that strongly fixes all elements in  $\sr(g)\partial E$.
\end{definition}
\begin{remark}
If $(\Gr,E)$ is faithful or pseudo free then it is tightly faithful, but in general all three properties are different.
The authors of \cite{Miller_Steinberg} say that an action is \emph{loosely faithful} if $K=N$. 
Hence, an action is faithful if and only if it is both tightly and loosely faithful.
\end{remark}

\begin{proposition}\label{prop:embeddings_of_algebras} Let $(\Gr,E,\sigma)$ be a twisted self-similar action and let $\emptyset \neq P\subseteq [1,\infty]$. 
\begin{enumerate}  
\item\label{enu:embeddings_of_algebras2}  The subalgebra $B(\tw^{\rd,P} )\subseteq \TT^P_{\red}(\Gr,E,\sigma)$ is an exotic algebra of $(\Gr,\sigma_{\Gr})$, in the sense that we have a canonical representation 
$B(\tw^{\rd,P} )\to F^P_{\red}(\Gr,\sigma_{\Gr})$.
\item\label{enu:embeddings_of_algebras3}  If   $E$ is row-finite, then the subalgebra $B(\tw^{\es,P} )\subseteq \TT^P_{\ess}(\Gr,E,\sigma)$ is an exotic algebra of $(\Gr,\sigma_{\Gr})$ in the sense that we have a canonical representation 
$B(\tw^{\es,P} )\to F^P_{\red}(\Gr,\sigma_{\Gr})$.
\item\label{enu:embeddings_of_algebras3.5} Assume $\sigma_{\bowtie}\equiv 1$ is trivial. Then 
the $\sigma_{\Gr}$-twisted representation $\Gr\ni g\mapsto \ww^{\rd,P}_{g}\in  \OO^P_{\red}(\Gr,E,\sigma)$ is injective if and only if $(\Gr,E)$ is tightly faithful.
In general, it descends to an injective representation 
$\Gr/K\ni gK\mapsto  \ww^{\rd,P}_{g}\in\OO^P_{\red}(\Gr,E,\sigma)$ 
of $(\Gr/K,\sigma/K)$.
\item \label{enu:embeddings_of_algebras4} If $(\Gr,E)$ is pseudo free, then  $B(\ww^{\rd,P} )$ is exotic in the sense that we have a canonical representation 
$B(\ww^{\rd,P} )\to F^P_{\red}(\Gr,\sigma_{\Gr})$.
\end{enumerate}
\end{proposition}
\begin{proof}

\ref{enu:embeddings_of_algebras2}. 
Note that the closed subgroupoid 
\begin{equation}\label{eq:subgroupoid_image_of_G}
\widetilde{\Gg}(\Gr) \coloneqq \{[\rg(g),g,\sr(g);\sr(g)]: g\in \Gr\}
\end{equation}
of $\widetilde{\Gg}(\Gr,E)$  is isomorphic to $\Gr$,  and the subspace $\ell^p(\widetilde{\Gg}(\Gr))\subseteq \ell^p(\widetilde{\Gg}(\Gr,E))$ is invariant under the action of $B(\tw^{\rd,p} )$. Moreover, restricting this action to  $\ell^p(\widetilde{\Gg}(\Gr))\cong \ell^p(\Gr)$,  we get a map $\Gr\in g\mapsto    \tw^{\rd,p}_{g}\in \Bound(\ell^p(\widetilde{\Gg}(\Gr)))$ which is equivalent to the
regular representation of $(\Gr,\sigma)$. This gives a natural representation $B(\tw^{\rd,p} ) \to F^p_{\red}(\Gr,\sigma_{\Gr})$.

\ref{enu:embeddings_of_algebras3}. If  $E$ is row-finite (there are no infinite receivers), than
every point in the subgroupoid $\widetilde{\Gg}(\Gr)$, given by \eqref{eq:subgroupoid_image_of_G}, is open in $\widetilde{\Gg}(\Gr,E)$. Hence, $\widetilde{\Gg}(\Gr)\subseteq \widetilde{\Gg}(\Gr,E)_{\Hau}$, as the latter is comeager in $\widetilde{\Gg}(\Gr,E)$.
As in \ref{enu:embeddings_of_algebras2}, the supspace $\ell^p(\widetilde{\Gg}(\Gr))\subseteq \ell^p(\widetilde{\Gg}(\Gr,E)_{\Hau})$ is invariant under the action of $B(\tw^{\es,p} )$, 
and the restriction to $\ell^p(\widetilde{\Gg}(\Gr))\cong \ell^p(\Gr)$ gives  representation $B(\tw^{\es,p} )\to F^p_{\red}(\Gr,\sigma)$.

\ref{enu:embeddings_of_algebras3.5}. By Lemma~\ref{lem:divide_by_K}, $\ww^{\rd,P}$ descends to $\Gr/K$ and so it cannot be injective 
if $K$ is nontrivial. So assume  $(\Gr,E)$ is tightly faithful and let $g\in \Gr\setminus \Gr^0$. 
If $g\not\in \Gr_{\reg}$, that is $\sr(g)\in E^0_{\sing}$, then $\gamma \coloneqq [\rg(g), g, \sr(g), \sr(g)]$ is an element of $\Gg(\Gr,E)$  different than $[\sr(g), \sr(g), \sr(g), \sr(g)]$ which we identify with $\sr(g)\in \partial E$. Since 
$\ww^{\rd,P}_{g} \bone_{\sr(g)}=\bone_{\gamma}\neq \bone_{\sr(g)}=\ww^{\rd,P}_{\sr(g)} \bone_{\sr(g)} $, we get that $\ww^{\rd,P}_{g}\neq   \ww^{\rd,P}_{\sr(g)}$. 
If $\gamma \in \Gr_{\reg}$, then by tight faithfulness, there is $\mu\in \sr(g)\partial E$, which is not  strongly $g$-fixed. Then $\gamma \coloneqq [\rg(g), g, \sr(g), \mu]$ and 
$\eta \coloneqq [\sr(g), \sr(g), \sr(g), \mu]$ are different elements of $\Gg(\Gr,E)$. Since 
$\ww^{\rd,P}_{g} \bone_{\eta}=\bone_{\gamma}\neq \bone_{\eta}=\ww^{\rd,P}_{\sr(g)} \bone_{\eta} $, we get that $\ww^{\rd,P}_{g}\neq   \ww^{\rd,P}_{\sr(g)}$.
Hence, $\ww^{\rd,P}$ is an injective representation of $(\Gr, \sigma_{\Gr})$.

\ref{enu:embeddings_of_algebras4}. 
Consider $f=\sum_{g\in F} \alpha_g \bone_{g}\in \Contc(\Gr,\sigma_{\Gr})$ and 
$\xi= \sum_{h\in H} \beta_h \bone_{h}\in \ell^p(\Gr)$ where $F, H\subseteq \Gr$ are finite and $\alpha_g$'s and $\beta_h$'s are complex numbers.
For each $v\in E^0=\Gr^0$  choose any $\xi_v\in v\partial E$.
By pseudo freeness, see  Proposition~\ref{prop:pseudo_freeness}\ref{enu:pseudo_freeness2.5},
we have $[(\rg(g),g,\sr(g),\xi]=[(\rg(h),h,\sr(h),\xi']$ if and only if $\xi=\xi'$ and $g=h$. Hence, the map $H\ni h\mapsto [\rg(h),h,\sr(h);\xi_{\sr(h)}]\in \Gg(\Gr,E)$
is injective, so for $\xi'= \sum_{h\in H} \beta_h \bone_{[\rg(h),h,\sr(h);\xi_{\sr(h)}]} \in \ell^p(\Gg(\Gr,E))$ we have
$\|\xi\|_{p}=\|\xi'\|_{p}$, and also
\begin{align*}
\|\Lambda_p(f) \xi\|_{p}&= \|f * \xi\|_{p}=
\begin{cases}
\big(\sum_{k\in \Gr} |\sum_{gh=k}\sigma_{\Gr}(g,h) \alpha_g\beta_h|^p\big)^{1/p}, & p<\infty
\\
\sup_{k\in \Gr} |\sum_{gh=k}\sigma_{\Gr}(g,h) \alpha_g\beta_h|, & p=\infty
\end{cases}
\\
&=\left \|\sum_{k\in \Gr} \sum_{gh=k} \sigma_{\Gr}(g,h) \alpha_g\beta_h \bone_{[\rg(k),k,\sr(k);\xi_{\sr(k)}]} \right\|_{p}= \left \|\sum_{g\in F} \alpha_g \ww_{g}^{\rd,p} \xi'\right\|_{p}.
\end{align*}
This implies that $\|\Lambda_p(f)\|_{\red}^{p} \leq \|\sum_{g\in F} \alpha_g \ww_{g}^{\rd,p}\|$. Hence, 
the map $B(\ww^{\rd,p})\ni \sum_{g\in F} \alpha_g \ww_{g}^{\rd,p} \mapsto \sum_{g\in F} \alpha_g \bone_{g}\in F^{p}_{\red}(\Gr,\sigma_{\Gr})$ is well-defined and contractive.
\end{proof}
\begin{remark}\label{rem:various_inclusions}
Without the row-finiteness assumption item \ref{enu:embeddings_of_algebras3} above fails, see Remark~\ref{rem:row_finiteness_and_toeplitz_groupoids} below.
Pseudo freeness in  \ref{enu:embeddings_of_algebras4} seems too strong, while tight faithfulness in \ref{enu:embeddings_of_algebras3.5} seems too weak
to show that $B(\ww^{\rd,P} )$ is an exotic algebra of $(\Gr,\sigma_{\Gr})$. For amenable actions with tight faithfulness we could prove that 
the $C^*$-algebraic singular ideal vanishes, in the same manner as in Corollary~\ref{cor:row_finiteness_and_toeplitz_groupoids} below. Thus, we have the following  open problem, which is already important in the $C^*$-algebraic case.
\end{remark}
\begin{question}\label{problem:embed_into_OO}
For which twisted self-similar actions $(\Gr,E,\sigma)$ is the closure of an image of $\Cont_c(\Gr,\sigma_{\Gr})$ 
in $\OO^P(\Gr,E,\sigma)$ an exotic Banach algebra of $(\Gr,\sigma_{\Gr})$?
\end{question}

\subsection{Main structural results}

The diagonal subalgebras $\Cont_0(\partial E)$ and  $\Cont_0( E^{\leq \infty})$ in Remark~\ref{rem:various_inclusions}, both contain a copy of the algebra $\Cont_0(E^0)$. Namely,  it is the closed linear span of characteristic functions $\{\bone_{Z(v)\cap \partial E}\}_{v\in E^0} $ and $\{\bone_{Z(v)}\}_{v\in E^0} $, respectively.
The algebra $\Cont_0(E^0)$ plays a crucial role in ``uniqueness theorem'' for Cuntz algebras. For Toeplitz algebras we instead need to consider the larger
subalgebra
$
\clsp\{ \bone_{Z(\alpha)}: \alpha \in E^0 \cup E^1 \}
$
of  $\Cont_0(E^{\le \infty})$. Its spectrum can be naturally identified with the discrete space $E^{\le 1} \coloneqq E^0 \cup E^1$, so we denote this subalgebra by $\Cont_0(E^{\le 1})$.

\begin{lemma}\label{lem:isometric_on_boundary_path}
Every representation $\psi:\OO^P(E)\to B$ that is nonzero on each projection corresponding to vertices in $E^0$ (is injective on $C_0(E^0)$) is isometric on $\Cont_0(\partial E)$.
\end{lemma}
\begin{proof}
Let $(W,T)$ be the  Cuntz--Krieger $E$-family  generating $\OO^P(E)$. Put $V_{\mu} \coloneqq \psi(T_{\mu}T_{\mu}^*)$ for $\mu \in E^*\setminus E^0$ and 
$V_v \coloneqq \psi(W_{v})$ for $v\in E^0$. By assumption   $V_v\neq 0$ for every $v\in E^0$. 
Since  $\psi(T_e T_e^*)=\psi(W_{\rg(e)})\neq 0$ we have $\psi(T_e)\neq 0$. 
If $(e,f)\in E^2$, then the $E$-family relations imply that $T_e^* (T_{ef} T_{ef}^*) T_{f}=W_{\rg(f)}$.
This implies that $\psi(T_{ef})\neq 0$. Proceeding inductively one concludes that $\psi(T_{\mu})\neq 0$ and so also $V_{\mu}\neq 0$ for every $\mu \in E^*$.
By  the last part of Proposition~\ref{prop:representations_inverse_semigroups_vs_actions}
we need to show that  $\prod_{\beta\in F} (V_\mu -V_\beta) \neq 0$ for every $\mu\in E^*$ and finite $F\subseteq \mu E^*$ that does not cover $\mu$ in the semigroup $E^*\cup 0$.
As explained in Example \ref{ex:spectrum_graph_inverse_semigroup}, the latter means that there is an extension  $\alpha$ of $\mu$ which is not comparable with any path in $F$.
Assuming this, $E$-family relations imply that $V_{\alpha} \cdot  V_{\mu} = V_{\alpha}\neq 0$ and 
$V_{\alpha}\cdot V_{\beta}= 0$ for all $\beta \in F$. 
Hence,  $V_{\alpha}\cdot \prod_{\beta\in F} (V_{\mu} -V_{\beta})=V_{\alpha} \neq 0$, and so $\prod_{\beta\in F} (V_{\mu} -V_{\beta}) \neq 0$. 
\end{proof}

The following proposition for $C^*$-algebras was proved in \cite{Fowler_Raeburn:Toeplitz}*{Theorem 4.1} using $C^*$-correspondence techniques (cf. also \cite{Carlsen_Kwasniewski_Ortega}*{Theorem 9.1}).
We use groupoid models.
\begin{proposition}\label{prop:Coburn-Toeplitz_for_graphs}
For any representation $\psi:\TT^P(E) \to B$  the following are equivalent:
\begin{enumerate}
\item\label{enu:Coburn-Toeplitz_for_graphs1} $\psi$ is injective on $\TT^P(E)$;
\item\label{enu:Coburn-Toeplitz_for_graphs2}  $\psi$ is injective on $C_0(E^{\le \infty})$;
\item\label{enu:Coburn-Toeplitz_for_graphs3}  $\psi$ is injective on the algebra  $\Cont_0(E^{\leq 1})=\clsp\{ \bone_{Z(\alpha)}: \alpha \in E^0 \cup E^1 \}\subseteq C_0(E^{\le \infty})$; 
\item\label{enu:Coburn-Toeplitz_for_graphs4} the representation $\psi$ satisfies the following condition
\begin{equation}\label{eq:Coburn--Toeplitz_condition}
\psi(\bone_{Z(v)})\neq\sum_{e \in \rg^{-1}(v)} \psi(\bone_{Z(e)})\text{ for all }v\in E^{0}_{\reg}\text{ and }\psi(\bone_{Z(v)})\neq 0\text{ for all }v\in E^0_{\sing}.
\end{equation}
\end{enumerate}
In particular, for a representation  $(W,T)$ of $E$ on an $L^p$-space $Y$ the associated representation $\TT^{p}(E)\to \Bound(Y)$ is injective if and only if 
$W_v\neq 0$ for all $v\in E^0$ and
\[
W_v\neq \sum_{e \in \rg^{-1}(v)} T_e T_e^* \qquad \text{ for all }v\in E^{0}_{\reg}.
\]
\end{proposition}
\begin{proof} Implications 
\ref{enu:Coburn-Toeplitz_for_graphs1}$\Rightarrow$\ref{enu:Coburn-Toeplitz_for_graphs2}$\Rightarrow$\ref{enu:Coburn-Toeplitz_for_graphs3}$\Rightarrow$\ref{enu:Coburn-Toeplitz_for_graphs4} are obvious. Also since $\TT^P(E)\cong F^P(\widetilde{\Gg}(E))$, by Theorem~\ref{thm:presentations_of_twisted_self_similar_algebras}, and  $\widetilde{\Gg}(E)$ is topologically free, by Theorem~\ref{thm:topological_freeness_self_similar_transformation_groupoids}\ref{enu:topological_freeness1},  the implication \ref{enu:Coburn-Toeplitz_for_graphs2}$\Rightarrow$\ref{enu:Coburn-Toeplitz_for_graphs1} holds by 
Theorem~\ref{thm:groupoid_simplicity_pure_infiniteness}\ref{enu:groupoid_simplicity_pure_infiniteness1}.

Thus, it suffices to prove that \ref{enu:Coburn-Toeplitz_for_graphs4}$\Rightarrow$\ref{enu:Coburn-Toeplitz_for_graphs1}.
Assume  \eqref{eq:Coburn--Toeplitz_condition}. Denote by $(W,T)$  the   $E$-family  generating $\TT^P(E)$, and  put  $V_{\mu} \coloneqq \psi(T_{\mu}T_{\mu}^*)$ for $\mu \in E^*\setminus E^0$, and 
$V_v \coloneqq \psi(W_{v})=\psi(\bone_{Z(v)})$ for $v\in E^0$. By \eqref{eq:Coburn--Toeplitz_condition},  $V_v\neq 0$ for all $v\in E^0$
and 
as in the proof Lemma~\ref{lem:isometric_on_boundary_path} we  get that  $V_{\mu}\neq 0$ for all $\mu \in E^*$.
By  the last part of Proposition~\ref{prop:representations_inverse_semigroups_vs_actions}, $\psi$ is isometric on  $C_0(E^{\le \infty})$ 
if and only if   $\prod_{\beta\in F} (V_\mu -V_\beta) \neq 0$ for every   $\mu\in E^*$ and every finite $F\subseteq \mu E^*\setminus\{\mu\}$. 
This condition readily implies that  $\psi(W_v)> \sum_{e \in \rg^{-1}(v)} \psi(T_e T_e^*)$ for  any  $v\in E^{0}_{\reg}$ (take $\mu=v$ and $F=\rg^{-1}(v)$).
Conversely, suppose that $\psi(W_v)> \sum_{e \in \rg^{-1}(v)} \psi(T_e T_e^*)$ for  every  $v\in E^{0}_{\reg}$. 
Take any   $\mu\in E^*$ and any finite $F\subseteq \mu E^*\setminus\{\mu\}$.
If $v \coloneqq \sr(\mu)$ is a source, then $F$ has to be empty. If $v$ is an infinite receiver, then there is an edge $e\in \rg^{-1}(v)$ which is not  a prefix of any path in $F$.
Then $V_{\mu e} \cdot  V_{\mu} = V_{\mu e}\neq 0$ and 
$V_{\mu e}\cdot V_{\beta}= 0$ for all $\beta \in F$ which implies that  $\prod_{\beta\in F} (V_\mu -V_\beta)$ is nonzero (it contains $V_{\mu e}$ as a subprojection).
If $v=\sr(\mu)\in E^{0}_{\reg}$ is regular, then by the assumption $V_v> \sum_{e \in \rg^{-1}(v)} V_{e}$. 
Since $\phi(T_{\mu}^*)V_{\mu}\phi(T_{\mu}^*)=V_{v}$ and  $\phi(T_{\mu}^*)V_{\mu e}\phi(T_{\mu}^*)=V_{e}$, this implies that 
$V_{\mu}- \sum_{e \in \rg^{-1}(v)}V_{\mu e}\neq 0$. This latter (nonzero) projection is a subprojection of $\prod_{\beta\in F} (V_\mu -V_\beta)$.
\end{proof}
\begin{definition}[cf. \cite{BK}*{Definition 5.6}]\label{def:visible} 
We say that an inclusion $A\subseteq B$ of Banach algebras  has the \emph{intersection property}, or that $A$ \emph{detects ideals} in $B$, if for every nonzero ideal $J$ in $B$ we have $J\cap A \neq \{ 0 \}$. 
\end{definition}
\begin{remark}An inclusion $A\subseteq B$ has the intersection property if and only if every representation $\psi$ of $B$ which is injective on $A$, is injective on $B$.
Results assuring the latter are often called ``uniqueness theorems''.
Proposition  \ref{prop:Coburn-Toeplitz_for_graphs} says that 
$\Cont_0(E^{\leq 1})\subseteq   \TT^P(E)$ has the intersection property.
Lemma~\ref{lem:isometric_on_boundary_path} implies that $\Cont_0(E^0) \subseteq \OO^P(E)$ has the intersection property if and only if $\Cont_0(\partial E)\subseteq \OO^P(E)$ has the intersection property.
\end{remark}
\begin{theorem}[Intersection properties]\label{thm:uniqueness_results} 
Let $(\Gr,E)$ be a self-similar groupoid action with a twist $\sigma$ and  let $P\subseteq [1,\infty]$ be a non-empty set. 
\begin{enumerate}
\item\label{enu:Cuntz--Krieger uniqueness1'} \Evr~ and \Cyc~ imply that  every ideal $I$ in $\OO^P(\Gr,E,\sigma)$ with $I\cap \Cont_0(E^0)=\{0\}$ is contained in the kernel of the canonical  map $\OO^P(\Gr,E,\sigma) \to  \OO^P_{\ess}(\Gr,E,\sigma)$.
The converse implication holds when the twist is trivial.

\item\label{enu:Cuntz--Krieger uniqueness1} \Evr~ and \Cyc~ imply that  $\Cont_0(E^0)\subseteq \OO^P_{\ess}(\Gr,E,\sigma)$ has the intersection property. 
If $\Cont_0(E^0)\subseteq \OO^P(\Gr,E)$ has the intersection property, then  $\Evr$ and $\Cyc$ hold.

\item\label{enu:Cuntz--Krieger uniqueness2'} \Evr~and \Rec~ imply that  every kernel of a representation $\psi$ of $\TT^P(\Gr,E,\sigma)$ satisfying \eqref{eq:Coburn--Toeplitz_condition} is contained in the kernel of the canonical  map $\TT^P(\Gr,E,\sigma) \to  \TT^P_{\ess}(\Gr,E,\sigma)$.
The converse implication holds when the twist is trivial.

\item\label{enu:Cuntz--Krieger uniqueness2} $\Evr$~and \Rec~ imply that    $C_0(E^{\leq 1})\subseteq \TT^P_{\ess}(\Gr,E,\sigma)$ has the intersection property.
If  $C_0(E^{\leq 1})\subseteq\TT^P(\Gr,E)$ has the intersection property, then $\Evr$~and $\Rec$ hold.

\item\label{enu:Cuntz--Krieger uniqueness3} \Evr~  implies that the two inclusions
$\Cont_0(\partial E)\subseteq \OO^P_{\ess}(\Gr,E,\sigma)_0,\, \OO^P_{\ess}(\Gr,E,\sigma)_{00}$ 
have the intersection property.

\item\label{enu:Cuntz--Krieger uniqueness4} $\Evr$ and $\Rec$  imply that   
the inclusions $\Cont_0( E^{\leq \infty})\subseteq \TT^P_{\ess}(\Gr,E,\sigma)_{*}$ have  the intersection property for all $*=\Space , 0, 00$.
\end{enumerate}

\end{theorem}
\begin{proof} 
\ref{enu:Cuntz--Krieger uniqueness1'}. 
Recall that $\OO^P(\Gr,E,\sigma)\cong F^P(\Gg(\Gr,E),\LL_{\sigma})$, $\OO^P_{\ess}(\Gr,E,\sigma)\cong F^P_{\ess}(\Gg(\Gr,E),\LL_{\sigma})$ and $\Gg(\Gr,E)$ is topologically free if and only if 
$\Evr$ and $\Cyc$ hold, by  Theorem~\ref{thm:topological_freeness_self_similar_transformation_groupoids}.
By  Lemmas~\ref{lem:isometric_on_boundary_path}, \ref{lem:embeddings_of_graph_algebras}, an ideal $I$ in $\OO^P(\Gr,E,\sigma)$ satisfies  $I\cap \Cont_0(E^0)=\{0\}$
if and only if $I\cap \Cont_0(\partial E)=\{0\}$. We get the assertion by Theorem~\ref{thm:groupoid_simplicity_pure_infiniteness}\ref{enu:groupoid_simplicity_pure_infiniteness0}.

\ref{enu:Cuntz--Krieger uniqueness1}. We get the assertion in a similar way as in \ref{enu:Cuntz--Krieger uniqueness1} but we need to appeal to Theorem~\ref{thm:groupoid_simplicity_pure_infiniteness}\ref{enu:groupoid_simplicity_pure_infiniteness1}. In particular, by  Lemma~\ref{lem:isometric_on_boundary_path}, \ref{lem:embeddings_of_graph_algebras},
$\Cont_0(E^0)\subseteq \OO^P_{\ess}(\Gr,E,\sigma)$ has the intersection property 
if and only if $\Cont_0(\partial E)\subseteq \OO^P_{\ess}(\Gr,E,\sigma)$ has the intersection property.

\ref{enu:Cuntz--Krieger uniqueness2'}.  Recall that $\TT^P(\Gr,E,\sigma)\cong F^P(\widetilde{\Gg}(\Gr,E),\LL_{\sigma})$, $\TT_{\ess}^P(\Gr,E,\sigma)\cong F^P_{\ess}(\widetilde{\Gg}(\Gr,E),\LL_{\sigma})$ and $\widetilde{\Gg}(\Gr,E)$ is topologically free if and only if 
$\Evr$ and $\Rec$ hold, by Theorem~\ref{thm:topological_freeness_self_similar_transformation_groupoids}.
Let $I$ be a kernel of a representation $\psi$ of $\TT^P(\Gr,E,\sigma)$ (every ideal in $\TT^P(\Gr,E,\sigma)$ is of this form).
By Lemmas~\ref{lem:isometric_on_boundary_path}, \ref{lem:embeddings_of_graph_algebras}, 
$I\cap \Cont_0(E^{\leq 1})=\{0\}$ if and only if $I\cap \Cont_0(E^{\leq \infty})=\{0\}$. 
Thus
we get the assertion by Theorem~\ref{thm:groupoid_simplicity_pure_infiniteness}\ref{enu:groupoid_simplicity_pure_infiniteness0}.

\ref{enu:Cuntz--Krieger uniqueness2}.  The argument in \ref{enu:Cuntz--Krieger uniqueness2'} and Theorem~\ref{thm:groupoid_simplicity_pure_infiniteness}\ref{enu:groupoid_simplicity_pure_infiniteness1} 
gives the claim.

\ref{enu:Cuntz--Krieger uniqueness3}. Let $*=0,00$.
By Theorem~\ref{thm:topological_freeness_self_similar_transformation_groupoids}, topological freeness of  $\Gg_{*}(\Gr,E)$ 
is equivalent to $\Evr$.
By Remark~\ref{remembeddings_of_algebras},   $\OO^P_{\ess}(\Gr,E,\sigma)_{*}$  is an essential Banach algebra of the twisted groupoid 
$(\Gg_{*}(\Gr,E), \LL_{\sigma})$.
Hence, the assertion follows from  \cite{BKM2}*{Theorem 5.10(2)}.

\ref{enu:Cuntz--Krieger uniqueness4}. Let $*=0,00$. By Theorem~\ref{thm:topological_freeness_self_similar_transformation_groupoids}, topological freeness of  $\widetilde{\Gg}_{*}(\Gr,E)$ 
is equivalent to $\Evr$ and $\Rec$.
By Remark~\ref{remembeddings_of_algebras},    $\TT^P_{\ess}(\Gr,E,\sigma)_*$ is an essential Banach algebras of the twisted groupoid 
$(\widetilde{\Gg}_{*}(\Gr,E), \LL_{\sigma})$.
Hence, the assertion follows from  \cite{BKM2}*{Theorem 5.10(2)}.
\end{proof}
\begin{corollary}\label{cor:detection_of_ideals_amenable}
Theorem~\ref{thmx:Detection of ideals II}  in the introduction holds.
\end{corollary}
\begin{proof} 
Theorem \ref{thm:uniqueness_results}\ref{enu:Cuntz--Krieger uniqueness1'} implies  that
$\Cont_0(E^0)\subseteq \OO^P(\Gr,E)$ has the intersection property if and only if the   map $\OO^P(\Gr,E) \to  \OO^P_{\ess}(\Gr,E)$ is injective
and both $\Evr$ and $\Cyc$ hold.   Similarly, Theorem \ref{thm:uniqueness_results}\ref{enu:Cuntz--Krieger uniqueness2'} gives that 
$\Cont_0(E^{\leq 1})\subseteq \TT^P(\Gr,E)$ has the intersection property if and only if  the   map $\TT^P(\Gr,E) \to  \TT^P_{\ess}(\Gr,E)$ is injective
and both $\Evr$ and $\Cyc$ hold. 

By the assumption in Theorem~\ref{thmx:Detection of ideals II} and Corollary \ref{cor:finite_Hausdorffness}, both groupoids $\Gg(\Gr,E)$ and $\widetilde{\Gg}(\Gr,E)$
are finitely non-Hausdorff. Hence, by Theorem \ref{thmx:singular_ideal},  injectivity of $\OO^P_{\red}(\Gr,E) \to  \OO^P_{\ess}(\Gr,E)$ and $\TT^P_{\red}(\Gr,E) \to  \TT^P_{\ess}(\Gr,E)$
is equivalent to $\Hum$ for $\Gg(\Gr,E)$ and $\widetilde{\Gg}(\Gr,E)$, respectively. 

Amenability of $\Gg_{00}(\Gr,E)$  or  $\Gr$ is equivalent to amenability of $\Gg(\Gr,E)$  and  $\widetilde{\Gg}(\Gr,E)$, respectively (see \cite{Miller_Steinberg}*{Theorem 2.18} and Theorem \ref{thm:Toeplitz_nuclearity} below).  
Assuming this we get that $\OO^P_{}(\Gr,E)=\OO^P_{\red}(\Gr,E)$ and $\TT^P(\Gr,E)=\TT^P_{\red}(\Gr,E)$, respectively, see Remark \ref{rem:amenablity_of_1_and_infty}.

All this holds independently of the choice of $P$. Hence the assertion follows.
\end{proof}
In the setting of $C^*$-algebras a Cartan inclusion is a structure consisting of an algebra $B$ together with a maximal abelian subalgebra  $A\subseteq B$ equipped with a faithful contractive projection $E:B\onto A$. 
We now describe analogous structures for $L^P$-algebras associated to self-similar actions.
We recall that the symbol $\Space$ stands for the empty space.

\begin{theorem}\label{thm:Cartan_algebras} 
Let $(\Gr,E,\sigma)$ be a  twisted self-similar groupoid action.  Let $\emptyset\neq P\subseteq [1,\infty]$ and $*=\Space\,,0, 00$.
Assume that $(\Gr,E)$ satisfies \Fin. Then  
$$\OO^P_{\red}(\Gr,E,\sigma)_{*}=\OO^P_{\ess}(\Gr,E,\sigma)_{*} \,\,\text{ and }\,\, \TT^P_{\red}(\Gr,E,\sigma)_{*}=\TT^P_{\ess}(\Gr,E,\sigma)_{*}.
$$
These algebras  are equipped with canonical 
faithful contractive projections $\OO^P_{\red}(\Gr,E,\sigma)_{*}\onto \Cont_0(\partial E)$  and
$\TT^P_{\red}(\Gr,E,\sigma)_{*}\onto \Cont_0(E^{\leq \infty})$, and 
\begin{enumerate}
\item\label{item:Cartan_algberas1}  $\Cont_0(\partial E)\subseteq \OO^P_{\red}(\Gr,E,\sigma)$ is maximal abelian if and only if $\Evr$ and $\Cyc$ hold; 
\item\label{item:Cartan_algberas2}  $\Cont_0(\partial E)\subseteq \OO^P_{\red}(\Gr,E,\sigma)_{0}$ is maximal abelian if and only if $\Cont_0(\partial E)\subseteq \OO^P_{\red}(\Gr,E,\sigma)_{00}$ is maximal abelian  if and only if 
$\Evr$ hold;
\item\label{item:Cartan_algberas3}  
for any of the inclusions $C_0( E^{\le \infty})\subseteq  \TT_{\red}^P(\Gr,E,\sigma), \TT_{\red}^P(\Gr,E,\sigma)_{0},  \TT_{\red}^P(\Gr,E,\sigma)_{00}$  being maximal abelian is equivalent to $\Evr$  and $\Rec$. 
\end{enumerate}
\end{theorem}
\begin{proof}By Proposition~\ref{prop:Hausdorff_extended_groupoid} the groupoids   $\Gg(\Gr,E)_{*}$ and $\widetilde{\Gg}(\Gr,E)_{*}$ are Hausdorff.
Hence, the first part follows. In particular, the algebras $\OO^P_{\red}(\Gr,E,\sigma)_{*}$ and $\TT^P_{\red}(\Gr,E,\sigma)_{*}$ are reduced Banach algebras of 
 $\Gg(\Gr,E)_{*}$ and $\widetilde{\Gg}(\Gr,E)_{*}$, respectively, cf.  Remark~\ref{rem:reduced_Lp_algebras}.
Thus statements \ref{item:Cartan_algberas1}--\ref{item:Cartan_algberas3} follow from Theorem~\ref{thm:topological_freeness_self_similar_transformation_groupoids} and
\cite{BKM2}*{Proposition 5.11}, cf. Theorem~\ref{thm:groupoid_simplicity_pure_infiniteness}\ref{enu:groupoid_simplicity_pure_infiniteness1.5}.
\end{proof}
\begin{corollary}\label{cor:Theorem C}
Theorem~\ref{thmx:Cartan} in the introduction holds.
\end{corollary}
\begin{proof}
Composing the canonical generalised expectations with quotients by meager support functions we  get contractive maps 
$\mathbb{E}:\OO_{\red}(\Gr,E,\sigma)_{*}\to \mathcal{D}(\partial E)$ and $\widetilde{\mathbb{E}}: \TT_{\red}(\Gr,E,\sigma)_{*}\to \mathcal{D}( E^{\leq \infty})$, 
which are pseudo-expectations in the sense of \cite{Kwasniewski-Meyer:Aperiodicity}.
By  \cite{Kwasniewski-Meyer:Aperiodicity}*{Theorem 3.6} and \cite{Kwasniewski-Meyer:Cartan}*{Corollary 7.6}
a Cartan $C^*$-inclusion has a unique pseudo-expectation and so it has to be the genuine faithful expectation onto 
the masa subalgebra. Thus, if $\Cont_0(\partial E)\subseteq \OO_{\red}(\Gr,E,\sigma)_{*}$   or $\Cont_0( E^{\leq \infty})\subseteq \TT_{\red}(\Gr,E,\sigma)_{*}$ is Cartan
then either $\mathbb{E}$ takes values in $\Cont_0(\partial E)$ or  $\widetilde{\mathbb{E}}$ takes values in $\Cont_0( E^{\leq \infty})$. 
The latter is equivalent to that either $\Gg(\Gr,E)$ or $\widetilde{\Gg}(\Gr,E)$ is Hausdorff, but each of these alternatives is equivalent
to $\Fin$ by Proposition~\ref{prop:Hausdorff_extended_groupoid}. This shows necessity of $\Fin$.
The remaining part of the assertion of Theorem~\ref{thmx:Cartan} follows directly from Theorem~\ref{thm:Cartan_algebras}.
\end{proof}

\begin{theorem}
\label{thm:simplicity}
Let $(\Gr,E,\sigma)$ be a  twisted self-similar groupoid action and let $\emptyset\neq P\subseteq [1,\infty]$.
\begin{enumerate}
\item If $(\Gr,E)$ is cofinal and satisfies  $\Evr$ and $\Cyc$, then  $\OO^P_{\ess}(\Gr,E,\sigma)$ is simple.  
If  in addition it satisfies $\Con$, then $\OO^P_{\ess}(\Gr,E,\sigma)$ is purely infinite simple.

\item  $\OO^P(\Gr,E)$ is simple if and only if $(\Gr,E)$  is cofinal, satisfies $\Evr$, $\Cyc$    and the canonical map $\OO^P(\Gr,E)\to \OO^P_{\ess}(\Gr,E) $
is injective. If $\OO^P(\Gr,E)$ is simple and $\Con$ holds, 
then $\OO^P(\Gr,E)$ is purely infinite.
\end{enumerate}
\end{theorem}
\begin{proof} 
By Theorem  \ref{thm:topological_freeness_self_similar_transformation_groupoids}, 
$\Evr$ and $\Cyc$ is equivalent to topological freeness of $\Gg(\Gr,E)$.
By Proposition~\ref{prop:groupoid_minimality}, $(\Gr,E)$ is cofinal if and only if  $\Gg(\Gr,E)$
is minimal. Hence, the assertion follows from 
Theorem~\ref{thm:groupoid_simplicity_pure_infiniteness}\ref{enu:groupoid_simplicity_pure_infiniteness2},\ref{enu:groupoid_simplicity_pure_infiniteness3} and
Proposition  \ref{prop:locally_contracting_semigroup}.
\end{proof}


\begin{corollary}\label{cor:contracting_results}
Corollary~\ref{corx:contracting_results} in the introduction holds.
\end{corollary}
\begin{proof} 
By  \cite{Miller_Steinberg}*{Corollary 2.19} and Corollary~\ref{cor:contracting_non_Hausdorff}
the groupoid $\Gg(\Gr,E)$ is  amenable and finitely non-Hausdorff. Hence the equivalence \ref{enux:Cuntz--Krieger uniqueness0}$\Leftrightarrow$\ref{enux:Cuntz--Krieger uniqueness1} 
in Theorem \ref{thmx:Detection of ideals II} 
implies that simplicity of  $\OO^P(\Gr,E)$  is equivalent to the set of conditions $\Evr$, $\Cyc$, $\Min$  and $\Hum$.
This set of conditions is equivalent to simplicity of $A_{\mathbb{C}}(\Gg(\Gr,E))$ by \cite{Steinberg_Szakacs}*{Theorem A'} and 
\cite{Brix-Gonzales-Hume-Li:Hausdorff_covers}*{Theorem 4.2}, cf. Table \ref{table:tight groupoid} and Theorem \ref{thmx:singular_ideal}.  
\end{proof}

\section{The \texorpdfstring{$C^*$}{C*}-correspondence analysis}
\label{sect:C-star}

Throughout this section we fix a twist $\sigma = (\sigma_G,\sigma_{\bowtie})$  for a self-similar action $(G,E)$.
Here, we specialise our analysis to the case $P=\{2\}$ and so the associated Banach algebras become $C^*$-algebras and   representations 
can be considered on Hilbert spaces or in $\Cst$-algebras. As it  is customary, and as we did in the introduction, in this context we omit writing the subscript $\{2\}$. 
Thus, we write $\TT(\Gr,E,\sigma) \coloneqq \TT^2(\Gr,E,\sigma)$ and $\OO(\Gr,E,\sigma) \coloneqq \OO^2(\Gr,E,\sigma)$ for
the \emph{Toeplitz $C^*$-algebra} and the \emph{$\Cst$-algebra}  of $(\Gr,E,\sigma)$, respectively,
and we  adopt a similar convention for $\Cst$-algebras 
$\TT_{\red}(\Gr,E,\sigma)$,   
$\TT_{\ess}(\Gr,E,\sigma)$,
$\OO_{\red}(\Gr,E,\sigma)$ and  $\OO_{\ess}(\Gr,E,\sigma)$ covered by Definition~\ref{def:reduced_and_essential_self-similar-algebras}. 
We will model these algebras as relative Cuntz--Pimsner algebras.

We recall that a \emph{$\Cst$-correspondence} from a $\Cst$-algebra $A$ to a $\Cst$-algebra $B$ is a right Hilbert $B$-module $X$ together with a left action of $A$ implemented by
a $*$-homomorphism $\phi:A\to \LL(X)$ into the $\Cst$-algebra of adjointable operators on $X$, see \cites{lance, Brix-Mundey-Rennie:Graph-moves}.
A \emph{frame} for a right Hilbert $A$-module $X$ is a family $\{x_i\}_{i\in I}\subseteq X$ such that for each $\xi \in X$, we have $\xi = \sum_{i} x_i \cdot \langle x_i \mid \xi \rangle_A$ with convergence in norm. 
When $A=B$ we say $X$ is a $\Cst$-correspondence over $A$. We will unify the following two examples:
\begin{example}
The  \emph{graph correspondence} $X(E)$ of $E=(E^0,E^1,\rg,\sr)$  is a $\Cst$-correspondence over $\Cont_0(E^0)$. It is the completion of $\Contc(E^1)$ in the norm induced by the $\Cont_0(E^0)$-valued inner product
\[
\langle \xi \mid \eta \rangle (v) = \sum_{\sr(e) = v} \overline{\xi(e)}\eta(e), \quad \text{for } \xi,\,\eta \in C_c(E^1) \text{ and } v \in E^0,
\]
with left and right actions of $\Cont_0(E^0)$ determined by
\[
a \cdot \xi \cdot b (e) = a(\rg(e)) \xi(e) b(\sr(e)), \quad \text{for } \xi \in C_c(E^1),\, a,b \in \Cont_0(E^0), \text{ and } e \in E^1.
\]
The point mass functions on edges $\{\bone_e\}_{e \in E^1}$
form a frame for $X(E)$.
\end{example}
\begin{example}
Let $\Cst_{\lambda}(\Gr,\sigma_G)$ be a $\Cst$-algebra obtained as a Hausdorff completion    of the $\sigma_G$-twisted convolution $*$-algebra $\Contc(\Gr,\sigma_G)$
in some $\Cst$-seminorm $\|\cdot\|_{\lambda}$. We may  treat $\Cst_{\lambda}(\Gr,\sigma_G)$ as a trivial $\Cst$-correspondence over itself, with the inner product $\langle a \mid b \rangle \coloneqq a^*b$, $a,b\in \Cst_{\lambda}(\Gr,\sigma_G)$. 
Let $\{\bone_g\}_{g \in \Gr}$ be point mass functions on arrows of $\Gr$. The image of $\{\bone_x\}_{x \in \Gr^0}$ in $\Cst_{\lambda}(\Gr,\sigma_G)$
is a frame for $\Cst_{\lambda}(\Gr,\sigma_G)$.
\end{example}
\begin{remark}\label{rem:tensor_correspondence}
We may link the above examples using that $E^0=\Gr^0$.  Namely, the inclusion $\Contc(E^0)\subseteq \Contc(\Gr,\sigma_G)$ 
induces a $*$-homomorphism from  $\Cont_0(E^0)$ to a Hausdorff completion $\Cst_{\lambda}(\Gr,\sigma_G)$ of $\Contc(\Gr,\sigma_G)$.
Thus, we may  view $\Cst_{\lambda}(\Gr,\sigma_G)$  as a $\Cst$-correspondence from  $\Cont_0(E^0)$ to $\Cst_{\lambda}(\Gr,\sigma_G)$, and we may form 
the (internal) tensor product
$$
X(E) \ox_{\Cont_0(E^0)} \Cst_\lambda(\Gr,\sigma_G).
$$
This is naturally  a $\Cst$-correspondence from  $\Cont_0(E^0)$ to $\Cst_{\lambda}(\Gr,\sigma_G)$.
It follows from \cite{Brix-Mundey-Rennie:Graph-moves}*{Proposition~2.16} that the image of 
$\{\bone_{e}\otimes \bone_{\sr(e)}\}_{e \in E^1}$ in $X(E) \ox_{\Cont_0(E^0)} \Cst_\lambda(\Gr,\sigma_G)$ is a frame.
This frame is \emph{orthogonal} in the sense that 
$$
\langle \bone_{e}\otimes \bone_{\sr(e)} \mid \bone_{f}\otimes \bone_{\sr(f)}\rangle = [e=f]\bone_{\sr(e)}, \qquad e,f\in E^{1}.
$$
We may use it to define a left action of $\Cst_{\lambda}(\Gr,\sigma_G)$ on  $X(E) \ox_{\Cont_0(E^0)} \Cst_\lambda(\Gr,\sigma_G)$,
whenever the completion $\Cst_{\lambda}(\Gr,\sigma_G)$ is ``self-similar'' in the following sense.
\end{remark}
\begin{definition}
An algebraic \emph{correspondence over a pre-$\Cst$-algebra $A_0$}, cf. \cite{Nekrashevych:Cuntz--Pimsner}*{Definition 3.1}, is an $A_0$-bimodule $X_0$ together with a right $A_0$-valued inner product    such that $\langle a\xi, \eta\rangle=\langle \xi, a^*\eta\rangle$, for $a\in A_0$, $\xi, \eta\in X_0$. For any $\Cst$-seminorm $\|\cdot\|_{\lambda}$ on $A_0$ we may consider Hausdorff completions  $A_{\lambda}$ and $X_{\lambda}$  of $A$ and $X$ in $\|\cdot\|_{\lambda}$ and $\sqrt{\|\langle \cdot,\cdot \rangle\|_{\lambda}}$, respectively. This  produces a right Hilbert $A_{\lambda}$-module  $X_{\lambda}$, see \cite{lance}*{page 4}, which may fail to be a $\Cst$-correspondence over $A_\lambda$ in the sense that the left action of $A_0$ on $X_0$ may not induce the  left action of $A_\lambda$ on $X_{\lambda}$. Extending \cite{Nekrashevych:Cuntz--Pimsner}*{Definition 3.4}, we  say that the $\Cst$-seminorm $\|\cdot\|_{\lambda}$ is \emph{self-similar} for $X_0$ if $X_{\lambda}$ is a $\Cst$-correspondence over $A_{\lambda}$, that is if
$\|\langle a\xi, a\xi\rangle\|_{\lambda}\leq\|a\|_{\lambda}^2\cdot \|\langle \xi,\xi \rangle\|_{\lambda}$ for $a\in A_0$, $\xi\in X_0$.
\end{definition}
\begin{lemma}\label{lem:self-similar_vs_positive}
Let $X$  be a $C^*$-correspondence over $A$, which is a completion of a correspondence $X_0$  over $A_0$.
Let $A_\lambda$ be a Hausdorff completion of $A_0$ in a $\Cst$-seminorm $\|\cdot\|_{\lambda}$ not exceeding the one on $A$. 
Then $\|\cdot\|_{\lambda}$ is self-similar if and only if the kernel $I$ of the canonical  $*$-epimorphism 
$A\onto A_\lambda$ is positively $X$-invariant, i.e. $IX\subseteq XI$.
%
\end{lemma}
\begin{proof}
Note that $A_\lambda\cong A/I$ and recall that $X/X I$ is naturally  a right Hilbert $A/I$-module with the structure induced by 
the quotient maps $q^{XI}:X\to X/XI$   and $q^{I}:A\onto A/I$,  cf. \cite{fmr}*{Lemma 2.1}. In particular, 
for $f\in X_0$ we get $\|f\|_{\lambda}^2=\|\langle f, f\rangle \|_{\lambda}=\|q^I(\langle f, f\rangle) \|= \|q^{XI}(f)\|^2$. Thus,  $X_0\ni f \mapsto q^{XI}(f)\in X/X I$
induces an isometry  $X_{\lambda}\to X/XI$, and it is  easy to see that in fact it is an isomorphism of Hilbert modules $X_\lambda\cong X/X I$.
The left action of $A$ on $ X/X I$ descends to a well-defined action of $A_\lambda\cong A/I$ if and only if $I$ is positively $X$-invariant, cf. \cite{fmr}*{Lemma 2.3}
\end{proof}
Below we 
use a convention 
that an expression \([sentence]\) is zero if
the \(sentence\) is false and 1 otherwise. We denote by $\{\bone_{e,g}\}_{(e,g)\in E^1 \fibre{\sr}{\rg} \Gr}$
the obvious linear basis for $\Contc(E^1 \fibre{\sr}{\rg} \Gr)$.
\begin{proposition}\label{prop:universal_completion_is_self-similar}
The space $\Contc(E^1 \fibre{\sr}{\rg} \Gr)$ is an algebraic correspondence over $\Contc(\Gr,\sigma_G)$
with operations given on basis elements by the formulas
\begin{align}%
\bone_{e,g} \cdot \bone_h &= [\rg(h) = \sr(g)]  \sigma_G(g,h) \bone_{e,gh}, \label{equ:universal_completion_is_self-similar1}
\\
\langle \bone_{e,g} \mid \bone_{f,h}\rangle&=[e = f] \ol{\sigma_G(g,g^{-1}h)} \ \bone_{g^{-1}h}, \label{equ:universal_completion_is_self-similar2}
\\
\bone_g \cdot \bone_{e,h}&=[\sr(g)=\rg(e)] \sigma_{\bowtie}(g,e)\sigma_{\Gr}(g|_{e},h) \bone_{ge,g|_eh} \label{equ:universal_completion_is_self-similar3}
\end{align}
for all $(e,g), (f,h)\in E^1 \fibre{\sr}{\rg} \Gr$.  
The maximal $\Cst$-norm on $\Contc(\Gr,\sigma_G)$ is self-similar, that is $\Contc(E^1 \fibre{\sr}{\rg} \Gr)$ completes to a $\Cst$-correspondence $X(\Gr,E,\sigma)$
over $\Cst(\Gr,\sigma_G)$.  The map $\bone_{e,g}\mapsto \bone_{e}\otimes \bone_{g}$ extends to an isomorphism of right Hilbert $\Cst$-modules
\begin{equation}\label{eq:universal_correspondence_isomosrphism}
X(\Gr,E,\sigma) \cong X(E) \ox_{\Cont_0(E^0)} \Cst(\Gr,\sigma_G).
\end{equation}
\end{proposition}
\begin{proof}
It is straightforward to see that the map $\bone_{e,g}\mapsto \bone_{e}\otimes \bone_{g}$ extends to a linear isomorphism 
$\Contc(E^1 \fibre{\sr}{\rg} \Gr)\stackrel{\cong}{\rightarrow} \operatorname{span}\{\bone_{e}\otimes \bone_{g} \mid (e,g) \in E^1\fibre{\sr}{\rg}\Gr \}\subseteq X(E) \ox_{\Cont_0(E^0)} \Cst(\Gr,\sigma_G)$. 
Thus, we use it to identify $\Contc(E^1 \fibre{\sr}{\rg} \Gr)$ with the dense subspace of $X(E) \ox_{\Cont_0(E^0)} \Cst(\Gr,\sigma_G)$.
Under this identification one readily sees that \eqref{equ:universal_completion_is_self-similar1} holds. 
For all $(e,g) , (f,h)\in E^1 \fibre{\sr}{\rg} \Gr$ with $\rg(h) = \rg(g)$ the $2$-cocycle conditions \eqref{eq:groupoid_cocycle_identities} for $\sigma_G$ give
$
\sigma_{G}(g^{-1},g)=\sigma_{G}(g,g^{-1})=\sigma_G(g,g^{-1}h)\sigma_G(g^{-1},h)
$.
Hence,
\begin{align}\label{eq:inner_product_alternativ}
\langle \bone_{e,g} \mid \bone_{f,h}\rangle
&=[e = f]    \bone_{g}^* \bone_{h}
=[e = f] \ol{\sigma_G(g^{-1},g)} \sigma_G(g^{-1}, h) \bone_{g^{-1}h} \\
&=[e = f] \ol{\sigma_G(g,g^{-1}h)} \ \bone_{g^{-1}h}. \nonumber
\end{align}
This shows that formulas  \eqref{equ:universal_completion_is_self-similar1} and \eqref{equ:universal_completion_is_self-similar2} determine the structure 
of pre-Hilbert $\Contc(\Gr,\sigma_G)$-module on $\Contc(E^1 \fibre{\sr}{\rg} \Gr)$  whose completion induced by the universal norm on $\Contc(\Gr,\sigma_G)$
is a Hilbert $\Cst(\Gr,\sigma_G)$-module $X(\Gr,E,\sigma)$ for which the isomorphism \eqref{eq:universal_correspondence_isomosrphism} holds.

Thus, we identify $X(\Gr,E,\sigma)$ with $X(E) \ox_{\Cont_0(E^0)} \Cst(\Gr,\sigma_G)$. We need to show that 
\eqref{equ:universal_completion_is_self-similar3} determines a homomorphism $\Cst(\Gr,\sigma_G)\to \LL(X(\Gr,E,\sigma))$.
By  universality of $\Cst(\Gr,\sigma_G)$, this boils down to showing that  \eqref{equ:universal_completion_is_self-similar3}
determines a $\sigma_{\Gr}$-twisted unitary representation  $g \mapsto U_g$ of $\Gr$ in $\LL(X(G,E,\sigma))$.
To this end, we use the orthogonal frame $\{\bone_{e,\sr(e)}\}_{e \in E^1}$  for $X(G,E,\sigma)$, cf.  Remark~\ref{rem:tensor_correspondence}. 
For each $g\in \Gr$,  the multiplication by $\bone_{g}$, given by \eqref{equ:universal_completion_is_self-similar3}, defines a linear operator $U_g: \Contc(E^1 \fibre{\sr}{\rg} \Gr)\to\Contc(E^1 \fibre{\sr}{\rg} \Gr)$ which is a right $\Contc( \Gr,\sigma_{\Gr})$-module map. In fact, we have
\begin{equation}\label{eq:unitary_rep}
U_g(\eta) 
=  \sum_{e \in \sr(g)E^1} \sigma_{\bowtie}(g,e) \bone_{g \cdot e,g|_e} \cdot \langle \bone_{e,\sr(e)} \mid \eta \rangle, \qquad \eta \in \Contc(E^1 \fibre{\sr}{\rg} \Gr). 
\end{equation}
In particular, $U_g$ is uniquely determined by its values on the frame   $\{\bone_{e,\sr(e)}\}_{e \in E^1}$ --
it is a unique linear  right $\Contc( \Gr,\sigma_{\Gr})$-module map such that    
$U_g(\bone_{e,\sr(e)})=  \bone_g \cdot \bone_{e,\sr(e)}=[\sr(g)=\rg(e)] \sigma_{\bowtie}(g,e)\bone_{ge,g|_{e}}$.
For any $(g,h)\in \Gr^2$ and $e\in E^1$ we have	 
\begin{align*}
U_{g}U_{h} \bone_{e,\sr(e)}
&\stackrel{\eqref{equ:universal_completion_is_self-similar3}}{=} U_{g} [\sr(h)=\rg(e)] \sigma_{\bowtie}(h,e) \bone_{he,h|_e}
\\
&\stackrel{\eqref{equ:universal_completion_is_self-similar3}}{=}[\sr(h)=\rg(e)] \sigma_{\bowtie}(h,e)\sigma_{\bowtie}(g,he) \sigma_{\Gr}(g|_{he},h|_{e}) \bone_{ghe,g|_{he}h|_e} 
\\
&\stackrel{\eqref{eq:ssc_edges_G}}{=}[\sr(gh)=\rg(e)] \sigma_{\bowtie}(gh,e)  \sigma_{\Gr}(g,h) \bone_{ghe,gh|_{e}}
\\
&\stackrel{\eqref{equ:universal_completion_is_self-similar3}}{=} \sigma_{\Gr}(g,h) U_{gh} \bone_{e,\sr(e)}.
\end{align*}
Hence, $U_{g}U_{h}=\sigma_{\Gr}(g,h) U_{gh}$. Moreover, for any $e,f\in E^1$
\begin{align*}
\langle U_g ( \bone_{e,\sr(e)}) \mid  \bone_{f,\sr(f)}\rangle
&\stackrel{\eqref{equ:universal_completion_is_self-similar3}}{=} [\sr(g)=\rg(e)] \ol{\sigma_{\bowtie}(g,e)}  \langle \bone_{ge,g|_e} \mid \bone_{f,\sr(f)} \rangle 
\\
&\stackrel{\eqref{equ:universal_completion_is_self-similar2}}{=} [\sr(g)=\rg(e)] [ge=f]\ol{\sigma_{\bowtie}(g,e)}\, \ol{\sigma_G(g|_e,(g|_e)^{-1})}   \bone_{(g|_e)^{-1}}  
\\
&
\stackrel{\eqref{equ:inverse_vs_restriction}}{=} [\sr(g)=\rg(e)] [ge=f]\ol{\sigma_{\bowtie}(g,e)}\, \ol{\sigma_G(g|_e,g^{-1}|_{ge})}   \bone_{g^{-1}|_{ge}}.
\\
&\stackrel{\eqref{eq:ssc_edges_G}}{=}
[\sr(g^{-1})=\rg(f)] [e=g^{-1}f]\ol{\sigma_{\Gr}(g,g^{-1}) }\sigma_{\bowtie}(g^{-1},f)  \bone_{g^{-1}|_f}  
\\
&\stackrel{\eqref{equ:universal_completion_is_self-similar2}}{=} [\sr(g^{-1})=\rg(f)]\ol{\sigma_{\Gr}(g,g^{-1}) } \sigma_{\bowtie}(g^{-1},f)  \langle \bone_{e,\sr(e)} \mid \bone_{g^{-1}f,g^{-1}|_f}  \rangle 	
\\
&\stackrel{\eqref{equ:universal_completion_is_self-similar3}}{=} \langle   \bone_{e,\sr(e)}\mid  \ol{\sigma_{\Gr}(g,g^{-1}) }U_{g^{-1}}\bone_{f,\sr(f)}\rangle.
\end{align*}	
Using that $\{\bone_{e,\sr(e)}\}_{e \in E^1}$ is a frame, $\langle \cdot\mid\cdot\rangle$ is a $\Contc( \Gr,\sigma)$-valued sesquilinear form, and $U_g$ and $U_{g^{-1}}$ are right $\Contc( \Gr,\sigma)$-module maps, one concludes that 
\begin{align*}
\langle U_g (\zeta) \mid \eta \rangle
=  \langle \zeta \mid  \ol{\sigma_{\Gr}(g,g^{-1}) }U_{g^{-1}} \eta \rangle, \qquad\text{for any $\zeta,\eta \in \Contc(E^1 \fibre{\sr}{\rg} \Gr)$}.
\end{align*} 
The relations we proved, show that for any $\eta \in \Contc(E^1 \fibre{\sr}{\rg} \Gr)$ we have
\begin{align*}
\langle U_g (\eta) \mid U_g (\eta) \rangle&=
\langle U_g (\eta) \mid \ol{\sigma_{\Gr}(g,g^{-1}) }U_{g^{-1}}U_g (\eta) \rangle=\langle  \eta \mid U_{\sr(g)} (\eta) \rangle.
\end{align*}	
Since  $U_{\sr(g)}$ is an orthogonal projection onto the subspace $\Contc(\sr(g)E^1 \fibre{\sr}{\rg} \Gr)$ in $\Contc(E^1 \fibre{\sr}{\rg} \Gr)$,
it extends to the contractive projection in $\LL(X(G,E,\sigma))$. Thus, it follows that $U_{g}$ extends to an adjointable partial isometry on $X(G,E,\sigma)$
and the map $\Gr\ni g \mapsto U_g\in \LL(X(G,E,\sigma))$  is a $\sigma_{\Gr}$-twisted unitary representation   of $\Gr$.
\end{proof}
\begin{remark}
The correspondence structure on  $\Contc(E^1 \fibre{\sr}{\rg} \Gr)$ in terms of functions $\xi, \eta\in\Contc(E^1 \fibre{\sr}{\rg} \Gr)$, $a\in \Contc(\Gr,\sigma_G)$,
is given by 
the formulas, for $(e,g)\in E^1 \fibre{\sr}{\rg} \Gr$,
\begin{align*}
(\xi \cdot a)(e,g)&=\sum_{h\in \Gr \sr(g)} \sigma_{\Gr}(gh^{-1},h) \xi(e,gh^{-1})a(h),
\\
\langle \xi \mid \eta\rangle (g) &= \sum_{ 
	(e,h)\in  E^1 \fibre{\sr}{\rg} \Gr \sr(g)}  \overline{\sigma_G(hg^{-1},g)}\, \overline{\xi(e , hg^{-1})} \eta(e,h),
\\
(a \cdot \xi) (e,g) 
&=  \sum_{h\in  \rg(g)\Gr} \sigma_{\bowtie}(g,e)\sigma_{G}(h|_{h^{-1}e}, (h|_{h^{-1}e})^{-1} g)  a(h) \xi(h^{-1}e,(h|_{h^{-1}e})^{-1} g).
\end{align*}
However, as we noticed in the proof above, it is determined by its values on the frame  $\{\bone_{e,\sr(e)}\}_{e \in E^1}$
and then even the formulas \eqref{equ:universal_completion_is_self-similar1}-\eqref{equ:universal_completion_is_self-similar3} simplify. 
\end{remark}

\begin{corollary}\label{cor:self-similar_completions}
Let $\Cst_{\lambda}(\Gr,\sigma_G)$ be a Hausdorff completion of  $\Contc(\Gr,\sigma_G)$ in a $\Cst$-seminorm $\|\cdot\|_{\lambda}$. 
Let $X_{\lambda}(\Gr,E,\sigma)$ be the corresponding Hausdorff completion of $\Contc(E^1 \fibre{\sr}{\rg} \Gr)$, 
and let $I_\lambda$ be the kernel of the canonical  $*$-epimorphism $\Cst(\Gr,\sigma_G)\to \Cst_{\lambda}(\Gr,\sigma_G)$.
We have natural isomorphisms of  right Hilbert $\Cst_{\lambda}(\Gr,\sigma_G)$-modules 
$$
X_{\lambda}(\Gr,E,\sigma)\cong X(\Gr,E,\sigma)/X(\Gr,E,\sigma) I_{\lambda}  \cong X(E) \ox_{\Cont_0(E^0)} \Cst_\lambda(\Gr,\sigma_G) 
$$
and $X_{\lambda}(\Gr,E,\sigma)$ is naturally a $\Cst$-correspondence from $\Cst(\Gr,\sigma_G)$ to $\Cst_\lambda(\Gr,\sigma_G)$.
The seminorm $\|\cdot\|_{\lambda}$ is self-similar if and only if the ideal $I_\lambda$ is positively $X(\Gr,E,\sigma)$-invariant.
\end{corollary}
\begin{proof}
That the identity on $\Contc(E^1 \fibre{\sr}{\rg} \Gr)$ induces an isomorphism $X_{\lambda}(\Gr,E,\sigma)\cong   X(E) \ox_{\Cont_0(E^0)} \Cst_\lambda(\Gr,\sigma_G)$ is 
straightforward. In particular, for 
$\sum_{i=1}^n a_i\otimes b_i\in \Contc(E^1 \fibre{\sr}{\rg} \Gr)$ where $a_i\in 
\Contc(E^1)$ and  $b_i\in \Contc(\Gr)$ we have 
$
\|\sum_{i=1}^n a_i\otimes b_i\|_{\lambda}^2=\| \sum_{i,j=1}^n \langle a_i\otimes b_i, a_j\otimes b_j\rangle \|_{\lambda}=
\| \sum_{i,j=1}^n b_i^* \langle a_i, a_j\rangle b_j \|_{\lambda}
=\|\sum_{i=1}^n a_i\otimes b_i\|_{X(E) \ox \Cst_\lambda(\Gr,\sigma_G) }^2
$. The remaining part follows from Proposition~\ref{prop:universal_completion_is_self-similar} and (the proof of) Lemma~\ref{lem:self-similar_vs_positive}.
\end{proof}
Let $X$ be a $\Cst$-correspondence over a $\Cst$-algebra $A$. 
Recall that a \emph{representation  of  $X$}  in a $C^*$-algebra $B$   is a pair $(\pi,\psi)$ where  $\pi:A\to B$ is  a representation and  $\psi:X\to B$ is a linear (necessarily contractive) map  such that
$$
\psi(a\cdot x)=\pi(a)\psi(x), \quad \psi(x\cdot b)=\psi(x)\pi(b),\quad \psi(x)^*\psi(y)=\pi(\langle x, y\rangle_A),
$$
for all   $a,b\in A,\, x\in X$. 
When $B=\Bound(H)$ for a Hilbert space $H$, we say that $(\pi,\psi)$ is a representation on the Hilbert space $H$.
The $C^*$-algebra generated by $\pi(A)\cup \psi(X)$ is denoted by $C^*(\pi,\psi)$.
The \emph{Toeplitz algebra}  of $X$ is $\TT(X) \coloneqq C^*(i_A,i_X)$ where $(i_A,i_X)$ is a universal representation of $X$ in the sense that 
for any other representation $(\pi,\psi)$ the maps $i_A(a)\mapsto \pi(a)$ and $i_X(x)\mapsto \psi(x)$ determine a representation $\TT(X)\to  C^*(\pi,\psi)$.
Recall that (generalised) compact operators $\Comp(X)$ is the closed linear span of rank one operators on $X$ given by
$\Theta_{x,y}(z) \coloneqq x\langle y,z\rangle_A$, for $x,y,z\in X$. 
Any representation $(\pi,\psi)$ of $X$ in $B$ induces a representation $\psi^{(1)}:\Comp(X)\to B$ where $\psi^{(1)}(\Theta_{x,y})=\psi(x)\psi(y)^*$, $x,y\in X$.
For any ideal $J$  in $J(X) \coloneqq \phi^{-1}(\Comp(X))$, we say that a representation $(\pi,\psi)$ of $X$ is 
$J$-\emph{covariant} if $\pi(a)=\psi^{(1)}(\phi(a))$ for every $a\in J$. The associated  \emph{relative Cuntz--Pimsner algebra}  is 
$\OO(J,X) \coloneqq C^*(j_A,j_X)$
where  $(j_A,j_X)$ is a universal $J$-covariant representation    of $X$.
It is equipped with the \emph{gauge circle action} $\gamma:\T \to \Aut(\OO(J,X))$ 
determined by $\gamma_z(j_A(a))=j_A(a)$ and $\gamma_z(j_X(x))=zj_X(x)$, for $a\in A$, $x\in X$, $z\in \T$.
Note that $\TT(X)=\OO(\{0\},X)$.
\begin{proposition}\label{prop:correspondence_btw_correspondence_reps}
We have a bijective correspondence between representations  $(\pi,\psi)$ of  the $\Cst(\Gr,\sigma_{\Gr})$-correspondence $X(\Gr,E,\sigma)$    on a Hilbert space $H$
and 	$\sigma$-twisted representations $(W,T)$  of  $(\Gr,E)$ on $H$ given by 
\[
\pi(\bone_g) = W_g  \quad \text{and} \quad \psi(\bone_{e,\sr(e)}) = T_e \quad \text{for all } g \in \Gr \text{ and } e \in E^1. 
\]
Moreover, $(W,T)$ is covariant if and only if $(\pi,\psi)$ is $J_{\reg}$-covariant where 
$$
J_{\reg}  \coloneqq  \Cst(\Gr_{\reg}, \sigma_{\Gr}|_{\Gr_{\reg}})
$$ sits naturally as an ideal in  $C^*(\Gr,\sigma_{\Gr})$. 
\end{proposition}
\begin{proof} Let $(W,T)$ be representation of  $(\Gr,E,\sigma)$ on $H$. 
Since $\{W_g \mid g \in \Gr\}$ is a $\sigma_{G}$-twisted unitary representation of $\Gr$, the universal property of $C^*(\Gr,\sigma_{G})$ induces a $*$-homomorphism $\pi \colon C^*(\Gr,\sigma_{G}) \to \Bound(H)$ such that $\pi(\bone_g) = W_g$. 
Recall that $\{ \bone_{e,g} \mid (e,g)\in  E^1 * \Gr\}$ densely spans $X(\Gr,E,\sigma)$. Define $\psi$ on this dense set by
$
\psi (\bone_{e,g}) = T_e W_g. 
$
If $(e,g),(f,h) \in E^1* \Gr$, then
\begin{align*}
\psi(\bone_{e,\sr(e)} \cdot \bone_g)^* \psi(\bone_{f,\sr(f)} \cdot \bone_h) 
&= W_{g}^* T_e^* T_f W_h\stackrel{\textup{\ref{defn:Cuntz--Krieger1}}}{=}[e = f]  W_{g}^* W_{\sr(e)} W_h 
\\
&
\stackrel{\eqref{eq:inverse_of_W_g}}{=}[e = f]  \ol{\sigma_{G}(g^{-1},g)}W_{g^{-1}}W_h
\\
&
\stackrel{\textup{\ref{enu:reps_of_twisted_self_similar1}}}{=}[e = f]  \ol{\sigma_{G}(g^{-1},g)}\sigma_{G}(g^{-1},h)W_{g^{-1} h}
\\
&	 	\stackrel{\eqref{eq:inner_product_alternativ} }{=} 	\pi (\langle \bone_{e,g}   \mid \bone_{f,h} \rangle).
\end{align*}
It follows that $\psi$ extends to an isometric linear map $\psi \colon X(\Gr,E,\sigma) \to \Bound(H)$ satisfying $\psi(\xi)^*\psi(\eta) = \pi(\langle \xi \mid \eta \rangle)$ for all $\xi,\eta \in X(\Gr,E,\sigma)$. 
The definition of $\psi$ makes it clear that $\psi(\xi) \pi(a) = \psi(\xi \cdot a)$ for all $\xi \in X(\Gr,E,\sigma)$ and $a \in C^*(\Gr, \sigma_{G})$, cf. \eqref{equ:universal_completion_is_self-similar1} and \ref{enu:reps_of_twisted_self_similar1}, \ref{enu:reps_of_twisted_self_similar2}.

For the left action, observe that for $(g,e) \in \Gr * E^1$, 
\[
\pi(\bone_g) \psi(\bone_{e,\sr(e)}) = W_g T_e \stackrel{\textup{\ref{enu:reps_of_twisted_self_similar3}}}{=}  \sigma_{\bowtie}(g,e)T_{g e} W_{g|_e}
= \psi(\sigma_{\bowtie}(g,e)\bone_{g \cdot e,g|_{e}})\stackrel{\eqref{equ:universal_completion_is_self-similar3}}{=}\psi(\bone_g\cdot \bone_{e,\sr(e)}),
\] 
so $\pi(a) \psi(\xi) = \psi(a \cdot \xi)$ for all $a \in C^*(\Gr,\sigma_G)$ and $\xi \in X(\Gr,E,\sigma)$. 

Conversely, let $(\pi,\psi)$ be any representation of  $X(\Gr,E,\sigma)$  in $B(H)$ and  define $(W,T)$ by the displayed formula in the assertion.
Then the  map $g \mapsto W_g$ is a $\sigma_{\Gr}$-twisted unitary representation of $\Gr$ on $H$. In particular, the projections  $W_{v}$, $v\in E^0$, are mutually orthogonal. 
For  $e,f\in E^{1}$ we have
\[
T_e^*T_f=\psi(\bone_{e,\sr(e)})^*\psi( \bone_{f,\sr(f)})=\pi(\langle \bone_{e,\sr(e)},\bone_{f,\sr(f)}\rangle) \stackrel{\eqref{equ:universal_completion_is_self-similar2}}{=}[e=f]\pi(\bone_{\sr(e)})=[e=f]W_{\sr(e)}.
\]
Hence, $\{T_{e}\}_{e\in E^1}$ are partial isometries with orthogonal range projections and $T_e^*T_e=W_{\sr(e)}$, $e\in E^1$.
Moreover, for each $(g,e) \in \Gr * E^1$, we have
\begin{align*}
W_g T_e 
&= \pi(\bone_g)\psi(\bone_{e,\sr(e)})=\psi(\bone_g\cdot \bone_{e,\sr(e)}) 
\stackrel{\eqref{equ:universal_completion_is_self-similar3}}{=}  \psi(\sigma_{\bowtie}(g,e) \bone_{g \cdot e, g|_{e} })
\\
&=
\sigma_{\bowtie}(g,e)  \psi(\bone_{g \cdot e, \sr(g \cdot e)}) \pi(\bone_{g|_{e}}) = \sigma_{\bowtie}(g,e)  T_{g|_e} W_{g \cdot e}.
\end{align*}
It follows that $(W,T)$  is a representation of  $(\Gr,E,\sigma)$. This proves the first part of the assertion.

Recall that $\Gr=\Gr_{\reg}\sqcup \Gr_{\sing}$ decomposes into a disjoint union of groupoids, see page~\pageref{page:groupoid_decomposition_singular_regular}.
This implies that $\Cst(\Gr, \sigma_{\Gr})=\Cst(\Gr_{\reg}, \sigma_{\Gr}|_{\Gr_{\reg}})\oplus \Cst(\Gr_{\sing}, \sigma_{\Gr}|_{\Gr_{\sing}})$
decomposes into a direct sum of $C^*$-algebras.
The expression \eqref{eq:unitary_rep} defining the representation $U$ of $(\Gr, \sigma_{\Gr})$ in  $\LL(X(\Gr,E,\sigma))$, which is in fact the left action of the unitaries $\{\bone_{g}\}_{g\in \Gr}$,  implies that 
\begin{equation}\label{eq:compact_sum_rep}
\phi(\bone_{g}) = \sum_{e \in \sr(g)E^1} \sigma_{\bowtie}(g,e) \Theta_{\bone_{g  e, g|_{e}}, \bone_{e,\sr(e)}}
\end{equation}
where the series is strongly convergent.
If $g \in \Gr_{\reg}$, then the sum in \eqref{eq:compact_sum_rep} is finite and so  $\phi(\bone_{g})\in \Comp(X(\Gr,E,\sigma)$ is compact. 
Therefore, $J_{\reg}=\Cst(\Gr_{\reg}, \sigma_{\Gr}|_{\Gr_{\reg}}) \subseteq \phi^{-1}(\Comp(X(\Gr,E,\sigma))$. 
Also for the corresponding representations  $(W,T)$ and  $(\pi,\psi)$, for any $g \in \Gr_{\reg}$ we have 
\begin{align*}
\psi^{(1)}(\phi(\bone_{g}))&=\sum_{e \in \sr(g)E^1} \sigma_{\bowtie}(g,e) \psi(\bone_{g  e, g|_{e}}) \psi(\bone_{e,\sr(e)})^*
\\
&=\sum_{e \in \sr(g)E^1} \sigma_{\bowtie}(g,e) T_{ge} W_{g|_{e}} T_{e}^*  
\\
&\stackrel{\ref{enu:reps_of_twisted_self_similar3}} {=}	W_{g} \sum_{e \in \sr(g)E^1} T_{e} T_{e}^*=\pi(\bone_{g}) \sum_{e \in \sr(g)E^1} T_{e} T_{e}^*.
\end{align*}
Hence, we see that $(W,T)$ is covariant if and only if $\psi^{(1)}(\phi(\bone_{g}))=\pi(\bone_{g})$ for all $g\in \Gr_{\reg}$ if 
and only if $(\pi,\psi)$ is  $J_{\reg}$-covariant.
\end{proof}
\begin{corollary}\label{cor:Cuntz--Pimsner_picture}
For any twisted self-similar action $(\Gr,E,\sigma)$ we have natural gauge-invariant isomorphisms
$$
\TT(\Gr,E,\sigma) \cong\TT(X(\Gr,E,\sigma)),\qquad  \OO(\Gr,E,\sigma)\cong \OO(X(\Gr,E,\sigma),J_{\reg}).
$$ 
For each $*=\red, \ess$, the above isomorphisms descend to isomorphisms 
of $\TT_{*}(\Gr,E,\sigma)$ and  $\OO_{*}(\Gr,E,\sigma)$ with some relative Cuntz--Pimsner algebras
associated to self-similar Hausdorff completions of $\Cont_c(\Gr,\sigma_{G})$. In particular,  
\begin{enumerate}  
\item\label{thm:Cuntz--Pimsner_picture1}  $
\TT_{\red}(\Gr,E,\sigma)$ is a relative Cuntz--Pimsner algebra of a $C^*$-correspondence  $X_{\tc,\red}(\Gr,E,\sigma)$ associated to an  exotic self-similar  completion   $C^*_{\tc,\red}(\Gr,\sigma_{G})$  of $\Cont_c(\Gr,\sigma_{G})$ and an ideal in the kernel of the canonical homomorphism $C^*_{\tc,\red}(\Gr,\sigma_{G})\onto C^*_{\red}(\Gr,\sigma_{G})$.

\item\label{thm:Cuntz--Pimsner_picture2}  If   $E$ is row-finite, then $\TT_{\ess}(\Gr,E,\sigma)$ is a relative Cuntz--Pimsner algebra of a $C^*$-correspondence  $X_{\tc,\ess}(\Gr,E,\sigma)$ over an  exotic self-similar  completion $C^*_{\tc,\ess}(\Gr,\sigma_{G})$  of $\Cont_c(\Gr,\sigma_{G})$ and 
an ideal in the kernel of  $C^*_{\tc,\ess}(\Gr,\sigma_{G})\onto C^*_{\red}(\Gr,\sigma_{G})$.

\end{enumerate}

\end{corollary}
\begin{proof}
The first isomorphisms follow from Proposition~\ref{prop:correspondence_btw_correspondence_reps}  
and the universal description  of the considered algebras.
By Remark~\ref{rem:gauge_actions_on_reduced_essential} the algebras  $\TT_{*}(\Gr,E,\sigma)$ and  $\OO_{*}(\Gr,E,\sigma)$, for  $*=\red,\ess$,
are equipped with canonical gauge-actions. Suppose that $(W,T)$ is a representation of $(\Gr,E,\sigma)$ generating one of these algebras. Let
$$
X_{W,T} \coloneqq \clsp\{T_{e}W_g: (e,g)\in E^{1}*\Gr\}, \qquad C^*_{W,T}(\Gr,\sigma_{G}) \coloneqq \clsp\{W_{g} : g\in \Gr\},
$$
and 
$J_{W,T} \coloneqq C^*_{W,T}(\Gr,\sigma_{G})\cap 	\overline{X_{W,T}	X_{W,T}^*}$ where 
$$
\overline{X_{W,T}	X_{W,T}^*}=\clsp\{T_{e}W_g T_f^*: (e,g,f) \in  E^{1}\fibre{\sr}{\rg}\Gr\fibre{\sr}{\sr}E^{1}\}.
$$ Then $X_{W,T}$ is a  $C^*$-correspondence over $C^*_{W,T}(\Gr,\sigma_{G})$ and $J_{W,T}$ is an ideal in $C^*_{W,T}(\Gr,\sigma_{G})$ that acts faithfully and by compacts on the left
of 	$X_{W,T}$. Hence, the relative Cuntz--Pimsner algebra $\OO(J_{W,T},X_{W,T})$ is isomorphic to the 
given algebra by the gauge-invariance uniqueness theorem, see \cite{Kakariadis} for instance.

\ref{thm:Cuntz--Pimsner_picture1}.  Assume that $(W,T)=(\tw^{\rd,2},\ttt^{\rd,2})$ is the representation on  $\ell^2(\widetilde{\Gg}(\Gr,E))$ that generates  $\TT_{\red}(\Gr,E,\sigma)$.
By Proposition~\ref{prop:embeddings_of_algebras}\ref{enu:embeddings_of_algebras2},  $C^*_{W,T}(\Gr,\sigma_{G})$ is 
an exotic $C^*$-algebra of $\Cont_c(\Gr,\sigma_{G})$. Note that
denoting by $P_{g}$ the orthogonal projection onto the one dimensional space  $\mathbb{C}\bone_{[\rg(g),g,\sr(g);\sr(g)]}$ we get 
that the formula 
$
E(b)=\sum_{g\in \Gr}P_{g}bP_{g}
$,  $b\in C^*_{W,T}(\Gr,\sigma_{G})$, defines a canonical conditional expectation $C^*_{W,T}(\Gr,\sigma_{G})\onto c_0(\Gr^0)$.
Since every $T_f^*$, for $f\in E^1$ kills every  $\bone_{[\rg(g),g,\sr(g);\sr(g)]}$ we see that 
$J_{W,T} \subseteq \ker E$. As the kernel of $C^*_{W,T}(\Gr,\sigma_{G})\onto C^*_{\red}(\Gr,\sigma_{G})$
is the largest ideal in $\ker E$ the assertion follows. 

\ref{thm:Cuntz--Pimsner_picture2}. The same argument as in the proof \ref{thm:Cuntz--Pimsner_picture1} works,
as for row-finite graphs elements $[\rg(g),g,\sr(g);\sr(g)]$ are in $\widetilde{\Gg}(\Gr,E)_{\Hau}$, cf. the proof of Proposition~\ref{prop:embeddings_of_algebras}\ref{enu:embeddings_of_algebras2}.
\end{proof}

\begin{theorem}\label{thm:Toeplitz_nuclearity}
Let $(\Gr,E,\sigma)$ be a twisted self-similar groupoid action.
The following conditions are equivalent:
\begin{enumerate}
\item\label{enu:Toeplitz_nuclearity1} $\Gr$ is amenable;
\item\label{enu:Toeplitz_nuclearity2} $\widetilde{\Gg}(\Gr,E)$ is amenable;
\item\label{enu:Toeplitz_nuclearity3} $\widetilde{\Gg}(\Gr,E)_{0}$ is amenable;
\item\label{enu:Toeplitz_nuclearity4} $\TT(\Gr,E,\sigma)$ is nuclear; 
\item\label{enu:Toeplitz_nuclearity5} $\TT(\Gr,E,\sigma)_0$ is nuclear; 
\item\label{enu:Toeplitz_nuclearity6} $C^*(\Gr,\sigma_{\Gr})$ is nuclear;
\item\label{enu:Toeplitz_nuclearity7} $C^*_{\red}(\Gr,\sigma_{\Gr})$ is nuclear; 
\item\label{enu:Toeplitz_nuclearity8} $\TT_{\red}(\Gr,E)$ is nuclear;
\item\label{enu:Toeplitz_nuclearity9} $\TT_{\red}(\Gr,E)_0$ is nuclear.  
\end{enumerate}
Assume these equivalent conditions hold. Then $\TT_{\red}(\Gr,E,\sigma)_{*}=\TT(\Gr,E,\sigma)_{*}$ for $*=\Space\, ,0$, and  these algebras as well as  $\TT_{\ess}(\Gr,E,\sigma)_{*}$,
for $*=\Space\, ,0$, are nuclear. 
Moreover,  $\TT(\Gr,E,\sigma)$ satisfies the UCT and in fact is $KK$-equivalent to $C^*(\Gr,\sigma_{\Gr})$.
If, in addition, the graph $E$ is row-finite, then 
$$
\TT(\Gr,E,\sigma)=\TT_{\ess}(\Gr,E,\sigma).
$$
\end{theorem} 
\begin{proof}
Assume \ref{enu:Toeplitz_nuclearity1}. 
Then $C^*(\Gr,\sigma_{\Gr})=C^*_{\red}(\Gr,\sigma_{\Gr})$ and they are nuclear by  \cite{BKMS}*{Theorem 11.7}  
In particular, in the notation of Corollary~\ref{cor:Cuntz--Pimsner_picture}\ref{thm:Cuntz--Pimsner_picture1}
we have $C^*_{\tc,\red}(\Gr,\sigma_{G})=C^*(\Gr,\sigma_{\Gr})=C^*_{\red}(\Gr,\sigma_{G})$ and therefore  $X_{\tc,\red}(\Gr,E,\sigma)=X(\Gr,E,\sigma)$,
cf. Corollary~\ref{cor:self-similar_completions}.  Thus,
$\TT(\Gr,E,\sigma)= \TT_{\red}(\Gr,E,\sigma)$ as they are both the Toeplitz algebra $\TT(X(\Gr,E,\sigma))$.
If $E$ is row-finite then the same argument, and Corollary~\ref{cor:Cuntz--Pimsner_picture}\ref{thm:Cuntz--Pimsner_picture2},
gives also $\TT(\Gr,E,\sigma)=\TT_{\ess}(\Gr,E,\sigma)$.
Applying \cite{Katsura}*{Theorem 7.2} to this Toeplitz algebra presentation we get 
that $\TT(\Gr,E,\sigma)= \TT_{\red}(\Gr,E,\sigma)$ and  $\TT(\Gr,E,\sigma)_0= \TT_{\red}(\Gr,E,\sigma)_0$ are nuclear.
Concluding \ref{enu:Toeplitz_nuclearity1} implies  the conditions  \ref{enu:Toeplitz_nuclearity4}--\ref{enu:Toeplitz_nuclearity9}.

By Corollary~\ref{cor:Cuntz--Pimsner_picture}, $\TT(\Gr,E,\sigma) \cong\TT(X(\Gr,E,\sigma))$. Hence, \ref{enu:Toeplitz_nuclearity4}--\ref{enu:Toeplitz_nuclearity6}
are equivalent by \cite{Katsura}*{Theorem 7.2}. They imply each of the conditions \ref{enu:Toeplitz_nuclearity7}--\ref{enu:Toeplitz_nuclearity9}, as nuclearity passes to quotients. By \cite{Takeishi}*{Theorem 5.4}, \ref{enu:Toeplitz_nuclearity7} and \ref{enu:Toeplitz_nuclearity1} are equivalent. 
Hence, \ref{enu:Toeplitz_nuclearity1} and \ref{enu:Toeplitz_nuclearity4}--\ref{enu:Toeplitz_nuclearity7} are all equivalent.
Note that they are independent of twist since \ref{enu:Toeplitz_nuclearity1} is. Hence, 
they are equivalent to nuclearity of $\TT(\Gr,E)$ which is the groupoid $C^*$-algebra $C^*(\widetilde{\Gg}(\Gr,E))$ by Theorem  \ref{thm:presentations_of_twisted_self_similar_algebras}. Therefore, these conditions are equivalent to \ref{enu:Toeplitz_nuclearity1} and to \ref{enu:Toeplitz_nuclearity8} by 
\cite{Buss_Martinez}*{Theorem A} or \cite{Brix-Gonzales-Hume-Li:Hausdorff_covers}*{Theorem F}. Conditions \ref{enu:Toeplitz_nuclearity8} and \ref{enu:Toeplitz_nuclearity9}
are equivalent  because we have a conditional expectation $\TT_{\red}(\Gr,E)\onto \TT_{\red}(\Gr,E)_0$, cf. Remark~\ref{rem:gauge_actions_on_reduced_essential}.
Then \ref{enu:Toeplitz_nuclearity9} and \ref{enu:Toeplitz_nuclearity3} are equivalent again by  
\cite{Buss_Martinez}*{Theorem A} or \cite{Brix-Gonzales-Hume-Li:Hausdorff_covers}*{Theorem F}, cf. Remark~\ref{remembeddings_of_algebras}.
Alternatively, one can get the equivalence of  \ref{enu:Toeplitz_nuclearity3} and \ref{enu:Toeplitz_nuclearity2} by \cite{Miller_Steinberg}*{Proposition 2.17(5)}.
This proves the equivalence of all the conditions \ref{enu:Toeplitz_nuclearity1}--\ref{enu:Toeplitz_nuclearity9}, as well as the second part of the assertion.
\end{proof}

We highlight the following open problems.

\begin{question}\label{problem:nuclearity}
Let $(\Gg,\LL)$ be a   non-Hausdorff, say second countable, twisted \'etale groupoid.
Show that nuclearity of $C^*_{\red}(\Gg,\LL)$, nuclearity of $C^*(\Gg,\LL)$, and  amenability of $\Gg$ are all equivalent, and if they hold, then
   $C^*_{\red}(\Gg,\LL)=C^*(\Gg,\LL)$. If $C^*_{\red}(\Gg,\LL)$ is nuclear, does it always satisfy the UCT?
\end{question}

\begin{remark}\label{rem:non-Hausdorff_nuclear}
The UCT part of Problem~\ref{problem:nuclearity} asks for generalisation of Tu's celebrated result \cite{Tu}, see \cite{Barlak_Li}, to the non-Hausdorff setting. The characterisation of nuclearity in Problem~\ref{problem:nuclearity} asks for a generalisation of \cite{Buss_Martinez}*{Theorem A} and \cite{Brix-Gonzales-Hume-Li:Hausdorff_covers}*{Theorem F} to the twisted case. If we knew  this, then we could  simplify the above proof, and we could also add twists to items 
\ref{enu:Toeplitz_nuclearity8} and \ref{enu:Toeplitz_nuclearity9}. Instead, our proof is based on  the theory of Cuntz--Pimsner algebras and results for discrete groupoids.  An important upshot of this method is the last part of the assertion, that we now reformulate as a separate corollary.
\end{remark}
\begin{corollary}\label{cor:row_finiteness_and_toeplitz_groupoids}
For a twisted self-similar action $(\Gr,E,\sigma)$ of an amenable groupoid on a row-finite graph, the $C^*$-algebraic singular ideal 
for the universal groupoid $\widetilde{\Gg}(\Gr,E)$ vanishes.
\end{corollary}
\begin{remark}\label{rem:row_finiteness_and_toeplitz_groupoids}
Without row-finiteness assumption the above corollary as well as Corollary~\ref{cor:Cuntz--Pimsner_picture}\ref{thm:Cuntz--Pimsner_picture2}
and Proposition~\ref{prop:embeddings_of_algebras}\ref{enu:embeddings_of_algebras3} fail. 
Indeed, note that in Example~\ref{ex:infinitely_many_edges} (where the graph is not row-finite) we have  $\widetilde{\Gg}(\Gr,E)=\Gg(\Gr,E)$ and this groupoid is amenable. Thus,  we have 
$$\TT^P(\Gr,E)=\TT^P_{\red}(\Gr,E)\neq \TT^P_{\ess}(\Gr,E),
$$
for any non-empty $P\subseteq [1,\infty]$.
The power of Corollary~\ref{cor:row_finiteness_and_toeplitz_groupoids} lies in that  it allows one to construct ``arbitrarily'' non-Hausdorff groupoids for which the singular ideal vanishes. The recent example from \cite{Martinez_Szakacs}*{Section 5} does not fall into this class, as  
the group acting there is non-amenable and the graph is not row-finite.
\end{remark}

Applying the isomorphism $\OO(\Gr,E,\sigma)\cong \OO(X(\Gr,E,\sigma),J_{\reg})$ is much harder, but we record some
consequences for future reference.
\begin{proposition}
Let  $(\Gr,E,\sigma)$	be a twisted self-similar action and let $*=\Space,\red, \ess$. 
Denote by $A_{*}$, $J_{*}$ and $X_*$ the images of $C^*(\Gr,\sigma_{\Gr})$, $J_{\reg}$  and $X(\Gr,E,\sigma)$ in $\OO_{*}(\Gr,E,\sigma)$, respectively. 
We have the six-term  exact sequence
$$
\begin{xy}
\xymatrix{
	K_0(J_{*}) \ar[rr]_{K_0(\iota)- K_0(X_{*})} & & K_0(A_*) \ar[rr]_{K_0(\iota)\,\,\,\,\,\,\,\,\,}  & &  \ar[d]  K_0(\OO_{*}(\Gr,E,\sigma))
	\\
	K_1(\OO_{*}(\Gr,E,\sigma)) \ar[u]  & & K_1(A_*)  \ar[ll]_{\,\,\,\,\,\,\,\,\,\,\,\,K_1(\iota)}  &  &  \ar[ll]_{K_1(\iota)- K_1(X_{*})}  K_1(J_{*})
} 
\end{xy}, 
$$
where $\iota$ stands for inclusion and $K_i(X_*)$, $i = 0,1$, is the homomorphism given by taking the Kasparov product with the Kasparov class associated to $X_*$. The $C^*$-algebra $\OO_{*}(\Gr,E,\sigma)$
is nuclear if and only if the inclusion $A_*\subseteq \OO_{*}(\Gr,E,\sigma)_0$ is nuclear.
If this holds and both $A_*$ and $J_*$ satisfy the UCT, then $\OO_{*}(\Gr,E,\sigma)$ satisfies the UCT.
\end{proposition}
\begin{proof} By Corollary~\ref{cor:Cuntz--Pimsner_picture} we may view $\OO_{*}(\Gr,E,\sigma)$ as
the relative Cuntz--Pimsner algebra $\OO(X_{*},J_{*})$. The assertion follows by applying \cite{Katsura}*{Propositions 8.7, 8.8, and Theorem 7.3}.	
\end{proof}
In the untwisted case we can get a result similar to 
Theorem~\ref{thm:Toeplitz_nuclearity} by using groupoid models and recent results of \cite{Miller_Steinberg}.
\begin{theorem}\label{thm:Cuntz--Pimsner_nuclearity}
Let $(\Gr,E)$ be a  self-similar groupoid action.
The following are equivalent:
\begin{enumerate}
\item\label{enu:Cuntz--Pimsner_nuclearity1} $\Gg(\Gr,E)_{00}$ is amenable;
\item\label{enu:Cuntz--Pimsner_nuclearity2} $\Gg(\Gr,E)$ is amenable;
\item\label{enu:Cuntz--Pimsner_nuclearity3} $\Gg(\Gr,E)_{0}$ is amenable;
\item\label{enu:Cuntz--Pimsner_nuclearity4} $\OO(\Gr,E)$ is nuclear; 
\item\label{enu:Cuntz--Pimsner_nuclearity5} $\OO(\Gr,E)_{0}$ is nuclear; 
\item\label{enu:Cuntz--Pimsner_nuclearity6} $\OO(\Gr,E)_{00}$ is nuclear;
\item\label{enu:Cuntz--Pimsner_nuclearity7} $\OO_{\red}(\Gr,E)$ is nuclear; 
\item\label{enu:Cuntz--Pimsner_nuclearity8} $\OO_{\red}(\Gr,E)_0$ is nuclear;
\item\label{enu:Cuntz--Pimsner_nuclearity9} $\OO_{\red}(\Gr,E)_{00}$ is nuclear.  
\end{enumerate}
If the above equivalent conditions hold, then   for any $*=\Space\, ,0, 00$ and non-empty $P\subseteq [1,\infty]$
\begin{equation}\label{eq:eqaulity_from_nuclearity}
\OO_{\red}^{P}(\Gr,E)_{*}=\OO^P(\Gr,E)_{*}.
\end{equation}
The above equivalent conditions always hold when $(\Gr,E)$ is contracting.
\end{theorem}
\begin{proof}
Conditions \ref{enu:Cuntz--Pimsner_nuclearity1}--\ref{enu:Cuntz--Pimsner_nuclearity3} are equivalent by \cite{Miller_Steinberg}*{Theorem 2.18},
as amenability passes to open subgroupoids, cf. also \cite{Miller_Steinberg}*{Proposition 2.17(5)}.
By Remark~\ref{remembeddings_of_algebras} for each $*=0,00$ the core subalgebras $\OO^P_{\red}(\Gr,E)_{*}$ are reduced. 
Hence, conditions \ref{enu:Cuntz--Pimsner_nuclearity1}--\ref{enu:Cuntz--Pimsner_nuclearity9} are all equivalent 
by \cite{Buss_Martinez}*{Theorem A} or \cite{Brix-Gonzales-Hume-Li:Hausdorff_covers}*{Theorem F}.
Equality \eqref{eq:eqaulity_from_nuclearity}  follows from \cite{Gardella_Lupini17}*{Theorem 6.19}, cf. Remark~\ref{rem:amenability}.
By \cite{Miller_Steinberg}*{Corollary 2.19} conditions \ref{enu:Cuntz--Pimsner_nuclearity1}--\ref{enu:Cuntz--Pimsner_nuclearity3} hold for contracting actions.
\end{proof}
\begin{remark} Recall the action $\Gr \curvearrowright \partial E$ from Remark~\ref{rem:transformation_groupoids_from_groupoid_actions}.
If the transformation groupoid $\Gr\rtimes \partial E$ is amenable, then the above equivalent conditions  \ref{enu:Cuntz--Pimsner_nuclearity1}--\ref{enu:Cuntz--Pimsner_nuclearity9} hold. When $(\Gr,E)$ is pseudo free the converse implication is also true, as then $\Gg(\Gr,E)_{00}\cong\Gr\rtimes \partial E$, see Proposition~\ref{prop:pseudo_freeness}.
When the action $(\Gr,E)$ is faithful, or even only loosely faithful, then
$\Gg(\Gr,E)_{00}$ is isomorphic to the groupoids of germs of $\Gr \curvearrowright \partial E$, see \cite{Miller_Steinberg}.
\end{remark}
\begin{remark}\label{rem:Cuntz--Pimsner_nuclearity}
Recall that Hausdorffness of any of the groupoids in  \ref{enu:Cuntz--Pimsner_nuclearity1}--\ref{enu:Cuntz--Pimsner_nuclearity3}
is equivalent to condition $\Fin$. If this holds, then the above conditions remain equivalent if we consider  twisted algebras in items \ref{enu:Cuntz--Pimsner_nuclearity4}--\ref{enu:Cuntz--Pimsner_nuclearity9}, for any twist $\sigma$ of $(\Gr,E)$, as the solution to the Problem~\ref{problem:nuclearity} in the Hausdorff case is known, cf. \cite{Takeishi} and Remark~\ref{rem:non-Hausdorff_nuclear}. 
We could also add twist \eqref{eq:eqaulity_from_nuclearity}, if Problem~\ref{problem:amenability_implies_twisted_coincide} is solved.
\end{remark}

\appendix

\section{Gauge actions, fixed-point subalgebras and 1-cocycles}\label{sec:gauge_actions}
Let $(\Gamma,+)$ be a discrete abelian group and let $\widehat{\Gamma}$ be its dual compact group. In this paper we are interested in the case where $\Gamma=\Z$ and $\widehat{\Gamma}=\T$,
but since all the arguments remain valid in this slightly more general picture we keep it for possible future reference. By an \emph{action of $\widehat{\Gamma}$ on a Banach algebra} $A$ we mean a group homomorphism $\kappa:\widehat{\Gamma}\to \Aut(A)$ into a group of isometric automorphisms on $A$ such that
for each $a\in A$, the map $\widehat{\Gamma}\ni z\mapsto \gamma_z(a)\in A$ is continuous. Then
$$
A_{t} \coloneqq \{a\in A: \kappa_{z}(a)= z(t) a\}, \qquad  t\in \Gamma,
$$
are Banach subspaces that we call \emph{spectral subspaces}. They clearly satisfy $A_t\cdot A_s\subseteq A_{s+t}$ and so in particular, 
$A_0$ is a Banach subalgebra of $A$ called the \emph{fixed-point subalgebra} of $A$. 
\begin{lemma}\label{lem:grading_from_gauge_action}
For any action $\kappa:\widehat{\Gamma}\to \Aut(A)$  on a Banach algebra $A$, we have $A=\overline{\bigoplus_{t\in \Gamma} A_t}$. That is, the spectral subspaces are linearly independent and their closed linear span is $A$. For each $t\in \Gamma$ there is a unique  contractive linear projection $E_t:A\to A_t\subseteq A$ such that $E_t(A_s)=0$ for $s\neq t$. Moreover, we have $E_t(ab)=aE_t(b)$ and $E_t(ba)=E_t(b)a$  for  $a\in A_0$, $b\in B$, $t\in \Gamma$.
\end{lemma}
\begin{proof}
Let $\mu$ be the (normalised) Haar measure on $\widehat{\Gamma}$.
For each $t$ and $a\in A$ the function $\widehat{\Gamma}\ni z \mapsto \kappa_z(a) \overline{z(t)}\in A$ is continuous. 
Since $\widehat{\Gamma}$ is compact and  $A$ is a Fr\'echet space, by \cite{Rudin}*{Theorems 3.27 and 3.20(c)} the weak (or Gelfand-Pettis) integral  $E_t(a) \coloneqq \int_{\widehat{\Gamma}} \kappa_z(a) \overline{z(t)}\, d\mu$ exists. This means that  $E_t(a)$ is a unique element in $A$ such that $f(E_t(a))=\int_{\widehat{\Gamma}} f(\kappa_z(a) \overline{z(t)})\, d\mu$ for every $f\in A'$. In particular, this implies that $\|E_t(a)\|\leq \|a\|$, $E_t$ is linear and $E_{t}|_{A_t}=\text{id}_{A_t}$. 
Routine calculations using $\widehat{\Gamma}$-invariance of  $\mu$ show that $E_{t}(A)\subseteq A_t$, $E_{t}(A_s)=0$ for $s\neq t$ (we have $\int_{\widehat{\Gamma}}  z(s)\overline{z(t)})\, d\mu(t)=0$). Hence, $E_t:A\to A_t\subseteq A$ is a contractive linear projection with desired properties. In particular,  the spaces $\{A_t\}_{t\in \Gamma}$ are linearly independent. The simple argument in the proof of \cite{Brown- Fuller-Pitts-Reznikof:Graded}*{Lemma 3.5}, that uses only Hahn-Banach theorem and injectivity of the Fourier transform, shows that $\bigoplus_{t\in \Gamma} A_t$ is dense in $A$.
This proves the first part of the assertion. The second part is straightforward in view of the integral formula for $E_t$.
\end{proof}
It is natural to call the contractive $A_0$-bimodule projection $E_1:A\to A_0$ from Lemma~\ref{lem:grading_from_gauge_action}, 
a \emph{conditional expectation}.
When $A$ is a $C^*$-algebra, then this projection is necessarily faithful (there is no nonzero ideal in $A$ contained in $\ker E_0$), cf. \cite{Brown- Fuller-Pitts-Reznikof:Graded}*{Lemma 3.3}. In the Banach algebra setting this is not automatic.

In \cite{Brown- Fuller-Pitts-Reznikof:Graded}*{Definition 2.9} the authors call a twisted groupoid (using Renault-Kumjian twists) $\Gamma$-graded
if there are two groupoid homomorphisms as consistent pair of continuous  groupoid homomorphisms, one defined on $\Gg$ the other on the twist.
However, every continuous groupoid homomorphism $c:\Gg\to \Gamma$ uniquely determines the relevant homomorphism on the twist (so this additional structure is automatic). In particular, we obtain a much simpler proof of \cite{Brown- Fuller-Pitts-Reznikof:Graded}*{Lemma 2.9}, which shows that $\Gamma$ action on the twisted groupoid induces a ``dual action'' of $\widehat{\Gamma}$ on the reduced $C^*$-algebra, and in fact we can prove it for the associated
full, reduced and essential $L^P$-operator algebras (and the proof for full algebras is nontrivial). In addition, we identify the structure of the fixed-point algebras.
\begin{lemma}\label{lem:reduced_essential_passes_to_subgroupoids}
Let $(\Gg,\LL)$ be a twisted \'etale groupoid and let $\Gg_0$ be a wide open subgroupoid of $\Gg$.
Let $F_{\RR}(\Gg,\LL)$ be a Banach algebra completion of $\mathfrak{C}_c(\Gg,\LL)$ and let $F_{\RR}(\Gg,\LL)_0$ be 
the closure of the image of $\mathfrak{C}_c(\Gg_0,\LL|_{\Gg_0})$ in $F_{\RR}(\Gg,\LL)$. 
If $F_{\RR}(\Gg,\LL)$ is a Banach algebra, a reduced Banach algebra, or an essential  Banach algebra of $(\Gg,\LL)$, then $F_{\RR}(\Gg,\LL)_0$ is a  Banach algebra, a reduced Banach algebra, or an essential  Banach algebra of $(\Gg_0,\LL|_{\Gg_0})$, respectively.
\end{lemma}
\begin{proof} 
We may assume that $\Cont_0(X)$ is a subalgebra of $F_{\RR}(\Gg,\LL)$.
If the map $\mathfrak{C}_c(\Gg,\LL)\ni f\mapsto f|_{X}\in \mathcal{B}(X)$ extends to the contractive operator $F_{\RR}(\Gg,\LL)\to \mathcal{B}(X)$,
this operator restricts to the operator
$F_{\RR}(\Gg,\LL)_0\to \mathcal{B}(X)$ showing that $F_{\RR}(\Gg,\LL)_0$ is a groupoid Banach algebra.

By \cite{BKM2}*{Remark 3.18}, $F_{\RR}(\Gg,\LL)$ is a reduced groupoid Banach algebra  if and only if the inclusion $\mathfrak{C}_c(\Gg,\LL)\subseteq \mathcal{B}(\Gg,\LL)$ extends to an injective contractive map 
$j: F_{\RR}(\Gg,\LL) \to \mathcal{B}(\Gg,\LL)$. If such map exists it restricts to the injective
map $j: F_{\RR}(\Gg,\LL)_0 \to \mathcal{B}(\Gg_0,\LL|_{\Gg_0})$ showing that 
$F_{\RR}(\Gg,\LL)_0$ is a reduced groupoid Banach algebra of $(\Gg_0,\LL|_{\Gg_0})$.

By \cite{BKM2}*{Remark 4.13}, $F_{\RR}(\Gg,\LL)$ is an  essential groupoid Banach algebra  if and only if the inclusion $\mathfrak{C}_c(\Gg,\LL)\subseteq \mathcal{B}(\Gg,\LL)$ induces  an injective contractive map 
$j: F_{\RR}(\Gg,\LL) \to \mathcal{D}(\Gg,\LL)=\mathcal{B}(\Gg,\LL)/\mathfrak{M}(\Gg,\LL)$. If such map exists it restricts to the injective
map $j: F_{\RR}(\Gg,\LL)_0 \to \mathcal{D}(\Gg_0,\LL|_{\Gg_0})$ showing that 
$F_{\RR}(\Gg,\LL)_0$ is an essential groupoid Banach algebra of $(\Gg_0,\LL|_{\Gg_0})$.
\end{proof}
\begin{theorem}\label{thm:gauge_actions_on_groupoids}
Let $(\Gg,\LL)$ be a twisted \'etale groupoid, and let $P\subseteq [1,\infty]$ be non-empty. For any continuous groupoid homomorphism $c:\Gg\to \Gamma$ the formula
\begin{equation}\label{eq:kappa_gauge_action}
	\kappa_z(f)(\gamma)=z(c(\gamma))f(\gamma), \qquad f\in \mathfrak{C}_c(\Gg,\LL),
\end{equation}
determines actions  of $\widehat{\Gamma}$ on  $F^P(\Gg , \LL)$,  $F^P_{\red}(\Gg , \LL)$,  and  $F^P_{\ess}(\Gg , \LL)$. Moreover,  $\Gg_{0} \coloneqq c^{-1}(0)$ is a wide clopen subgroupoid of $\Gg$ and 
the fixed-point subalgebras  $F^P(\Gg , \LL)_0$, $F^P_{\red}(\Gg , \LL)_0$, and  $F^P_{\ess}(\Gg , \LL)_0$
are respectively a Banach algebra, a reduced Banach algebra and an essential Banach algebra of $(\Gg_0 , \LL|_{\Gg_0})$.
The associated conditional expectations $F^P_{\red}(\Gg , \LL)\onto F^P_{\red}(\Gg , \LL)_0$ and
$F^P_{\ess}(\Gg , \LL)\onto F^P_{\ess}(\Gg , \LL)_0$ are faithful.
\end{theorem}
\begin{proof}
Firstly note that \eqref{eq:kappa_gauge_action} yields a well-defined automorphism  $\kappa_z$ of the algebra $\mathfrak{C}_c(\Gg,\LL)$.
Also, clearly $\kappa_{z_1}\circ \kappa_{z_2}=\kappa_{z_1z_2}$  for $z_1,z_2\in \widehat{\Gamma}$. Hence, $\kappa:  \widehat{\Gamma}\to  \Aut(\mathfrak{C}_c(\Gg,\LL))$
is a group homomorphism. Thus,  for the first part of the assertion it suffices to show that $\kappa_{z}$ extends (induces) an isometric automorphism on the relevant (Hausdorff) completion of  $\mathfrak{C}_c(\Gg,\LL)$.
Secondly, it suffices to consider the case when $P=\{p\}$, as then one gets the assertion by passing to appropriate direct sums.
Recall the regular representation $\Lambda_p : \mathfrak{C}_c(\Gg,\LL) \to \Bound(\ell^{p}(\Gg,\LL))$ from Example~\ref{ex:regular_representation}.
For $z\in \widehat{\Gamma}$  the multiplication operator $V_{z}\xi  (\gamma) \coloneqq  z(c(\gamma))\xi(\gamma)$, $\xi \in \ell^{p}(\Gg,\LL)$,
is an invertible isometry on  $\ell^{p}(\Gg,\LL)$. A simple calculation shows that $V_{z}\Lambda_p(f)V_{z}^{-1}=\Lambda_p(\kappa_{z}(f))$ for every $f\in \mathfrak{C}_c(\Gg,\LL)$. This implies that \eqref{eq:kappa_gauge_action} determines an isometric automorphism  $\kappa_{z}^{\red}$ of $F^p_{\red}(\Gg , \LL)$.
Since the subspace $\ell^p(\Gg_{\Hau} , \LL)\subseteq \ell^p(\Gg, \LL)$ is invariant for both  $\Lambda_p$ and $V_z$ the same reasoning shows 
that \eqref{eq:kappa_gauge_action} determines an isometric automorphism  $\kappa_{z}^{\ess}$ of $F^p_{\ess}(\Gg , \LL)$.
Hence, 
$\kappa^{\red}$ and $\kappa^{\ess}$ are the desired actions on  $F^p_{\red}(\Gg , \LL)$ and $F^p_{\ess}(\Gg , \LL)$.

For the universal algebras we need to use the disintegration-integration theorem from \cite{BKM} in a slightly stronger form than Proposition~\ref{prop:disintegration_theorem}.
Let  $S$ be the family of bisections where $\LL$ is topologically trivial. This is a wide inverse subsemigroup of $\Bis(\Gg)$.
For each $U\in S$ we fix a unitary section $c_{U} \in \Contu(U,\LL)$ which we treat as a global section of $\LL$ by letting $c_{U}$
to be zero outside $U$. For each $a\in \Cont_0(\rg(U))$ we put $a\delta_{U} \coloneqq  a* c_U$. 
Then $\mathfrak{C}_c(\Gg , \LL)=\linspan\{a\delta_U: a\in \Cont_0(\rg(U)), \,\, U\in S\}$. 
We may treat $(\Gg,\LL)$ as $(S\ltimes_{h} X,\LL_{u})$ where $h$ is the restriction of the standard action of $\Bis(G)$ on $X$ and  $u(U,V) \coloneqq c_{U}*c_{V}* c_{UV}^*$ for $U,V\in S$, cf. \cite{BKM}*{Subsection 4.3}.  We may also naturally treat $(h,u)$ as an action of the spectrum of $\B(X)$, cf. \cite{BKM}*{Page  40}.
By \cite{BKM}*{Theorem 5.19(1)}  there is a covariant representation of $(h,u)$  on an $L^p$-space $Y$
as in Definition~\ref{defn:covariant_representation_in_algebra}, except that representation $\pi:\B(X)\to \Bound(Y)$ is defined on $\B(X)$ rather than on $\Cont_0(X)$ (in fact we may assume $\pi$ acts by multiplication operators),
and such that the formula
$$
\pi\times v(a\delta_U)=\pi(a)v_U, \qquad a \in \Cont_0(\rg(U)), U\in S
$$
determines an isometric representation  $\pi\times v:F^p(\Gg , \LL)\to B(Y)$.
Note that for each $z\in \widehat{\Gamma}$ we have $z_{U} \coloneqq z\circ c\circ r|_{U}^{-1}\in \B(\rg(U))\subseteq \B(X)$, and putting 
$w_{z,U} \coloneqq \pi(z_{U})v_{U}$ for $U\in S$  one readily checks that  the pair $(\pi,w_z)$ shares the same properties as  $(\pi,v)$, i.e. it is a covariant representation of the action $(h,u)$.
Hence, it integrates  a representation $\pi\times w_z:F^p(\Gg , \LL)\to B(Y)$.  Moreover, for $a \in \Cont_0(\rg(U))$ and $U\in S$ we have
$$
\pi\times v(\kappa_z(a\delta_U))=\pi\times v((a\cdot z_{U})\delta_U))=\pi(a\cdot z_{U})v_U=\pi\times w_z(a\delta_U). 
$$
This, and the fact that $\pi\times v$ is isometric, imply that $\kappa_z \coloneqq (\pi\times v)^{-1}\circ \pi\times w_z$ is a well-defined contractive homomorphism
$\kappa_z:F^p(\Gg , \LL)\to F^p(\Gg , \LL)$ satisfying \eqref{eq:kappa_gauge_action}.  Since $\kappa_{z^{-1}}$ is a contractive inverse of $\kappa_{z}$,  we see that  $\kappa_z$ is an isometric automorphism of  $F^p(\Gg , \LL)$. This finishes the proof of the first part of the assertion.

Let $E_0:F^P(\Gg , \LL)\to F^P(\Gg , \LL)_0$ be the associated conditional expectation.  Note that $\mathfrak{C}_c(\Gg_0,\LL|_{\Gg_0})$ is naturally a subalgebra of $\mathfrak{C}_c(\Gg,\LL)$ and $E_0$ restricts to a projection $E_0:\mathfrak{C}_c(\Gg,\LL)\to \mathfrak{C}_c(\Gg_0,\LL|_{\Gg_0})$, which is given by 
$E_0(f)=f|_{\Gg_0}$, for $f\in \mathfrak{C}_c(\Gg,\LL)$. In particular, $\mathfrak{C}_c(\Gg_0,\LL|_{\Gg_0})\subseteq F^P(\Gg , \LL)_0$.
To see that $\mathfrak{C}_c(\Gg_0,\LL|_{\Gg_0})$ is dense in $F^P(\Gg , \LL)_0$ take any $f\in F^P(\Gg , \LL)_0$ and choose a sequence $\{f_n\}_{n=1}^{\infty}\subseteq \mathfrak{C}_c(\Gg,\LL)$ which converges in norm to $f$. Then $\{E_0(f_n)\}_{n=1}^{\infty}\subseteq \mathfrak{C}_c(\Gg_0,\LL|_{\Gg_0})$ converges to $E_0(f)=f$.
Hence, $F^P(\Gg , \LL)_0$ is a closure of the image of $\mathfrak{C}_c(\Gg_0,\LL|_{\Gg_0})$.
Similar arguments show that $F^P_{\red}(\Gg,\LL)_0$ and $F^P_{\ess}(\Gg,\LL)_0$ are closures of images of  $\mathfrak{C}_c(\Gg_0,\LL|_{\Gg_0})$. 
Hence, the middle part of the assertion follows from Lemma~\ref{lem:reduced_essential_passes_to_subgroupoids}.

Finally, notice that composing the conditional expectation $E_{0}^{\red}:F^P_{\red}(\Gg , \LL)\onto F^P_{\red}(\Gg , \LL)_0$ with the canonical generalised expectation
$F^P_{\red}(\Gg , \LL)_0\onto \B(X)$ coincides with the associated faithful map $F^P_{\red}(\Gg , \LL)\to \B(X)$. Hence, $E_{0}^{\red}$
is faithful. Similarly, $E_{0}^{\ess}:F^P_{\ess}(\Gg , \LL)\onto F^P_{\ess}(\Gg , \LL)_0$ composed with the canonical map
$F^P_{\ess}(\Gg , \LL)_0\onto \mathcal{D}(X)$ coincides with the canonical faithful map $F^P_{\ess}(\Gg , \LL)\to \mathcal{D}(X)$. Hence, $E_{0}^{\ess}$
is faithful. 
\end{proof}
The actions of $\widehat{\Gamma}$ described in Theorem~\ref{thm:gauge_actions_on_groupoids} in the context of algebras defined in terms of generators and relations are often called \emph{gauge-actions}.  Therefore, this name is even more justified when applied to the  inverse semigroups algebras that we defined in  the previous subsection.  Using Lemma~\ref{lem:from_cocycles_to_homomorphisms} we can translate the above result to this context.

\begin{definition}\label{def:gauge_action_semigroup_representations}
Let $(S,\omega)$ be a twisted inverse semigroup equipped with a $1$-cocycle $c:S\setminus \{0\}\to \Gamma$ 
with values in an abelian group $\Gamma$ (so we have $c(st)=c(s)+ c(t)$ whenever $st\neq 0$). 
We say that a representation $v:S\to B(E)$ of  $(S,\omega)$ \emph{admits a gauge action} induced by $c$  if
$$
\kappa_z(v_t)=z(c(t)) v_t, \qquad z\in \widehat{\Gamma}, t\in S,
$$
determines an action of $\widehat{\Gamma}$ on the range $B(v)= \clsp\{v_t : \, t\in S \}$ of $v$.
\end{definition}
\begin{corollary}\label{cor:gauge_action_inverse_semigroup}
Let $(S,\omega)$ be a twisted inverse semigroup,  $c:S\setminus \{0\}\to \Gamma$ a  $1$-cocycle,  $\emptyset \neq P\subseteq [1,\infty]$ and let $*=\Space,\red,\ess$.
Representations generating $\TT^P_{*}(S)$ and $\OO^P_{*}(S)$ admit gauge actions induced by $c$. Moreover,
\begin{enumerate}
	\item The associated
	fixed-point subalgebras $\TT^P_{*}(S)_0$ and $\OO^P_{*}(S)_0$ are ranges of restrictions of the corresponding representations
	to the inverse subsemigroup  $S_0 \coloneqq  c^{-1}(0)\cup \{0\}\subseteq S$.
	
	\item  
	The algebras $\TT^P(S)_0$, $\TT^P_{\red}(S)_0$ and $\TT^P_{\ess}(S)_0$  are respectively, a Banach algebra, a reduced Banach algebra, and an essential Banach algebra  of the groupoid $\widetilde{\Gg}(S)_0 \coloneqq S_0 \ltimes_{\widetilde{h}} \widehat{\EE}\subseteq\widetilde{\Gg}(S)$.
	\item The algebras $\OO^P(S)_0$, $\OO^P_{\red}(S)_0$ and $\OO^P_{\ess}(S)_0$  are respectively, a Banach algebra, a reduced Banach algebra, and an essential Banach algebra  of $\Gg(S)_0 \coloneqq S_0 \ltimes_{h} \partial \EE\subseteq \Gg(S)$.
\end{enumerate}
\end{corollary}
\begin{proof} Apply 
Theorem~\ref{thm:gauge_actions_on_groupoids} to the groupoid models in Corollary~\ref{cor:groupoid_presentation_inverse_semigroup_algebras}
equipped with the associated groupoid homomorphism to $c$ via  Lemma~\ref{lem:from_cocycles_to_homomorphisms}.
\end{proof}

\end{document}